\documentclass[12pt,draft]{amsart}
\usepackage{amsmath,amsthm,amsxtra,amssymb,mathrsfs,enumerate, verbatim,upref,eepic}
\usepackage[T1]{fontenc}
\usepackage{lmodern}
\usepackage{textcomp}
\usepackage{mathtools}
\usepackage{stmaryrd}
\usepackage{dsfont}
\newfont{\wncyr}{wncyr10 scaled \magstep0}
\usepackage{calligra}

\AtBeginDocument{%
\ensuremath{
\fontdimen13 \textfont 2=.43em   
\fontdimen13 \scriptfont 2=0.26em 
\fontdimen17 \textfont 2=.3em 
\fontdimen17 \scriptfont 2= .25em  
\fontdimen14\textfont2= \fontdimen13\textfont2 
\fontdimen14 \scriptfont 2= \fontdimen13 \scriptfont 2
\fontdimen14 \scriptscriptfont 2= \fontdimen13 \scriptscriptfont 2
\fontdimen16\textfont2= \fontdimen17\textfont2 
\fontdimen16 \scriptfont 2= \fontdimen17 \scriptfont 2 
\fontdimen16 \scriptscriptfont 2= \fontdimen17 \scriptscriptfont 2
}}

\numberwithin{equation}{section}
\usepackage[dvips]{graphicx}

\swapnumbers
\theoremstyle{definition}
\newtheorem{defn}{Definition} [section]
\newtheorem{rem}[defn]{Remark}

\newtheorem{ex}[defn]{Example}
\theoremstyle{plain}
\newtheorem{thm}[defn]{Theorem}
\newtheorem{prop}[defn]{Proposition}
\newtheorem{cor}[defn]{Corollary}
\newtheorem{lem}[defn]{Lemma}
\newcommand{\Ker}{\operatorname{Ker}}
\newcommand{\OOO}{\mathscr{O}^{\raise0.4ex\hbox{${\scriptstyle  (0,n,0)}$}}_{\smash{\widehat{X}^2}}}
\newcommand{\ms}[1]{\raise0.3ex\hbox{${\scriptstyle  #1}$}}
\newcommand{\Oo}{\mathscr{O}^{\ms{(0,n)}}_{\smash{X^2}}}
\newcommand{\WX}{\smash{\widehat{X}}}
\newcommand{\WXX}{\smash{\widehat{X}^2}}
\newcommand{\Hom}{\mathscr{H} \hspace{-0.41em}\text{\large \calligra om}\,}
\newcommand{\RHom}{\boldsymbol{R}\Hom}
\newcommand{\tens}{\mathop{\otimes}\limits}
\newcommand{\Sum}{\mathop{\textstyle\sum}\limits}
\newcommand{\Bcup}{\mathop{\textstyle\bigcup}\limits}
\newcommand{\Bcap}{\mathop{\textstyle\bigcap}\limits}

\newcommand{\iim}{\mathop{\scriptscriptstyle \sqrt{-1\,}}}
\renewcommand{\atop}[2]{\genfrac{}{}{0pt}{}{#1}{#2}}
\newcommand{\cset}{\,\smashoperator{\mathop{\Subset}_{\rm conic}}\,}
\newcommand{\Cl}{\operatorname{Cl}}
\newcommand{\wick}[1]{\textrm{\textup{{\Large :}}}#1 \textrm{\textup{\Large :}}}

\usepackage[all]{xy}
\newdir{ >}{{}*!/-5pt/@{>}}
\UseComputerModernTips 
\newcommand{\dsim}{\text{\smash{\lower0.9ex
\hbox{\normalsize $\sim$}}}}

\newcommand{\abs}[1]{\lvert#1 \rvert}
\newcommand{\Abs}[1]{\lVert#1 \rVert}
\renewcommand{\Re}{\operatorname{Re}}
\renewcommand{\Im}{\operatorname{Im}}

\newcommand{\im}{\mathop{\text{\footnotesize  $\sqrt{-1\,}$}}}
\newcommand{\earrow}
{\mathrel{\text{\smash{\lower.477ex\hbox{$\overset{\text{\normalsize \smash{\lower.55ex\hbox{$\sim\,$}}}}{\to}$}}}}}
\newcommand{\eleftarrow}
{\mathrel{\text{\smash{\lower.477ex\hbox{$\overset{\text{\normalsize \smash{\lower.55ex\hbox{$\,\sim$}}}}{\leftarrow}$}}}}}
\makeatletter
\def\@settitle{\begin{center}%
  \LARGE
    \bfseries
  \@title
  \end{center}}
\def\@setauthors{%
  \begingroup
  \trivlist
  \centering \fontsize{11}{13\p@}\selectfont\@topsep13\p@\relax \large
  \item\relax
  \andify\authors
  \def\\{\protect\linebreak}%
\authors
  \endtrivlist
  \endgroup}
\makeatother

\begin{document}
\title[Analytic Pseudodifferential Operators. I]{Foundation of Symbol Theory\\ for  Analytic Pseudodifferential Operators. I}
\author[T. Aoki]{Takashi \textsc{Aoki}}
\address[T. Aoki]{Department of Mathematics, Kinki University,
Higashi-Osaka 577--8502, Japan.}
\author[N. Honda]{Naofumi \textsc{Honda}}
\address[N. Honda]{Department of Mathematics, 
Faculty of Science, 
Hokkaido University, Sapporo 060--0810, Japan.}

\author[S. Yamazaki]{Susumu \textsc{Yamazaki}} 
\address[S. Yamazaki]{Department of General Education, College of Science and Technology, 
Nihon  University, Chiba 274--8501, Japan.}

\begin{abstract}
A new symbol theory for pseudodifferential operators in the complex analytic category is given.
This theory provides a cohomological foundation of symbolic calculus.
\end{abstract}
\thanks{This work was supported by JSPS KAKENHI Grant Numbers 22540210, 23540178 and 20540191.}
\maketitle
\baselineskip=1\baselineskip plus 0.5pt minus 0.5pt

\section*{Introduction}
The aim of this series of papers is to establish a complete symbol theory for the sheaf ${\mathscr E}^\mathbb R_X$
of pseudodifferential operators in the complex analytic category. 
Here we distinguish a little difference
between the usage of hyphenation in the words ``pseudodifferential''  and ``pseudo-differential.''   The latter 
might be
more familiar than the former for
most of readers. To clarify this distinction,
we have to mention some of history. 

The notion of the pseudo-differential operators in the analytic category was introduced by 
Boutet de Monvel and Kre\'e \cite{bk} and by Boutet de Monvel \cite{b} 
for the real domain and by Sato, Kawai and Kashiwara \cite{SKK} for the complex domain about forty
years ago. Note that \cite{bk} introduced the notion in the category of ultradifferentiable functions
of Gevrey class which contains 
the analytic category for a special case and treated 
operators of finite order. On the other hand, \cite{b} and \cite{SKK} considered operators of infinite order
and these operators play an essential role in \cite{SKK} and in Kashiwara and Kawai \cite{K-K1}.
The notion is effectively used not only in the analysis of differential equations in analytic
category but in various fields of mathematics and now it becomes one of the most basic tools in analysis
as well as pseudo-differential operators in $C^\infty$ category. There are a number of references
which use pseudo-differential operators and we can not cite all of them. We only cite Bj\"ork \cite{bj},
H\"ormander \cite{h}, Kashiwara and Schapira \cite{K-S2},
Kumano-go \cite{ku}, Liess \cite{Li}, Schapira \cite{sch} 
and Tr\`eves \cite{TR}, and refer the readers to the references cited in those books or manuscripts.

The definition of the pseudo-differential operators given in \cite{b} 
used oscillatory integrals and analytic symbols, 
while \cite{SKK} employed the cohomology theory. One of the advantages of the latter theory is invariance which
comes from the cohomology theory. But most of analysts are familiar with the former theory because
it is comprehensible as an analogy of pseudo-differential operators in $C^\infty$ category
through symbol theory and it does not require heavy algebraic tools such as derived categories and
exact sequences. Symbol theory for pseudo-differential operators was also developed in  \cite{SKK},
where the sheaf of them was denoted by $\mathscr{P}_X$. This sheaf
is recently denoted by $\mathscr E^\infty_X$ after the work  of Kashiwara and Schapira \cite{ks} and  called the sheaf of 
microdifferential operators (of finite or infinite order). Note that it is a subsheaf of 
 ${\mathscr E}^\mathbb R_X$. The notion of symbols defined in  \cite{SKK} is different from
that of  \cite{b}. A symbol of a microdifferential operator (or a pseudo-differential operator 
in the sense of \cite{SKK}) is a sequence with an index in $\mathbb Z$ of holomorphic
functions defined on the cotangent bundle $T^*X $ with some homogeneity and growth
conditions. On the other hand, a pseudo-differential operator in the sense of  \cite{b} is defined
by a total symbol $p(x,\xi)$ which is real analytic in $(x,\xi)$ satisfying a growth condition in $\xi$
variables. The relation between those two theories was clarified by Kataoka \cite{kat}. 
He defined symbols of operators in $\mathscr E^{\mathbb R}_X$ 
by using the Radon transform and through his theory, we knew that
pseudo-differential operators of \cite{b} is obtained by restriction of $\mathscr E^{\mathbb R}_X$ to
the real domain.  

 The essential idea of the definition of ${\mathscr E}^\mathbb R_X$ was given
 in \cite{SKK} but the definition itself was not given there explicitly. The definition first appeared in 
 the work of Kashiwara and Kawai \cite{K-K},
 where the notation ${\mathscr P}^{\mathbb R}_X$ was used, although the name of the sections of the
 sheaf was not given. As well as the case of $\mathscr E^\infty_X$, we use the notation 
${\mathscr E}^{\mathbb R}_X$  instead of ${\mathscr P}^{\mathbb R}_X$ after \cite{ks} and 
we call the sections of ${\mathscr E}^{\mathbb R}_X$ pseudodifferential operators
after \cite{a2}.

Since the symbol theory developed in \cite{kat} was not published, some parts of it were supplied by
the first author \cite{a2} and the theory played a role in the analysis of operators of infinite order 
(cf.\ Aoki \cite{a1}, \cite{a3}, Aoki, Kawai, Koike and Takei \cite{akkt}, Aoki, Kawai and Takei \cite{akt},
Kajitani and Wakabayashi \cite{k-w}, Kataoka \cite{kat2}, Uchikoshi \cite{u1}). 
Some systematic description of the theory has been
included in the book of Aoki, Kataoka and Yamazaki \cite{aky}. 
The foundation of the symbol theory of ${\mathscr E}^{\mathbb R}_X$ 
at the present stage is, however, quite unsatisfactory.
There are two issues: first one is that, as Kamimoto and Kataoka have pointed out in 
their work \cite[Example 1.1]{kaka1}, the space of
the kernel functions which comes from standard \v{C}ech representation of cohomology groups is
not closed under composition of kernel functions defined by naive integration employed in \cite{a2}, \cite{aky}. 
Regarding this issue, \cite{kaka1} gives a possible solution
by introducing the notion of formal kernels. Second issue is that
the relation between the action of operators by integration of kernel functions and
canonical action through cohomological definition was not clarified. 
We note that the notion of formal kernels given in \cite{kaka1} has not yet given a solution to this issue
directly.
Thus we think we have to provide a complete symbol theory of ${\mathscr E}^{\mathbb{R}}_X$
which solves these issues.

We mention that some modifications of the symbol theory are given
by Uchikoshi \cite{u2} and by Ishimura \cite{I} for microlocal operators and non-local
operators in the analytic category, respectively. But there are analogous issues in these theories.

In this series of papers, we establish a new symbol theory of ${\mathscr E}^{\mathbb R}_X$
which completely fits in the cohomological definition of the sheaf. In the first part, we present
a foundation of symbol theory for ${\mathscr E}^{\mathbb R}_X$.
Our main idea is to introduce
a redundant parameter, which we call an apparent parameter, 
in the definitions of (real) holomorphic microfunctions and symbols.
By introducing this parameter, cohomological definition of operation such as
composition, formal adjoint, coordinate transformation, etc.\  are directly related to
those of symbols (see Kashiwara-Schapira \cite{K-S1}, \cite{K-Sa}). To clarify the relation between \v{C}ech cohomology classes and symbols,
we fully use the theory of the action of microdifferential operators on microfunctions established
by Kashiwara and Kawai (\cite{K-K}, \cite{K-K1}).
We also develop a theory of formal symbols which was firstly introduced for operators of infinite order by \cite{b}
and generalized by \cite{a1}, \cite{a2} and by Laurent \cite{L}. 
The formal symbol theory established in this article is exactly based on the cohomological
definition of ${\mathscr E}^{\mathbb R}_X$. To develop this theory, we employ
an idea introduced by \cite{L}.
Our forthcoming second paper will be devoted to the symbol theory
for operators with Gevrey growth and the cohomology theory for Whitney holomorphic
functions. It will be useful for applications.

The plan of this paper, the first part,  is as follows. In Section \ref{sec:coh}, we prepare a proposition of the
local cohomology group theory on a vector space which we shall use in this article.
Section \ref{sec:rep-micro} gives a new formulation of the sheaf of (real) holomorphic microfunctions
utilizing an apparent parameter.
Applying this formulation, we give a cohomological representation of
pseudodifferential operators in Section \ref{sec:actions_ER}. In Section \ref{sec:symbol-with-ap}, we define symbol spaces
with an apparent parameter. The relation between symbols and cohomological
representation of pseudodifferential operators is clarified in Section \ref{sec:kernel-symbol}.
Sections \ref{sec:cl-formal-symbol} and \ref{sec:formal-symbol} are devoted to establish a theory 
of formal symbols with an apparent parameter for
pseudodifferential operators. 
Logically we can skip these two sections since our theory developed up to Section 
\ref{sec:kernel-symbol} provides the equivalence of our present theory and the
symbol theory given in \cite{a2}, \cite{aky} without cohomology theoretical foundation.
But we think we need these two sections since the relation between formal symbols 
and kernel functions can be understood  directly from the viewpoint of these sections.
Of course, they provides a unified
treatment of symbols of pseudodifferential operators with previous sections.
Basic algebraic operations in terms of symbols are given
with cohomology theoretical foundation in these sections.

In Appendix \ref{ap:actions}, we confirm the compatibility
of actions of pseudodifferential operators on the sheaf of holomorphic microfunctions.
Appendix \ref{ap:general_C} gives a general construction of the sheaf of microfunctions 
which can manage the symbol mapping on the space of kernel functions with respect to arbitrary coverings.

This work, especially, the idea of introducing redundant parameter, 
is inspired by \cite{kaka1}. The authors would like to express their sincere thanks to
Professor K.\ Kataoka and Dr.\ S.\ Kamimoto. They also thank Professor T.\  Kawai and
Professor Y.\ Okada for encouragement to them.

\section{Local Cohomology Groups on a Vector Space}\label{sec:coh}

We denote by $\mathbb{Z}$, $\mathbb{R}$ and $\mathbb{C}$ the sets of  integers, of real numbers  and of complex numbers  respectively. 
Further, set 
 $\mathbb{N}:=\{m\in \mathbb{Z};\,m >  0\}$, $\mathbb{N}_0^{}:=\mathbb{N}\cup\{0\}$ and 
$\mathbb{C}^{\times}:=\{c\in \mathbb{C};\,c\ne 0\}$.  

Let $X$ be a finite dimensional $\mathbb{R}$-vector space,  and define  an open proper sector $ S \subset \mathbb{C}$   by
\[
S:=\{\eta \in \mathbb{C};\, a < \arg \eta < b,\, 0 < |\eta| < r\} 
\]
for some $0 < b-a < \pi$ and $r > 0$.
We set $\widehat{X} := X \times \mathbb{C}_\eta$
with  coordinates $(x,\eta)$, and let   $\pi_\eta \colon  \widehat{X} 
\ni (x,\eta)\mapsto x\in  X$ be  the canonical projection.
Let $G \subset X$ be a closed subset (not necessarily convex) and
 $U \subset X$ an open neighborhood of the origin.
In this section we give
another representation of 
local cohomology groups $H^k_{G\cap U}(U; \mathcal{F})$ for a sheaf $\mathcal{F}$
on $X$. For this purpose, we need some preparations. 
Let $Z$ be a closed subset in $X$ and
$\varphi \colon  X \times [0,1] \to X$  a continuous deformation mapping  
which satisfies the following conditions:
\begin{enumerate}[(i)]
	\item $\varphi(x,1) = x$ for any $x \in X$
and $\varphi(z,s) = z$ for any $z \in Z$.
\item $\varphi(\varphi(x,s),0) = \varphi(x,0)$ for any $s \in [0,1]$ and $x \in X$.
\item We set 
\[
\rho^{}_\varphi(x,s) := | \varphi(x,s) - \varphi(x,0) |.
\]
Then  $\rho^{}_\varphi(x,s)$ is a strictly increasing function of $s$
outside $Z$, i.e.\ if $s^{}_1 < s^{}_2$, we have   $\rho^{}_\varphi(x,s^{}_1) <  \rho^{}_\varphi(x,s^{}_2)$ for any $x \in X \smallsetminus Z$.
\end{enumerate}
We set, for short
\[
\rho^{}_\varphi(x) := \rho^{}_\varphi(x,1) = | \varphi(x,1) - \varphi(x,0) |
= | x - \varphi(x,0) |.
\]
Here we remark
\[
\rho^{}_\varphi(\varphi(x,s)) = | \varphi(x,s) - \varphi(\varphi(x,s),0) |=| \varphi(x,s) - \varphi(x,0) | =\rho^{}_\varphi(x,s).
\]
Let us see a typical example of such a deformation mapping.
\begin{ex}{\label{es:typical-deformation}}
	Let $\zeta$ be a unit vector in $X:= \mathbb{C}^n$ and $Z = \{x \in X;\, \langle x,\zeta\rangle = 0\}$
with $\langle x,\zeta\rangle := \smashoperator{\Sum_{i=1}^n} x^{}_i \zeta^{}_i$.
Define the deformation mapping
$\varphi \colon X \times [0,1] \to X$ by
\[
\varphi(x,s) := x + (s-1)\langle x,\zeta\rangle \overline{\zeta}.
\]
Here $\overline{\zeta}$ denotes the complex conjugate of $\zeta$.
Note that $\varphi(x,1) = x$ and 
$\varphi(x,0)$ gives the orthogonal projection to the
complex hyperplane $Z$ with respect to
the standard Hermitian metric  $|x|= \langle x,\overline{x}\rangle^{1/2}$.  
\end{ex}

Let $\varrho > 0$ a positive constant.
We define the subsets in $\widehat{X}$ by
\begin{equation}\label{eq:def-hat-UV-general}
\begin{aligned}
\widehat{G} &:= \{(\varphi(x,s),\eta) \in \widehat{X};\, \rho^{}_\varphi(x) 
\leqslant \varrho|\eta|,\, 0 \leqslant s \leqslant 1,\,
x \in G\}, \\
\widehat{U} &:= 
\{(x,\eta) \in  U \times  S;\, \rho^{}_\varphi(x) <\varrho |\eta|\}.
\end{aligned}
\end{equation}
Note that $\widehat{G} \cap \widehat{U}$ is a closed subset in $\widehat{U}$.
Then we have the following proposition.
\begin{prop}\label{prop:iso-cohomology-groups-for-general}
Let $\mathcal{F}$ be a complex of Abelian sheaves  on $X$.  
Assume that $U$ satisfies 
\begin{equation}\label{conrho}
\smashoperator{\sup_{x \in U}} \rho^{}_\varphi(x) <  \varrho r.
\end{equation}
Then there exists  the following isomorphism\textup{:}
\[
	\boldsymbol{R}\varGamma^{}_{G \cap U}(U; \mathcal{F})
\earrow 
\boldsymbol{R}\varGamma_{\widehat{G} \cap \widehat{U}}(\widehat{U};\pi_\eta^{-1}\!\mathcal{F}).
\]
\end{prop}
\begin{proof}
Since  $  \pi_{\eta}^{-1}(G)\cap\widehat{U}$ is  closed in $\widehat{G} \cap\widehat{U} $ and  $\widehat{U}$ is   open in $\pi_\eta^{-1}(U)$, 
we obtain 
\[
\mathbb{Z}_{\widehat{G}\cap\widehat{U}  }
\to
\mathbb{Z}_{\pi^{-1}_{\eta}(G)\cap\widehat{U}  }
\to
\mathbb{Z}_{\pi^{-1}_{\eta}(G)\cap \pi^{-1}_{\eta}(U) }
=
\pi_\eta^{-1}\mathbb{Z}^{}_{G\cap U  } 
=
\pi_\eta^{!}\mathbb{Z}^{}_{G\cap U} [-2],
\]
and this induces the canonical morphism
\begin{equation}\label{eq.1.3}
\boldsymbol{R}\pi^{}_{\eta !} \mathbb{Z}_{\widehat{G}\cap\widehat{U}  }
\to \mathbb{Z}^{}_{G\cap U}[-2].
\end{equation}
Hence, we have
\allowdisplaybreaks
\begin{align*}
\boldsymbol{R}\varGamma_{G \cap U}(U; \mathcal{F})&
\simeq 
\boldsymbol{R}\textrm{Hom}_{\mathbb{Z}^{}_X}\!(\mathbb{Z}^{}_{G\cap U},\mathcal{F})
\to \boldsymbol{R}\textrm{Hom}_{\mathbb{Z}_X}\!
(\boldsymbol{R}\pi^{}_{\eta !}\mathbb{Z}_{\widehat{G}\cap\widehat{U}}, \mathcal{F})[-2]
\\
&
\simeq
\boldsymbol{R}\textrm{Hom}_{\mathbb{Z}_{\WX}}
(\mathbb{Z}_{\widehat{G}\cap\widehat{U}},\pi_\eta^!\mathcal{F})[-2] 
\simeq
\boldsymbol{R}\varGamma_{\widehat{G} \cap \widehat{U}}(\widehat{U};\pi_\eta^{-1}\! \mathcal{F}),
\end{align*}
and to show Proposition \ref{prop:iso-cohomology-groups-for-general},  it suffices to prove that \eqref{eq.1.3}
is an isomorphism. 
We first give some properties of $\varphi$ and $\rho^{}_\varphi$.

\begin{enumerate}[(1)]
\item If $\varphi(x,s) = \varphi(x',s')$ holds, we have
$ \varphi(x,0)=\varphi(\varphi(x,s),0) =\varphi(\varphi(x',s'),0)  = \varphi(x',0)$. 
In particular, $\rho^{}_\varphi(x,s) = \rho^{}_\varphi(x',s')$.
\item For any  $x^* \in X$ we set
\[
G(x^*) := \{(g,t) \in G \times [0,1];\,\varphi(g, t) = x^*\}.
\]
If $G(x^*)\ne \emptyset$, there exists $(x,s) \in G(x^*)$ such that  $\rho^{}_\varphi(x)$ attains the value
\[
a(x^*) := \inf\{ \rho_\varphi(g);\,(g,t) \in G(x^*)\}.
\]
\end{enumerate}
Let us compute $\boldsymbol{R}\pi_{\eta !}\mathbb{Z}_{\widehat{G}\cap\widehat{U}}$\,. 
If $x^* \notin U$, clearly we have
$(\boldsymbol{R}\pi_{\eta !}\mathbb{Z}_{\widehat{G}\cap\widehat{U}})_{x^*} = 0$. 
Hence in what follows, we assume $x^* \in U$,  in particular,
$\pi_\eta^{-1}(x^*) \cap \widehat{U} \ne \emptyset$ holds thanks to the assumption. Then
it follows from the definition and the properties above that if $\pi_\eta^{-1}(x^*) \cap \widehat{G}\ne \emptyset$,
\[
\pi_\eta^{-1}(x^*) \cap \widehat{G}
\simeq
\{\eta \in \mathbb{C};\, a(x^*) \leqslant \varrho|\eta| \}. 
\]
We also have
\[
\pi_\eta^{-1}(x^*) \cap \widehat{U} 
\simeq
\{\eta \in  S;\, \rho^{}_\varphi(x^*) <  \varrho|\eta|\}.
\]
By taking these observations into account, we can calculate
$(\boldsymbol{R}\pi_{\eta !}\mathbb{Z}_{\widehat{G}\cap\widehat{U}})_{x^*}$ for $x^* \in U$
as follows. If $x^* \in G$, we see that the subset
$\pi_\eta^{-1}(x^*) \cap \widehat{G}\ne \emptyset $ and we get
\[
\pi_\eta^{-1}(x^*) \cap \widehat{G}
\simeq
\{\eta \in \mathbb{C};\,\rho^{}_\varphi(x^*) \leqslant \varrho|\eta|\} 
\]
because of $(x^*,1) \in G(x^*)$ and 
\[
\rho^{}_\varphi(x^*) = \rho^{}_\varphi(x^*,1) =
\rho^{}_\varphi(x,s) \leqslant \rho^{}_\varphi(x, 1) = \rho^{}_\varphi(x)
\]
for any $(x,s) \in G(x^*)$.
Hence we have
\[
\pi_\eta^{-1}(x^*) \cap \widehat{G}\cap\widehat{U}
\simeq
\{\eta \in  S;\, \rho^{}_\varphi(x^*) < \varrho|\eta| \},
\]
which implies
\[
(\boldsymbol{R}\pi^{}_{\eta !}\mathbb{Z}_{\widehat{G}\cap\widehat{U}})_{x^*} =
\boldsymbol{R}\varGamma_{\rm c}(\pi_\eta^{-1}(x^*) \cap \widehat{G}\cap\widehat{U}; \mathbb{Z}_{\WX}) = 
\mathbb{Z}[-2].
\]
On the other hand, if $x^* \notin G$, we obtain
\[
(\boldsymbol{R}\pi^{}_{\eta !}\mathbb{Z}_{\widehat{G}\cap\widehat{U}})_{x^*} =
\boldsymbol{R}\varGamma_{\rm c}(\pi_\eta^{-1}(x^*) \cap \widehat{G}\cap\widehat{U}; \mathbb{Z}^{}_{\WX}) =  0.
\]
As a matter of fact, if $\pi_\eta^{-1}(x^*) \cap \widehat{G}= \emptyset $, 
the claim clearly holds. Otherwise, we have
$
\rho^{}_\varphi(x^*) < a(x^*)
$
which can be shown by the following argument.
Let $(x,s)$ be a point in $G(x^*)$ with $\rho_\varphi(x) = a(x^*)$.
Since $x^* \notin G$, $x \in G$ 
and $x^* = \varphi(x,\,s)$, we have $x \notin Z$ and $s < 1$. 
From these facts, 
\[
\rho^{}_\varphi(x^*) = \rho^{}_\varphi(x^*,1) = \rho^{}_\varphi(x,s) < 
\rho^{}_\varphi(x,1) = \rho^{}_\varphi(x) = a(x^*)
\]
follows. Hence we have
\[
\pi_\eta^{-1}(x^*) \cap \widehat{G}\cap\widehat{U}
\simeq
\{\eta \in  S;\, a(x^*) \leqslant  \varrho|\eta| \},
\]
which implies the claim.

Summing up, we have obtained
\[
(\boldsymbol{R}\pi^{}_{\eta !}\mathbb{Z}_{\widehat{G}\cap\widehat{U}})_{x^*} =
\left\{
\begin{aligned}
	&\mathbb{Z}[-2]& (x^* \in G\cap U),\\
	&0 & \text{(otherwise)},
\end{aligned}
\right.
\]
hence \eqref{eq.1.3} 
is an isomorphism. This completes the proof.
\end{proof}

\begin{rem}\label{rem1.3}
Without  \eqref{conrho}, we have the following claim by
the same argument as that in the proof above: Set $U':=\{x \in U;\,\rho_\varphi(x) <  \varrho r\}$. Then there exists the canonical
isomorphism
\[
\boldsymbol{R}\varGamma_{G \cap U'}
(U'; \mathcal{F})
\earrow \boldsymbol{R}\varGamma_{\widehat{G} \cap \widehat{U}}
(\widehat{U};\pi_\eta^{-1}\mathcal{F}).
\]
\end{rem}

%

\section{Holomorphic Microfunctions with an Apparent  Parameter}\label{sec:rep-micro}

Let $X$ be an $n$-dimensional $\mathbb{C}$-vector space 
with the coordinates $z=(z_1,\dots,z_n)$, and
$Y$  the closed complex submanifold of $X$ defined by $\{z'=0\}$ where $z = (z',z'')$ with $z':=(z_1,\dots,z_d)$ for 
some $1 \leqslant d \leqslant n$. Set $\widehat{ X}:= X\times \mathbb{C}$, and
let $\pi_\eta \colon \widehat{ X} \ni (z, \eta) \mapsto z\in  X$ be  the canonical projection as in Section \ref{sec:coh}. 
In what follows, we denote an object defined on the space $\widehat{X}$ by
a symbol with $\,\widehat{\cdot}\,$ like $\widehat{U}_{\boldsymbol{\kappa}}$ etc. 
For any $z\in \mathbb{C}^n$, we set  $\Abs{z}:= \smashoperator{\max_{1 \leqslant i \leqslant n}}\{\abs{z^{}_i}\}$. 
Let $\mathscr{O}^{}_X$ be the sheaf of \textit{holomorphic functions} on $X$,  and 
 $\mathscr{C}^{\mathbb{R}}_{Y|X}$  the sheaf of \textit{real  holomorphic microfunctions} along $Y$ 
on the conormal bundle $T^*_YX$ to $Y$. 
Let $z_0 = (0, z''_0) \in Y$ 
and $z^*_0=(z''_0;\zeta'_0) \in T^*_YX$ with $|\zeta'_0| = 1$. 
Set
\[
f^{}_1(z) := \langle z',\zeta'_0\rangle, \quad
f'(z) := z' - \langle z',\zeta'_0\rangle \overline{\zeta}{}'_0\,.
\]
\begin{rem}
The subsequent arguments can be applied to a general family of  a function $f_1(z)$ and a mapping $f'(z)$, that
enables us to develop the theory not only on a vector space but also on a complex manifold.
It is, however, rather technical. Hence we put such a generalization in Appendix \ref{ap:general_C}.
\end{rem}
By the definition of $\mathscr{C}^{\mathbb{R}}_{Y|X}$,
we have
\[
\mathscr{C}^{\mathbb{R}}_{Y|X,z^*_0}
=
\smashoperator[r]{\varinjlim_{\varrho,L,U}}H^d_{G_{\varrho,L}\cap U}(U;\mathscr{O}_{X}).
\]
Here $U \subset X$ ranges  through open neighborhoods of $z_0$,  
and  $G_{\varrho,L}$ denotes the closed set
\[
G_{\varrho,L}:=\{z \in X;\,\varrho^2|f'(z)| \leqslant |f_1(z)|,\,
f^{}_1(z) \in L\},
\]
where 
$L \subset \mathbb{C}$ 
ranges  through closed convex cones 
with $L \subset \{\tau \in \mathbb{C};\, \Re \tau > 0\} \cup \{0\}$. 
Now we apply the result in the previous section to the case above.
We take the open sector $ S_{r,\theta}$  defined by
\[
 S_{r,\theta}:=\{\eta \in \mathbb{C};\,\lvert\arg\eta| < \theta,\, 0< |\eta| < r\}
\]
for $0 < \theta < \pi/2$ and $r > 0$
as an $ S$ in the previous section. The vector $\zeta$ 
is taken to be the image of $\zeta'_0$ by the canonical mapping  
$(T^*_YX)_{z_0} \to (T^*X)_{z_0} = \mathbb{C}^n$.
We adopt the deformation mapping given in Example \ref{es:typical-deformation} and
assume that $U$ is sufficiently small so that the assumption of 
Proportion \ref{prop:iso-cohomology-groups-for-general} is satisfied.
Therefore there exists the canonical isomorphism
\[
\boldsymbol{R}\varGamma^{}_{G_{\varrho,L}\cap U}(U;\mathscr{O}^{}_{X})
\earrow
\boldsymbol{R}\varGamma_{\widehat{G}_{\varrho,L}\cap\widehat{U}_{\varrho,r,\theta}}\!
(\widehat{U}_{\varrho,r,\theta}; \pi^{-1}_\eta\! \mathscr{O}^{}_{X}),
\]
where $\widehat{G}_{\varrho,L}$ and $\widehat{U}_{\varrho,r,\theta}$ are defined by
 \eqref{eq:def-hat-UV-general} with respect to $G=G_{\varrho,L}$ and $U$.
By easy computations, these sets are given by
\begin{equation}{\label{eq:def-hat-G-U}}
\begin{aligned}\widehat{G}_{\varrho,L} &= \{(z,\eta) \in \widehat{X};\,
\varrho|f'(z)| \leqslant |\eta|,\,
f_1(z) \in L\},
\\
\widehat{U}_{\varrho,r,\theta} &= 
\{(z,\eta) \in  U \times  S_{r,\theta};\,|f_1(z)| < \varrho |\eta|\}, 
\end{aligned}
\end{equation}
respectively. Thus, from the  exact sequence
\[
0\to \pi_{\eta}^{-1}\!\mathscr{O}_{X} \to \mathscr{O}_{\WX} \xrightarrow{\partial_\eta} 
\mathscr{O}_{\WX} \to 0,
\]
we obtain the following distinguished triangle:
\[
\boldsymbol{R}\varGamma_{G_{\varrho,L}\cap U}(U;\mathscr{O}_{X})\to  
\boldsymbol{R}\varGamma_{\widehat{G}_{\varrho,L}\cap \widehat{U}_{\varrho,r,\theta}}\!
(\widehat{U}_{\varrho,r,\theta};\mathscr{O}_{\WX}) 
\xrightarrow{\partial_\eta} 
\boldsymbol{R}\varGamma_{\widehat{G}_{\varrho,L}\cap \widehat{U}_{\varrho,r,\theta}}\!
( \widehat{U}_{\varrho,r,\theta};\mathscr{O}_{\WX}) \xrightarrow{+1}.
\]
We will see later the fact
\begin{equation}\label{eq:vanishing-cohomology-edge}
\smashoperator{\varinjlim_{\varrho,r,\theta,L,U}}
H^k_{\widehat{G}^{}_{\varrho,L}\cap \widehat{U}_{\varrho,r,\theta}}\!( \widehat{U}_{\varrho,r,\theta};\mathscr{O}_{\WX}) = 0
\quad (k \ne d).
\end{equation}
Hence we have reached the following definition and theorem.
\begin{defn}
We define
\[
\widehat{C}^{\,\mathbb{R}}_{{Y|X},z^*_0} :=
\smashoperator{\varinjlim_{\varrho,r,\theta,L,U}}
H^d_{\widehat{G}^{}_{\varrho,L} \cap\widehat{U}_{\varrho,r,\theta}}\!( \widehat{U}_{\varrho,r,\theta};\mathscr{O}_{\WX}),
\]
where $U \subset X$ and $L \subset \mathbb{C}$ 
range through open neighborhoods of $z_0$ and
closed convex cones in $\mathbb{C}$
with $L \subset \{\tau \in \mathbb{C};\, \Re \tau > 0\} \cup \{0\}$ 
respectively, and the subsets $\widehat{U}^{}_{\varrho,r,\theta}$ and $\widehat{G}^{}_{\varrho,L}$
are given in \eqref{eq:def-hat-G-U}.
Further we define
\[
C^{\mathbb{R}}_{Y|X,z^*_0} := \Ker(\partial^{}_\eta\colon 
\widehat{C}^{\,\mathbb{R}}_{Y|X,z^*_0} \to \widehat{C}^{\,\mathbb{R}}_{Y|X,z^*_0}).
\]
\end{defn}
Therefore,  we obtain:
\begin{thm}
There exists the following canonical isomorphism
\[
\mathscr{C}^{\mathbb{R}}_{Y|X,z^*_0}  \earrow
C^{\mathbb{R}}_{Y|X,z^*_0}.
\]
\end{thm}
Let us show \eqref{eq:vanishing-cohomology-edge}.
We may assume 
$z^*_0 = (z''_0;\zeta'_0) = (0;1,0,\dots,0)$.  
Let $\boldsymbol{\kappa}:=(r,r',\varrho, \theta) \in \mathbb{R}^{4}$ be a $4$-tuple of positive constants with
\begin{equation}{\label{eq:kappa-conditions}}
0<\theta < \frac{\,\pi\,}{2},\quad 0<\varrho < 1,\quad 0< r < \varrho r'.
\end{equation}
Then we set 
\begin{equation}\label{eq:def-Gamma}
S_{\boldsymbol{\kappa}}:= S^{}_{r,\theta/4}=\{\eta \in \mathbb{C};\,0< |\eta| < r,\,\lvert\arg\eta| < \dfrac{\,\theta\,}{4}\}
\end{equation} 
and  define the open subset
\[
\widehat{U}_{\boldsymbol{\kappa}} := \smashoperator{\Bcap_{i=2}^{n}}\{(z,\eta) \in X \times S_{\boldsymbol{\kappa}};\,
 |z_1| < \varrho|\eta|,\, |z_i| < r'\}.
\]
We also define the closed cone
\[
\widehat{G}_{\boldsymbol{\kappa}}
:= \smashoperator{\Bcap_{i=2}^{d}}\{(z,\eta) \in \widehat{ X};\, \lvert\arg z_1| \leqslant \frac{\,\pi\,}{2} - \theta,\,
\varrho|z_i| \leqslant |\eta| \}.
\]
By using these subsets, we introduce objects corresponding to
$\widehat{C}^{\,\mathbb{R}}_{Y|X,z^*_0}$ and $C^{\mathbb{R}}_{Y|X,z^*_0}$
at $z^*_0$, which are easily manipulated by
\v{C}ech cohomology groups. 
\begin{defn}\label{def2.3}
We define
\begin{align*}
\widehat{C}^{\,\mathbb{R}}_{Y|X}(\boldsymbol{\kappa}) & := 
H^{d}_{\widehat{G}_{\boldsymbol{\kappa}}\cap \widehat{U}_{\boldsymbol{\kappa}}}\!
(\widehat{U}_{\boldsymbol{\kappa}};\mathscr{O}_{\WX}),
\\
C^{\mathbb{R}}_{Y|X}(\boldsymbol{\kappa}) &:= 
\Ker(\partial^{}_\eta\colon 
\widehat{C}^{\,\mathbb{R}}_{Y|X}(\boldsymbol{\kappa}) \to \widehat{C}^{\,\mathbb{R}}_{Y|X}(\boldsymbol{\kappa})).
\end{align*}
\end{defn}
Clearly we have
\[
\widehat{C}^{\,\mathbb{R}}_{Y|X,z^*_0} = \smashoperator{\varinjlim_{\boldsymbol{\kappa}}} \widehat{C}^{\,\mathbb{R}}_{Y|X}(\boldsymbol{\kappa})
\,\,\text{ and }\,\,
C^{\mathbb{R}}_{Y|X,z^*_0} = \smashoperator{\varinjlim_{\boldsymbol{\kappa}}} C^{\mathbb{R}}_{Y|X}(\boldsymbol{\kappa}),
\]
since families of closed cones and open subsets appearing in 
inductive limits of the both sides are equivalent with respect to inclusion of sets.
\begin{prop}\label{prop:vanishing-hol-micro-param}
If $k\ne d$, then
\[
H^{k}_{\widehat{G}_{\boldsymbol{\kappa}}\cap \widehat{U}_{\boldsymbol{\kappa}}}\!(\widehat{U}_{\boldsymbol{\kappa}};\mathscr{O}_{\WX}) = 0.
\]
In particular,  \eqref{eq:vanishing-cohomology-edge} holds.
\end{prop}
\begin{proof}
We set 
\begin{equation}
\label{cech1} 
\begin{aligned}
\widehat{V}^{(1)}_{\boldsymbol{\kappa}} &:= 
\{(z,\eta) \in \widehat{U}_{\boldsymbol{\kappa}};\,
\dfrac{\,\pi\,}{2} - \theta <  \arg z_1 < \dfrac{\,3\pi\,}{2} + \theta\}, \\
\widehat{V}^{(i)}_{\boldsymbol{\kappa}}&:= 
\{(z,\eta) \in \widehat{U}_{\boldsymbol{\kappa}};\,
\varrho|z_i| > |\eta|\} \qquad (2 \leqslant i \leqslant d).
\end{aligned}
\end{equation}
Since each  $\widehat{V}^{(i)}_{\boldsymbol{\kappa}}$ is  pseudoconvex  and 
 $\widehat{U}_{\boldsymbol{\kappa}}\smallsetminus \widehat{G}_{\boldsymbol{\kappa}}
=\smashoperator{\Bcup_{i=1}^d}\widehat{V}^{(i)}_{\boldsymbol{\kappa}}$, we have
$
H^{k}_{\widehat{G}_{\boldsymbol{\kappa}}\cap \widehat{U}_{\boldsymbol{\kappa}}}\!(\widehat{U}_{\boldsymbol{\kappa}};\mathscr{O}_{\WX})  = 0$ for $k > d$.
Let us show the assertion for $k < d$. As $\varrho < 1$ and $r < \varrho r'$ hold, we have
\[
\boldsymbol{R}\varGamma_{\widehat{G}_{\boldsymbol{\kappa}}\cap \widehat{U}_{\boldsymbol{\kappa}}}\!(\widehat{U}_{\boldsymbol{\kappa}};\mathscr{O}_{\WX})
\simeq
\boldsymbol{R}\varGamma_{\widehat{G}_{\boldsymbol{\kappa}}\cap \widehat{D}}(\widehat{D};\mathscr{O}_{\WX}),
\]
where 
\[
\widehat{D} :=\smashoperator{\Bcap_{i=d+1}^n} \{(z,\eta)\in X \times   S_{\boldsymbol{\kappa}};\,
|z_1| < \varrho|\eta|,\, |z_i| < r'\}.
\]
Let us consider the holomorphic mapping on $\widehat{ X}$ defined by
\[
\varphi(z, \eta) := (z_1,\eta z_2,\dots, \eta z_d,z'',\eta).
\]
Since $\varphi$ is bi-holomorphic on $X \times \mathbb{C}^{\times}$,
we have
\[
\boldsymbol{R}\varGamma_{\widehat{G}_{\boldsymbol{\kappa}}\cap \widehat{D}}(\widehat{D};\mathscr{O}_{\WX})
\simeq
\boldsymbol{R}\varGamma_{\widehat{K}\cap \widehat{D}}(\widehat{D};\mathscr{O}_{\WX}).
\]
Here we set $\widehat{K}:=\widehat{K}_1\cap\widehat{K}_2$ with
\[
\widehat{K}_1 := \{(z,\eta)\in \widehat{ X};\, \lvert \arg z_1| \leqslant \frac{\,\pi\,}{2} - \theta\} , 
\qquad
\widehat{K}_2 := 
\smashoperator{\Bcap_{i=2}^{d}}\{(z,\eta) \in \widehat{ X};\,
\varrho|z_i| \leqslant 1 \}.
\] 
Then we have the distinguished triangle
\[
\boldsymbol{R}\varGamma_{\widehat{K}\cap \widehat{D}}
( \widehat{D};\mathscr{O}_{\WX})
\to
\boldsymbol{R}\varGamma_{\widehat{K}_2\cap \widehat{D}}
( \widehat{D};\mathscr{O}_{\WX})
 \to
\boldsymbol{R}\varGamma_{(\widehat{K}_2\smallsetminus \widehat{K}_1)\cap \widehat{D}}( \widehat{D} \smallsetminus \widehat{K}_1;\mathscr{O}_{\WX})
\xrightarrow{+1}.
\]
Hence the claim of the proposition 
follows from the following well-known lemma.
\end{proof}
\begin{lem}
Let $\mathbb{D}$ be a closed disk with positive radius in $\mathbb{C}$ and
$U$ a pseudoconvex open subset in $\mathbb{C}^m$. Then 
\[
H^\nu_{\mathbb{D}^{k} \times U}(\mathbb{C}^{k} \times U; 
\mathscr{O}_{\mathbb{C}^{k+m}}) = 0\qquad (\nu \ne k).
\]
Furthermore, for any pseudoconvex open subsets $U_1 \subset U_2$
in $\mathbb{C}^{m}$ which are non-empty and connected, the following canonical morphism is injective\textup{:}
\[
H^{k}_{\mathbb{D}^{k} \times  U_2}\!(\mathbb{C}^{k} \times U_2;
\mathscr{O}_{\mathbb{C}^{k+m}}) \to
H^{k}_{\mathbb{D}^{k} \times  U_1}\!(\mathbb{C}^{k} \times U_1;
\mathscr{O}_{\mathbb{C}^{k+m}}).
\]
\end{lem}
Next,  we set
\begin{align*}
U_{\boldsymbol{\kappa}}& :=\smashoperator{\Bcap_{i=2}^{n}}
\{z \in X;\, |z_1| < \varrho r,\, |z_i| < r' \},
\\
G_{\boldsymbol{\kappa}} & :=\smashoperator{\Bcap_{i=2}^{d}}\{
z \in X;\, \lvert\arg z_1| \leqslant \frac{\,\pi\,}{2} - \theta,\, \varrho^2|z_i| \leqslant |z_1|\}.
\end{align*}
\begin{cor}\label{cor2.6}
If $k\ne d$, then
\[
H^{k}_{G_{\boldsymbol{\kappa}}\cap U_{\boldsymbol{\kappa}}}\!(U_{\boldsymbol{\kappa}};\mathscr{O}_{X })  = 0,
\]
and there exists the following exact sequence\textup{:}
\begin{equation}\label{eq:1.22}
0\to 
H^d_{G_{\boldsymbol{\kappa}}\cap U_{\boldsymbol{\kappa}}}\!(U_{\boldsymbol{\kappa}}; \mathscr{O}_{X}) \to
\widehat{C}^{\,\mathbb{R}}_{Y|X}(\boldsymbol{\kappa})
\xrightarrow{\partial_\eta}
\widehat{C}^{\,\mathbb{R}}_{Y|X}(\boldsymbol{\kappa})
\to 0.
\end{equation}
\end{cor}
\begin{proof}
We set
\allowdisplaybreaks
\begin{equation}
\label{cech2} 
\begin{aligned}
V^{(1)}_{\boldsymbol{\kappa}}&:= 
\{z \in U_{\boldsymbol{\kappa}};\,
\dfrac{\,\pi\,}{2} - \theta <  \arg z_1 < \dfrac{\,3\pi\,}{2} + \theta\}, \\
V^{(i)}_{\boldsymbol{\kappa}} &:= 
\{z \in U_{\boldsymbol{\kappa}};\,\varrho^{2}|z_i| > |z_1| \} \qquad (2 \leqslant i \leqslant d).
\end{aligned}
\end{equation}
Since each  $V^{(i)}_{\boldsymbol{\kappa}}$ is  pseudoconvex  and 
 $U_{\boldsymbol{\kappa}}\smallsetminus G_{\boldsymbol{\kappa}}=\smashoperator{\Bcup_{i=1}^d}V^{(i)}_{\boldsymbol{\kappa}}$, 
we have
$
H^{k}_{G_{\boldsymbol{\kappa}}\cap U_{\boldsymbol{\kappa}}}\!(U_{\boldsymbol{\kappa}};\mathscr{O}_{X })  = 0$ for $k > d$.
By Proposition \ref{prop:iso-cohomology-groups-for-general}  and Remark \ref{rem1.3}, we have
the following distinguished triangle
\[
\boldsymbol{R}\varGamma_{G_{\boldsymbol{\kappa}}\cap U_{\boldsymbol{\kappa}}}\!(U_{\boldsymbol{\kappa}}; \mathscr{O}_{X}) \to
\boldsymbol{R}\varGamma_{\widehat{G}_{\boldsymbol{\kappa}}\cap \widehat{U}_{\boldsymbol{\kappa}}}\!
(\widehat{U}_{\boldsymbol{\kappa}};\mathscr{O}_{\WX})
\xrightarrow{\partial_\eta}
\boldsymbol{R}\varGamma_{\widehat{G}_{\boldsymbol{\kappa}}\cap \widehat{U}_{\boldsymbol{\kappa}}}\!
(\widehat{U}_{\boldsymbol{\kappa}};\mathscr{O}_{\WX}) 
\xrightarrow{+1}.
\]
By Definition \ref{def2.3} and Proposition \ref{prop:vanishing-hol-micro-param}, we have \eqref{eq:1.22} and
$H^{k}_{G_{\boldsymbol{\kappa}}\cap U_{\boldsymbol{\kappa}}}\!(U_{\boldsymbol{\kappa}};\mathscr{O}_{X })  = 0$ for $k < d$.
\end{proof}
Note that, since
\[
\widehat{U}_{\boldsymbol{\kappa}} \subset 
\pi_{\eta}^{-1}(U_{\boldsymbol{\kappa}}), \quad 
\pi_{\eta}^{-1}(G_{\boldsymbol{\kappa}}) \cap\widehat{U}_{\boldsymbol{\kappa}} \subset
\widehat{G}_{\boldsymbol{\kappa}} \cap\widehat{U}_{\boldsymbol{\kappa}}, 
\]
the  morphism $H^d_{G_{\boldsymbol{\kappa}}\cap U_{\boldsymbol{\kappa}}}\!(U_{\boldsymbol{\kappa}}; \mathscr{O}_{X}) \to
\widehat{C}^{\,\mathbb{R}}_{Y|X}(\boldsymbol{\kappa})$ is defined 
by a natural way associated with inclusion of sets.
By Proposition \ref{prop:vanishing-hol-micro-param} and \eqref{eq:1.22}, we obtain the following corollary.
\begin{cor}\label{cor:iso-hol-micro-at-1}
Let $z^*_0 = (0;1,0,\dots,0)$. Then there exist  isomorphisms
\begin{equation}\label{eq:canonical-morphisms-for-G}
\xymatrix @C=.5em @R=2ex{
H^d_{G_{\boldsymbol{\kappa}}\cap U_{\boldsymbol{\kappa}}}\!
(U_{\boldsymbol{\kappa}};\mathscr{O}_{X}) \ar@<-.4ex>[r]^-{\dsim} \ar[d]
&
C^{\mathbb{R}}_{Y|X}(\boldsymbol{\kappa}) \ar[d]
\\
\mathscr{C}^{\mathbb{R}}_{Y|X,z^*_0}
\ar@<-.4ex>[r]^-{\dsim} &
\smashoperator{\varinjlim\limits_{\boldsymbol{\kappa}}}
C^{\mathbb{R}}_{Y|X}(\boldsymbol{\kappa}).
}\end{equation}
\end{cor}
We now  consider 
a \v{C}ech representation
of $C^{\mathbb{R}}_{Y|X}(\boldsymbol{\kappa})$. Recall $\widehat{V}^{(i)}_{\boldsymbol{\kappa}}\subset \widehat{X}$ of \eqref{cech1} 
and $V^{(i)}_{\boldsymbol{\kappa}}\subset X$  of \eqref{cech2} for  $1\leqslant i \leqslant d$.  
Let $\mathcal{P}_d$ be the set of all the subsets of $\{1,\dots,d\}$ and
$\mathcal{P}^\vee_d\subset \mathcal{P}_d$ consisting of $\alpha \in \mathcal{P}_d$
with $\#\alpha = d-1$ ($\#\alpha$ denotes the number of elements in $\alpha$).
For $\alpha \in \mathcal{P}_d$, we define
\begin{equation}{\label{eq:def-alpha-V}}
\widehat{V}^{(\alpha)}_{\boldsymbol{\kappa}}
:= \smashoperator{\Bcap_{i \in \alpha}} \widehat{V}^{(i)}_{\boldsymbol{\kappa}},
\qquad V^{(\alpha)}_{\boldsymbol{\kappa}} 
:= \smashoperator{\Bcap_{i \in \alpha}} V^{(i)}_{\boldsymbol{\kappa}}.
\end{equation}
In what follows, the symbol $*$ denotes the set $\{1,\dots,d\}$ 
by convention, for example,
\[
\widehat{V}^{(*)}_{\boldsymbol{\kappa}} := \widehat{V}^{(\{1,\dots,d\})}_{\boldsymbol{\kappa}}
= \smashoperator{\Bcap_{i=1}^d} \widehat{V}^{(i)}_{\boldsymbol{\kappa}}. 
\]
As each
$\widehat{V}^{(\alpha)}_{\boldsymbol{\kappa}}$ 
(resp.\ $V^{(\alpha)}_{\boldsymbol{\kappa}}$) is pseudoconvex, we have
\begin{align*}
\widehat{C}^{\,\mathbb{R}}_{Y|X}(\boldsymbol{\kappa}) & =
\varGamma(\widehat{V}^{(*)}_{\boldsymbol{\kappa}};\mathscr{O}^{}_{\WX})\big/
\smashoperator{\Sum_{\alpha \in\mathcal{P}^\vee_d}}
\varGamma(\widehat{V}^{(\alpha)}_{\boldsymbol{\kappa}};\mathscr{O}^{}_{\WX}), 
\\
{C}^{\mathbb{R}}_{Y|X}(\boldsymbol{\kappa}) & =\{u \in 
\widehat{C}^{\,\mathbb{R}}_{Y|X}(\boldsymbol{\kappa});\,\partial_\eta u  = 0 \},\\
H^d_{G_{\boldsymbol{\kappa}}\cap U_{\boldsymbol{\kappa}}}\! (U_{\boldsymbol{\kappa}};\mathscr{O}_{X}) &
=
\varGamma(V^{(*)}_{\boldsymbol{\kappa}};\mathscr{O}_{X})\big/
\smashoperator{\Sum_{\alpha \in \mathcal{P}^\vee_d}}
\varGamma(V^{(\alpha)}_{\boldsymbol{\kappa}};\mathscr{O}_{X}).
\end{align*}
Since $\widehat{V}^{(\alpha)}_{\boldsymbol{\kappa}} 
\subset \pi_\eta^{-1}(V^{(\alpha)}_{\boldsymbol{\kappa}})$ holds,
we can regard a holomorphic function $\varphi$ on $V^{(\alpha)}_{\boldsymbol{\kappa}}$
as that on $\widehat{V}^{(\alpha)}_{\boldsymbol{\kappa}}$, and thus, we have the natural morphism
 $\varGamma(V^{(\alpha)}_{\boldsymbol{\kappa}};\mathscr{O}_{X})\to 
\varGamma(\widehat{V}^{(\alpha)}_{\boldsymbol{\kappa}};\mathscr{O}_{\WX})$.
This induces the canonical morphism between the \v{C}ech cohomology groups
\[
H^d_{G_{\boldsymbol{\kappa}}\cap U_{\boldsymbol{\kappa}}}\! (U_{\boldsymbol{\kappa}};\mathscr{O}_{X})
=
\dfrac{\varGamma(V^{(*)}_{\boldsymbol{\kappa}};\mathscr{O}_{X})}
{\smashoperator{\Sum_{\alpha \in \mathcal{P}^\vee_d}}
\varGamma(V^{(\alpha)}_{\boldsymbol{\kappa}};\mathscr{O}_{X})} 
\to
\{u \in \dfrac{\varGamma(\widehat{V}^{(*)}_{\boldsymbol{\kappa}};\mathscr{O}_{\WX})}
{\smashoperator{\Sum_{\alpha \in \mathcal{P}^\vee_d}}
\varGamma(\widehat{V}^{(\alpha)}_{\boldsymbol{\kappa}};\mathscr{O}_{\WX})};\,
\partial_\eta u = 0\}
= C^{\mathbb{R}}_{Y|X}(\boldsymbol{\kappa}).
\]
Clearly this morphism coincides with 
\eqref{eq:canonical-morphisms-for-G}, hence  it gives
an isomorphism by Corollary \ref{cor:iso-hol-micro-at-1}. 

%

\section{Cohomological Representation of $\mathscr{E}^{\mathbb{R}}_X$ with an Apparent  Parameter}\label{sec:actions_ER}
We inherit  the same notation from the previous section. Set $X^2:= X \times X$ with the coordinates $(z,w)$,  and let 
$(z,w,\eta)$ be coordinates of $\widehat{ X}^2:= X^2 \times \mathbb{C}$.  
Let
$\varDelta\subset X^2$  be the diagonal set. We identify  $X$  with $\varDelta $, and
\[
T^*X=\{(z;\zeta)\} \simeq\{(z,z;\zeta,-\zeta)\}= T_ \varDelta ^*X^2.
\]
Let $\mathscr{E}^\mathbb{R}_X$ denote the sheaf of pseudodifferential operators on
the cotangent bundle $T^*X$ of $X$, 
and $z^*_0 = (z^{}_0;\zeta^{}_0) \in T^*X$ with $|\zeta^{}_0| = 1$. 
Set
\[
 f_{\varDelta,1}(z,w) := \langle z-w,\zeta^{}_0 \rangle, \qquad
 f'_{\varDelta}(z,w) := z -w- \langle z-w,\zeta^{}_0\rangle \overline{\zeta}{}_0.
\]
See also Appendix \ref{ap:general_C} for a generalization of the  mappings above and the following arguments on
a complex manifold. For a closed convex cone $L \subset \mathbb{C}$, set
\[
G^{}_{\varDelta,\varrho,L} := \{(z,w) \in X^2;\,
\varrho^2|f'_{\varDelta}(z,w)| \leqslant |f_{\varDelta,1}(z,w)|,\,
f_{\varDelta,1}(z,w) \in L\}.
\]
Then it follows from the definition of $\mathscr{E}^{\mathbb{R}}_X$ that
we have
\[
\mathscr{E}^{\mathbb{R}}_{X,z^*_0}
= \smashoperator{\varinjlim_{\varrho,L,U}}
H^n_{G_{\varDelta,\varrho,L} \cap U}(U;\Oo).
\]
Here $\Oo$ is the sheaf of holomorphic $n$-forms 
with respect to $dw_1, \dots, dw_n$,
$U \subset X^2$ and $L \subset \mathbb{C}$ 
range through open neighborhoods of $(z^{}_0, z^{}_0)$ and
closed convex cones in $\mathbb{C}$
with $L \subset \{\tau \in \mathbb{C};\, \Re \tau > 0\} \cup \{0\}$ 
respectively.

Now we introduce the corresponding cohomology group with an apparent  parameter.
Set, for an open subset $U \subset X^2$ and a closed convex cone $L \subset \mathbb{C}$,
\begin{align*}
\widehat{U}_{\varDelta,\varrho,r,\theta} &:= \{(z,w,\eta) \in U \times S_{r,\theta};\,
|f_{\varDelta,1}(z,w)| < \varrho |\eta|\}, \\
\widehat{G}_{\varDelta,\varrho,L} &:= \{(z,w,\eta) \in\widehat{ X}^2;\,
\varrho|f'_{\varDelta}(z,w)| \leqslant |\eta|,\,
f_{\varDelta,1}(z,w) \in L
 \}.
\end{align*}
\begin{defn}
We set
\[
\widehat{E}^{\mathbb{R}}_{X,z^*_0} :=
\smashoperator{\varinjlim_{\varrho,r,\theta,L,U}}
H^n_{\widehat{G}_{\varDelta,\varrho,L}\cap \widehat{U}_{\varDelta,\varrho,r,\theta}}\!
(\widehat{U}_{\varDelta,\varrho,r,\theta}; \OOO).
\]
Here $\OOO$ is the sheaf of holomorphic $n$-forms 
with respect to $dw_1, \dots, dw_n$,
$U \subset X^2$ and $L \subset \mathbb{C}$ 
range through open neighborhoods of $(z_0,z^{}_0)$ and
closed convex cones in $\mathbb{C}$
with $L \subset \{\tau \in \mathbb{C};\, \Re \tau > 0\} \cup \{0\}$ 
respectively.
Further we define
\[
E^{\mathbb{R}}_{X,z^*_0} :=
\Ker(\partial^{}_\eta\colon \widehat{E}^{\mathbb{R}}_{X,z^*_0} \to \widehat{E}^{\mathbb{R}}_{X,z^*_0}).
\]
\end{defn}
From the consequence of the previous section, 
the following theorem immediately follows.
\begin{thm}
There exists  the canonical isomorphism
\[
\mathscr{E}^{\mathbb{R}}_{X,z^*_0} \earrow
E^{\mathbb{R}}_{X,z^*_0}.
\]
\end{thm}
We assume $z^*_0=(z^{}_0;\zeta^{}_0)=(0;1,0,\dots,0)$ in what follows and
consider a \v{C}ech representation of $E^{\mathbb{R}}_{X,z^*_0}$.
Let $\boldsymbol{\kappa} = (r, r',\varrho, \theta) \in 
\mathbb{R}^{4}$ be parameters satisfying the conditions \eqref{eq:kappa-conditions}.
Then we define
\begin{align*}
\widehat{U}_{\varDelta,\boldsymbol{\kappa}} & := \smashoperator{\Bcap_{i=2}^n}
\{(z,w,\eta) \in \widehat{X}^2;\, \|z\| < r',\,\eta \in S_{\boldsymbol{\kappa}},\,|z^{}_1-w^{}_1| < \varrho |\eta|,\,|z_i-w^{}_i| < r'\},
\\
\widehat{G}_{\varDelta,\boldsymbol{\kappa}} &:=
\smashoperator{\Bcap_{i=2}^n}\{(z, w, \eta) \in \widehat{X}^2;\, 
\lvert\arg (z^{}_1 - w^{}_1)| \leqslant \frac{\,\pi\,}{2} - \theta,\,
\varrho|z^{}_i - w^{}_i| \leqslant |\eta| \}.
\end{align*}
We also set
\begin{align*}
U_{\varDelta,\boldsymbol{\kappa}} &:= \smashoperator{\Bcap_{i=2}^n}
\{(z,w) \in X^2;\, \|z\| < r',\, |z^{}_1 -w^{}_1| < \varrho r,\,
 |z^{}_i -w^{}_i| <r'\},
\\ 
G_{\varDelta,\boldsymbol{\kappa}} & :=\smashoperator{\Bcap_{i=2}^n}
\{(z, w) \in X^2;\, \lvert\arg (z^{}_1 - w^{}_1)| \leqslant \frac{\,\pi\,}{2} - \theta,\,
\varrho^2|z^{}_i - w^{}_i| \leqslant |z^{}_1-w^{}_1| \}.
\end{align*}
\begin{defn}
We define
\begin{align*}
\widehat{E}^{\mathbb{R}}_{X}(\boldsymbol{\kappa}) &:= 
H^{n}_{\widehat{G}_{\varDelta,\boldsymbol{\kappa}} \cap \widehat{U}_{\varDelta,\boldsymbol{\kappa}}}\!
(\widehat{U}_{\varDelta,\boldsymbol{\kappa}}; \OOO), 
\\
E^{\mathbb{R}}_{X}(\boldsymbol{\kappa}) &:= 
\Ker(\partial_\eta\colon\widehat{E}^{\mathbb{R}}_{X}(\boldsymbol{\kappa})\to \widehat{E}^{\mathbb{R}}_{X}(\boldsymbol{\kappa})).
\end{align*}
\end{defn}
Note that 
\[
\widehat{E}^{\mathbb{R}}_{X,z^*_0} = \smashoperator{\varinjlim_{\boldsymbol{\kappa}}} \widehat{E}^{\mathbb{R}}_{X}(\boldsymbol{\kappa})
\,\,\text{ and }\,\,
E^{\mathbb{R}}_{X,z^*_0} = \smashoperator{\varinjlim_{\boldsymbol{\kappa}}} E^{\mathbb{R}}_{X}(\boldsymbol{\kappa})
\]
hold.  Then by employing the coordinates transformation
$(z,w)\mapsto (z,z- w)$, it follows from
Proposition \ref{prop:vanishing-hol-micro-param},  
Corollaries \ref{cor2.6} and  \ref{cor:iso-hol-micro-at-1} that the both complexes
\begin{align*}
&\boldsymbol{R}\varGamma_{\widehat{G}_{\varDelta,\boldsymbol{\kappa}}\cap \widehat{U}_{\varDelta,\boldsymbol{\kappa}}}\!
(\widehat{U}_{\varDelta,\boldsymbol{\kappa}}; \OOO),
\\
&
\boldsymbol{R}\varGamma_{G_{\varDelta,\boldsymbol{\kappa}}\cap U_{\varDelta,\boldsymbol{\kappa}}}\!
(U_{\varDelta,\boldsymbol{\kappa}}; \Oo)\simeq
\boldsymbol{R}\varGamma_{\widehat{G}_{\varDelta,\boldsymbol{\kappa}}\cap \widehat{U}_{\varDelta,\boldsymbol{\kappa}}}\!
(\widehat{U}_{\varDelta,\boldsymbol{\kappa}};
\RHom_{\mathscr{D}_{\WXX}}\!(\mathscr{D}^{}_{\WXX}
/\mathscr{D}^{}_{\WXX}\partial^{}_\eta,\OOO))
\end{align*}
are concentrated in degree $n$, and we have the canonical isomorphism
\[
H^{n}_{G_{\varDelta,\boldsymbol{\kappa}}\cap U_{\varDelta,\boldsymbol{\kappa}}}\!
(U_{\varDelta,\boldsymbol{\kappa}}; \Oo)  \earrow  E^{\mathbb{R}}_X(\boldsymbol{\kappa}).
\]
Furthermore we have
\[
\mathscr{E}^\mathbb{R}_{X,z^*_0} = 
\smashoperator{\varinjlim_{\boldsymbol{\kappa}}} H^{n}_{G_{\varDelta,\boldsymbol{\kappa}}\cap U_{\varDelta,\boldsymbol{\kappa}}}\!
(U_{\varDelta,\boldsymbol{\kappa}}; \Oo),
\]
By these facts,  we get
\[
\mathscr{E}^\mathbb{R}_{X,z^*_0} = 
\smashoperator{\varinjlim_{\boldsymbol{\kappa}}} E^{\mathbb{R}}_X(\boldsymbol{\kappa}).
\]
Now we give the \v{C}ech representations of these cohomology groups.
Recall that the open subset $\widehat{U}_{\varDelta,\boldsymbol{\kappa}}\subset \widehat{X}^2$ is defined by
\[
\Bcap_{i=2}^n \!\{(z,w,\eta) \in X^2\times  S_{\boldsymbol{\kappa}};\,
	\|z\| < r',\,  |z_1 - w_1 | < \varrho |\eta|,\, |z_i - w_i| < r'\}.
\]
Here the open sector $ S_{\boldsymbol{\kappa}}$ was given by \eqref{eq:def-Gamma}.
Set
\begin{align*}
\widehat{V}_{\varDelta,\boldsymbol{\kappa}}^{(1)} &:= 
\{(z,w,\eta) \in\widehat{U}_{\varDelta,\boldsymbol{\kappa}};\,
\dfrac{\,\pi\,}{2} - \theta <  \arg(z_1 - w_1) < \dfrac{\,3\pi\,}{2} + \theta\}, \\
\widehat{V}^{(i)}_{\varDelta,\boldsymbol{\kappa}} &:= 
\{(z,w,\eta) \in \widehat{U}_{\varDelta,\boldsymbol{\kappa}};\,
\varrho|z_i-w_i| > |\eta|\}\quad (2 \leqslant i \leqslant n).
\end{align*}
We also set
\begin{align*}
V_{\varDelta,\boldsymbol{\kappa}}^{(1)} &:= 
\{(z,w) \in U_{\varDelta,\boldsymbol{\kappa}};\,
\dfrac{\,\pi\,}{2} - \theta <  \arg(z_1 - w_1) < \dfrac{\,3\pi\,}{2} + \theta\},\\
V_{\varDelta,\boldsymbol{\kappa}}^{(i)} &:= 
\{(z,\,w) \in U_{\varDelta,\boldsymbol{\kappa}};\,
\varrho^{2}|z_i - w_i| > |z_1 - w_1|\} \quad  (2 \leqslant i\leqslant n).
\end{align*}
For any $\alpha \in \mathcal{P}_n$, the subset 
$\widehat{V}^{(\alpha)}_{\varDelta,\boldsymbol{\kappa}}$,
$V^{(\alpha)}_{\varDelta,\boldsymbol{\kappa}}$ etc.\  are defined in the same way as those in 
\eqref{eq:def-alpha-V}.  Then, using these coverings, we have
\begin{align*}
\widehat{E}^{\mathbb{R}}_{X}(\boldsymbol{\kappa})& =\varGamma( \widehat{V}_{\varDelta,\boldsymbol{\kappa}}^{(*)}; 
\OOO)
\big/\smashoperator{\Sum_{\alpha \in \mathcal{P}^\vee_n}}
\varGamma( \widehat{V}^{(\alpha)}_{\varDelta,\boldsymbol{\kappa}}; \OOO), 
\\
{E}^{\mathbb{R}}_{X}(\boldsymbol{\kappa}) &=
\{K \in 
\widehat{E}^{\mathbb{R}}_{X}(\boldsymbol{\kappa});\,
\partial_\eta K  = 0 \},
\\
H^n_{G_{\varDelta,\boldsymbol{\kappa}}\cap U_{\varDelta,\boldsymbol{\kappa}}}\! (U_{\varDelta,\boldsymbol{\kappa}};\Oo)
&=
\varGamma(V^{(*)}_{\varDelta,\boldsymbol{\kappa}}; \Oo)
\big/\smashoperator{\Sum_{\alpha \in \mathcal{P}^\vee_n}}
\varGamma( V_{\varDelta,\boldsymbol{\kappa}}^{(\alpha)}; \Oo).
\end{align*}
Let us take any  $K(z,w)\,dw =[\psi(z,w,\eta)\,dw] \in E_X^{\mathbb{R}}(\boldsymbol{\kappa})$
and  $f(z)= [u(z,\eta)] \in C_{Y|X}^{\mathbb{R}}(\boldsymbol{\kappa})$  with  representatives
$\psi(z,w,\eta)\,dw \in \varGamma( \widehat{V}_{\varDelta,\boldsymbol{\kappa}}^{(*)}; 
\OOO)$  
and $u(z,\eta) \in \varGamma( \widehat{V}^{(*)}_{\boldsymbol{\kappa}}; \mathscr{O}^{}_{\smash{\widehat{X}}})$ respectively,  
which were introduced in the previous 
section.
We will define the action $\mu_K$ on $C_{Y|X}^{\mathbb{R}}(\boldsymbol{\kappa})$ 
associated with the kernel
$K(z,w)\,dw$. For that purpose, we first introduce the paths of the integration related
to $\mu_K$. Let $(z,\eta) \in  \widehat{X}$. Set 
$\beta^{}_0:=\dfrac{\,\varrho\,}{2} \,e^{-\iim (\pi + \theta)/2}$ and 
 $\beta^{}_1:= \dfrac{\,\varrho\,}{2} \,e^{\iim (\pi + \theta)/2} $, 
and we define,
for a sufficiently small $\varepsilon > 0$,
the path 
$\gamma_1(z, \eta;\varrho, \theta)$ in $\mathbb{C}_{w_1}$ by
\begin{align*}
\{w_1 = z_1 + t\beta^{}_0\eta;\,
 1 \geqslant t \geqslant \varepsilon\} 
&\vee
\{w_1 = z_1 +   \dfrac{\,\varepsilon\varrho \eta\,}{2} \,e^{-\iim (\pi + \theta)t/2};\,
- 1\leqslant t \leqslant 1\} 
\\
&\vee
\{w_1 = z_1 + t\beta^{}_1\eta;\, 
 \varepsilon \leqslant t \leqslant 1\}.
\end{align*}
Note that $\gamma_1(z,\eta;\varrho, \theta)$ joins the two points
$
z_1 + \beta^{}_0\eta
$
and
$z_1 + \beta^{}_1\eta$,
which depend on the variables $z_1$ and $\eta$ holomorphically.
We introduce another path
$\overline{\gamma}_1(z,\eta;\varrho,\theta)$ in $\mathbb{C}_{w_1}$ by
the straight segment from
$z_1 + \beta^{}_0\eta$ to
$z_1 + \beta^{}_1\eta$.
We also define the path $\gamma_i(z,\eta;\varrho)$ in $\mathbb{C}_{w_i}$ 
($i=2,\dots,n$) by the circle with center at $z_i$ and radius
$\dfrac{\,|\eta|\,}{\varrho} + \varepsilon$, i.e.
\[
\gamma_i(z,\eta;\varrho) :=
\{w_i = z_i + \Bigl(\dfrac{|\eta|}{\varrho} + \varepsilon\Bigr) e^{2\pi \iim t};\,
0 \leqslant t \leqslant 1\}.
\]
Define the real $n$-dimensional chain in $X$ made from these paths by
\begin{align*}
\gamma(z,\eta;\varrho,\theta) &:=
\gamma_1(z, \eta;\varrho, \theta) \times \gamma^{}_2(z, \eta;\varrho)\times \dots \times
\gamma_n(z, \eta;\varrho) \subset X, \\
\overline{\gamma}(z,\eta;\varrho,\theta) &:=
\overline{\gamma}_1(z, \eta;\varrho, \theta) \times 
\gamma_2(z, \eta;\varrho) \times 
\dots \times
\gamma_n (z, \eta;\varrho)\subset X. 
\end{align*}
Let $\widehat{\pi}^{}_2 \colon \widehat{X}^2 \ni (z,w,\eta) \mapsto (w,\eta)\in \widehat{X}$ be the canonical projection.
For $\alpha \in \mathcal{P}_n$ and $\beta \in \mathcal{P}_d$, we set
\begin{align*}
\widehat{W}^{(\alpha,\,\beta)}_{\boldsymbol{\kappa}}&:= 
\widehat{V}^{(\alpha)}_{\varDelta,\boldsymbol{\kappa}} \cap
\widehat{\pi}^{\,-1}_2(\widehat{V}^{(\beta)}_{\boldsymbol{\kappa}}),\\
\widehat{W}^{(*,*)}_{\boldsymbol{\kappa}} &:= 
\widehat{V}^{(*)}_{\varDelta,\boldsymbol{\kappa}} \cap
\widehat{\pi}_2^{\,-1}(\widehat{V}^{(*)}_{\boldsymbol{\kappa}}).
\end{align*}
We also set
$\widehat{W}^{(\alpha,*)}_{\boldsymbol{\kappa}}:= \widehat{W}^{(\alpha, \{1,\dots,d\})}_{\boldsymbol{\kappa}}$
and
$\widehat{W}^{(*,\,\beta)}_{\boldsymbol{\kappa}} := \widehat{W}^{(\{1,\dots,n\},\,\beta)}_{\boldsymbol{\kappa}}$.
Then the following lemma is easily obtained by elementary computations.
\begin{lem}{\label{lem:path-fundamental}}
Let $\widetilde{\boldsymbol{\kappa}} = (\tilde{r}, \tilde{r}', \tilde{\varrho},\tilde{\theta}) \in \mathbb{R}^{4}$ satisfying
\[
0<\tilde{r} < r,\quad
0<\tilde{r}' < \dfrac{r'}{2},\quad 0<\tilde{\theta} < \dfrac{\,\theta\,}{4},\quad
0<\tilde{\varrho} < \dfrac{\,\varrho\,}{2} \sin\frac{\,\theta\,}{4},
\]
and the corresponding conditions to \eqref{eq:kappa-conditions}. Then 
the following hold for sufficiently small $\varepsilon > 0$\textup{:}
\begin{enumerate}[\rm (1)]
\item For any $(z,\eta) \in \widehat{V}^{(*)}_{\widetilde{\boldsymbol{\kappa}}}$, 
in $\widehat{X}^2$
\[
\{z\} \times \gamma(z,\eta;\varrho,\theta) \times \{\eta\}
\subset \widehat{W}^{(*,*)}_{\boldsymbol{\kappa}}.
\]
Here $\{z\} \times \gamma(z,\eta;\varrho,\theta) \times \{\eta\}$ 
denotes the product of these three subsets 
in $\widehat{X}^2= X \times X\times \mathbb{C}$.
\item 
For any $(z,\eta) \in \widehat{V}^{(\beta)}_{\widetilde{\boldsymbol{\kappa}}}$
with $\beta \in \mathcal{P}^\vee_d$, 
\[
\{z\} \times \gamma(z,\eta;\varrho,\theta) \times \{\eta\}
\subset \widehat{W}^{(*,\,\beta)}_{\boldsymbol{\kappa}}.
\]
\item 
For any $(z,\eta) \in \widehat{V}^{(\{2,\dots,d\})}_{\widetilde{\boldsymbol{\kappa}}}$, 
\[
\{z\} \times \overline{\gamma}(z,\eta;\varrho,\theta) \times \{\eta\}
\subset \widehat{W}^{(\{2,\dots,n\}, *)}_{\boldsymbol{\kappa}}.
\]
Furthermore 
\[
\{z\} \times \partial \gamma(z,\eta;\varrho,\theta) \times \{\eta\}
\subset \widehat{W}^{(*, *)}_{\boldsymbol{\kappa}},
\]
where $\partial \gamma(z, \eta;\varrho, \theta)$ denotes the boundary of $\gamma(z, \eta;\varrho, \theta)$ as a real $n$-dimensional 
chain.
\end{enumerate}
\end{lem}
Now we are ready to define the action $\mu_K$ of $K(z,w)\, dw \in E^{\mathbb{R}}_X(\boldsymbol{\kappa})$ 
on $C^{\mathbb{R}}_{Y|X}(\boldsymbol{\kappa})$.
\begin{thm}{\label{th:action-morphism}}
The bi-linear morphism
\[
\mu \colon E^{\mathbb{R}}_X(\boldsymbol{\kappa}) \tens_{\mathbb{C}} C^{\mathbb{R}}_{Y|X}(\boldsymbol{\kappa})
\to C^{\mathbb{R}}_{Y|X}(\widetilde{\boldsymbol{\kappa}})
\]
defined by
\begin{multline*}
K(z,w)\,dw \tens f(z) =
[\psi(z,w,\eta)\,dw] \otimes [u(z,\eta)]
\\ \mapsto 
\mu(K dw\tens f):=
[\smashoperator{\int\limits_{\hspace{6ex}\gamma(z,\eta;\varrho,\theta)} }
\psi(z,w,\eta)\,u(w,\eta)\, dw]
\end{multline*}
is well defined. 
Here $\widetilde{\boldsymbol{\kappa}}$ is  a $4$-tuple of  positive constants satisfying the conditions
given in Lemma \ref{lem:path-fundamental}. In particular, there exists the following linear morphism\textup{:}
\[
\mu^{}_K \colon  C^{\mathbb{R}}_{Y|X}(\boldsymbol{\kappa}) \ni f(z) \mapsto \mu(Kdw \tens f) \in   C^{\mathbb{R}}_{Y|X}(\widetilde{\boldsymbol{\kappa}}).
\]
\end{thm}
\begin{rem}
The same result holds for $\psi(z,w,\tau,\eta)dw$ and
$u(w,\tau,\eta)$ with additional holomorphic parameters $\tau$.
\end{rem}
\begin{proof}[Proof of Theorem \ref{th:action-morphism}]
For any $\varphi(z,w,\eta) \in  \varGamma(\widehat{W}^{(*,*)}_{\boldsymbol{\kappa}};\mathscr{O}_{\WXX})$, set
\[
\mu(\varphi)(z,\eta) :=
\smashoperator{\int\limits_{\hspace{6ex}\gamma(z,\eta;\varrho,\theta)}} \varphi(z,w,\eta) \,dw.
\]
Note that, by Lemma \ref{lem:path-fundamental} (1) we have
$\mu(\varphi)(z,\eta) \in 
\varGamma(\widehat{V}^{(*)}_{\widetilde{\boldsymbol{\kappa}}};\mathscr{O}_{\WX})$.

\begin{lem}\label{lem.1.16}
Assume that $\varphi(z,w,\eta) \in  \varGamma(\widehat{W}^{(\alpha,\,\beta)}_{\boldsymbol{\kappa}};\mathscr{O}_{\WXX})$ with
$\alpha \in \mathcal{P}^\vee_n$ and $\beta=*$ or with
$\alpha = *$ and $\beta \in \mathcal{P}^\vee_d$.
Then  $\mu(\varphi)(z,\eta) \in \varGamma(\widehat{V}^{(\beta)}_{\widetilde{\boldsymbol{\kappa}}};\mathscr{O}_{\WX})$
for some $\beta \in \mathcal{P}^\vee_d$.
\end{lem}
\begin{proof}
If
$
\varphi(z,w,\eta) \in
\varGamma(\widehat{W}^{(\alpha,*)}_{\boldsymbol{\kappa}};\mathscr{O}_{\WXX})$
for some $\alpha \in \mathcal{P}^\vee_n$,  we have
\[
\left\{
\begin{array}{ll}
	\mu(\varphi)(z,\eta) \in \varGamma(\widehat{V}^{(\{2,\dots,d\})}_{\widetilde{\boldsymbol{\kappa}}};\mathscr{O}_{\WX})
\qquad & (\alpha = \{2,\dots,n\}),\\
\mu(\varphi) (z,\eta)= 0 \qquad &(\text{otherwise}).
\end{array}
\right.
\]
Here we remark that the first fact comes from  Lemma \ref{lem:path-fundamental} (3) by
deforming the path of integration to $\overline{\gamma}(z,\eta;\varrho,\theta)$. In the same way, 
by Lemma \ref{lem:path-fundamental} (2), 
it follows that 
if $
\varphi(z,w,\eta) \in
\varGamma(\widehat{W}^{(*,\,\beta)}_{\boldsymbol{\kappa}};\mathscr{O}_{\WXX})$
for some $\beta \in \mathcal{P}^\vee_d$, 
\[
\mu(\varphi)(z,\eta) \in
\varGamma(\widehat{V}^{(\beta)}_{\widetilde{\boldsymbol{\kappa}}};\mathscr{O}_{\WX}).
\qedhere\]
\end{proof}
It follows from Lemma \ref{lem.1.16} that $\mu$ induces the canonical morphism
\[
\mu \colon \dfrac{\varGamma(\widehat{W}^{(*,*)}_{\boldsymbol{\kappa}};\OOO)}
{\smashoperator[r]{\Sum_{(\alpha,\,\beta) \in \varLambda}}
 \varGamma(\widehat{W}^{(\alpha,\,\beta)}_{\boldsymbol{\kappa}};\OOO)}
\to
\dfrac{
\varGamma(\widehat{V}^{(*)}_{\widetilde{\boldsymbol{\kappa}}};\mathscr{O}^{}_{\WX})}
{\smashoperator[r]{\Sum_{\beta \in  \mathcal{P}^\vee_d}}
\varGamma(\widehat{V}^{(\beta)}_{\widetilde{\boldsymbol{\kappa}}};\mathscr{O}^{}_{\WX})}
= \widehat{C}^{\,\mathbb{R}}_{Y|X}(\widetilde{\boldsymbol{\kappa}})
\]
where $\varLambda:=\{(\alpha,*);\,\alpha \in \mathcal{P}^\vee_n\}\sqcup 
\{(*,\,\beta);\,\beta \in \mathcal{P}^\vee_d\}$.
Furthermore, we have the canonical morphism
\[
E^{\mathbb{R}}_X(\boldsymbol{\kappa}) \underset{\mathbb{C}}{\otimes} C^{\mathbb{R}}_{Y|X}(\boldsymbol{\kappa})
\to
\dfrac{\varGamma(\widehat{W}^{(*,*)}_{\boldsymbol{\kappa}};\OOO)}
{\smashoperator[r]{\Sum_{(\alpha,\,\beta) \in \varLambda}}
 \varGamma(\widehat{W}^{(\alpha,\,\beta)}_{\boldsymbol{\kappa}};\OOO)}
\]
by $[\psi(z,w,\eta)\,dw] \otimes [u(z,\eta)] \mapsto [\psi(z,w,\eta)\,u(w,\eta)\,dw]$.
Hence we have obtained the morphism
\[
\mu \colon  E^{\mathbb{R}}_X(\boldsymbol{\kappa}) \underset{\mathbb{C}}{\otimes} C^{\mathbb{R}}_{Y|X}(\boldsymbol{\kappa})
\ni [\psi(z,w,\eta)\,dw] \otimes [u(z,\eta)]  \mapsto 
[\smashoperator{\int\limits_{\hspace{6ex}\gamma(z,\eta;\varrho,\theta)}}\psi(z,w,\eta)\,u(w,\eta)\,dw]
\in
\widehat{C}^{\,\mathbb{R}}_{Y|X}(\widetilde{\boldsymbol{\kappa}}).
\]
Thus to complete the proof, it suffices to show the image of $\mu$ is
contained in $C^{\mathbb{R}}_{Y|X}(\widetilde{\boldsymbol{\kappa}})$.
We have
\begin{align*}
\partial_\eta \smashoperator{\int\limits_{\hspace{6ex}\gamma(z,\eta;\varrho,\theta)}} 
\psi(z,w,\eta)\,u(w,\eta)\,dw
={}& \smashoperator{\int\limits_{\hspace{14ex}\gamma_2(z, \eta;\varrho) \times \dots \times \gamma_n(z, \eta;\varrho)}} 
\big[\tau \psi(z,z_1 + \tau\eta, w',\eta)\,
u(z_1 + \tau\eta, w',\eta)
\big]^{\beta^{}_1}_{\tau = \beta^{}_0}\,
dw_2 \cdots dw_n \\
& + \smashoperator{\int\limits_{\hspace{6ex}\gamma(z,\eta;\varrho,\theta)}}
\partial_\eta \psi(z,w,\eta)\, u(w,\eta)
\,dw
+ \smashoperator{\int\limits_{\hspace{6ex}\gamma(z,\eta;\varrho,\theta)} }
\psi (z,w,\eta)\,\partial_\eta u(w,\eta)
\,dw.
\end{align*}
By Lemma \ref{lem:path-fundamental} (3), the first term belongs to
$\varGamma(\widehat{V}^{(2,\dots,d)}_{\widetilde{\boldsymbol{\kappa}}};\mathscr{O}_{\WX})$.
For the second and third terms, as each integrand belongs to
$\smashoperator[r]{\Sum_{(\alpha,\,\beta) \in \varLambda}}
 \varGamma(\widehat{W}^{(\alpha,\,\beta)}_{\boldsymbol{\kappa}};\mathscr{O}^{}_{\WXX})$,
the corresponding integral also belongs to
$\smashoperator[r]{\Sum_{\beta \in  \mathcal{P}^\vee_d}}
\varGamma(\widehat{V}^{(\beta)}_{\widetilde{\boldsymbol{\kappa}}};\mathscr{O}^{}_{\WX})$.
Hence we have obtained
$\dfrac{\partial}{\partial \eta} \,\mu([\psi dw] \otimes [u]) = 0\in \widehat{C}^{\,\mathbb{R}}_{Y|X}(\widetilde{\boldsymbol{\kappa}})$, which implies
$\mu([\psi dw] \otimes [u])  \in C^{\mathbb{R}}_{Y|X}(\widetilde{\boldsymbol{\kappa}})$.
The proof is complete.
\end{proof}
As a corollary of the theorem, we have the result on
the composition on $E^{\mathbb{R}}_X(\boldsymbol{\kappa})$.
\begin{cor}
Let $\widetilde{\boldsymbol{\kappa}} = (\tilde{r},\tilde{r}',\tilde{\varrho}, \tilde{\theta}) \in \mathbb{R}^{4}$ satisfying
\[
0<\tilde{r} < r,\quad
0<\tilde{r}' < \dfrac{r'}{8},\quad 0<\tilde{\theta} < \dfrac{\,\theta\,}{4},\quad
0<\tilde{\varrho} < \dfrac{\,\varrho\,}{2} \sin\frac{\,\theta\,}{4},
\]
and the corresponding conditions to \eqref{eq:kappa-conditions}. Then 
there exists the bi-linear morphism
\[
\mu\colon  E^{\mathbb{R}}_X(\boldsymbol{\kappa}) \tens_{\mathbb{C}}E^{\mathbb{R}}_{X}(\boldsymbol{\kappa})
\to E^{\mathbb{R}}_{X}(\widetilde{\boldsymbol{\kappa}})
\]
defined by
\[
[\psi_1(z,w,\eta)\,dw] \otimes [\psi_2(z,w,\eta)\,dw] \mapsto \Bigl[\Bigl(
\smashoperator{\int\limits_{\hspace{6ex}\gamma(z,\eta;\varrho,\theta)}} 
\psi_1(z,\widetilde{w},\eta)\,\psi_2(\widetilde{w},w,\eta)\,d\widetilde{w}\Bigr)dw\Bigr].
\]
\end{cor}
\begin{proof}
By employing the coordinates transformation $z = z' + w$, $\widetilde{w} = \widetilde{w}'  + w$ and $w = w$, 
the  integration above becomes
\[
\smashoperator{\int\limits_{\hspace{6ex}\gamma(z',\eta;\varrho,\theta)} }
\psi_1(z'+w ,\widetilde{w}'+w,\eta)\,\psi_2(\widetilde{w}'+w, w,\eta)\,d\widetilde{w}'.
\]
Then, under the new coordinates $(z', \widetilde{w}', w)$,  
the $\psi_2$ (resp.\ the result of the integration) can be regarded a holomorphic microfunction along 
$\{\widetilde{w}' = 0\}$ (resp. $\{z' = 0\}$). Hence, by noticing the simple fact that
$|w^{}_i| < \dfrac{\,r'\,}{2}$ and $|\widetilde{w}'_i| < \dfrac{\,r'\,}{2}$ imply
$|\widetilde{w}_i| < r'$, we can easily obtain the result by the theorem.
\end{proof}

The following theorem can be shown by the same arguments as 
in  Kashiwara-Kawai \cite{K-K1}. We give 
the detailed proof in Appendix \ref{ap:actions} 
for the reader's convenience.
See also Theorem \ref{thm:generic-point-action} for the 
corresponding claim at an arbitrary point in $T^*X$.
\begin{thm}{\label{th:action-commutative}}
The action
\[
\mathscr{E}^{\mathbb{R}}_{X,z^*_0} \tens_{\mathbb{C}}
\mathscr{C}^{\mathbb{R}}_{Y|X,z^*_0}
=
\smashoperator{\varinjlim_{\boldsymbol{\kappa} }} (
E^{\mathbb{R}}_X(\boldsymbol{\kappa}) \tens_{\mathbb{C}} C^{\mathbb{R}}_{Y|X}(\boldsymbol{\kappa}))
\xrightarrow{\mu}
\smashoperator{\varinjlim_{\boldsymbol{\kappa}}} C^{\mathbb{R}}_{Y|X}(\boldsymbol{\kappa})
= \mathscr{C}^{\mathbb{R}}_{Y|X,z^*_0}
\]
coincides with the cohomological action 
of $\mathscr{E}^{\mathbb{R}}_{X,z^*_0}$ on $\mathscr{C}^{\mathbb{R}}_{Y|X,z^*_0}$.
\end{thm}
As an immediate corollary, we have:
\begin{cor}
The multiplication of the ring $\mathscr{E}^{\mathbb{R}}_{X,z^*_0}$ coincides with
the composition  defined by
\[
\mathscr{E}^{\mathbb{R}}_{X,z^*_0} \tens_{\mathbb{C}}
\mathscr{E}^{\mathbb{R}}_{X,z^*_0}
=
\smashoperator{\varinjlim_{\boldsymbol{\kappa}}}( E^{\mathbb{R}}_X(\boldsymbol{\kappa})
 \tens_{\mathbb{C}} E^{\mathbb{R}}_{X}(\boldsymbol{\kappa}))
\xrightarrow{\mu}
\smashoperator{\varinjlim_{\boldsymbol{\kappa}}} E^{\mathbb{R}}_{X}(\boldsymbol{\kappa})
= \mathscr{E}^{\mathbb{R}}_{X,z^*_0}.
\]
\end{cor}

\section{Symbols with an Apparent Parameter}\label{sec:symbol-with-ap}
Let $X:= \mathbb{C}^n$ and consider 
$ T^*X \simeq X \times \mathbb{C}^n =\{(z;\zeta)\}$.  
Let $\pi\colon T^*X\to X$ be  the canonical projection. 
If $V \subset T^*X$ is a conic set and   $d>0$,  we set 
\[
V[d]  :=\{(z;\zeta)\in V;\,\|\zeta\| \geqslant d \}.
\]
For any open conic subset 
 $\varOmega \subset T^*X$ and $\rho \geqslant 0$, we set
\[
   \varOmega^{}_ \rho
:= \Cl\Bigl[\smashoperator[r]{\Bcup_{(z,\zeta)\in\varOmega}}\{(z+z';\zeta+\zeta')\in \mathbb{C}^{2n};\,
                                  \|z'\|\leqslant \rho,\,\|\zeta'\|\leqslant \rho\|\zeta\|\}\Bigr].
\]
Here $\Cl$ means the closure. 
In particular, $\varOmega^{}_0= \Cl \varOmega$. 
For any $d>0$ and  $\rho \in \bigl[0,1\bigr[$,  we set for short:
\[
d^{}_{\rho}:=d(1-\rho).
\]
Let  $U$, $V$ be  conic subsets of $T^*X$.  
Then we write   $V \cset U$  if $V$ is generated  by  a compact subset of $U$. 
We recall the definition of symbols  of $\mathscr{E}^{\mathbb{R}}_X$ : 
\begin{defn}[see \cite{a2}, \cite{aky}]
Let $\varOmega\cset T^*X$ be an open conic  subset. 
\par
(1) We call  $P(z,\zeta)$ a \emph{symbol} on  $\varOmega$ if 
there exist   $d>0$ and  $\rho\in \bigl]0,1\bigr[$   such that  $
P(z,\zeta) \in \varGamma(\varOmega^{}_ \rho[d^{}_ \rho];\mathscr{O}^{}_{T^*X})$, 
and for any  $h >0$ there exists  $C_{h}>0$ such that 
\[
|P(z,\zeta)| \leqslant C_{h}e^{h\|\zeta\|} \quad ((z;\zeta) \in \varOmega^{}_ \rho [d^{}_ \rho]).
\]
We denote by  $\mathscr{S}(\varOmega)$ the set of symbols on $\varOmega$.
\par 
(2)  We call  $P(z,\zeta)$ a \emph{null-symbol} on  $\varOmega$ if 
there exist $d>0$ and  $\rho\in \big]0,1\bigr[$ such that  $
P(z,\zeta) \in \varGamma(\varOmega^{}_ \rho [d^{}_ \rho];\mathscr{O}^{}_{T^*X})$, and
there exist $C$, $\delta>0$ such that 
\[
|P(z,\zeta)| \leqslant Ce^{-\delta\|\zeta\|} \quad ((z;\zeta) \in \varOmega^{}_ \rho[d^{}_ \rho]).
\]
We denote by  $\mathscr{N}(\varOmega)$ the set of null-symbols on $\varOmega$.
\par
(3) For any $z^*_0\in T^*X$, we set
\[
\mathscr{S}^{}_{z^*_0}  :=  \smashoperator{ \varinjlim_{\varOmega\owns z^*_0}}
\mathscr{S}(\varOmega) \supset  
\mathscr{N}^{}_{z^*_0}  :=  \smashoperator{\varinjlim_{\varOmega\owns z^*_0}}\mathscr{N}(\varOmega)
\]
where  $\varOmega \cset T^*X$ ranges through open conic neighborhoods of $z^*_0$. 
\end{defn}
Next, set for short
\begin{equation}\label{S.eq1.1}
 S:= S_{\boldsymbol{\kappa}}
\end{equation}
for some  $r$, $\theta\in \bigl]0,\dfrac{1}{\,2\,}\bigr[$ (recall \eqref{eq:def-Gamma}). 
In particular we always assume that  $|\eta|< \dfrac{1}{\,2\,}$ for any $\eta \in  S$.  
For $Z \Subset S$, we set $m^{}_Z:= \min\limits_{\eta\in Z}|\eta|>0$.  
\begin{defn}
We define  a set $\mathfrak{N} (\varOmega; S)$ as follows:   $P(z, \zeta, \eta) \in \mathfrak{N} (\varOmega; S)$ if 
\begin{enumerate}[(i)]
\item $P(z, \zeta, \eta) \in \varGamma(\varOmega^{}_ \rho [d^{}_ \rho]\times S;
\mathscr{O}^{}_{T^*X \times \mathbb{C}})$ for some $d>0$ and  $\rho\in \bigl]0,1\bigr[$,
\item
there exists $\delta>0$ so that 
for any  $Z\Subset  S$, 
 there exists a constant $C_Z >0$
 satisfying
\begin{equation}\label{S.eq1.2a}
|P(z , \zeta, \eta)| \leqslant     C_Z e^{-\delta  \|\eta\zeta\|} \quad 
((z;\zeta,\eta) \in 
\varOmega^{}_ \rho[d^{}_ {\rho}] \times Z).
\end{equation}
\end{enumerate}
\end{defn}
\begin{lem}\label{S.lem1.3}
If  $P(z,\zeta,\eta)\in \mathfrak{N} (\varOmega; S)$, it follows that 
 $\partial_\eta  P(z,\zeta, \eta)\in \mathfrak{N} (\varOmega; S)$.
\end{lem}
\begin{proof}
We assume  that $P(z, \zeta, \eta) \in \varGamma(\varOmega^{}_ \rho [d^{}_ \rho]\times S;
\mathscr{O}^{}_{T^*X \times \mathbb{C}})$.
For any $Z \Subset  S$,  we take  $\delta'\in \big]0, \dfrac{\delta}{\,2\,}\big[$ as 
\begin{equation}\label{S.eq.1.3a}
Z':=\smashoperator{\Bcup_{\eta\in Z}}\{\eta'\in \mathbb{C};\, |\eta-\eta'|\leqslant \delta'|\eta|\}  \Subset  S.
\end{equation}
Then by the Cauchy inequality, there exists a  constant  $C_{Z'}>0$  such that 
\[
|\partial_\eta P(z,\zeta, \eta)| \leqslant 
\dfrac{1}{\delta'|\eta|}\smashoperator[r]{\sup_{|\eta-\eta'|= \delta'|\eta|}}
|P(z,\zeta, \eta')| \leqslant \dfrac{C_{Z'} e^{-\delta \|\eta\zeta\|/2}}{\delta'm^{}_Z}\quad 
((z;\zeta,\eta) \in \varOmega^{}_{\rho}[d^{}_ {\rho}] \times Z).
\qedhere
\]
\end{proof}
\begin{prop}\label{S.prop1.4}
Let $P(z, \zeta, \eta) \in \varGamma(\varOmega^{}_ \rho [d^{}_ \rho]\times S;
\mathscr{O}^{}_{T^*X \times \mathbb{C}})$. Assume 
that $\partial_\eta  P(z, \zeta, \eta)\in \mathfrak{N} (\varOmega; S)$.

$(1)$ The following conditions are equivalent\textup{:}
\begin{enumerate}[\rm (i)]
\item  there exists a constant  $\upsilon>0$ satisfying  the following\textup{:} 
for any    $Z\Subset  S$  there exists a constant  $C^{}_Z>0$  such that 
\begin{equation}\label{S.eq4.4a}
|P(z , \zeta, \eta)| \leqslant     C^{}_Z e^{\upsilon |\eta| \|\zeta\|}\quad ((z;\zeta,\eta) \in \varOmega^{}_ \rho[d^{}_{\rho}] \times Z).
\end{equation}
\item
for any  $h>0$ and  $Z\Subset  S$  there exists constant  $C_{h,Z}>0$  such that 
\[
|P(z , \zeta, \eta)| \leqslant     C_{h,Z} e^{h \|\zeta\|} \quad ((z;\zeta,\eta) \in \varOmega^{}_ \rho[d^{}_{\rho}] \times Z).
\]
\end{enumerate}

$(2)$ Assume that
 $P(z, \zeta, \eta)$ satisfies the equivalent conditions of \textup{(1)} \textup{(}resp.\ 
$P(z,\zeta,\eta)\in \mathfrak{N} (\varOmega; S)$\textup{)}. Then    for any  $\eta^{}_0 \in  S$, it follows that
$P(z,\zeta,\eta^{}_0)\in \mathscr{S} (\varOmega)$ \textup{(}resp.\ $P(z,\zeta,\eta^{}_0)\in \mathscr{N} (\varOmega)$\textup{)} 
and further
$P(z,\zeta,\eta)- P(z,\zeta,\eta^{}_0) \in \mathfrak{N} (\varOmega; S) $. 
\end{prop}
\begin{proof}
(1) (i) $\Longrightarrow$ (ii). 
For any $h>0$, we choose $\eta^{}_0\in  S \cap \mathbb{R}$ as $\upsilon \eta^{}_0< h$. Then  
there exists a  constant  $C_{\eta^{}_0}>0$  such that 
\[
|P(z ,\zeta,\eta^{}_0)| \leqslant     C_{\eta^{}_0} e^{ \upsilon \eta^{}_0  \|\zeta\|} 
\leqslant C_{\eta^{}_0}  e^{h\|\zeta\|} \quad ((z;\zeta) \in \varOmega^{}_ \rho[d^{}_{\rho}]).
 \]  
For any $Z \Subset  S$, let   $Z' \Subset  S$ be  the convex hull of 
 $Z \cup\{\eta^{}_0\}$. 
 Since
$\partial_\eta  P(z,\zeta,\eta) \in \mathfrak{N}(\varOmega; S)$, 
we can find   $\delta>0$ and a constant  $C_{Z'}>0$  such that 
for any $(z;\zeta,\eta)  \in \varOmega^{}_ \rho[d^{}_{\rho}] 
\times Z \subset \varOmega^{}_ \rho[d^{}_{\rho}] \times Z'$ the following holds:  
\begin{align*}
|P(z,\zeta,\eta)|& =|P(z, \zeta, \eta^{}_0) + \int_{\eta^{}_0}^{\eta}
\partial_\eta  P(z,\zeta,\tau)\,d\tau| 
\leqslant C^{}_{\eta^{}_0}e^{h\|\zeta\|}+ |\eta-\eta^{}_0|  C^{}_{Z'}e^{-\delta m^{}_{Z'}\|\zeta\|}
\\
&\leqslant (C^{}_{\eta^{}_0}+ r  C^{}_{Z'})e^{h \|\zeta\|}.
\end{align*}

(ii) $\Longrightarrow$ (i). 
For any $Z\Subset  S$, we take  $0<h \leqslant m^{}_ Z$. Then 
there exists   $C_{h,Z}>0$   such that 
\[
|P(z ,\zeta,\eta)| \leqslant C_{h,Z} e^{h  \|\zeta\|} \leqslant C_{h,Z}e^{|\eta| \|\zeta\|} 
\quad ((z;\zeta,\eta) \in \varOmega^{}_ \rho[d^{}_{\rho}] \times Z).
\]

(2)
Taking $Z=\{\eta_0\}$, we see that 
$P(z,\zeta,\eta^{}_0)\in \mathscr{S} (\varOmega)$ by (1). 
Set $\delta^{}_0:= \delta|\eta^{}_0|$. 
As in the proof of (i) $\Longrightarrow$ (ii) in (1), we see that  for any  $(z;\zeta,\eta) 
\in \varOmega ^{}_ \rho[d ^{}_ {\rho}]\times Z 
\subset \varOmega ^{}_ \rho[d ^{}_ {\rho}]\times Z' $ the following holds: if
 $|\eta| \geqslant |\eta^{}_0|$
\[
  |P(z,\zeta,\eta)-P(z,\zeta,\eta^{}_0)|
 = \Bigl| \int_{\eta^{}_0}^{\eta}
\partial_\eta  P _\nu(z,\zeta,\tau)\,d\tau\Bigr|
 \leqslant |\eta-\eta^{}_0|  
C_{h,Z'}e^{-\delta\|\eta^{}_0\zeta\|}
\leqslant  r C_{h,Z'}e^{-\delta^{}_0\|\eta\zeta\|},
\]
and if $|\eta|\leqslant  |\eta^{}_0|$
\[
  |P(z,\zeta,\eta)-P(z,\zeta,\eta^{}_0)|
\leqslant  r C_{h,Z'}e^{-\delta\|\eta\zeta\|}
\leqslant r C_{h,Z'}e^{-\delta^{}_0 \|\eta\zeta\|}.
\]
Hence $P(z,\zeta,\eta)- P(z,\zeta,\eta^{}_0) \in \mathfrak{N} (\varOmega; S) $. 
If $P(z,\zeta,\eta)\in \mathfrak{N} (\varOmega; S)$, the proof is same.
\end{proof}

\begin{defn}
(1) We define  a set $\mathfrak{S} (\varOmega; S)$ as follows: 
  $P(z, \zeta, \eta) \in  \mathfrak{S}(\varOmega; S)$  if   
\begin{enumerate}[(i)] 
\item $ P(z, \zeta, \eta) \in \varGamma(\varOmega^{}_ \rho [d^{}_ \rho]\times S;
\mathscr{O}^{}_{T^*X \times \mathbb{C}})$ for some $d>0$ and $\rho\in \bigl]0,1\bigr[$,
\item  $\partial_\eta  P(z, \zeta, \eta)\in \mathfrak{N} (\varOmega; S)$,
\item $P(z, \zeta, \eta)$ satisfies the equivalent conditions of Proposition \ref{S.prop1.4}.
\end{enumerate}
Note that $\mathfrak{N} (\varOmega; S)  \subset \mathfrak{S} (\varOmega; S) $ 
holds by Lemma \ref{S.lem1.3}.

(2) For $z^*_0\in \dot{T}^*X$, we set 
\[
\mathfrak{S}^{}_{z^*_0}
:= \smashoperator{\varinjlim_ {\varOmega,   S}}
\mathfrak{S} (\varOmega; S) \supset
\mathfrak{N}^{}_{z^*_0}  := \smashoperator{\varinjlim_ {\varOmega,   S}}
\mathfrak{N} (\varOmega; S). 
\]
Here  $\varOmega \cset T^*X$ ranges through open conic neighborhoods of $z^*_0$, 
and the inductive limits with respect to $ S$ are taken by $r$, $\theta\to 0$ 
in \eqref{S.eq1.1}.
\end{defn}
 We  call each element of $ \mathfrak{S} (\varOmega; S) $ (resp.\ 
$\mathfrak{N} (\varOmega; S)$) a \textit{symbol}  (resp.\ \textit{null-symbol}) 
\textit{on $\varOmega$ with an apparent parameter in $ S$}.
It is easy to see that $\mathfrak{S} (\varOmega; S)$ is  a $\mathbb{C}$-algebra under 
the ordinary operations of functions, and $\mathfrak{N} (\varOmega; S)$ is a subalgebra. 
By definition,  we can regard that   
\begin{align*} 
\mathscr{S}(\varOmega) & =\{ P(z,\zeta,\eta)\in 
\mathfrak{S}(\varOmega;  S);\, \partial_\eta  P(z,\zeta,\eta)
=0\} \subset \mathfrak{S}(\varOmega;  S),
\\  
\mathscr{N}(\varOmega) & =
\mathscr{S}(\varOmega) \cap  \mathfrak{N}(\varOmega;  S)
\subset \mathfrak{N}(\varOmega;  S). 
\end{align*}
 Hence we have an  injective 
mapping $\mathscr{S} (\varOmega)/\mathscr{N} (\varOmega) \hookrightarrow 
 \mathfrak{S} (\varOmega;  S) /\mathfrak{N} (\varOmega;  S)$.
Moreover

\begin{prop}
There exists  the following isomorphism\textup{:}
\[
\mathscr{S} (\varOmega)/\mathscr{N} (\varOmega)
\earrow \mathfrak{S} (\varOmega;  S) /\mathfrak{N} (\varOmega;  S).
\]\end{prop}
\begin{proof}

Let $P(z,\zeta,\eta)\in \mathfrak{S} (\varOmega; S)$. We   fix $\eta^{}_0 \in S$. 
Then by Proposition \ref{S.prop1.4}, we have  
$P(z,\zeta,\eta^{}_0) \in \mathscr{S}(\varOmega)$ and
 $[P(z,\zeta,\eta)] =[P(z,\zeta,\eta^{}_0)] \in \mathfrak{S} (\varOmega; S)  /\mathfrak{N} (\varOmega; S)$.
\end{proof}
\begin{defn}
We set  
\[
\wick{P(z,\zeta,\eta)}: = P(z,\zeta,\eta) \bmod 
\mathfrak{N}(\varOmega; S) \in \mathfrak{S}(\varOmega; S)
\big/\mathfrak{N}(\varOmega; S)
\]
which is called   the \textit{normal product} or the  \textit{Wick product} of $P(z,\zeta,\eta)$. 
\end{defn}

\section{Kernel Functions and Symbols}\label{sec:kernel-symbol}
In this section, we shall establish the correspondence of kernel functions and symbols.  
For this purpose, first we define two mappings that give the correspondence above.   
Set  
$z^*_0 = (z_0;\zeta_0) := (0;1,0,\dots,0)$. 
 Take any element $  K(z,w)\,dw = [ \psi(z,w,\eta)\,dw]
 \in \smash{\varinjlim\limits_{\boldsymbol{\kappa}} E^{\mathbb{R}}_X(\boldsymbol{\kappa})}
$.
Then a representative $\psi(z,z+w,\eta)$ of  $K(z,z+w) $  is holomorphic on
\[
\smashoperator[r]{\Bcap_{i=2}^n}\{(z,w,\eta) \in \mathbb{C}^{2n}\times  S;\,
\|z\| < r',\,  \dfrac{1}{\,\varrho\,}|\eta| <| w^{}_i| < r',\, |w_1 | < \varrho |\eta|,\,  \lvert\arg  w_1| < \dfrac{\,\pi\,}{2} + \theta\}.
\]
\begin{defn}\label{defsigma}
We set 
\[
\sigma(\psi)(z,\zeta,\eta):=\smashoperator{\int\limits_{\hspace{6ex}\gamma(0,\eta;\varrho,\theta)}}
 \psi(z,z+w,\eta) \,e^{\langle w, \zeta\rangle} dw.
\]
In Proposition  \ref{S.rem5.2} below, we show that $\sigma$ induces a mapping 
$\mathscr{E}^{\smash{\mathbb{R}}}_{X, z^*_0} \to \mathfrak{S} ^{}_{z^*_0}/\mathfrak{N}^{}_{z^*_0}$\,. 
\end{defn}
In order to construct the inverse of $\sigma$, we make full use of the following family of functions (see Laurent \cite[p.39]{L}): 
\begin{defn}\label{S.defn2.1}
We set
\[
\varGamma^{}_\nu(\tau,\eta) := 
\left\{\begin{aligned}
&1
& &(\nu=0),
\\
&\frac{1}{\,(\nu-1)!\,}
 \int_{0}^{\eta} \hspace{-1ex} e^{-s\tau} s^{\nu-1}ds& & (\nu\in \mathbb{N}).
\end{aligned}\right.
\]
\end{defn}

Let $z^*_0=(0;1,0,\dots,0)\in \dot{T}^*X$, 
and $P(z,\zeta,\eta)\in\mathfrak{S} ^{}_{z^* _0}$\,. By Proposition \ref{S.prop1.4}, for any sufficiently small $\eta^{}_0>0$ we have 
$P(z,\zeta,\eta)-P(z,\zeta,\eta^{}_0)\in \mathfrak{N} ^{}_{z^* _0}$\,. 
We may assume that 
$\|\zeta\|=|\zeta^{}_1|$ on a neighborhood  
of $z^*_0$. 
We develop  $P(z,\zeta,\eta^{}_0)$ into  the Taylor series with respect to
$\zeta'/\zeta^{}_1=(\zeta^{}_2/\zeta^{}_1,\dots, \zeta ^{}_n/\zeta ^{}_1)$:
\begin{equation}
\label{S.eq2.6q}
P(z,\zeta,\eta^{}_0)=\smashoperator[r]{\Sum_{\alpha\in{\mathbb{N}_0^{n-1}}}
}P^{}_{\alpha}(z,\zeta^{}_1,\eta^{}_0)
\biggl(\frac{\zeta'}{\zeta^{}_1}\biggr)^{\!\!\alpha}.
\end{equation}
Then we set 
$P^{\mathcal{B}}_{\alpha}(z,\zeta^{}_1,\eta):=P^{}_{\alpha}(z,\zeta^{}_1,\eta^{}_0)\,
\zeta^{|\alpha|}_1\varGamma^{}_{|\alpha|}(\zeta^{}_1,\eta)$ and 
\begin{equation}\label{S.eq2.8a}
P^{\mathcal{B}}(z,\zeta,\eta): =
\smashoperator{\Sum_{\alpha\in{\mathbb{N}_0^{n-1}}}}P^{\mathcal{B}}_{\alpha}(z,\zeta^{}_1,\eta)
\biggl(\frac{\zeta'}{\zeta^{}_1}\biggr)^{\!\!\alpha}.
\end{equation}
\begin{defn}\label{defpi}
Under the preceding notation, we set  
\begin{equation}
\label{S.eq2.12}
\begin{aligned}
\varpi^{}_\alpha(P)(z,w^{}_1,\eta)
:= {} &
\int_{d}^\infty \hspace{-1.5ex} P^{\mathcal{B}}_\alpha(z,\zeta^{}_1,\eta)
\,\dfrac{e^{- w^{}_1 \zeta^{}_1}}{\zeta^{|\alpha|}_1}\,d\zeta^{}_1 
\\
={}&\int_{d}^\infty \hspace{-1.5ex}P^{}_{\alpha}(z,\zeta^{}_1,\eta^{}_0)\,
\varGamma^{}_{|\alpha|}(\zeta^{}_1,\eta)\,e^{- w^{}_1 \zeta^{}_1}\,d\zeta^{}_1,
\end{aligned}
\end{equation}
and further  define 
\begin{equation}
\label{S.eq2.13}
\varpi(P)(z,z+w,\eta)
:= \Sum_{\alpha\in \mathbb N_0^{n-1}}
\frac{\,\alpha!\, \varpi^{}_\alpha(P)(z,w^{}_1,\eta)} 
{\,(2\pi\im\,)^{n}\,
(w')^{\alpha+ \mathbf{1}_{n-1}}\,}\,.
\end{equation}
Here we set  $w':=(w^{}_2,\dots,w^{}_n)$ and $\mathbf{1}^{}_{n-1}:=(1,\dots,1) $. 
In  Proposition \ref{welldefomga} below, we show that $\varpi$ induces a mapping 
$\mathfrak{S} ^{}_{z^*_0}/\mathfrak{N}^{}_{z^*_0} \to \mathscr{E}^{\smash{\mathbb{R}}}_{X, z^*_0}$\,. 
 \end{defn}

\begin{prop}\label{S.rem5.2}
The $\sigma$ in Definition \ref{defsigma} induces the linear  mapping
\[
\xymatrix @C=1em @R=.01ex{
\sigma\colon \mathscr{E}^{\smash{\mathbb{R}}}_{X, z^*_0} \ar@{}[d]|{\rotatebox{90}{$\in$}}\ar@{=}[r]& 
\varinjlim\limits_{\boldsymbol{\kappa}} E^{\mathbb{R}}_X(\boldsymbol{\kappa}) \ar@{}[d]|{\rotatebox{90}{$\in$}} \ar[r] &
\mathfrak{S} ^{}_{z^*_0}/\mathfrak{N}^{}_{z^*_0} \ar@{}[d]|{\rotatebox{90}{$\in$}}
\\   K(z,w)\,dw \ar@{=}[r]& [ \psi(z,w,\eta)\,dw]
\ar@{|->}[r] & \sigma(K)(z,\zeta)=
 [\sigma(\psi)(z,\zeta,\eta)].
}\]
The $\sigma$ does not depend on the choice of the path of the integration.
\end{prop}
We call $\sigma $  the \textit{symbol mapping}, and  
 $\sigma(K)$  the \textit{symbol} of $ K(z,w)\,dw\in \mathscr{E}^{\smash{\mathbb{R}}}_{X, z^*_0} $.  
\begin{proof}
We expand 
\[
\psi(z,z+w,\eta)= \Sum_{\alpha\in\mathbb{Z}^{n-1}}
  \frac{\,\psi^{}_{\alpha}(z,w^{}_1,\eta)\,}{(2\pi \im\,)^{n-1}\,(w')^{\alpha+ \mathbf{1}^{}_{n-1}}}
\]
If  $\alpha^{}_i+1\leqslant 0$  for some  $2 \leqslant i \leqslant n$, this term  is zero in 
$\smashoperator{\varinjlim\limits_{\boldsymbol{\kappa}}} E^{\mathbb{R}}_X(\boldsymbol{\kappa}) $, hence we may assume from the beginning that 
\begin{equation}\label{S.eq2.1}
\psi(z,z+w,\eta)= \Sum_{\alpha\in\mathbb{N}_0^{n-1}}
  \frac{\,\psi^{}_{\alpha}(z,w^{}_1,\eta)\,}{ (2\pi \im\,)^{n-1}\,(w')^{\alpha + \mathbf{1}^{}_{n-1}}}\,.
\end{equation}
Here 
\begin{equation}\label{S.eq2.2}
\psi^{}_{\alpha}(z,w^{}_1,\eta) :=\smashoperator{\oint\limits_{\hspace{14ex}| \widetilde{w}^{}_2|=c|\eta|,\dots,| \widetilde{w}^{}_n|=c|\eta|} }
\hspace{.5ex}
\psi(z,z^{}_1+w^{}_1,z'+\widetilde{w}',\eta)(\widetilde{w}')^{\alpha}\,d \widetilde{w}'
\end{equation}
for  $c>\dfrac{1}{\,\varrho\,}$. 
Hence  we may assume that $\psi(z,z+w,\eta) $ is holomorphic on
\[
\smashoperator[r]{\Bcap_{i=2}^n}\{(z,w,\eta) \in \mathbb{C}^{2n}\times  S;\,
\|z\| < r',\, |w_1 | <\varrho |\eta|,\, \vert\arg  w_1|< \dfrac{\,\pi\,}{2} + \theta,\,\dfrac{|\eta| }{\, | w_i|\,} <\varrho\}, 
\]
Take  $c'>0$, an open conic neighborhood 
  $\varOmega \cset T^*X$ and $0<\rho\ll 1 $ as
\[
\varOmega^{}_{\rho}\; \cset\;
\{(z,\zeta) \in \mathbb{C}^{2n};\, 
\|z\| \leqslant r',\, \|\zeta'\| \leqslant c'|\zeta^{}_1|,\, \lvert\arg\zeta_1| \leqslant r''\}.
\]
Taking $c'$ small enough, we can assume that $\|\zeta\| = |\zeta_1|$ on $\varOmega^{} _{\rho} $. 
We have
\begin{align*}
\sigma(\psi)(z,\zeta,\eta)={}& \int\limits_{\gamma^{}_1(0,\eta;\varrho,\theta)}\hspace{-3ex} dw^{}_1\oint\limits_{\gamma^{}_2(0,\eta;\varrho)}
\hspace{-1ex}\dotsi\hspace{-1ex}\oint\limits_{\gamma^{}_n(0,\eta;\varrho)}
    \Sum_{\alpha\in\mathbb{N}_0^{n-1}}
  \frac{\,\psi^{}_{\alpha}(z,w^{}_1,\eta)\,}{ (2\pi \im\,)^{n-1}\,(w')^{\alpha +\mathbf{1}^{}_{n-1}}} \,
              e^{\langle w,\zeta\rangle}dw'
\\
={}  & \smashoperator[l]{\Sum_{\alpha\in \mathbb {N}_0^{n-1}}}
            \smashoperator{\int\limits_{\hspace{6ex}\gamma^{}_1(0,\eta;\varrho,\theta)}}
                     \psi^{}_{\alpha}(z,w^{}_1,\eta) \,\frac{\, e^{w^{}_1\zeta^{}_1}\,}{\,\alpha!}\,
              \partial_{w'}^{\,\,\alpha}  e^{\langle w',\zeta'\rangle}\big|{}^{}_{w'=0}\, dw^{}_1
\\
=   {}      &  \smashoperator[l]{\Sum_{\alpha\in \mathbb {N}_0^{n-1}}}
             \frac{(\zeta')^{\alpha}}{\alpha!}
            \smashoperator{\int\limits_{\hspace{6ex}\gamma_1(0,\eta;\varrho,\theta)}}
             \psi^{}_{\alpha}(z,w^{}_1,\eta) \,    e^{w^{}_1 \zeta^{}_1}dw^{}_1  \,.
\end{align*}
We can change 
 $\gamma_i(0, \eta;\varrho)= \{w_i = |\eta|s' e^{2\pi \iim t};\, 0 \leqslant t \leqslant 1\}$ with 
$0<\varrho^{-1}<s'$. 
Deforming $\gamma_1(0, \eta;\varrho,\theta)$, we can see  that for any $h>0$ we have $e^{\Re\langle w^{}_1, \zeta^{}_1\rangle} \leqslant 
e^{h|\eta\zeta^{}_1|}$ holds if $\lvert \arg\zeta_1| <r''$ and $w^{}_1 \in \gamma^{}_1(0,\eta;\varrho,\theta) $. 
Thus 
we have
\begin{align*}
\bigl|e^{\langle w,\zeta\rangle}\bigr| = e^{\Re \langle w,\zeta\rangle} &\leqslant  \exp \bigl(\Re \langle w^{}_1,\zeta^{}_1\rangle + \smashoperator{ \Sum_{i=2}^n} |w^{}_i\,\zeta^{}_i|\bigr)
\\& 
\leqslant  e^{h|\eta\zeta^{}_1| + (n-1) c's'|\eta\zeta^{}_1| } 
=  e^{(h+(n-1)c's')\|\eta\zeta\|}.
\end{align*}
Fix $d>0$. 
Take $Z\Subset  S$. 
Then there exists a constant $C>0$ such that
\[
\Big| \smashoperator{\int\limits_{\hspace{6ex}\gamma(0,\eta;\varrho,\theta)}}\psi(z,z+w,\eta) \,e^{\langle w, \zeta\rangle} dw
\Big|\leqslant  Ce^{(h+(n-1)c's')\|\eta\zeta\|}\quad ((z;\zeta,\eta) \in 
\varOmega^{}_{\rho}[d^{}_\rho] \times Z),
\]
that is,   we can see that $\sigma(\psi)(z,\zeta,\eta) \in \varGamma(\varOmega^{}_ \rho [d^{}_ \rho]\times S;
\mathscr{O}^{}_{T^*X \times \mathbb{C}})$ and satisfies  \eqref{S.eq4.4a}.
If $[\psi(z,w,\eta)\,dw]= 0\in \smashoperator{\varinjlim\limits_{\boldsymbol{\kappa}}} E^{\mathbb{R}}_X(\boldsymbol{\kappa}) $, 
we may assume that there is $\delta'>0$ such that $|e^{\langle w^{}_1, \zeta^{}_1\rangle}| \leqslant 
e^{-\delta'|\eta\zeta^{}_1|}$ holds if $\lvert \arg\zeta_1| <r''$ 
and $w^{}_1 \in \overline{\gamma}^{}_1(0,\eta;\varrho,\theta) $. We choose $c'$ so small that  
 $\delta :=\delta'-(n-1)c's'   >0$.
Then there exists a constant $C>0$ such that
\[
\Big|\smashoperator{\int\limits_{\hspace{6ex}\gamma(0,\eta;\varrho,\theta)}}\psi(z,z+w,\eta) \,e^{\langle w, \zeta\rangle} dw
\Big|\leqslant C e^{((n-1)c's'-\delta')|\eta\zeta^{}_1|} 
\leqslant C e^{-\delta\|\eta\zeta\|}\quad 
((z;\zeta,\eta) \in 
\varOmega^{}_{\rho}[d^{}_ \rho]  \times Z),
\]
that is,   $\sigma(\psi)(z,\zeta,\eta) \in \mathfrak{N}^{}_{z^*_0}$. 
Further 
 we can prove that 
\[ \smashoperator{\int\limits_{\hspace{6ex}\gamma(0,\eta;\varrho,\theta)}}\psi(z,z+w,\eta) \,e^{\langle w, \zeta\rangle} dw
- \smashoperator{\int\limits_{\hspace{7ex}\gamma(0,\eta;\varrho ^{}_1,\theta^{}_1)}}\psi(z,z+w,\eta) \,e^{\langle w, \zeta\rangle} dw \in 
\mathfrak{N}^{}_{z^*_0}. 
\] 
Note that
\begin{align*}
\partial_\eta \sigma(\psi) (z,\zeta,\eta)
= {} &\smashoperator{\int\limits_{\hspace{15ex}\gamma_2 (0, \eta;\varrho)\times \dots \times \gamma_n(0, \eta;\varrho)}}
\hspace{1ex}
\big[\tau \psi(z, z_1+\tau\eta, z'+w',\eta)\, e^{\tau\eta  \zeta^{}_1+ \langle w', \zeta'\rangle} 
\big]^{ \beta_1}_{\tau = \beta_0}\,
dw_2 \cdots dw_n
\\
&  + \smashoperator{\int\limits_{\hspace{6ex}\gamma(0, \eta;\varrho,\theta)}}
\partial_\eta  \psi(z,z+w,\eta)\, e^{\langle w, \zeta\rangle} 
\,dw.
\end{align*}
Since $\partial_\eta \psi(z,w,\eta)$ is a zero class,  $\partial_\eta \sigma(\psi)(z,\zeta,\eta) \in 
\mathfrak{N}^{}_{z^*_0}$.  
Thus we see that $\sigma(\psi)(z,\zeta,w)\in \mathfrak{S}^{}_{z^*_0}$  and   $\sigma$  is well defined.
\end{proof}

\begin{prop} \label{welldefomga}
The   $ \varpi$ in Definition \ref{defpi} induces the linear mapping 
\begin{equation}\label{S.eq2.14}
\varpi\colon\mathfrak{S}^{}_{z^* _0} /\mathfrak{N}^{}_{z^* _0} 
\ni \wick{P(z,\zeta,\eta)}\mapsto \varpi(\wick{P}):= [\varpi(P)(z,w,\eta)\,dw] \in \mathscr{E}^{\smash{\mathbb{R}}}_{X, z^*_0}\,.
\end{equation}
This mapping is independent of the  choice of either $\eta_0$ or the path of the integration. 
\end{prop}
We call $\varpi(\wick{P}) $  the \textit{kernel} of  $\wick{P} \in \mathfrak{S}^{}_{z^* _0} /\mathfrak{N}^{}_{z^* _0} $.  
\begin{proof}
We need the following estimate to prove that $\varpi$  is  well defined:
\begin{lem}
Assume that $\Re (\eta \tau) \geqslant 2\delta^{}_0 |\eta\tau|>0$ for some $\delta^{}_0 \in \bigl]0, \dfrac{1}{\;2\;}\bigr[$\,. Then 
for any $\nu\in \mathbb{N}$,
\begin{align}\label{S.eq2.6}
|\varGamma^{}_\nu(\tau,\eta)| & \leqslant \frac{\,|\eta|^{\nu}\,}{\,\nu!\,},
\\
\label{S.eq2.5}
|1- \tau^\nu\varGamma^{}_\nu(\tau,\eta)| &
      \leqslant \dfrac{\,e^{-\delta^{}_0|\eta\tau|}\,}{\delta_0^{\nu-1}}.
\end{align}
\end{lem}
\begin{proof}
We have \eqref{S.eq2.6} as follows:
\[
|\varGamma^{}_\nu(\tau,\eta)| \leqslant \dfrac{1}{(\nu-1)!}\int_{0}^{|\eta|}\hspace{-1ex}
|  e^{-s \tau}s^{\nu-1}|\,ds \leqslant  \dfrac{1}{(\nu-1)!}\int_{0}^{|\eta|}\hspace{-1ex}
s^{\nu-1}\,ds  = \frac{\,|\eta|^{\nu}\,}{\,\nu!\,}.
\]
By   the definition of $\varGamma$-function and   induction on $\nu$, we have
\[
1- \tau^\nu\varGamma^{}_\nu(\tau,\eta)   =   
        \frac{\tau^{\nu}}
       {\,(\nu-1)! \,} \int_{\eta}^{\infty}\hspace{-2ex}e^{- s \tau}s^{\nu-1}\,ds=
\smashoperator{\Sum_{k=0}^{\nu-1}}
\dfrac{\,(\eta\tau)^k\,}{k!}\,e^{-\eta\tau}.
\]
Therefore,  we have
\begin{align*}
|1- \tau^\nu\varGamma^{}_\nu(\tau,\eta)| 
& = \bigl|\smashoperator[r]{\Sum_{k=0}^{\nu-1}}
\dfrac{\,(\delta_0\eta\tau)^k\,}{k!\delta_0^{\;k} }\,e^{-\eta\tau}\bigr|
\leqslant  
\smashoperator{\Sum_{k=0}^{\nu-1}}
\dfrac{\,(\delta_0|\eta\tau|)^k\,}{k!}\,  \dfrac{e^{-2\delta^{}_0|\eta\tau|} }{\,\delta_0^{\nu-1}}
\\&     
 \leqslant e^{\delta^{}_0|\eta\tau|} \dfrac{e^{-2\delta^{}_0|\eta\tau|} }{\,\delta_0^{\nu-1}}
   = \dfrac{\,e^{-\delta^{}_0|\eta\tau|}\,}{\delta_0^{\nu-1}}.
\qedhere
\end{align*}
\end{proof}
Recall \eqref{S.eq2.6q} and \eqref{S.eq2.8a}. 
There exist sufficiently small $r^{}_0$, $\theta'>0$ and sufficiently large $d>0$ such that $P^{}_{\alpha}(z,\zeta^{}_1,\eta^{}_0)$ is
holomorphic on a common neighborhood of $D$ for each $\alpha \in \mathbb{N}_0^{n-1}$, where 
\[
D:=\{(z,\zeta^{}_1)\in  \mathbb{C}^{n+1};\,\|z\|\leqslant r^{}_0,\, \lvert\arg\zeta^{}_1| \leqslant \theta',\,
 |\zeta^{}_1|\geqslant d \}.
\]
It follows from the Cauchy inequality
that we can take  a constant  $K>0$ 
so that for each  $h>0$  there exists $C^{}_{h}>0$ such that 
for every $\alpha \in \mathbb{N}_0^{n-1}$,
\begin{equation}
\label{S.eq2.8}
|P^{}_{\alpha}(z,\zeta^{}_1,\eta^{}_0)|
\leqslant
C^{}_{h} K^{|\alpha|}e^{h|\zeta^{}_1|} \quad ((z,\zeta^{}_1) \in D).
\end{equation} 
We take  $\delta^{}_0 \in \bigl]0, \dfrac{1}{\;2\;}\bigr[$ as 
$\Re (\eta \zeta^{}_1) \geqslant 2\delta^{}_0 |\eta\zeta^{}_1|>0$ if $\eta\in S$ and $\lvert\arg\zeta^{}_1| \leqslant \theta'$. 
Take $\varepsilon>0$ as $0<\dfrac{K\varepsilon}{\delta^{}_0} < \dfrac{1}{\,2\,}$. For any  $Z \Subset S$,  
 we chose  $h=\dfrac{\delta^{}_0m^{}_Z}{2}$. Then  by
\eqref{S.eq2.5},  for $(z,\zeta^{}_1) \in D \times Z$ and $|\zeta^{}_i| \leqslant \varepsilon |\zeta^{}_1|$ $(2 \leqslant i \leqslant n)$ 
we have
\begin{equation}
\label{S.eq2.11}
\begin{split}
|P(z,\zeta,\eta^{}_0)- P^{\mathcal{B}}(z,\zeta,\eta)| &= 
\Big|\smashoperator{\Sum_{|\alpha|=1}^\infty}P^{}_{\alpha}(z,\zeta^{}_1,\eta^{}_0)
\,(1- \zeta_1^{|\alpha|}\varGamma^{}_{|\alpha|}(\zeta^{}_1,\eta))
\biggl(\frac{\zeta'}{\zeta^{}_1}\biggr)^{\!\!\alpha}
\Bigr|
\\
&
 \leqslant \delta^{}_0C^{}_{h} e^{-\delta^{}_0|\eta\zeta^{}_1|/2}\smashoperator{\Sum_{|\alpha|=1}^\infty} \,
\Bigl(\dfrac{K\varepsilon}{\delta^{}_0}\Bigr)^{\!|\alpha|}\leqslant 2^{n-1}\delta^{}_0C^{}_{h}e^{-\delta^{}_0|\eta\zeta^{}_1|/2},
\end{split}
\end{equation}
where we remark that 
$\#\{\alpha\in \mathbb{N}_0^{n-1};\,|\alpha|=i\}=\dbinom{n+i-2}{i} \leqslant  2^{n+i-2}$. 
Therefore we see that $P(z,\zeta,\eta)-P^{\mathcal{B}}(z,\zeta,\eta)=
P(z,\zeta,\eta)-P(z,\zeta,\eta^{}_0)+P(z,\zeta,\eta^{}_0)-P^{\mathcal{B}}(z,\zeta,\eta)\in \mathfrak{N} ^{}_{z^* _0}$\,. 
Further by \eqref{S.eq2.6} and \eqref{S.eq2.8}, 
 there exists a constant  $K>0$ 
so that for each  $h>0$ there exists $C^{}_{h}>0$ such that 
for every $\alpha \in \mathbb{N}_0^{n-1}$ and   $(z,\zeta^{}_1,\eta) \in D \times  S$, we have
\begin{equation}
\label{S.eq2.10}
\dfrac{|P^{\mathcal{B}}_\alpha(z,\zeta^{}_1,\eta)|}{|\zeta^{}_1|^{|\alpha|}}
\leqslant
\dfrac{C^{}_{h} (K|\eta|)^{|\alpha|}e^{h|\zeta^{}_1|} }{|\alpha|!}.
\end{equation}

We can take a sufficiently small $\delta^{}_1$, $\delta'>0$ such that   
\[
\{w^{}_1\in \mathbb{C};\,\lvert \arg w^{}_1|<\delta'+\dfrac{\,\pi\,}{2}\} 
\subset  \;\smashoperator{\Bcup_{ \lvert\arg\zeta^{}_1| \leqslant\theta'}}\;\{w^{}_1\in \mathbb{C};\,
\Re(w^{}_1 \zeta^{}_1)\geqslant \delta^{}_1|w^{}_1 \zeta^{}_1|\},
\]
and we set
\begin{equation}\label{S.eq2.16}
L:=\{(z,w^{}_1)\in \mathbb{C}^{n+1};\, \|z\|< r^{}_0,\,\lvert \arg w^{}_1|<\delta'+\dfrac{\,\pi\,}{2} \}.
\end{equation}
By \eqref{S.eq2.10},  for any $k\in \mathbb{N}$  there exists $C^{}_{k}>0$ such that 
for every $\alpha \in \mathbb{N}_0^{n-1}$,
\begin{equation}
\label{S.eq2.17}
\dfrac{|P^{\mathcal{B}}_\alpha(z,\zeta^{}_1,\eta)|}{|\zeta^{}_1|^{|\alpha|}}
\leqslant
\dfrac{C^{}_{k} (K|\eta|)^{|\alpha|}e^{\delta^{}_1|\zeta^{}_1|/k} }{|\alpha|!}
\quad ((z,\zeta^{}_1,\eta) \in D \times  S).
\end{equation}
By changing the direction of the integration in the complex $\tau$-plane,
$\varpi^{}_\alpha(P)(z,w^{}_1,\eta)$ extends analytically  to the domain $L \times  S$. Set 
\begin{equation}\label{S.eq2.18}
L^{}_{k}:=\{(z,w^{}_1)\in \mathbb{C}^{n+1};\, \|z\|< r^{}_0,\,\lvert \arg w^{}_1| < \delta'+\dfrac{\,\pi\,}{2},\, 
\dfrac{2}{\,k\,} <|w^{}_1|  \}.
\end{equation}
Then by \eqref{S.eq2.17} and \eqref{S.eq2.6}  for any $\eta\in S$ we have 
\begin{equation}
\label{S.eq2.19}
\sup\{|\varpi^{}_\alpha(P)(z,w^{}_1,\eta)|;\,(z,w^{}_1)\in L^{}_{k}\}
\leqslant
\dfrac{2kC^{}_{k}}{\, \delta^{}_1 |\alpha|!\,}\, (K|\eta|)^{|\alpha|}.
\end{equation}
Therefore the right-hand side of \eqref{S.eq2.13} converges locally uniformly in
\begin{equation}
\label{S.eq2.20}
V^{}_{k}
:= \smashoperator[r]{\Bcap_{2\leqslant i \leqslant n}}\{(z,w,\eta)\in\mathbb
C^{2n}\times  S;\, (z,w^{}_1)\in L^{}_k,\, K |\eta|<|w^{}_i|\,
 \}.
\end{equation}
 Hence $\varpi(P)(z,z+w,\eta)$ is a holomorphic function
 defined on the set
\begin{equation}
\label{S.eq2.21}
V:= \smashoperator{\Bcup _{k=1}^{\infty}}
V^{}_{k}
= \smashoperator[r]{\Bcap_{2\leqslant i\leqslant n}}\{(z,w,\eta)\in \mathbb{C}^{2n} \times  S;\, 
\|z\|<r^{}_0,\,\lvert \arg w^{}_1| < \delta'+\dfrac{\,\pi\,}{2},  \,  K|\eta| < |w^{}_i|\}. 
\end{equation}
Next,   we have
\[
\partial^{}_{\eta}\varpi^{}_\alpha(P)(z,w^{}_1,\eta)
=\frac{1}{\,(|\alpha|-1)!\,}\int_{d}^\infty\hspace{-2ex} P^{}_\alpha(z,\zeta^{}_1,\eta^{}_0)\,e^{-(\eta+w^{}_1)\zeta^{}_1} \eta^{|\alpha|-1}d\zeta^{}_1\,.
\]
Let $Z\Subset S$. 
Then choosing $h = \delta^{}_0m^{}_Z$ in \eqref{S.eq2.8},  for 
$\|z\|<r^{}_0$,  $\eta\in Z$ and $|w^{}_1|<\dfrac{\delta^{}_0 |\eta|}{2}$, we have
\allowdisplaybreaks
\begin{align*}
|\partial^{}_{\eta}\varpi^{}_\alpha(P)(z,w^{}_1,\eta)|& \leqslant 
\frac{C^{}_{\delta^{}_0m^{}_Z} (K|\eta|)^{|\alpha|}}{\,(|\alpha|-1)!\,}\int_{d}^\infty\hspace{-1ex} 
e^{ \delta^{}_0m^{}_Z|\zeta^{}_1|-\Re(\eta\zeta^{}_1)+|w^{}_1\zeta^{}_1|}d|\zeta^{}_1|
\\
& \leqslant 
\frac{C^{}_{\delta^{}_0m^{}_Z} (K|\eta|)^{|\alpha|}}{\,(|\alpha|-1)!\,}\int_{d}^\infty\hspace{-1ex} 
e^{-(\delta^{}_0|\eta|-|w^{}_1|)|\zeta^{}_1|}d|\zeta^{}_1|
\leqslant\frac{2C^{}_{\delta^{}_0m^{}_Z} (K|\eta|)^{|\alpha|}}{\delta^{}_0 |\eta|\,(|\alpha|-1)!\,}.
\end{align*}
Hence 
$\partial^{}_{\eta}\varpi(P)(z,z+w,\eta)$  
 is  holomorphic on 
\begin{multline*}
\Bcup _{Z\Subset S}\smashoperator[r]{\Bcap_{2\leqslant i\leqslant n}}\{(z,w,\eta)\in \mathbb{C}^{2n} \times  Z;\, 
\|z\|<r^{}_0,  \,|w^{}_1|<\dfrac{\delta^{}_0 |\eta|}{2},\,  K|\eta| < |w^{}_i|\}
\\
=\smashoperator[r]{\Bcap_{2\leqslant i\leqslant n}}\{(z,w,\eta)\in \mathbb{C}^{2n} \times  S;\, 
\|z\|<r^{}_0,  \,|w^{}_1|<\dfrac{\delta^{}_0 |\eta|}{2},\,  K|\eta| < |w^{}_i|\}.
\end{multline*}
This entails that $[\partial^{}_{\eta}\varpi(P)(z,w,\eta)\,dw]=0 \in \mathscr{E}^{\mathbb{R}}_{X,z^*_0}$\,. 
If  $P(z,\zeta,\eta)\in \mathfrak{N}^{}_{z^*_0}$, 
there exists  a constant  $\delta$, $C$, $K>0$ 
so that 
for every $\alpha \in \mathbb{N}_0^{n-1}$,
\[
|P^{}_{\alpha}(z,\zeta^{}_1,\eta^{}_0)|
\leqslant
C K^{|\alpha|}e^{-\delta|\zeta^{}_1|} \quad ((z,\zeta^{}_1) \in D).
\]
 Thus if $|w^{}_1|< \dfrac{ \,\delta\,}{2}$, we have
\[
|\varpi^{}_\alpha(P)(z,w^{}_1,\eta)|
\leqslant \frac{\,C(K|\eta|)^{|\alpha|}\,}{\,|\alpha|!\,}
\int_{d}^\infty \hspace{-1.5ex} e^{-(\delta-|w^{}_1|)|\zeta^{}_1|}d|\zeta^{}_1| 
\leqslant \frac{\,2C(K|\eta|)^{|\alpha|}\,}{\,\delta|\alpha|!\,}\,.
\]
Thus $\varpi(P)(z,z+w,\eta)$   is  holomorphic on 
\[
\smashoperator[r]{\Bcap_{2\leqslant i\leqslant n}}\{(z,w,\eta)\in \mathbb{C}^{2n} \times  S;\, 
\|z\|<r^{}_0,  \,|w^{}_1|<\dfrac{\,\delta\,}{2},\,  K|\eta| < |w^{}_i|\},
\]
hence $[\varpi(P)(z,w,\eta)\,dw]=0$. If we change $\eta^{}_0$ or  $d$ in \eqref{S.eq2.12}, 
for the same reasoning as above, we see that $[\varpi(P)(z,w,\eta)\,dw]=0$. Therefore 
we obtain a well-defined linear mapping \eqref{S.eq2.14}.\end{proof}

Now we shall prove our fundamental theorem for the symbol theory:

\begin{thm} \label{S.thm2.1} 
The mappings $\sigma$ and $\varpi$ are   inverse to each other. 
In particular 
\[
\sigma\colon \mathscr{E}^{\smash{\mathbb{R}}}_{X, z^*_0}\earrow  \mathfrak{S} ^{}_{z^* _0} /\mathfrak{N} ^{}_{z^* _0}\, .
\] 
\end{thm}
\begin{proof}
We may assume that $z^*_0=(0;1,0,\dots,0)$, and we may also assume $\|\zeta\|=|\zeta^{}_1|$ on a neighborhood  
of $z^*_0$ in the course of proof. 
\paragraph{\textbf{Step 1.}}
We shall show $\sigma\cdot\varpi=\mathds{1} \colon
\mathfrak{S}^{}_{z^*_0}/\mathfrak{N}^{}_{z^*_0} 
\rightarrow \mathfrak{S}^{}_{z^*_0}/\mathfrak{N}^{}_{z^*_0}$. Let $P(z,\zeta,\eta)\in
\mathfrak{S}^{}_{z^*_0}$. 
Assume that  $P^{\mathcal{B}}(z,\zeta,\eta)\in \mathfrak{S}_{z^*_0}$ is holomorphic on a neighborhood of
\[
\widetilde{V}
:= \smashoperator[r]{\Bcap_{2 \leqslant i\leqslant n}}\{(z,\zeta,\eta)\in \mathbb C^{2n} \times  S;\, \|z\| < r^{}_0,\, 
|\zeta^{}_1|\geqslant d,\,d\lvert \arg \zeta^{}_1| \leqslant 1,\,d|\zeta^{}_i|
\leqslant |\zeta^{}_1| \}. 
\]
 By the definition of $\varpi$, we have
\allowdisplaybreaks
\begin{align*}
\sigma \cdot \varpi(P)(z,\zeta,\eta) & = \smashoperator{\int\limits_{\hspace{6ex}\gamma(0, \eta;\varrho,\theta)}} 
\varpi(P)(z,z+w, \eta)\,e^{\langle w,\zeta \rangle}dw
\\
&
 = 
  \int\limits_{\gamma(0, \eta;\varrho,\theta) } \hspace{-1ex}dw\smashoperator[l]{\Sum_{\alpha\in \mathbb {N}_0^{n-1}}}\!
  \frac{\alpha!\,e^{\langle w,\zeta \rangle}}
       {\,(2\pi\im\,)^{n}\,
          (w')^{\alpha+\mathbf{1}^{}_{n-1}}}
\int_{d}^\infty  \hspace{-1.5ex} P^{\mathcal{B}}_\alpha(z,\xi^{}_1,\eta)\,
\dfrac{\,e^{- w^{}_1\xi^{}_1}}{\xi^{|\alpha|}_1}\,d\xi^{}_1
              \\
&
= \smashoperator[r]{\Sum_{\alpha\in \mathbb{ N}_0^{n-1}}}\,  (\zeta')^{\alpha}
   \int\limits_{\gamma^{}_1(0, \eta;\varrho,\theta) } \hspace{-2ex}dw^{}_1\,\frac{\,e^{w^{}_1\zeta^{}_1}} {2\pi\im}
\int_{d}^\infty \hspace{-2ex}   P^{\mathcal{B}}_\alpha(z,\xi^{}_1,\eta)\,
\dfrac{\,e^{- w^{}_1\xi^{}_1}}{\xi^{|\alpha|}_1}\,d\xi^{}_1\,.
\end{align*}
We   set 
\[
\widetilde{V}_{\varepsilon}
:= \smashoperator[r]{\Bcap_{i=2}^n}\{(z,\zeta)\in \mathbb C^{2n};\, \|z\| < r^{}_0,\,
|\zeta^{}_1|\geqslant \frac{d}{\, \varepsilon\,},\,
\lvert\arg\zeta^{}_1| \leqslant \varepsilon,\,
|\zeta^{}_i|
\leqslant \varepsilon |\zeta^{}_1|\}. 
\]
\begin{figure}[t]
\centering
{\footnotesize
\begin{picture}(450,180)(-100,-90)
{\thicklines \bezier{100}(19,14)(25,9)(25,0)}
{\thicklines\dottedline{4}(35,14)(90,35)}
{\thicklines\dottedline{4}(35,-14)(90,-35)}
{\thicklines \bezier{7}(35,14)(45,0)(35,-14)}
\put(75,15){$\lvert \arg\zeta^{}_1|\leqslant \varepsilon$ }
\put(-20,90){$\Im w^{}_1$ }
\put(90,-3){$\Re w^{}_1$}
\put(42,6){$d/\varepsilon$}
\put(37,-3){$\bullet$}
\put(20,-10){$d$}
\put(16,24){$d^+$}
\put(-80,0){\vector(1,0){165}}
\put(-11,-80){\vector(0,1){165}}
\put(0,75){$\Re(w^{}_1\xi^{}_1)=0$ $(|\xi^{}_1|>d)$}
\put(-13,0){ \line(2,1){30}}
\put(19,15){\thicklines \vector(2,1){70}}
\put(16,12){$\bullet$}
\put(22,-3){$\bullet$}
\put(45,50){$\xi^{}_1\in\varSigma^{}_{+}$}
\put(92,50){$\infty$}
\put(-65,-81){\line(2,3){54}}
\put(-11,0){\line(-2,3){54}}
\put(-52,-82){\line (1,2){75}}
{\thicklines \bezier{1000}(-37,-75)(0,-40)(-2,0)}
\put(0,-10){$a$}
\put(-5,-3){$\bullet$}
\put(-47,-85){$\beta^{}_0 \eta $}
\put(25,-60){$\Re(w^{}_1\xi^{}_1)>0$ $(|\xi^{}_1|>d)$}
\put(0,-30){$\gamma^{-}_1(0, \eta;\varrho,\theta)$}
\put(-38,-77){$\bullet$}
\put(-45,77){$\beta^{}_1 \eta $}
\put(-37,69){$\bullet$}
\put(140,0){\vector(1,0){144}}
\put(190,-80){\vector(0,1){165}}
\put(-80,0){\vector(1,0){165}}
\put(-11,-80){\vector(0,1){165}}
\put(201,-75){$\Re(w^{}_1\xi^{}_1)=0$ $(|\xi^{}_1|>d)$}
\put(190,0){\line(2,-1){30}}
\put(246,-52){$\xi^{}_1\in \varSigma^{}_-$}
\put(300,-60){$\infty$}
\put(136,-81){\line(2,3){54}}
\put(190,0){\line(-2,3){54}}
\put(149,82){\line (1,-2){75}}
{\thicklines \bezier{1000}(199,0)(200,30)(165,72)}
\put(197,-2){$\bullet$}
\put(205,5){$a $}
\put(155,-85){$\beta^{}_0 \eta $}
\put(163,-77){$\bullet$}
\put(156,77){$\beta^{}_1 \eta$}
\put(200,57){$\Re(w^{}_1\xi^{}_1)>0$ $(|\xi^{}_1|>d)$}
\put(200,30){$\gamma^+_1(0, \eta;\varrho,\theta)$}
\put(163,69){$\bullet$}
\put(180,90){$\Im w^{}_1$ }
\put(290,-3){$\Re w^{}_1$}
{\thicklines \bezier{100}(219,-14)(225,-9)(225,0)}
\put(220,6){$d$}
\put(222,-2){$\bullet$}
\put(217,-14){\thicklines \vector(2,-1){80}}
\put(212,-26){$d^-$}
\put(216,-17){$\bullet$}
{\thicklines\dottedline{4}(235,14)(290,35)}
{\thicklines\dottedline{4}(235,-14)(290,-35)}
{\thicklines \bezier{7}(235,14)(245,0)(235,-14)}
\put(245,6){$d/\varepsilon$}
\put(237,-2){$\bullet$}
\put(275,15){$\lvert \arg\zeta^{}_1|\leqslant \varepsilon$ }
\end{picture}
}
\caption{}
\label{S.pic3}
\end{figure}%
We deform the path of integration $\displaystyle\int_{d}^\infty \hspace{-2ex} d\xi^{}_1$  in two ways as follows:  
Let $\delta>0$ be a 
sufficiently small constant and $d^\pm$ intersection points of the circle 
$|\tau|= d$ and $\{\xi^{}_1\in \mathbb{C};\,  \pm \Im \xi^{}_1
= \delta\Re \xi^{}_1 > 0 \}$. 
 Let $\varSigma^{}_{\pm}$ be paths starting from 
$d$, first going to $d^{\pm}$ along the 
circle and next going to the infinity along the half lines 
$\{\xi^{}_1 \in \mathbb C;\,   \pm\Im \xi^{}_1
= \delta\Re \xi^{}_1 > 0 \}$ respectively (see Figure \ref{S.pic3}). 
According to these deformations, we divide the path $\gamma^{}_1(0, \eta;\varrho,\theta) $ into 
two parts:
\[
\gamma_1^{\pm}(0, \eta;\varrho,\theta) 
:= \gamma^{}_1(0, \eta;\varrho,\theta) \cap\{ w^{}_1 \in \mathbb C;\,\pm \Im w^{}_1>0 \}.
\]
We take $a\in \gamma^{}_1(0, \eta;\varrho,\theta)  \cap \mathbb{R}$.
Now we can change the order of integration in $I$ (cf.\ Figure~\ref{S.pic3}) and obtain:
\begin{align*}
\sigma \cdot \varpi(P)(z,\zeta,\eta)
=  {} &  \smashoperator[r]{\Sum_{\alpha\in\mathbb{N}_0^{n-1}}}\,
    (\zeta')^{\alpha}\Bigl(
    \int_{\varSigma^{}_-} \hspace{-1.5ex} d\xi^{}_1\,
    \frac{\,P^{\mathcal{B}}_{\alpha}(z, \xi^{}_1,\eta)\,}
         {2\pi\im\xi_1^{|\alpha|}}\!
    \int_{\gamma_1^+(0, \eta;\varrho,\theta) } \hspace{-6ex}
e^{w^{}_1(\zeta^{}_1-\xi^{}_1)} \,dw^{}_1
    \\*&
+     \! \int_{\varSigma^{}_+} \hspace{-1.5ex} d\xi^{}_1\,
      \frac{\,P^{\mathcal{B}}_{\alpha}(z, \xi^{}_1,\eta)\,}
           {2\pi\im\xi_1^{|\alpha|}}\!
      \int_{\gamma_1^-(0, \eta;\varrho,\theta) } \hspace{-6ex}e^{w^{}_1(\zeta^{}_1-\xi^{}_1)}\,dw^{}_1\Bigr)\\
= {}&  \smashoperator[r]{\Sum_{\alpha\in\mathbb{N}_0^{n-1}}}\,
    (\zeta')^{\alpha}
    \int_{\varSigma^{}_-}\!\!\!
    \frac{\,P^{\mathcal{B}}_{\alpha}(z, \xi^{}_1,\eta)\,
 (e^{a(\zeta ^{}_1-\xi^{}_1)}-e^{\,\beta^{}_1\eta(\zeta ^{}_1- \xi^{}_1)})\,}
         {2\pi\im \xi_1^{|\alpha|}(\xi^{}_1-\zeta ^{}_1)}\,  d\xi^{}_1
\\*
& +\smashoperator{\Sum_{\alpha\in\mathbb{N}_0^{n-1}}}\,
    (\zeta')^{\alpha}
    \int_{\varSigma^{}_+}\!\!\!
    \frac{\, P^{\mathcal{B}}_{\alpha}(z, \xi^{}_1,\eta) \,
(e^{\,\beta^{}_0\eta(\zeta ^{}_1- \xi^{}_1)}-e^{a(\zeta ^{}_1- \xi^{}_1)})\,}
         {2\pi\im \xi_1^{|\alpha|} (\xi^{}_1-\zeta ^{}_1)}\, d\xi^{}_1\,.
\end{align*}
 Here we remark that  $a>0$ can be taken as sufficiently  small.
Further we set
\allowdisplaybreaks
\begin{align*}
I
 & :=
\smashoperator[r]{ \Sum_{\alpha\in\mathbb{N}_0^{n-1}}}\,
  (\zeta')^{\alpha}
    \int_{\varSigma^{}_- - \varSigma^{}_+}
    \frac{\, P^{\mathcal{B}}_{\alpha}(z, \xi^{}_1,\eta) \,e^{a(\zeta^{}_1- \xi^{}_1)}\,}
         {\,2\pi\im \xi_1^{|\alpha|}(\xi^{}_1-\zeta^{}_1)\,}\,    d\xi^{}_1,
\\
I^{-}
& := -\smashoperator{\Sum_{\alpha\in\mathbb{ N}_0^{n-1}}}\,
  (\zeta')^{\alpha}
    \int_{\varSigma^{}_-}\!\!
    \frac{\,P^{\mathcal{B}}_{\alpha}(z, \xi^{}_1,\eta) \,e^{\,\beta_1\eta(\zeta^{}_1-\xi^{}_1)}}{\,2\pi\im \xi_1^{|\alpha|}(\xi^{}_1-\zeta^{}_1)\,}\,
    d\xi^{}_1,
\\
I^{+}
& := \smashoperator[r]{\Sum_{\alpha\in\mathbb{ N}_0^{n-1}}}\,
(\zeta')^{\alpha}
    \int_{\varSigma^{}_+}\!\!
    \frac{\,P^{\mathcal{B}}_{\alpha}(z, \xi^{}_1,\eta)\, e^{\,\beta^{}_0\eta(\zeta^{}_1-\xi^{}_1)}}
{\,2\pi\im \xi_1^{|\alpha|}(\xi^{}_1-\zeta^{}_1)\,}\,d\xi^{}_1.
\end{align*}
Then 
$\sigma \cdot \varpi(P)(z,\zeta,\eta) = I+I^{-} +I^{+}$.  
Let us recall that we have discussed in $\lvert\arg\zeta^{}_1| \leqslant \varepsilon$, 
$|\zeta^{}_1|\geqslant \dfrac{d}{\, \varepsilon\,}>d$ and 
$|\zeta^{}_i| \leqslant \varepsilon |\zeta^{}_1|$ $(2 \leqslant i\leqslant n)$.
We can find  $\varepsilon^{}_0$, $c>0$ such that   
$|\zeta^{}_1-\xi^{}_1| \geqslant c |\xi^{}_1| \geqslant cd$ and $\Re (\zeta^{}_1\beta^{}_1\eta) 
\leqslant -2c|\eta\zeta^{}_1|$ hold for any $\varepsilon\in
\bigl]0,\varepsilon^{}_0\bigr[$ and $(\xi^{}_1,\eta) \in \varSigma^{}_-\times  S$. Further there exists a constant 
 $h^{}_0>0$ such that   $\Re(\beta^{}_1\eta\xi^{}_1) \geqslant  2h^{}_0|\eta\xi^{}_1|$ holds  for 
any $\xi^{}_1 \in \varSigma^{}_-\smallsetminus \{|\xi^{}_1|=d\}$ and $\eta\in  S$. 
For any $Z\Subset  S$,  choose $h=h^{}_Z>0$ 
 as 
$h^{}_Z<h^{}_0 m^{}_Z$ in \eqref{S.eq2.10}. Hence replacing $\varepsilon>0$ as 
$2K\varepsilon \leqslant c$
on $\widetilde{V}_{\varepsilon} \times Z$  we have
\begin{align}
\label{S.eq2.23}
  |I^{-}|
\leqslant {} &
\smashoperator[r]{\Sum_{|\alpha|=0}^\infty}
\frac{\,C^{}_{h_Z}(K\varepsilon|\eta\zeta^{}_1|)^{|\alpha|}\,e^{-2c|\eta\zeta^{}_1|}\,}{\,2\pi c d \,|\alpha|!\,} 
\\*
& \times \Bigl(e^{(h^{}_Z+\lvert\Re(\beta^{}_1\eta)|) d}
 \int_{|\xi^{}_1| = d}\hspace{-0.4ex} |d\xi^{}_1|+
              \int_{d}^\infty \hspace{-1ex} e^{-(2h^{}_0|\eta| -h^{}_Z)|\xi^{}_1|}\,d|\xi^{}_1|
         \Bigr)
\notag\\
\leqslant {} &
\frac{\,2^{n-2}C^{}_{h_Z}e^{-c |\eta\zeta^{}_1|}\,}{\,c\,}\,
  \biggl( e^{(h^{}_Z+|\beta^{}_1|r) d}+
              \frac{e^{-h^{}_0dr}}
{\, 2\pi h^{}_0d m^{}_Z\,}
         \biggr).
\notag\end{align}
Hence  we see that $I^{-}\in \mathfrak{N}^{}_{z^*_0}$. 
 Similarly, we have
\begin{equation}
\label{S.eq2.23a}
  |I^{\,+}|
\leqslant 
\frac{\,2^{n-2}C^{}_{h_Z}e^{-c |\eta\zeta^{}_1|}\,}{\,c\,}\,
  \biggl( e^{(h^{}_Z+|\beta^{}_0|r) d}+
              \frac{e^{-h^{}_0dr}}
{\, 2\pi h^{}_0dm^{}_Z\,}
         \biggr),
\end{equation}
 hence $I^{\,+}\in \mathfrak{N}^{}_{z^*_0}$. 
Now we consider $I$. For any $K\Subset \{\zeta^{}_1\in \mathbb{C};\,
\lvert \arg\zeta^{}_1| \leqslant \varepsilon\}$,  we see that  the integral operator
\[
\int_{\varSigma^{}_{-}-\varSigma^{}_+}d\xi^{}_1 \,
\frac{\,e^{a(\zeta^{}_1-\xi^{}_1)}\,}{2\pi\im (\xi^{}_1-\zeta^{}_1)}
\]
has the Cauchy kernel with a damping factor since $-\Re(a\xi^{}_1) <0$. Hence, 
\[
I =
\smashoperator[r]{\Sum_{\alpha\in\mathbb N_0^{n-1}}}
\,\,\,P^{\mathcal{B}}_{\alpha}(z, \zeta^{}_1,\eta)
\biggl(\frac{\zeta'}{\zeta^{}_1}\biggr)^{\!\!\alpha} =P^{\mathcal{B}}(z, \zeta,\eta)
\]
holds if $\zeta^{}_1$ is located in the domain surrounded by 
$\varSigma^{}_{-}-\varSigma^{}_+$. 
Thus we have 
\[
\sigma\cdot\varpi(P)(z,\zeta,\eta)-P(z,\zeta,\eta)= 
\sigma\cdot\varpi(P)(z,\zeta,\eta)-P^{\mathcal{B}}(z,\zeta,\eta)+P^{\mathcal{B}}(z,\zeta,\eta)-P(z,\zeta,\eta)
 \in \mathfrak{N}^{}_{z^*_0}, 
\]
that is,
$
\sigma\cdot\varpi= \mathds{1}  \colon  \mathfrak{S}^{}_{z^*_0}/
\mathfrak{N}^{}_{z^*_0} 
\earrow \mathfrak{S}^{}_{z^*_0}/
\mathfrak{N}^{}_{z^*_0}$\,. 
\paragraph{\textbf{Step 2.}}
Let $P=[\psi(z,w,\eta)\,dw]\in \mathscr {E}^{\mathbb R}_{X,z^*_0}$. 
Then we can assume that a representative $\psi(z,z+w,\eta)$ has the
form as in \eqref{S.eq2.1}. By Proposition \ref{S.rem5.2}, each coefficient $P^{}_{\alpha}(z,\zeta^{}_1,\eta)$ in \eqref{S.eq2.6q}  
is  written as
\[
P^{}_{\alpha}(z,\zeta^{}_1,\eta)
= \frac{\,\zeta_1^{\,|\alpha|}}{\alpha!} 
\smashoperator{\int\limits_{\hspace{6ex}\gamma^{}_1(0, \eta;\varrho,\theta) }}
\psi^{}_{\alpha}(z,p,\eta)\, e^{p\zeta^{}_1}\,dp.
\]
We assume that each $\psi^{}_{\alpha}(z,p,\eta)$ is holomorphic on 
\[
\{(z,p,\eta) \in \mathbb{C}^{n+1}\times  S;\,
\|z\| < 2r^{}_0,\, |w_1 | <\varrho |\eta|,\, \lvert\arg  p|< \dfrac{\,\pi\,}{2} + \theta\}. 
\]
Fix $\eta^{}_0\in S\cap \mathbb{R}$, and 
we take $\varrho'<\varrho$ as $\varrho'|\eta| <|\eta^{}_0|$ for any $\eta\in S$. 
By \eqref{S.eq2.2}, there exists $c >0$ and   for any $Z\Subset  S$ there exists
  $C^{}_{Z}>0$ such that for any $\eta\in Z$,
\begin{equation}\label{S.eq2.24}
\sup\{|\psi^{}_{\alpha}(z,p,\eta)|;\,  \|z\| \leqslant  r^{}_0,\, p\in \gamma^{}_1(0, \eta^{}_0;\varrho,\theta)\}
\leqslant C_{Z}(c|\eta|)^{|\alpha|+n-1}.
\end{equation}
By the definition, we have
\begin{multline*}
\varpi\cdot\sigma(\psi)(z,z+w,\eta) 
\\
= \Sum_{|\alpha|=0}^\infty
\frac{1}{\,2\pi\im (w')^{\alpha+\mathbf{1}_{n-1}}\,} 
\int_{d}^\infty \hspace{-1.5ex}d\zeta^{}_1\,
\zeta_1^{|\alpha|}\varGamma^{}_{|\alpha|}(\zeta^{}_1,\eta)\,e^{- w^{}_1 \zeta^{}_1}\!\!
\smashoperator{\int\limits_{\hspace{7ex}\gamma^{}_1(0, \eta^{}_0;\varrho,\theta)}}\psi^{}_{\alpha}(z,p,\eta^{}_0)\,e^{p\zeta^{}_1}dp. 
\end{multline*}
We set
\[
\varpi'\cdot\sigma(\psi)(z,z+w,\eta):=
\Sum_{|\alpha|=0}^\infty
\frac{1}{\,2\pi\im (w')^{\alpha+\mathbf{1}_{n-1}}\,} 
\int_{d}^\infty \hspace{-1.5ex}d\zeta^{}_1\,e^{- w^{}_1 \zeta^{}_1}\!\!
\smashoperator{\int\limits_{\hspace{7ex}\gamma^{}_1(0, \eta^{}_0;\varrho,\theta)}} \psi^{}_{\alpha}(z,p,\eta)\,e^{p\zeta^{}_1}dp. 
\]
We assume that $\Re (\eta \zeta^{}_1) \geqslant 2\delta^{}_0 |\eta\zeta^{}_1|>0$ for some $\delta^{}_0 \in \bigl]0, \dfrac{1}{\;2\;}\bigr[$. 
We deform the path $\gamma^{}_1(0, \eta^{}_0;\varrho, \theta)$ as 
$|e^{p\zeta^{}_1}|\leqslant e^{-\delta|\eta\zeta^{}_1|/2}$ for $\lvert\arg \zeta^{}_1|\leqslant \varepsilon$. 
Then by \eqref{S.eq2.5} and \eqref{S.eq2.24}, 
 for any $Z\Subset  S$, there exists  $C^{}_{Z}>0$ such that if $2|w^{}_1|< \delta^{}_0|\eta|$ and 
$c|\eta|<\delta^{}_0 |w^{}_i|$ ($2 \leqslant  i \leqslant n$), we have
\begin{align*}
\Bigl|&\Sum_{|\alpha|=0}^\infty
\frac{1}{\,2\pi\im (w')^{\alpha+\mathbf{1}_{n-1}}\,} 
\int_{d}^\infty \hspace{-1.5ex}d\zeta^{}_1\,
(1- \zeta_1^{|\alpha|}\varGamma^{}_{|\alpha|}(\zeta^{}_1,\eta))\,e^{- w^{}_1 \zeta^{}_1}\!\!
\smashoperator{\int\limits_{\hspace{7ex}\gamma^{}_1(0, \eta^{}_0;\varrho,\theta)} } 
\psi^{}_{\alpha}(z,p,\eta)\,e^{p\zeta^{}_1}dp\Bigr|
\\
& \leqslant 
\Sum_{|\alpha|=1}^\infty
\frac{C^{}_Z|\gamma^{}_1(0, \eta^{}_0;\varrho,\theta)|(c|\eta|)^{|\alpha|+n-1}}{\,2\pi (w')^{\alpha+\mathbf{1}_{n-1}}\,} 
\int_{d}^\infty 
\dfrac{\,e^{-(\delta^{}_0|\eta|/2-|w^{}_1|)| \zeta^{}_1|}\,}{\delta_0^{|\alpha|-1}}\,d\zeta^{}_1
\\
& \leqslant 
\Sum_{|\alpha|=1}^\infty
\frac{C^{}_Z\delta^{\,n}_0|\gamma^{}_1(0, \eta^{}_0;\varrho,\theta)|}{\,\pi (\delta^{}_0|\eta|-2|w^{}_1|)\,} 
\Bigl|\Bigl(\dfrac{c\eta}{\delta^{}_0 w'}\Bigr)^{\!\!\alpha+\mathbf{1}_{n-1}}\Bigr|<\infty.
\end{align*}
Here $|\gamma^{}_1(0, \eta^{}_0;\varrho,\theta)|$ denotes the length of $\gamma^{}_1(0, \eta^{}_0;\varrho,\theta)$. 
Next we consider
\begin{equation}\label{5.21z}
\smashoperator{\int\limits_{\hspace{7ex}\gamma^{}_1(0, \eta^{}_0;\varrho,\theta)}} \psi^{}_{\alpha}(z,p,\eta)\,e^{p\zeta^{}_1}dp
-\smashoperator{\int\limits_{\hspace{7ex}\gamma^{}_1(0, \eta^{}_0;\varrho,\theta)}}
 \psi^{}_{\alpha}(z,p,\eta^{}_0)\,e^{p\zeta^{}_1}dp
= \smashoperator{\int\limits_{\hspace{7ex}\gamma^{}_1(0, \eta^{}_0;\varrho,\theta)} } dp\, e^{p\zeta^{}_1}
\int_{\eta^{}_0}^{\eta}\hspace{-1ex}\partial_{\eta}\psi^{}_{\alpha}(z,p,\tau)\,d\tau
\,.
\end{equation}
Since $\partial_{\eta}\psi(z,w,\eta)$ is holomorphic on $|w^{}_1|<\varrho |\eta|$, 
as in  \eqref{S.eq2.24}
there exists $c$ and for any $S\Subset Z$ there exists   $C^{}_Z>0$ such that for any $\eta\in Z$,
 \[
\sup\{|\psi^{}_{\alpha}(z,p,\eta)-\psi^{}_{\alpha}(z,p,\eta^{}_0)|;\,  \|z\| \leqslant  r^{}_0,\, |p| \leqslant 
\varrho'|\eta|\}
\leqslant C_{Z}(c|\eta|)^{|\alpha|+n-1}.
\]
Thus we can change  the path of the integration in \eqref{5.21z}  as    $|e^{p\zeta^{}_1}|\leqslant e^{-c'|\eta\zeta|}$ 
for $\lvert\arg \zeta^{}_1|\leqslant \varepsilon$. Hence if $\eta\in Z$,  $|w^{}_1|< c'|\eta|$ and $c|\eta|<\delta^{}_0 |w^{}_i|$ 
($2 \leqslant  i \leqslant n$),  we have 
\begin{align*}
\Bigl|&\Sum_{|\alpha|=0}^\infty
\frac{1}{\,2\pi\im (w')^{\alpha+\mathbf{1}_{n-1}}\,} 
\int_{d}^\infty \hspace{-1.5ex}d\zeta^{}_1\,e^{- w^{}_1 \zeta^{}_1}\!\!
\smashoperator{\int\limits_{\hspace{7ex}\gamma^{}_1(0, \eta^{}_0;\varrho,\theta) } } 
(\psi^{}_{\alpha}(z,p,\eta)-\psi^{}_{\alpha}(z,p,\eta^{}_0))\,e^{p\zeta^{}_1}dp\Bigr|
\\
&\leqslant 
\Sum_{|\alpha|=0}^\infty
\frac{ C^{}_{Z}|\gamma^{}_1(0, \eta^{}_0;\varrho,\theta)|(c|\eta|)^{|\alpha|+n-1}}{\,2\pi (w')^{\alpha+\mathbf{1}_{n-1}}\,} 
\int_{d}^\infty \hspace{-1.5ex}e^{- (c'|\eta|-|w^{}_1|) |\zeta^{}_1|}\,d\zeta^{}_1
\\
&\leqslant 
\Sum_{|\alpha|=0}^\infty
\frac{ C^{}_{Z}|\gamma^{}_1(0, \eta^{}_0;\varrho,\theta)|}{\,2\pi  (c'|\eta|-|w^{}_1|)\,} 
\Bigl|\Bigl(\dfrac{c\eta}{\delta^{}_0 w'}\Bigr)^{\!\!\alpha+\mathbf{1}_{n-1}}\Bigr|<\infty.
\end{align*}
Summing up we can prove that  
$[\varpi'\cdot\sigma(\psi)(z,w,\eta)\,dw]=[\varpi\cdot\sigma(\psi)(z,w,\eta)\,dw] \in 
\varinjlim\limits_{\boldsymbol{\kappa}} \widehat{E}^{\mathbb{R}}_{X}(\boldsymbol{\kappa})$.
We may assume  that  
$\varpi'\cdot \sigma(\psi)(z,z+w,\eta)$ is holomorphic on 
\[
\smashoperator[r]{\Bcap_{2\leqslant i \leqslant n}}\{(z,w,\eta)\in \mathbb{C}^{2n}\times  S;\,
\|z\|< r^{}_0,\, \lvert \arg w^{}_1| < \delta'+\dfrac{\,\pi\,}{2},\, c|\eta|<|w^{}_i| \}
\]
with   some constants $r^{}_0$, $\delta'$, $c>0$. 
\[
\smashoperator[r]{\Bcap_{2\leqslant i \leqslant n}}\{(z,w,\eta)\in \mathbb{C}^{2n}\times  S;\,
\|z\|< r^{}_0,\, 0 <|w^{}_1|,\,\lvert \arg w^{}_1| < \delta'+\dfrac{\,\pi\,}{2},\, c|\eta|<|w^{}_i| \}
\]
with   some constants $r^{}_0$, $\delta'$, $c>0$. 
Let $\gamma_1'$ be a path starting from $\beta^{}_0\eta^{}_0$, 
ending at $\beta^{}_1\eta^{}_0$  and detouring  
$w^{}_1$  clockwise as in Figure~\ref{S.pic1}. 
\begin{figure}
{\footnotesize
\begin{picture}(450,120)(-200,-60)
{\thicklines \bezier{1000}(-17,-35)(30,0)(-16,33)}
\put(3,-3){$\blacktriangle$}
\put(45,-5){$\bullet$}
\put(35,-10){$w^{}_1$}
\put(88,2){$\blacktriangledown$}
{\thicklines \bezier{1000}(-17,-35)(200,10)(-16,33)}
\put(-10,-3){\line(-1,2){30}}
\put(-10,-3){\line(-1,-2){30}}
\put(8,10){$\gamma^{}_1(0, \eta^{}_0;\varrho,\theta) $}
\put(98,2){$\gamma'_1$}
\put(-25,-50){$ \beta^{}_0\eta^{}_0$}
\put(-18,-37){$\bullet$}
\put(-25,40){$ \beta^{}_1\eta^{}_0$}
\put(-17,30){$\bullet$}
\end{picture}
}
\caption{}
\label{S.pic1}
\end{figure}
If $\Re \,(p-w^{}_1) <0$, we have
\[\displaystyle
  \int_{d}^\infty \hspace{-1ex}
  e^{(p-w^{}_1)\zeta^{}_1}d\zeta^{}_1
= -\frac{e^{(p-w^{}_1) d}}{p-w^{}_1},
 \]
and the right-hand side extends analytically. 
Thus  on the common domain of definition    we have
\allowdisplaybreaks
\begin{align*}
\varpi'\cdot \sigma(\psi)(z,z+w,\eta)
={} & \Sum_{\alpha\in \mathbb N_0^{n-1}}
\frac{-1}{\, (w')^{\alpha+\mathbf{1}_{n-1}}\,} 
\int\limits_{\gamma^{}_1(0, \eta^{}_0;\varrho,\theta) }\!\!\frac{\psi^{}_{\alpha}(z,p,\eta)\, e^{(p-w^{}_1) d}}{2\pi\im\,(p-w^{}_1)}\,dp
\\
={} &   \Sum_{\alpha\in \mathbb N_0^{n-1}} \!
\frac{-1}{\, (w')^{\alpha+\mathbf{1}_{n-1}}\,}
\oint\limits_{\gamma^{}_1(0, \eta^{}_0;\varrho,\theta) \vee\gamma'_1}\hspace{-1ex}
\frac{\, \psi^{}_{\alpha}(z,p,\eta)\, e^{(p-w^{}_1) d}\,}{2\pi\im\,(p-w^{}_1)}\,dp 
\\*
&+\Sum_{\alpha\in \mathbb N_0^{n-1}} \! 
\frac{1}{\, (w')^{\alpha+\mathbf{1}_{n-1}}\,}
\int_{\gamma'_1}\!\frac{\, \psi^{}_{\alpha}(z,p,\eta)\, e^{(p-w^{}_1) d}\,}{2\pi\im\,(p-w^{}_1)}\,dp 
\\
={} &\psi(z,z+w,\eta)+\varPi,
\end{align*}
where
\[
\varPi := \Sum_{\alpha\in \mathbb N_0^{n-1}} \! 
\frac{1}{\, (w')^{\alpha+\mathbf{1}_{n-1}}\,}
\int_{\gamma'_1}\!\frac{\, \psi^{}_{\alpha}(z,p,\eta)\, e^{(p-w^{}_1) d}\,}{2\pi\im\,(p-w^{}_1)}\,dp .
\]
As in \eqref{S.eq2.24},  there exist $c$,  $c^{}_1 >0$ and   for any $Z\Subset  S$, there exists
  $C^{}_{Z}>0$ such that 
\[
\Bigl|\frac{\,e^{(p-w^{}_1)d}\,}{p-w^{}_1}
 \Bigr|
        \leqslant C^{}_Z,\qquad |\psi^{}_{\alpha}(z,p,\eta)|\leqslant C^{}_Z(c|\eta|)^{|\alpha|+n-1},
\]
hold on $\{\|z \|<c^{}_1,\,|w^{}_1|<c^{}_1|\eta|,\, p\in\gamma_1' ,\,\eta \in Z\}$. 
Thus, on $\{\|z\| <c^{}_1, \, |w^{}_1|<c^{}_1|\eta|,\,\eta \in Z\}$  we have 
\[
|\varPi | \leqslant
\Sum_{\alpha\in \mathbb N_0^{n-1}} \!
\dfrac{\,C^{\;2}_Z|\gamma'_1|\,}
{2\pi}\,\Bigl|\!\Bigl(\dfrac{c|\eta|}{w'}\Bigr)^{\!\alpha+\mathbf{1}_{n-1}}\Bigr|.
\]
By taking $\delta>0$ as $c\delta< 1$ and $\delta<c^{}_1$, we see that  $\varPi$ is holomorphic on 
\begin{align*}
\Bcup_{Z\Subset  S}
\smashoperator[r]{\Bcap_{i=2}^n}\{(z,w,\eta)\in   \mathbb{C}^{2n}\times Z;\, \|z\|< \delta,\,
|w^{}_1| < \delta |\eta|<\delta^2 |w^{}_i|\}&
\\=
\smashoperator{\Bcap_{i=2}^n}\{(z,w,\eta)\in   \mathbb{C}^{2n}\times  S;\, \|z\|< \delta,\,
|w^{}_1| < \delta |\eta|<\delta^2 |w^{}_i|\}&.
\end{align*}
Thus   $\varpi\cdot\sigma=
 \mathds{1}\colon \mathscr{E}^{\smash{\mathbb{R}}}_{X, z^*_0} \earrow \mathscr{E}^{\smash{\mathbb{R}}}_{X, z^*_0}$. 

By Steps 1 and 2,
we see that  $\sigma^{-1}= \varpi $,  
hence $\sigma\colon \mathscr{E}^{\smash{\mathbb{R}}}_{X, z^*_0}\earrow  \mathfrak{S} ^{}_{z^* _0} /\mathfrak{N} ^{}_{z^* _0} $. 
\end{proof}

Let $P\in \mathfrak{S}_{z^*_0}$, and  consider 
 $[\varpi (P)(z,w,\eta)\,dw] \in  \varinjlim\limits_{\boldsymbol{\kappa}} E^{\mathbb{R}}_X(\boldsymbol{\kappa}) $.  
Here we can assume that $\varpi(P)(z,z+w,\eta)$ is  holomorphic  on $V$ in  \eqref{S.eq2.21}. 
Take $c^{}_0>1$ such that $c^{}_0\Re \zeta^{}_1 \geqslant |\zeta^{}_1|$ for $\lvert \arg \zeta^{}_1| \leqslant \theta'$. 
In \eqref{S.eq2.8}, we take  $\{\varepsilon^{}_{\nu} \}_{\nu=1}^\infty\subset \mathbb{R}^{}_{>0}$ and  $C>0$ as
\[
1\gg  \varepsilon^{}_1 >  \varepsilon^{}_2 > \cdots > \varepsilon^{}_\nu \xrightarrow[\nu]{}0, \qquad
  \frac{2C^{}_{\varepsilon^{}_{\nu}/2}}{\,\varepsilon^{}_\nu\,}\leqslant 2^\nu C,
\]
Set $\varepsilon^{}_{\alpha}:= \varepsilon^{}_{|\alpha|}$ for short, and  we define
\begin{align*}
\varpi_{0,\alpha}(P)(z,w^{}_1) 
&:=\int_{d}^\infty \hspace{-1.5ex}P^{}_{\alpha}(z,\zeta^{}_1,\eta^{}_0)\,
\varGamma^{}_{|\alpha|}(\zeta^{}_1,c^{}_0\varepsilon_\alpha-w^{}_1)\,e^{- w^{}_1 \zeta^{}_1}\,d\zeta^{}_1\,,
\\
\varpi_0(P)(z,z+w)
&:= \Sum_{\alpha\in \mathbb N_0^{n-1}}
\frac{\,\alpha!\, \varpi_{0,\alpha}(P)(z,w^{}_1)} 
{\,(2\pi\im\,)^{n}\,
(w')^{\alpha+ \mathbf{1}_{n-1}}\,}\,.
\end{align*}
\begin{thm}
\textup{(1)} The $\varpi_0$ induces the mapping 
$\varpi_0 \colon \mathfrak{S}_{z^*_0}\big/\mathfrak{N}_{z^*_0}\to \mathscr{E}^{\smash{\mathbb{R}}}_{X, z^*_0}$\,.

\textup{(2)} 
It follows that  $[\varpi (P)(z,w,\eta)\,dw]=[\varpi _0(P)(z,w)\,dw] \in  
\varinjlim\limits_{\boldsymbol{\kappa}} E^{\mathbb{R}}_X(\boldsymbol{\kappa})$, and  
the following diagram is commutative\textup{:}
\begin{equation}
\label{commA}
\vcenter{
\xymatrix @R=3ex{
\mathscr{S}_{z^*_0}\big/\mathscr{N}_{z^*_0}\ar[r]^-{\dsim}
\ar@{}[d]|{\rotatebox{-90}{$\simeq$}}
&\mathfrak{S}_{z^*_0}\big/\mathfrak{N}_{z^*_0}  \ar@ <.6ex>[d]^-{\varpi} \ar[ld]_-{\varpi_0}
\\
\mathscr{E}^{\smash{\mathbb{R}}}_{X, z^*_0} \ar[r]^-{\dsim}
& \varinjlim\limits_{\boldsymbol{\kappa}} E^{\mathbb{R}}_X(\boldsymbol{\kappa}).
 \ar@ <.6ex>[u]^-{\sigma}
}}\end{equation}
Here, the isomorphism $\mathscr{E}^{\smash{\mathbb{R}}}_{X, z^*_0} \earrow \mathscr{S}_{z^*_0}\big/\mathscr{N}_{z^*_0} $ 
is induced by 
\[
\psi(z,w)\,dw \mapsto  \sigma(\psi)(z,\zeta,\eta^{}_0)=  
\smashoperator{\int\limits_{\hspace{7ex}\gamma(0,\eta^{}_0;\varrho,\theta)}}
\psi(z,z+w) \,e^{\langle w, \zeta\rangle} dw
\] 
for any fixed $\eta^{}_0 \in S$. 
\end{thm}
\begin{rem}
(1)  The isomorphism 
$\mathscr{E}^{\smash{\mathbb{R}}}_{X, z^*_0} \earrow \mathscr{S}_{z^*_0}\big/\mathscr{N}_{z^*_0} $ 
is established  in \cite{a2},  \cite{A-K-Y} and \cite{aky}.

(2) From the diagram \eqref{commA},  we obtain 
an explicit description of the isomorphism $
\mathscr{E}^{\smash{\mathbb{R}}}_{X, z^*_0}\simeq
\varinjlim\limits_{\boldsymbol{\kappa}} E^{\mathbb{R}}_X(\boldsymbol{\kappa})$. 
\end{rem}
\begin{proof}
If $ \varepsilon_{\nu} < \delta^{}_1|w^{}_1|$,  $\Re( w^{}_1 \zeta^{}_1)\geqslant \delta^{}_1 |w^{}_1 \zeta^{}_1|$ 
and  $0\leqslant t\leqslant 1$,  we have 
\[
|e^{-(t(c^{}_0\varepsilon_\nu -w^{}_1)+w^{}_1)\zeta^{}_1}|= e^{-tc^{}_0\varepsilon_\nu\Re \zeta^{}_1-(1-t) \Re(w^{}_1\zeta^{}_1)}
\leqslant
 e^{-t\varepsilon_\nu|\zeta^{}_1|-(1-t) \varepsilon^{}_\nu|\zeta^{}_1|}=
 e^{-\varepsilon_\nu|\zeta^{}_1|}.
\]
Thus
\begin{align*}
|\varGamma^{}_{\nu}(\zeta^{}_1,c^{}_0\varepsilon_\nu-w^{}_1)\,e^{- w^{}_1 \zeta^{}_1}|&=
\Bigl|\frac{1}{\,(\nu-1)!\,}
 \int_{0}^{c^{}_0\varepsilon_\nu-w^{}_1} \hspace{-0.7ex} 
e^{-(s+w^{}_1)\zeta^{}_1} s^{\nu-1}ds\Bigr|
\\
&
= 
\Bigl|\frac{(c^{}_0\varepsilon_\nu -w^{}_1)^\nu}{\,(\nu-1)!\,}
\int_{0}^{1} \hspace{-1ex} e^{-(t(c^{}_0\varepsilon_\nu -w^{}_1)+w^{}_1)\zeta^{}_1}t^{\nu-1}dt\Bigr|
\\
&
\leqslant  \frac{(c^{}_0\varepsilon_\nu +|w^{}_1|)^\nu e^{-\varepsilon_{\nu}|\zeta^{}_1|}}{\,(\nu-1)!\,}
 \int_{0}^{1} \hspace{-1ex}  t^{\nu-1}dt = \frac{(c^{}_0\varepsilon_\nu +|w^{}_1|)^\nu e^{-\varepsilon_{\nu}|\zeta^{}_1|}}{\,\nu!\,}\,.
\end{align*}
Set 
\[   L'_{\alpha}
 :=  \{(z,w^{}_1)\in \mathbb{C}^{n+1};\,\|z\| < r^{}_0,\,\lvert \arg w^{}_1| < \delta'+\dfrac{\,\pi\,}{2},\, 
   \varepsilon_{\alpha} < \delta^{}_1|w^{}_1|   \}\,.                    
\]
Then   taking $h=\dfrac{\,\varepsilon^{}_{\alpha}\,}{2}$ in \eqref{S.eq2.8} we have
\allowdisplaybreaks
\begin{align*}
\sup_{  L'_{\alpha}}|\varpi_{0,\alpha}(P)(z,w^{}_1)|& \leqslant
 \dfrac{C^{}_{\varepsilon^{}_{\alpha}/2} (K(c^{}_0\varepsilon_\alpha+|w^{}_1|))^{|\alpha|}}{|\alpha|!}
\int_{d}^\infty \hspace{-1.5ex}e^{-\varepsilon^{}_{\alpha}| \zeta^{}_1|/2}\,d|\zeta^{}_1|
\\
&
\leqslant \dfrac{2C^{}_{\varepsilon^{}_{\alpha}/2} (K(c^{}_0\varepsilon_\alpha+|w^{}_1|))^{|\alpha|}}{\,\varepsilon^{}_{\alpha}\,
|\alpha|!}
 \leqslant \dfrac{C (2K(c^{}_0\varepsilon_\alpha+|w^{}_1|))^{|\alpha|}}{|\alpha|!}\,.
\end{align*}

If $(z,w^{}_1) \in L'_{\nu}$ and $2K(c^{}_0\varepsilon^{}_\nu+|w^{}_1|) <|w^{}_i| $,  we have
\begin{align*}
|\varpi_0(P)(z,z+w)|& 
\leqslant \Sum_{\alpha\in \mathbb N_0^{n-1}}
\dfrac{C}{\,(2\pi)^{n}|(w')^{\mathbf{1}_{n-1}}|\,}\,
\Bigl|\!\Bigl(\dfrac{2K(c^{}_0\varepsilon_\alpha+|w^{}_1|)}{w'}\Bigr)^{\!\alpha}\Bigr|<\infty,
\end{align*}
hence $\varpi_0(P)(z,z+w)$ is holomorphic on 
\begin{multline*}
\Bcup_{\nu=1}^\infty\smashoperator[r]{\Bcap_{2\leqslant i\leqslant n}}\{(z,w)\in\mathbb{C}^{2n};\,(z,w^{}_1) \in L'_{\nu}\,,\,
2K(c^{}_0\varepsilon^{}_\nu+|w^{}_1|)<|w^{}_i| \}
\\
=
\smashoperator[r]{\Bcap_{2\leqslant i\leqslant n}}\{(z,w)\in \mathbb{C}^{2n};\, 
\|z\|<r^{}_0,  \,  2K|w^{}_1| < |w^{}_i|,\,\lvert \arg w^{}_1| < \delta'+\dfrac{\,\pi\,}{2}\}.
\end{multline*}
Therefore   $\varpi_0(P)(z,w)\,dw$ defines  a germ of $\mathscr{E}^{\smash{\mathbb{R}}}_{X, z^*_0}$.  
There  exists $\delta>0$ such that $\Re(\eta \zeta^{}_1) \geqslant 2\delta|\eta \zeta^{}_1|$ for any $\eta \in S$ and 
$\lvert\arg\zeta^{}_1| \leqslant\theta'$. 
Suppose that $|w^{}_1|<\delta |\eta|$. 
For any $Z\Subset S$, there exists $N^{}_Z\in \mathbb{N}$ such that 
$\delta m^{}_Z \geqslant \varepsilon^{}_\nu$ for any $\nu \geqslant N^{}_Z$. 
Thus 
if  $0\leqslant t\leqslant 1$,  we have 
\begin{align*}
|e^{-((1-t)c^{}_0\varepsilon_\nu+t (\eta+w^{}_1))\zeta^{}_1}|& = 
e^{-(1-t)c^{}_0\varepsilon_\nu\Re \zeta^{}_1-t \Re((\eta+w^{}_1)\zeta^{}_1)}
\leqslant
 e^{-(1-t)\varepsilon_\nu|\zeta^{}_1|-t(2\delta|\eta|-|w^{}_1|) |\zeta^{}_1|}
\\
&\leqslant  e^{-(1-t)\varepsilon_\nu|\zeta^{}_1|-t\delta m^{}_Z |\zeta^{}_1|}
\leqslant e^{-\varepsilon_\nu|\zeta^{}_1|}.
\end{align*}
Thus we have 
\begin{align*}
|(\varGamma^{}_{\nu}(\zeta^{}_1,\eta)&-\varGamma^{}_{\nu}(\zeta^{}_1,c^{}_0\varepsilon_\nu-w^{}_1))\,e^{- w^{}_1 \zeta^{}_1}|
=
\Bigl|\frac{1}{\,(\nu-1)!\,}
 \int_{c^{}_0\varepsilon_\nu-w^{}_1} ^{\eta}\hspace{-6ex} e^{- (s+w^{}_1)\zeta^{}_1} s^{\nu-1}ds\Bigr|
\\
&
=
\Bigl|\frac{ (\eta+w^{}_1-c^{}_0\varepsilon_\nu)^\nu}{\,(\nu-1)!\,}
 \int_{0} ^{1} \hspace{-1ex} e^{-((1-t)c^{}_0\varepsilon_\nu+t (\eta+w^{}_1))\zeta^{}_1} t^{\nu-1}dt\Bigr|
\\
&
\leqslant
\frac{(|\eta|+\delta|\eta|+c^{}_0 \delta m^{}_Z)^\nu e^{-\varepsilon_\nu|\zeta^{}_1|}}{\,(\nu-1)!\,}
\int_{0}^{1} \hspace{-1ex} t^{\nu-1}dt
\leqslant\frac{((1+\delta+c^{}_0\delta )|\eta|)^\nu e^{-\varepsilon_\nu|\zeta^{}_1|}}{\,\nu!\,}\,.
\end{align*}
Set $K^{}_1:= K(1+ \delta+c^{}_0\delta)$.
Choosing $h=\dfrac{\,\varepsilon^{}_{\alpha}\,}{2}$ in \eqref{S.eq2.8} $(|\alpha| \geqslant N^{}_Z)$ we have
\begin{align*}
|\varpi^{}_\alpha(P)(z,w^{}_1,\eta)-\varpi_{0,\alpha}(P)(z,w^{}_1)|& \leqslant
\frac{C^{}_{\varepsilon^{}_{\alpha}/2}(K^{}_1|\eta|)^\nu e^{-\varepsilon_\nu|\zeta^{}_1|}}{\,|\alpha|!\,}
\int_{d}^\infty \hspace{-1.5ex}e^{-\varepsilon^{}_{\alpha}| \zeta^{}_1|/2}\,d|\zeta^{}_1|
\\
&
\leqslant \dfrac{2C^{}_{\varepsilon^{}_{\alpha}/2} (K^{}_1|\eta|)^{|\alpha|}}{\varepsilon^{}_{\alpha}\,|\alpha|!\,}
 \leqslant \dfrac{C (2K^{}_1|\eta|)^{|\alpha|}}{|\alpha|!}\,.
\end{align*}
Thus if $2K^{}_1|\eta|<|w^{}_i|  $,  we have 
\begin{align*}
|&\varpi(P)(z,z+w,\eta)-\varpi_0(P)(z,z+w)| 
\\ &\leqslant 
\Big|\smashoperator[r]{\Sum_{|\alpha|<N^{}_Z}}
\frac{\,\alpha! (\varpi_\alpha(P)(z,w^{}_1,\eta)-\varpi'_\alpha(P)(z,w^{}_1))} 
{\,(2\pi\im\,)^{n}\,
(w')^{\alpha+ \mathbf{1}_{n-1}}\,}\Bigr|
+
\Sum_{|\alpha|\geqslant N^{}_Z}
\dfrac{C}{\,(2\pi)^{n}|(w')^{\mathbf{1}_{n-1}}|\,}\,
\Bigl|\!\Bigl(\dfrac{2K^{}_1|\eta|}{w'}\Bigr)^{\!\!\alpha}\Bigr|
\\
&<\infty,
\end{align*}
hence $\varpi(P)(z,z+w,\eta)-\varpi_0(P)(z,z+w)$ is holomorphic on
\begin{multline*}
\Bcup_{Z\Subset S}\smashoperator[r]{\Bcap_{2\leqslant i\leqslant n}}\{(z,w,\eta)\in\mathbb{C}^{2n}\times Z;\,
\|z\|< r^{}_0,\,|w^{}_1|< \delta|\eta|,\,2K^{}_1|\eta|<|w^{}_i|  \}
\\
=
\smashoperator[r]{\Bcap_{2\leqslant i\leqslant n}}\{(z,w,\eta)\in\mathbb{C}^{2n}\times S;\,
\|z\|< r^{}_0,\,|w^{}_1|< \delta|\eta|,\,2K^{}_1|\eta|<|w^{}_i|  \}.
\end{multline*}
Therefore
 we have  $[\varpi (P)(z,w,\eta)\,dw]=[\varpi_0(P)(z,w)\,dw] \in  
\varinjlim\limits_{\boldsymbol{\kappa}} E^{\mathbb{R}}_X(\boldsymbol{\kappa})$.
\end{proof}

\section{Classical Formal Symbols with an Apparent Parameter}\label{sec:cl-formal-symbol}

\begin{defn}[see \cite{aky}, \cite{b}]  Let $t$ be an indeterminate. 

(1) 
$P(t;z,\zeta) = \Sum_{\nu=0}^\infty t^\nu P^{}_\nu(z,\zeta)$
is an element of $\widehat{\mathscr{S}} _{\rm cl}(\varOmega)$  if 
$
    P(t;z,\zeta)\in \varGamma(\varOmega^{}_ \rho[d^{}_ \rho];
\mathscr{O}^{}_{T^*X}) \llbracket t\rrbracket 
$ 
for some $d>0$ and $\rho\in \bigl]0,1\bigr[$,  and 
there exists a  constant $A >0$   satisfying  the following:
for any  $h >0$  there exists a  constant  $C_{h}>0$  such that 
\[
          |P^{}_\nu(z,\zeta)|
\leqslant \frac{\,C_{h}\nu!A^{\nu} e^{h \|\zeta\|}}{\|\zeta\|^{\nu}} \quad (\nu \in\mathbb{N}^{}_0,\, 
(z;\zeta) \in \varOmega^{}_ \rho[d^{}_ \rho]).
\]

(2) 
$  
  P(t;z,\zeta) = \Sum_{\nu=0}^\infty t^\nu P^{}_\nu(z,\zeta,\eta) 
\in \widehat{\mathscr{S}} _{\rm cl}(\varOmega)$
is an element of $\widehat{\mathscr{N}} _{\rm cl}(\varOmega) $  if 
there exists a    constant $A >0$   satisfying  the following: 
for any  $h >0$  there exists a  constant  $C_{h}>0$  
such that 
\[
          \Bigl|\smashoperator[r]{\Sum_{\nu=0}^{m-1}} P^{}_\nu(z,\zeta)\Bigr|
 \leqslant \frac{\,C_{h}m!A^{m}e^{h\|\zeta\|}}{\|\zeta\|^{m}}
\quad (m \in\mathbb{N},\,(z;\zeta) \in \varOmega^{}_ \rho[d^{}_ \rho]).
\]

(3) We set  
\[
\widehat{\mathscr{S}}^{}_{{\rm cl},z^*_0}  := \varinjlim_{ \varOmega}
\widehat{\mathscr{S}}_{\rm cl}(\varOmega) \supset
\widehat{\mathscr{N}}^{}_{{\rm cl}, z^*_0}  := \varinjlim _{ \varOmega}
\widehat{\mathscr{N}} _{\rm cl}(\varOmega). 
\]
\end{defn}
 We  call each element of $ \widehat{\mathscr{S}} _{\rm cl}(\varOmega)$ (resp.\ 
$\widehat{\mathscr{N}} _{\rm cl}(\varOmega)$) a \textit{classical formal symbol}  (resp.\ \textit{classical formal null-symbol}) 
\textit{on $\varOmega$}.

\begin{defn}  Let $t$ be an indeterminate. Then 
we define a set $\widehat{\mathfrak{N}} _{\rm cl}(\varOmega; S) $ as follows:
 $  
  P(t;z,\zeta,\eta) = \Sum_{\nu=0}^\infty t^\nu P^{}_\nu(z,\zeta,\eta) \in \widehat{\mathfrak{N}} _{\rm cl}(\varOmega; S) $  if 
\begin{enumerate}[(i)]
\item $ 
    P(t;z,\zeta,\eta)\in \varGamma(\varOmega^{}_ \rho [d^{}_ \rho]\times S;
\mathscr{O}^{}_{T^*X \times \mathbb{C}})\llbracket t\rrbracket 
$ for some $d>0$ and $\rho\in \bigl]0,1\bigr[$, 
\item
there exists a  constant $A >0$, and  
for any $Z\Subset  S$, $h>0$, there exists  $C_{h,Z}>0$ such that
\begin{equation}\label{S.eq3.1a}
   \Bigl|\smashoperator[r]{\Sum_{\nu=0}^{m-1}} P^{}_\nu(z,\zeta,\eta)\Bigr|
 \leqslant \frac{\, C_{h,Z}\,m!A^{m}e^{h\|\zeta\|}}{\|\eta\zeta\|^{m}} 
\quad (m \in\mathbb{N},\,(z;\zeta,\eta) \in \varOmega^{}_ \rho[d^{}_{\rho}] \times Z).
\end{equation}
 \end{enumerate}
\end{defn}

\begin{defn} 
 We define a set $\widehat{\mathfrak{S}} _{\rm cl}(\varOmega; S)$ as follows:
  $P(t;z,\zeta,\eta) = \Sum_{\nu=0}^\infty t^\nu P^{}_\nu(z,\zeta,\eta) \in 
\widehat{\mathfrak{S}} _{\rm cl}(\varOmega; S)$  if 
\begin{enumerate}[(i)]
\item $ 
    P(t;z,\zeta,\eta)\in \varGamma(\varOmega^{}_ \rho [d^{}_ \rho]\times S;
\mathscr{O}^{}_{T^*X \times \mathbb{C}}) \llbracket t\rrbracket 
$ for some $d>0$ and $\rho\in \bigl]0,1\bigr[$, 
\item
there exists a  constant $A >0$,  and  
for any $Z\Subset  S$, $h>0$  there exists  $C^{}_{h,Z} >0$ such that
\begin{equation}\label{S.eq3.2a}
          |P^{}_\nu(z,\zeta,\eta)|
\leqslant \frac{\,C_{h,Z}\,\nu!A^{\nu} e^{h\|\zeta\|}}{\|\eta\zeta\|^{\nu}}\,
\quad (\nu \in\mathbb{N}^{}_0,\,(z;\zeta,\eta) \in \varOmega^{}_ \rho[d^{}_{\rho}] \times Z).
\end{equation}
\item $ 
   \partial_\eta  P(t;z,\zeta,\eta)
\in \widehat{\mathfrak{N}} _{\rm cl}(\varOmega; S)
$.
\end{enumerate}
\end{defn}
 We  call each element of $ \widehat{\mathfrak{S}}^{}_{\rm cl} (\varOmega; S) $ (resp.\ 
$\widehat{\mathfrak{N}}^{}_{\rm cl} (\varOmega; S)$) a \textit{classical formal  symbol}  (resp.\ \textit{classical 
formal null-symbol}) 
\textit{on $\varOmega$ with an apparent parameter in $ S$}.
\begin{lem}
$ \widehat{\mathfrak{N}} _{\rm cl}(\varOmega; S) 
\subset  \widehat{\mathfrak{S}} _{\rm cl}(\varOmega; S)$. 
\end{lem}
\begin{proof} We assume \eqref{S.eq3.1a}. 
 Take $C'$, $B>0$ as  $2(\nu+1) A^{\nu+1} \leqslant d^{}_\rho C'B^{\nu}$ and  
$2A^{\nu} \leqslant C'B^{\nu} $
for any $\nu \in \mathbb{N}_0$. Then for any $\nu\in \mathbb{N}$ 
and $(z;\zeta,\eta) \in \varOmega ^{}_ \rho[d ^{}_{\rho}]\times Z$ we have
\begin{align*}
|P_\nu(z,\zeta,\eta)|& =  \Bigl|
\smashoperator[r]{\Sum_{i=0}^{\nu}}  
 P _i(z,\zeta,\eta)-\smashoperator{\Sum_{i=0}^{\nu-1}}  
  P _i(z,\zeta,\eta)\Bigr| 
\leqslant  
\Bigl|\smashoperator[r]{\Sum_{i=0}^{\nu}}
 P _i(z,\zeta,\eta)\Bigr| + \Bigl|\smashoperator[r]{\Sum_{i=0}^{\nu-1}} 
 P _i(z,\zeta,\eta)\Bigr| 
\\
&\leqslant \frac{ \,C_{h,Z}\,(\nu+1)!A^{\nu+1} e^{h\|\zeta\|}}{\|\eta\zeta\|^{\nu+1}}+
 \frac{\, C_{h,Z}\, \nu!A^{\nu}e^{h\|\zeta\|}}{\|\eta\zeta\|^{\nu}}\leqslant \frac{\,C' C_{h,Z}\,\nu!B^{\nu} 
e^{h\|\zeta\|}}{|\eta|\|\eta\zeta\|^{\nu}}.
\end{align*}
 Next,  for any $Z \Subset  S$,  
take  $\delta'$ and $Z'$ as in \eqref{S.eq.1.3a}.
Then by the Cauchy inequality, 
for any $h>0$ there exists a constant  $C_{h,Z'} >0$  such that 
for any $m\in \mathbb{N}$ and $(z;\zeta,\eta) \in \varOmega ^{}_ \rho[d ^{}_{ \rho}] \times Z$,
\[
   \Bigl|\smashoperator[r]{\Sum_{\nu=0}^{m-1}} \partial_\eta P^{}_\nu(z,\zeta,\eta)\Bigr|
 \leqslant \dfrac{1}{\delta'|\eta|}\sup_{|\eta-\eta'|= \delta'|\eta|}
\Bigl|\smashoperator[r]{\Sum_{\nu=0}^{m-1}}  P^{}_\nu(z,\zeta,\eta')\Bigr|
\leqslant\frac{\,C_{h,Z'}\,m!(2A)^{m}e^{h\|\zeta\|}}{\delta'm^{}_Z \|\eta\zeta\|^{m}}\,.
\qedhere
\]
\end{proof}
For any $z^*_0\in \dot{T}^*X$, we set  
\[
\widehat{\mathfrak{S}}^{}_{{\rm cl},z^*_0}   := \smashoperator{\varinjlim_{ \varOmega,  S}}
\widehat{\mathfrak{S}}_{\rm cl}(\varOmega; S) \supset
\widehat{\mathfrak{N}}_{{\rm cl }, z^*_0}  := \smashoperator{\varinjlim _{ \varOmega,  S}}
\widehat{\mathfrak{N}} _{\rm cl}(\varOmega; S). 
\]

\begin{prop}\label{S.prop3.5}
Let  $P(t;z,\zeta,\eta) \in \widehat{\mathfrak{S}} _{\rm cl}(\varOmega; S)$. 
Then   for any  $\eta^{}_0 \in  S$, it follows that 
$P(t;z,\zeta,\eta^{}_0)\in \widehat{\mathscr{S}} _{\rm cl}(\varOmega) $ and 
$P(t;z,\zeta,\eta)- P(t;z,\zeta,\eta^{}_0) \in \widehat{\mathfrak{N}} _{\rm cl}(\varOmega; S) $. 
\end{prop}
\begin{proof}
Set $A^{}_0:=A/|\eta^{}_0|>A$. 
For any $h>0$, there exists   a constant  $C_{h,\eta^{}_0}>0$ such that 
for any $(z;\zeta) \in \varOmega ^{}_ \rho[d^{}_{\rho}]$ 
the following holds:
\[
          |P^{}_\nu(z,\zeta, \eta^{}_0)|
\leqslant \frac{\,C _{h,\eta^{}_0}\nu! A_0^{\,\nu} e^{h\|\zeta\|}}{\|\zeta\|^{\nu}};
\]
that is, $P(t;z,\zeta,\eta^{}_0)\in \widehat{\mathscr{S}} _{\rm cl}(\varOmega) $.
For any $Z \Subset  S$,  let   $Z' \Subset  S$ be 
the convex hull
of   $Z \cup\{\eta^{}_0\}$.  
Since
\[P_\nu(z,\zeta,\eta)=P _\nu(z, \zeta, \eta ^{}_0) + \displaystyle \int_{\eta ^{}_0}^{\eta}
\partial_\eta  P _\nu(z,\zeta,\tau)\,d\tau,
\] 
and
 $\partial_\eta  P(t;z,\zeta,\eta) \in \widehat{\mathfrak{N}} _{\rm cl} (\varOmega; S) $, 
 there exists $A>0$  so that 
for any $h>0$
 we can find  a  constant  $C_{h,Z'}>0$ such that for any $m\in \mathbb{N}$ and $(z;\zeta,\eta) 
\in \varOmega ^{}_ \rho[d ^{}_ {\rho}]\times Z 
\subset \varOmega ^{}_ \rho[d ^{}_ {\rho}]\times Z' $ the following holds: if
 $|\eta| \geqslant |\eta^{}_0|$
\begin{align*}
  \Bigl|\smashoperator[r]{\Sum_{\nu=0}^{m-1}} (P^{}_\nu(z,\zeta,\eta)-P_\nu(z,\zeta,\eta^{}_0))\Bigr|
& = \Bigl| \smashoperator[r]{\Sum_{\nu=0}^{m-1}} \int_{\eta^{}_0}^{\eta}
\partial_\eta  P _\nu(z,\zeta,\tau)\,d\tau\Bigr|
 = \Bigl|  \int_{\eta^{}_0}^{\eta}\smashoperator[r]{\Sum_{\nu=0}^{m-1}}
\partial_\eta  P _\nu(z,\zeta,\tau)\,d\tau\Bigr|  
\\
&
 \leqslant |\eta-\eta^{}_0|  
\frac{\,C_{h,Z'}\,m!A^{m}e^{h\|\zeta\|}}{\|\eta^{}_0\zeta\|^{m}}
\leqslant  \frac{\,r C_{h,Z'}\, m!A_0^{\,m}e^{h\|\zeta\|}}{\|\eta\zeta\|^{m}},
\end{align*}
and if $|\eta|\leqslant  |\eta^{}_0|$,
\begin{align*}
  \Bigl|\smashoperator[r]{\Sum_{\nu=0}^{m-1}} (P^{}_\nu(z,\zeta,\eta)-P_\nu(z,\zeta,\eta^{}_0))\Bigr|
&
\leqslant  \frac{\,r C_{h,Z'}\, m!A^{m}e^{h\|\zeta\|}}{\|\eta\zeta\|^{m}}
\leqslant \frac{\,r C_{h,Z'}\, m!A_0^{\,m}e^{h\|\zeta\|}}{\|\eta\zeta\|^{m}}.
\end{align*}
Hence $P(t;z,\zeta,\eta)- P(t;z,\zeta,\eta^{}_0) \in \widehat{\mathfrak{N}} _{\rm cl}(\varOmega; S) $. 
\end{proof}
Since $\|\eta\zeta\|<\|\zeta\|$ for any $\eta\in  S$, 
we can regard that
\begin{align*} 
\widehat{\mathscr{S}}^{}_{\rm cl}(\varOmega) & =\{ P(t;z,\zeta,\eta)\in 
\widehat{\mathfrak{S}}^{}_{\rm cl}(\varOmega;  S);\, \partial_\eta  P(t;z,\zeta,\eta)
=0\} \subset \widehat{\mathfrak{S}}^{}_{\rm cl}(\varOmega;  S),
\\ \widehat{\mathscr{N}}^{}_{\rm cl}(\varOmega) & =
\widehat{\mathscr{S}}^{}_{\rm cl}(\varOmega) \cap  \widehat{\mathfrak{N}}^{}_{\rm cl}(\varOmega;  S)
\subset \widehat{\mathfrak{N}}^{}_{\rm cl}(\varOmega;  S). 
\end{align*}
  Hence we have an  injective 
mapping $
\widehat{\mathscr{S}}^
{}_{\rm cl} (\varOmega)/\widehat{\mathscr{N}}^{}_{\rm cl}(\varOmega)\hookrightarrow 
 \widehat{\mathfrak{S}}^{}_{\rm cl}(\varOmega;  S)  /\widehat{\mathfrak{N}}^{}_{\rm cl}(\varOmega;  S) $. 
Moreover
\begin{prop}

$\widehat{\mathscr{S}}^{}_{\rm cl}(\varOmega) 
\big/ \widehat{\mathscr{N}}^{}_{\rm cl}(\varOmega)\earrow 
\widehat{\mathfrak{S}}^{}_{\rm cl}(\varOmega;  S) 
\big/ \widehat{\mathfrak{N}}^{}_{\rm cl}(\varOmega;  S) $.
\end{prop}
\begin{proof}
Let us take any $P(t;z,\zeta,\eta)\in \widehat{\mathfrak{S}}^{}_{\rm cl}(\varOmega;  S)  $. 
We  
fix $\eta^{}_0 \in S$. 
Then by Proposition  \ref{S.prop3.5}, we have  
$P(t;z,\zeta,\eta^{}_0) \in \widehat{\mathscr{S}}^{}_{\rm cl}(\varOmega)$ and
 $[P(t;z,\zeta,\eta)] =[P(t;z,\zeta,\eta^{}_0)] \in  \widehat{\mathfrak{S}}^{}_{\rm cl}(\varOmega;  S) 
 / \widehat{\mathfrak{N}}^{}_{\rm cl}(\varOmega;  S)$.
\end{proof}

\begin{prop} \label{S.prop3.7}
$\widehat{\mathfrak{N}}_{\rm cl}(\varOmega; S)\cap \varGamma(\varOmega^{}_ \rho[d^{}_ \rho]\times   S ;
\mathscr{O}^{}_{T^*X \times \mathbb{C}})=
\mathfrak{N} (\varOmega; S)$.
\end{prop}
\begin{proof} 
Let  $P(z,\zeta,\eta)\in\widehat{\mathfrak{N}}_{\rm cl}(\varOmega; S) \cap \varGamma(\varOmega^{}_ \rho[d^{}_ \rho]\times   S ;
\mathscr{O}^{}_{T^*X \times \mathbb{C}})$.  
Then  there exists $A>0$ and 
for any $Z \Subset   S$, $h>0$ there exists  
 $C_{h,Z}>0$  such that 
\[
          |P(z,\zeta,\eta)|
\leqslant \frac{\,C_{h,Z}\,\nu!A^{\nu} e^{h\|\zeta\|}}{\|\eta\zeta\|^{\nu}}
 \quad (\nu \in\mathbb{N}^{}_0,\, (z;\zeta,\eta) \in \varOmega ^{}_ \rho[d ^{}_{\rho}] \times Z).
\]
Then for any $\zeta$ with $\|\zeta\| \geqslant d^{}_{\rho}$, taking   $\nu$ as the integer part of $\dfrac{\,\|\eta\zeta\|\,}{A}$, we have
\[
  |P(z,\zeta, \eta)|
\leqslant C^{}_{h,Z}\Bigl(\frac{2\pi \|\eta\zeta\|}{A}\Bigr)^{\!1/2}\,e^{(h-|\eta|/A)\|\zeta\|-1}.
\]
We choose $h$ as $2hA=m^{}_Z$. Hence $e^{(h-|\eta|/A)\|\zeta\|}\leqslant e^{-\|\eta\zeta\|/(2A)}$. 
Then we can find $\delta$, $C'_Z>0$ such that 
\[
C^{}_{h,Z}\Bigl(\frac{2\pi \|\eta\zeta\|}{A}\Bigr)^{\!1/2}e^{-\|\eta\zeta\|/(2A)} \leqslant C'_Ze^{-\delta\|\eta\zeta\|}.
\]
Here $\delta$ is not depend on  $Z$. Thus \eqref{S.eq1.2a} holds.
 Conversely, assume \eqref{S.eq1.2a}. 
Set $A:=\dfrac{1}{\,\delta\,}$. Then for any $m\in \mathbb{N}^{}_0$ and  $h>0$, we have
\[
|P(z,\zeta,\eta)| \leqslant C_Z e^{-\delta\|\eta\zeta\|}\leqslant \dfrac{\,C^{}_Zm!A^{m} }{\|\eta\zeta\|^m}
\leqslant \dfrac{\,C^{}_Zm!A^{m} e^{h\|\zeta\|}}{\|\eta\zeta\|^m}\,.
\]
Hence 
$P(z,\zeta,\eta) \in\widehat{\mathfrak{N}}_{\rm cl}(\varOmega; S) \cap \varGamma(\varOmega^{}_ \rho[d^{}_ \rho]\times   S ;
\mathscr{O}^{}_{T^*X \times \mathbb{C}}) $.
\end{proof}
By Proposition \ref{S.prop3.7}
we can also regard $ \mathfrak{S} (\varOmega; S) =\widehat{\mathfrak{S}}_{\rm cl}(\varOmega; S)\cap \varGamma(\varOmega^{}_ \rho[d^{}_ \rho]\times   S ;
\mathscr{O}^{}_{T^*X \times \mathbb{C}})$. Moreover 
\begin{thm}\label{S.thm3.8}
 Let  $\varOmega \cset T^*X$ be  any sufficiently  small neighborhood of $z^* _0=(z^{}_0;\zeta^{}_0)\in  \dot{T}^*X$.  
Then for any  $  P(t;z,\zeta,\eta) =\smashoperator{\Sum_{\nu=0}^\infty}  t^\nu P^{}_\nu(z,\zeta,\eta)\in 
\widehat{\mathfrak{S}}_{\rm cl}(\varOmega; S)
$,  there exists    $P(z,\zeta) \in \mathscr{S}(\varOmega)$ such that 
\[
P(t;z,\zeta,\eta) - P(z,\zeta) 
\in \widehat{\mathfrak{N}}_{\rm cl}(\varOmega; S).\]
\end{thm}
\begin{proof} 
We may assume that $\zeta^{}_0=(1,0,\dots,0)$, 
$\varOmega \cset \{(z;\zeta);\,\Re \zeta^{}_1\geqslant 2\delta^{}_0|\zeta^{}_1|\}$ for some 
$0<\delta^{}_0 <\dfrac{1}{\;2\;}$,  $\|\zeta\| = |\zeta_1|$ on $\varOmega^{}_ \rho[d^{}_ \rho]$, and 
$P^{}_\nu(z,\zeta,\eta) \in \varGamma(\varOmega^{}_ \rho [d^{}_ \rho]\times S;
\mathscr{O}^{}_{T^*X \times \mathbb{C}})$.
Fix $\eta^{}_0\in   S$. 
Thus $  P(t;z,\zeta,\eta^{}_0) \in 
\widehat{\mathscr{S}}_{\rm cl}(\varOmega) $ by Proposition  \ref{S.prop3.5}. 
Set $A^{}_0:= A/|\eta^{}_0|$. 
We take a constant $a$ as  
\begin{equation}\label{S.eq3.1}
0<a <\min\{1,\dfrac{\,1\,}{\,2A^{}_0\,}\}
\end{equation} 
and set
 $B:=\max\{\dfrac{\,1\,}{\,\delta^{}_0a\,},\dfrac{\,A^{}_0\,}{\,2\delta^{}_0\,}\}$. 
Using  the function $\varGamma^{}_\nu(\tau,a)$ in Definition \ref{S.defn2.1}, we set 
\[
P(z,\zeta) := \Sum_{\nu=0}^\infty P^{}_\nu(z,\zeta,\eta^{}_0) \,\zeta_1^{\nu}\,
\varGamma_{\nu}(\zeta^{}_1,a).
\]
By \eqref{S.eq3.1}  and \eqref{S.eq2.6} for any $h>0$  
  on $\varOmega ^{}_\rho[d^{}_ \rho] $ we have
\[
|P(z,\zeta)| 
 \leqslant \Sum_{\nu=0}^\infty \frac{\,C_{h}\nu! A_0^{\,\nu}e^{h\|\zeta\|}\,|\zeta_1|^{\nu}\,a^\nu\,}
{\|\zeta\|^\nu\,\nu!} 
= C_{h}e^{h\|{\zeta\|}}\Sum_{\nu=0}^\infty (aA^{}_0)^\nu
\leqslant 2C_{h}e^{h\|{\zeta\|}}.
\]
Therefore $ P(z,\zeta) \in \mathscr{S}(\varOmega)$. 
On the other hand, 
for any $m \in \mathbb{N}$,  on  $\varOmega ^{}_ \rho[d^{}_ \rho]$ we have
\[
|P(z,\zeta)  -  \smashoperator{\Sum_{\nu=0}^{m-1}} P^{}_\nu(z,\zeta, \eta^{}_0) | 
\leqslant   \smashoperator{\Sum_{\nu=1}^{m-1}} |P^{}_\nu(z,\zeta, \eta^{}_0)| \,|1-\zeta_1^{\nu}\,
\varGamma_{\nu}(\zeta^{}_1,a)|     + \smashoperator{\Sum_{\nu=m}^{\infty}} |
        P^{}_\nu(z,\zeta, \eta^{}_0)\, \zeta_1^{\nu}\,
\varGamma_{\nu}(\zeta^{}_1,a)| .
\]
For any $1 \leqslant \nu\leqslant m-1$, we have $
e^{-a\delta^{}_0 |\zeta_1|}= e^{-a\delta^{}_0 \|\zeta\|}
       \leqslant \dfrac{\,(m-\nu)!\,}{\,(a\delta^{}_0\|\zeta\|)^{m-\nu}\,}.
$
Thus by \eqref{S.eq2.5} and  \eqref{S.eq3.1} we have 
\allowdisplaybreaks
\begin{align*}
 \Sum_{\nu=1}^{m-1} & |P^{}_\nu(z,\zeta, \eta^{}_0)| \,|1-\zeta_1^{\nu}\,
\varGamma_{\nu}(\zeta^{}_1,a)|   
      \leqslant\Sum_{\nu=1}^{m-1}
\frac{\,\delta^{}_0C_{h}\nu! A_0^{\,\nu}e^{h\|\zeta\|}e^{-a\delta^{}_0 |\zeta_1|}\,}{(\delta^{}_0\|\zeta\|)^\nu}     
\\
&       \leqslant 
\Sum_{\nu=1}^{m-1}\frac{\,\delta^{}_0C_{h}\nu! \,(m-\nu)!A_0^{\,\nu} e^{h\|\zeta\|}\,}
{(\delta^{}_0\|\zeta\|)^\nu\,(a\delta^{}_0 \|\zeta\|)^{m-\nu}\,}
    \leqslant 
       \frac{\,\delta^{}_0C_{h} m!B^{m}e^{h\|\zeta\|}\,}{\,\|\zeta\|^m\,}
\Sum_{\nu=1}^{m-1}\!(aA^{}_0)^\nu
 \leqslant \frac{\,\delta^{}_0C_{h}  m!B^{m}e^{h\|\zeta\|}\,}{\,\|\zeta\|^m\,}\,.
\end{align*}
Next, since
\[
 \int_{0}^a e^{- 2\delta^{}_0s \|\zeta\|}s^{k+m-1}\,ds <
a^k\int_{0}^\infty e^{- 2\delta^{}_0s \|\zeta\|}  s^{m-1}\,ds = \frac{a^k(m-1)!\,}{( 2\delta^{}_0\|\zeta\|)^m},
\]
we have
\begin{align*}
\Sum_{\nu=m}^{\infty} &|P^{}_\nu(z,\zeta, \eta^{}_0)\,\zeta_1^{\nu}\,
\varGamma_{\nu}(\zeta^{}_1,a)|
\leqslant  C_{h} e^{h\|\zeta\|}\Sum_{\nu=m}^{\infty}\frac{\,\nu A_0^{\,\nu}|\zeta_1|^\nu \,}{\,\|\zeta\|^\nu\,} \,
\int_{0}^a e^{- 2\delta^{}_0s \|\zeta\|}  \,s^{\nu-1}\,ds
\\
&= C_{h} e^{h\|\zeta\|}\Sum_{k=0}^{\infty}(k+m)A_0^{k+m} \int_{0}^a e^{- 2\delta^{}_0s  \|\zeta\|} s^{k+m-1}\,ds
\\
&  \leqslant  \frac{\, C_{h}(m-1)!A_0^m e^{h\|\zeta\|}\,}{( 2\delta^{}_0 \|\zeta\|)^m}\Sum_{k=0}^{\infty}(k+m)(aA^{}_0)^{k}
\\
&
\leqslant  \frac{\,  C_{h}(m-1)! B^{m}e^{h\|\zeta\|}}{\|\zeta\|^m}\,2(m+1)
\leqslant  \frac{\,4   C_{h}\,m!B^{m} e^{h\|\zeta\|}}{\|\zeta\|^m}.
\end{align*}
Thus  by Proposition \ref{S.prop3.5}, we have
\begin{align*}
P(t;z,\zeta,\eta) - P(z,\zeta) &= (P(t;z,\zeta,\eta)-P(t;z,\zeta,\eta^{}_0)) 
+ (P(t;z,\zeta,\eta^{}_0)- P(z,\zeta)) 
\\
&\in  \widehat{\mathfrak{N}}_{\rm cl}(\varOmega; S)+
\widehat{\mathscr{N}}_{\rm cl}(\varOmega)\subset \widehat{\mathfrak{N}}_{\rm cl}(\varOmega; S).\qedhere
\end{align*}
\end{proof}
\begin{cor}\label{S.cor3.9}
 Let  $\varOmega \cset T^*X$ be  any sufficiently  small neighborhood of $z^*_0 =(z^{}_0;\zeta^{}_0)\in  \dot{T}^*X$.  
Then 
 $ \mathfrak{S} (\varOmega; S) \big/
 \mathfrak{N} (\varOmega; S) \earrow \widehat{\mathfrak{S}}^{}_{\rm cl}(\varOmega;  S) 
\big/ \widehat{\mathfrak{N}}^{}_{\rm cl}(\varOmega;  S) $.
\end{cor}
\begin{defn} 
As in the case of   $ \mathfrak{S}(\varOmega; S)$,
 for any $P(t;z,\zeta,\eta)\in \widehat{\mathfrak{S}}^{}_{\textrm{cl}}(\varOmega; S)$ we set 
\[\wick{P(t;z,\zeta,\eta)}: = P(t;z,\zeta,\eta) \bmod 
 \widehat{\mathfrak{N}}^{}_{\textrm{cl}}(\varOmega; S) \in 
 \widehat{\mathfrak{S}}^{}_{\textrm{cl}}(\varOmega; S)
\big/ \widehat{\mathfrak{N}}^{}_{\textrm{cl}}(\varOmega; S)
\]
which is also called   the \textit{normal product} or the  \textit{Wick product} of $P(t;z,\zeta,\eta)$. 
\end{defn}
Take  $\varOmega\cset T^*\mathbb{C}^n$. 
Let 
$z=(z^{}_1,\ldots,z^{}_n)$ and  $w=(w^{}_1,\ldots,w^{}_n)$ 
be local coordinates on a neighborhood of 
$\Cl \pi(\varOmega)\subset X$, and  $(z;\zeta)$, $(w; \lambda)$  corresponding  local coordinates 
on a neighborhood of    $\Cl \varOmega$. Let 
$z=\varPhi(w)$ be the coordinate transformation. 
We define $J^{*}_{\varPhi}(z',z)$  by the relation
 $\varPhi^{-1}(z')- \varPhi^{-1}(z)=J^{*}_{\varPhi}(z',z) (z'-z)$. 
Then   ${}^{\mathsf{t}}\! J^{*}_{\varPhi}(z,z) \lambda 
= {}^{\mathsf{t}}\!\bigl[\dfrac{\partial w}{\partial z}(z)\bigr] \lambda =\zeta$. 
 Here 
$\dfrac{\partial w}{\partial z}$ stands for the Jacobian matrix of $\varPhi^{-1}$. 
Let 
$P(t;z,\zeta,\eta)\in \widehat{\mathfrak{S}}^{}_{\textrm{cl}}(\varOmega; S) $ 
with respect to  $(z;\zeta)$.  Then we set $\varPhi^*P(t;w, \lambda, \eta)= 
\Sum_{\nu=0}^\infty t^\nu \varPhi^*P^{}_\nu(w, \lambda,\eta)$ by
\[
\varPhi^*P(t;w, \lambda,\eta)
:= e^{t\langle \partial_{\zeta'},\partial_{z'}\rangle} 
   P(t; \varPhi(w),\zeta'+ {}^{\mathsf{t}}\! J^{*}_{\varPhi}(\varPhi(w) +z', \varPhi(w)) \lambda, \eta)
\big{|}{}_{\atop{z'=0}{\zeta'=0}}\,,
\]
i.e.
\[
(\varPhi^*P)^{}_\nu(w, \lambda,\eta) = \Sum_{k+|\alpha|=\nu} 
  \dfrac{1}{\alpha!}\,
 \partial_{\zeta'}^{\,\,\alpha}\partial_{z'}^{\,\,\alpha}P^{}_k(\varPhi(w),
\zeta'+{}^{\mathsf{t}}\! J^{*}_{\varPhi}(\varPhi(w) +z', \varPhi(w)) \lambda,\eta)
\big{|}{}_{\atop{z'=0}{\zeta'=0}}\,.
\]
\begin{thm}\label{S.thm3.11} 
$(1)$
$\varPhi^*P(t;w, \lambda,\eta) $ defines an element of 
$ \widehat{\mathfrak{S}}^{}_{\rm cl}(\varOmega; S) $ 
with respect to  $(w; \lambda)$. Moreover if 
$P(t;z,\zeta,\eta) \in \widehat{\mathfrak{N}}^{}_{\rm cl}(\varOmega; S)  $, 
it follows that $\varPhi^*P(t;w, \lambda, \eta) \in \widehat{\mathfrak{N}}^{}_{\rm cl}(\varOmega; S)  $. 
\par
$(2)$ $\mathds{1}^*$ is the identity, and if  $z=\varPhi(w)$  and  $w=\varPsi(v)$ are complex coordinate transformations,
$\varPsi^*\varPhi^*=(\varPhi\varPsi)^*$ holds. 
\end{thm}
\begin{proof}
(1) 
Assume that 
$P^{}_\nu(z,\zeta,\eta)\in \varGamma(\varOmega^{}_ \rho [d^{}_ \rho]\times S;
\mathscr{O}^{}_{T^*X \times \mathbb{C}})$. 
Note that $\partial^{}_\eta\varPhi^*P(t;w, \lambda,\eta) =
\varPhi^* (\partial^{}_\eta P)(t;w, \lambda,\eta) $. We may assume that  on a neighborhood of $\varOmega^{}_\rho$, 
there exist  $c>1$ and $0<c'<1$ such that
\begin{equation}\label{S.eq3.2}
c'\|\lambda\| <  \|\zeta\| = \bigl\| {}^{\mathsf{t}}\!\bigl[\dfrac{\partial w}{\partial z}(z)\bigr]  \lambda\bigr\| < c\|\lambda\|
\end{equation}
We can choose 
$0<\varepsilon<1$ such that  $c'- c \varepsilon>0$, and that there exists  $\delta>0$ such that 
if $ \|z'\|\leqslant \delta$,  it follows that 
\begin{equation}\label{S.eq3.3}
\left\{
\begin{aligned}
&\|{}^{\mathsf{t}}\! J^{*}_{\varPhi}(z+z',z) \lambda-\zeta\|  \leqslant \varepsilon\|\zeta\| ,
\\
&c'\|\lambda\| \leqslant \|{}^{\mathsf{t}}\! J^{*}_{\varPhi}(z+z',z) \lambda\| \leqslant c\|\lambda\|.
\end{aligned}\right.
\end{equation}
Moreover take $\rho'\in \bigl]0, \rho\bigr[$ and replacing  $\varepsilon$, $\delta>0$ if necessary,  setting
\[
\varOmega'_{\rho'}:=\smashoperator[r]{\Bcup_{ (z;\zeta)\in\varOmega_{\rho'}}}
\{(z;\zeta'+{}^{\mathsf{t}}\! J^{*}_{\varPhi}(z+z',z) \lambda);\,
      \|z'\|\leqslant \delta,\,
   \|\zeta'\| \leqslant \varepsilon \|\zeta\| \},
\]
we have
\begin{equation}\label{S.eq3.4}
\varOmega'_{\rho'}[d^{}_{\rho'}]
 \cset \varOmega^{}_ \rho[d^{}_ \rho].
\end{equation}
Thus for any $h>0$,   $Z\Subset  S$ and $(z;\zeta'+{}^{\mathsf{t}}\! J^{*}_{\varPhi}(z+z',z) \lambda,\eta)
\in \varOmega'_{\rho'}[d^{}_{\rho'}] \times Z$, we have
 \[
 | P^{}_k(z,\zeta'+{}^{\mathsf{t}}\! J^{*}_{\varPhi}(z+z',z) \lambda,\eta)| 
\leqslant 
 \frac{C_{h,Z}\,k!A^{k}e^{h\|\zeta'+{}^{\mathsf{t}}\! J^{*}_{\varPhi}(z+z',z) \lambda\|}\,}
{\|\eta(\zeta'+{}^{\mathsf{t}}\! J^{*}_{\varPhi}(z+z',z) \lambda)\|^k}\,.
\]
Set $\varPhi(w): =z$. 
If $(z;\zeta,\eta) \in \varOmega^{}_{\rho'}[d^{}_{\rho'}]\times Z$,   we have
\allowdisplaybreaks
\begin{align}
\frac{1}{\,\alpha!\,}\bigl| &\partial_{\zeta'}^{\,\,\alpha}\partial_{z'}^{\,\,\alpha}
P^{}_k (z,\zeta'+{}^{\mathsf{t}}\! J^{*}_{\varPhi}(z +z', z) \lambda,\eta)
\big{|}{}_{\atop{z'=0}{\zeta'=0}}\bigr| 
\label{S.eq3.5}
\\
& = \biggl| \frac{\alpha!}{(2\pi\im\,)^{2n}}
  \smashoperator{\oint_{\atop{|z'_i|=\delta}
         {|\zeta'_i|= \varepsilon\|\zeta\|}}}
  \frac{P^{}_k(z,\zeta'+{}^{\mathsf{t}}\! J^{*}_{\varPhi}(z +z', z) \lambda,\eta)\,
 dz' d\zeta'}   {\,(\zeta')^{\alpha+\mathbf{1}^{}_n}\,(z')^{\alpha+\mathbf{1}^{}_n}\,}\Biggr| 
\notag
\\
& \leqslant
  \frac{\alpha!}{\,(\varepsilon\delta\|\zeta\|)^{|\alpha|}\,}
  \smashoperator{\sup_{\atop{|z'_i|=\delta}{|\zeta'_i|=\varepsilon\|\zeta\|}}}
|P^{}_k(z,\zeta'+{}^{\mathsf{t}}\! J^{*}_{\varPhi}(z +z', z) \lambda,\eta)| \notag
         \\
& \leqslant
  \frac{\,C^{}_{h,Z}\,\alpha!\, k! A^{k}\,}
       {(\varepsilon\delta c'\|\lambda\|)^{|\alpha|}}
   \smashoperator{\sup_{\atop{|z'_i|=\delta}{|\zeta'_i|=\varepsilon\|\zeta\|}}}\,
\frac{e^{h\|\zeta'+{}^{\mathsf{t}}\! J^{*}_{\varPhi}(z +z', z) \lambda\|}\,}
{\|\eta(\zeta'+{}^{\mathsf{t}}\! J^{*}_{\varPhi}(z +z', z) \lambda)\|^k} \notag
\\& 
\leqslant
   \frac{\,C^{}_{h,Z}\,\alpha!\,k!A^{k}e^{hc(1+\varepsilon)\|\lambda\|}\,}
       {(\varepsilon\delta c'\|\lambda\|)^{|\alpha|}((c'-c\varepsilon)\|\eta\lambda\|)^{k}}
\leqslant
   \frac{C^{}_{h,Z}\,\alpha!\,k!A^{k}e^{2hc\|\lambda\|}\,}
       {(\varepsilon \delta c'\|\lambda\|)^{|\alpha|}((c'-c\varepsilon)\|\eta\lambda\|)^{k}}\,. \notag
\end{align}
Set
$B:=\dfrac{2}{\, \varepsilon\delta c'\,}$ and 
replacing  $\varepsilon$, $\delta>0$ as $C := \dfrac{\, \varepsilon \delta c'A \,}{\,2(c'-c \varepsilon)\,}<1$. 
Then 
if $\|\zeta\|\geqslant c'\|\lambda\|\geqslant (\nu +1)d^{}_{\rho'}$, we have
\begin{align*}
| (\varPhi^*P)^{}_\nu(w, \lambda,\eta)| 
& \leqslant 
C^{}_{h,Z}\,e^{2hc\|\lambda\|}
\Sum_{k= 0}^\nu   \frac{\,k!A^{k} }
       {((c'-c \varepsilon)\|\eta\lambda\|)^{k}}
\Sum_{|\alpha|= \nu-k}   \frac{\,\alpha ! \,}
       {(\varepsilon \delta c'\|\lambda\|)^{|\alpha|}}
\\
&\leqslant C^{}_{h,Z}e^{2hc\|\lambda\|}
\Sum_{k= 0}^\nu   \frac{\, k!A^{k} }
       {((c'-c\varepsilon)\|\eta\lambda\|)^{k}}
\,\frac{\,2^{n+\nu-k-1} \,(\nu-k)!\,}
       {\,(\varepsilon \delta c'\|\lambda\|)^{\nu-k}\,}
\\
& \leqslant \frac{\, 2^{n-1}
C^{}_{h,Z}\,\nu! B^\nu e^{2hc\|\lambda\|}\,}
       {\|\eta\lambda\|^{\nu}\, }\smashoperator{\Sum_{k=0}^\infty} C^{k} 
\leqslant \frac{\, 2^{n-1}
C^{}_{h,Z} \,\nu! B^\nu  e^{2hc\|\lambda\|}\,}
       {\,(1-C)\,\|\eta\lambda\|^{\nu}\, }\,.
 \end{align*}
\par
Next, if  $P(t;z,\zeta,\eta)\in \widehat{\mathfrak{N}}^{}_{\rm cl}(\varOmega; S)  $, 
for any  $m \in \mathbb{N}$ 
\[
\dfrac{1}{\,\alpha!\,}\Bigl| \smashoperator[r]{\Sum_{k=0}^{m-1}}\partial_{\zeta'}^{\,\,\alpha}\partial_{z'}^{\,\,\alpha}
P^{}_k(z,\zeta'+{}^{\mathsf{t}}\! J^{*}_{\varPhi}(z +z', z) \lambda,\eta)
\big{|}{}_{\atop{z'=0}{\zeta'=0}}\Bigr|
\leqslant
   \frac{\,C_{h,Z}\,\alpha!\, m! A^{m} e^{2hc\|\lambda\|}\,}
       {(\varepsilon\delta c'\|\lambda\|)^{|\alpha|}((c'-c \varepsilon)\|\eta\lambda\|)^{m}}\,.
\]
Hence setting $\varPhi(w)=z$,  we have
\allowdisplaybreaks
\begin{align*}
\Big|\! \Sum_{\nu=0}^{m-1} &(\varPhi^*P)^{}_\nu(w, \lambda,\eta)\Big|
 = \Bigl| \Sum_{|\alpha|=0}^{m-1}\Sum_{k=0}^{m-1-|\alpha|}\!\!
\frac{1}{\,\alpha!\,}\partial_{\zeta'}^{\,\,\alpha}\partial_{z'}^{\,\,\alpha}P^{}_k
(z,\zeta'+{}^{\mathsf{t}}\! J^{*}_{\varPhi}(z+z', z) \lambda,\eta)
\big{|}{}_{\atop{z'=0}{\zeta'=0}}\Bigr|
\\
&
\leqslant 
C_{h,Z}e^{2hc\|\lambda\|}\Sum_{|\alpha|=0}^{m-1}\frac{\,\alpha!\,(m-|\alpha|)!A^{m-|\alpha|}\,}
       {(\varepsilon \delta c'\|\lambda\|)^{|\alpha|}((c'- c \varepsilon)\|\eta\lambda\|)^{m-|\alpha|}}
\\
&\leqslant 
\frac{\, C_{h,Z}e^{2hc\|\lambda\|}\,}{\|\eta\lambda\|^m}\smashoperator{\Sum_{\nu=0}^{m-1}}\frac{\,2^{n+\nu-1}m!A^{m-\nu}\,}
       {\,(\varepsilon \delta c')^{\nu}(c'- c \varepsilon)^{m-\nu}\,}
\\
&=
\frac{\,2^{n-1} C_{h,Z} \,m! B^m e^{2hc\|\lambda\|}\,}
       {\|\eta\lambda\|^{m}}
\smashoperator{\Sum_{\nu=0}^{m-1}}C^{m-\nu}
\leqslant 
\frac{\,2^{n-1} C_{h,Z} \,m!B^{m}e^{2hc\|\lambda\|}\,}
       {(1-C)\|\eta\lambda\|^{m}}\, .
\end{align*}
\par
(2) It is  trivial that $\mathds{1}^*$ is the identity. In order to prove that $\varPsi^*\varPhi^*=(\varPhi\varPsi)^*$, 
it is enough to show that $\varPsi^*\varPhi^* P(t;v,\xi,\eta) =(\varPhi\varPsi)^* P(t;v,\xi,\eta)$ 
for  any  $P(t;z,\zeta,\eta)\in \mathscr{O}_{T^*X \times S,(z^{}_0,\zeta^{}_0,\eta)} \llbracket t\rrbracket $ 
for any fixed $(z^{}_0;\zeta^{}_0)\in \Cl \varOmega$. Note that
 $ {}^{\mathsf{t}}\!\bigl[\dfrac{\partial w}{\partial z}(z)\bigr]  \lambda =\zeta$ and $ 
{}^{\mathsf{t}}\!\bigl[\dfrac{\partial v}{\partial w}(w)\bigr] \xi= \lambda $. 
\begin{lem}[see \cite{a4}, \cite{m}]\label{S.lem3.12}
For any $n$-tuple 
$A(t;z,\zeta)=(A^{}_1(z,\zeta),\ldots,A^{}_n(z,\zeta))$ of holomorphic functions, 
and  holomorphic function $Q(z,\zeta)$, the following holds\textup{:}
\[
  e^{\langle \partial_{\zeta},\partial_{z} \rangle}
 Q(z,\zeta)\,e^{\langle z,A(z,\zeta)\rangle}\big{|}{}^{}_{z=0}
 =  e^{\langle \partial_{\zeta},\partial_{z} \rangle}\smashoperator[l]{\Sum_{\alpha\in \mathbb{N}_0^{\,n}}}
 \frac{1}{\,\alpha!\,}\,\partial_{\zeta}^{\;\alpha}
\bigl( Q(z,\zeta)A(z,\zeta)^\alpha\bigr)\big{|}{}^{}_{z=0}\,.
\]
\end{lem}
\begin{proof}
We have:
\allowdisplaybreaks
\begin{align*}
 e^{\langle \partial_{\zeta},\partial_{z} \rangle}
 Q(z,\zeta)\,e^{\langle z,A(z,\zeta)\rangle}
 \big{|}{}^{}_{z=0} &=\Sum_{\alpha,\beta\in\mathbb{N}_0^{\,n}} 
    \frac{1}{\,\alpha!\,\beta!\,}\,   \partial_\zeta^{\,\beta}
     \partial_z^{\,\beta}( Q(z,\zeta)\,z^\alpha A(z,\zeta)^\alpha) \big{|}_{ z=0}
\\
&  =\Sum_{\alpha,\beta} \Sum_{\gamma \leqslant \beta}\binom{\beta}{\gamma}
    \frac{1}{\,\alpha!\,\beta!\,}\,   \partial_\zeta^{\,\beta}\Bigl(\dfrac{\partial^{\gamma} z^\alpha}{\partial z^{\gamma}}
\,\partial_z^{\,\beta-\gamma}( Q(z,\zeta)\,A(z,\zeta)^\alpha) \Bigr)
\big{|}_{ z=0}
\\
&  = \Sum_{\alpha \leqslant\beta} 
    \frac{1}{\,\alpha!\,(\beta-\alpha)!\,}\,   \partial_\zeta^{\,\beta}
   \partial_z^{\,\beta-\alpha}( Q(z,\zeta)\,A(z,\zeta)^\alpha) 
\big{|}_{ z=0}
\\
&  = e^{\langle \partial_{\zeta},\partial_{z} \rangle}\Sum_{\alpha \in\mathbb{ N}_0^{\,n}} 
    \frac{1}{\,\alpha!\,}\,   \partial_\zeta^{\,\alpha}( Q(z,\zeta)\,A(z,\zeta)^\alpha) 
\big{|}_{z=0}\,.
\qedhere
\end{align*}
\end{proof}
\begin{rem}\label{S.rem3.13}
Take $\tau\in \mathbb{C}$ and $z'\in \mathbb{C}^n$ such that $|\tau z'|$ is sufficiently small, hence 
$J^{*}_{\varPhi}(z +\tau z',z)$ is holomorphic. Then by Lemma \ref{S.lem3.12} for any holomorphic function 
$Q(\zeta)$,  we have
\begin{align*}
 e^{\langle \partial_{\zeta'},\partial_{z'}\rangle}
 Q(\zeta')\, 
e^{\langle z',{}^{\mathsf{t}}\! J^{*}_{\varPhi}(z +\tau z', z) \lambda - \zeta^{}_0\rangle}&
\big{|}{}_{\atop{z'=0}{\zeta'=0}}
\\
&= e^{\langle \partial_{\zeta'},\partial_{z'}\rangle}
\Sum_{\alpha} \frac{1}{\,\alpha!\,}\,\partial_{\zeta'}^{\,\,\alpha}
\bigl(Q(\zeta')
({}^{\mathsf{t}}\! J^{*}_{\varPhi}(z+\tau z', z) \lambda - \zeta^{}_0)^\alpha\bigr)
\big{|}{}_{\atop{z'=0}{\zeta'=0}}
\\
&=    e^{\langle \partial_{\zeta'},\partial_{z'}\rangle}
 Q(\zeta'+{}^{\mathsf{t}}\! J^{*}_{\varPhi}(z +\tau z', z) \lambda- \zeta^{}_0 )
\big{|}{}_{\atop{z'=0}{\zeta'=0}}
\\
 &= 
\Sum_{\alpha} \dfrac{1}{\alpha!}\, \,\partial_{z'}^{\,\,\alpha}\!
\,\partial_{\zeta'}^{\,\,\alpha}
Q(\zeta'+{}^{\mathsf{t}}\! J^{*}_{\varPhi}(z+\tau z', z) \lambda- \zeta^{}_0)\big{|}{}_{\atop{z'=0}{\zeta'=0}}
\\
 &= 
\Sum_{\alpha} \dfrac{ \tau^{|\alpha|}}{\alpha!}\, \,\partial_{z'}^{\,\,\alpha}\!
\,\partial_{\zeta'}^{\,\,\alpha}
Q(\zeta'+{}^{\mathsf{t}}\! J^{*}_{\varPhi}(z+ z', z) \lambda- \zeta^{}_0)\big{|}{}_{\atop{z'=0}{\zeta'=0}}\,.
\end{align*}
Hence  as a \textit{formal power series with respect to $t$}, we have  
\begin{align*}
 e^{\langle \partial_{\zeta'},\partial_{z'}\rangle}
Q(\zeta')\, 
e^{\langle z',{}^{\mathsf{t}}\! J^{*}_{\varPhi}(z +t z', z) \lambda - \zeta^{}_0\rangle}
\big{|}{}_{\atop{z'=0}{\zeta'=0}}
&= 
\Sum_{\alpha} \dfrac{ t^{|\alpha|}}{\alpha!}\, \,\partial_{z'}^{\,\,\alpha}\!
\,\partial_{\zeta'}^{\,\,\alpha}
Q(\zeta'+{}^{\mathsf{t}}\! J^{*}_{\varPhi}(z+ z', z) \lambda- \zeta^{}_0)\big{|}{}_{\atop{z'=0}{\zeta'=0}}
\\
& =
e^{t\langle \partial_{\zeta'},\partial_{z'}\rangle} 
Q(\zeta'+{}^{\mathsf{t}}\! J^{*}_{\varPhi}(z+ z', z) \lambda - \zeta^{}_0)\big{|}{}_{\atop{z'=0}{\zeta'=0}}\,.
\end{align*}
In what follows,  we use this type of arguments. 
\end{rem}
By Lemma \ref{S.lem3.12} and Remark \ref{S.rem3.13},  setting $\varPhi(w)=z$ we have 
\begin{align*}
 e^{\langle \partial_{\zeta'},\partial_{z'}\rangle}&
 P(t; z,\zeta^{}_0+\zeta',\eta)\, e^{\langle J^{*}_{\varPhi}(z +tz',z)z', \lambda \rangle}
e^{-\langle z',\zeta^{}_0\rangle}
\big{|}{}_{\atop{z'=0}{\zeta'=0}}
\\
 &= \Sum_{\nu=0}^\infty t^\nu e^{\langle \partial_{\zeta'},\partial_{z'}\rangle}
 P^{}_\nu(z,\zeta^{}_0+\zeta',\eta)\, 
e^{\langle z',{}^{\mathsf{t}}\! J^{*}_{\varPhi}(z +tz', z) \lambda - \zeta^{}_0\rangle}
\big{|}{}_{\atop{z'=0}{\zeta'=0}}
\\
&=  \Sum_{\nu=0}^\infty t^\nu  e^{t\langle \partial_{\zeta'},\partial_{z'}\rangle}
 P^{}_\nu(z,\zeta'+{}^{\mathsf{t}}\! J^{*}_{\varPhi}(z +z', z) \lambda,\eta )
\big{|}{}_{\atop{z'=0}{\zeta'=0}}
\\
 &= 
 e^{t\langle \partial_{\zeta'},\partial_{z'}\rangle} 
P(t; z,\zeta'+{}^{\mathsf{t}}\! J^{*}_{\varPhi}(z+ z', z) \lambda, 
\eta )\big{|}{}_{\atop{z'=0}{\zeta'=0}}
=
\varPhi^*P(t;w, \lambda,\eta).
\end{align*}
On the other hand, we have
\[
w  + J^{*}_{\varPhi}(z +z', z )z'  
 = \varPhi^{-1}(z)+ \varPhi^{-1}(z +z')- \varPhi^{-1}(z) = \varPhi^{-1}(z +z').
\]
Hence  we have
\begin{align*}
J^{*}_{\varPsi}(w+ J^{*}_{\varPhi}(z +z', z)z',w)\,
J^{*}_{\varPhi}&(z +z', z )z'
\\
& =
J^{*}_{\varPsi}(\varPhi^{-1}(z +z'),\varPhi^{-1}(z))\,J^{*}_{\varPhi}(z+z',z)(z+z'-z)
\\
& =
J^{*}_{\varPsi}(\varPhi^{-1}(z+z'),\varPhi^{-1}(z))\,(\varPhi^{-1}(z+z')-\varPhi^{-1}(z))
\\
& =\varPsi^{-1}\varPhi^{-1}(z+z')-\varPsi^{-1}\varPhi^{-1}(z)
= J^{*}_{\varPhi\varPsi}(z+z',z)z'.
\end{align*} 
Thus  as a formal power series  with respect  to $t$, we have 
\[
J^{*}_{\varPsi}(w+ J^{*}_{\varPhi}(z+tz',z)tz',w)\,J^{*}_{\varPhi}(z+tz',z)z'
= J^{*}_{\varPhi\varPsi}(z+tz',z)z'.
\]
Set  $z^{}_0:=\varPhi(w^{}_0)$ and  $ {}^{\mathsf{t}}\!\bigl[\dfrac{\partial w}{\partial z}(z^{}_0)\bigr]\lambda ^{}_0=\zeta^{}_0$. 
Then on a neighborhood  of  $(w^{}_0;\eta^{}_0)$, setting $z=\varPhi (w)=  \varPhi\varPsi(v) $  we have
\allowdisplaybreaks
\begin{align*}
\varPsi^* \varPhi^*P(t;v,\xi ,\eta) 
={}&
  e^{\langle \partial_{\lambda'},\partial_{w'}\rangle}\varPhi^*P(t; w,\lambda^{}_0+ \lambda',\eta)\, 
 e^{\langle J^{*}_{\varPsi}(w+tw',w)w',\xi\rangle}
 e^{-\langle w', \lambda ^{}_0\rangle}
\big{|}{}_{\atop{w'=0}{\lambda'=0}}
\\
={}& 
   e^{\langle \partial_{\lambda'},\partial_{w'}\rangle}e^{\langle \partial_{\zeta'},\partial_{z'}\rangle}
\bigl(P(t;z,\zeta^{}_0+\zeta',\eta)\, e^{\langle J^{*}_{\varPhi}(z+tz',z)z', \lambda ^{}_0+\lambda'\rangle}
 e^{-\langle z',\zeta^{}_0\rangle}
\\*
& \cdot
  e^{\langle J^{*}_{\varPsi}(w+tw',w)w',\xi\rangle}\,
 e^{-\langle w', \lambda ^{}_0\rangle}\bigr) 
\big{|}{}_{\atop{z'=w'=0}{\zeta'= \lambda'=0}}\,.
  \end{align*}
By Lemma  \ref{S.lem3.12} and Remark \ref{S.rem3.13} we have
\begin{align*}
  e^{\langle \partial_{\lambda'},\partial_{w'}\rangle}&
e^{\langle J^{*}_{\varPhi}(z+tz',z)z', \lambda ^{}_0+ \lambda'\rangle}
 e^{\langle J^{*}_{\varPsi}(w+tw',w)w',\xi\rangle}
 e^{-\langle w', \lambda ^{}_0\rangle}
\big{|}{}_{\atop{w'=0}{\lambda'=0}}
\\
&=
 e^{\langle \partial_{\lambda'},\partial_{w'}\rangle}
e^{\langle J^{*}_{\varPhi}(z+tz',z)z', \lambda'\rangle} 
e^{\langle J^{*}_{\varPsi}(w+tw',w)w',\xi\rangle}
 e^{\langle J^{*}_{\varPhi}(z+tz',z)z'-w', \lambda ^{}_0\rangle}
\big{|}{}_{\atop{w'=0}{\lambda'=0}}
\\
  & =   e^{\langle \partial_{\lambda'},\partial_{w'}\rangle}\Sum_{\alpha} \frac{1}{\,\alpha!\,}\,\partial_{w'}^{\,\,\alpha}
\bigl((J^{*}_{\varPhi}(z+tz',z)z')^\alpha
\,e^{\langle J^{*}_{\varPsi}(w+tw',w)w',\xi\rangle}
 e^{\langle J^{*}_{\varPhi}(z+tz',z)z'-w', \lambda ^{}_0\rangle}\bigr)\big{|}{}_{\atop{w'=0}{\lambda'=0}}
\\
  &=\Sum_{\alpha} \frac{1}{\,\alpha!\,}(J^{*}_{\varPhi}(z+tz',z)z')^\alpha
\,\partial_{w'}^{\,\,\alpha}\bigl( e^{\langle J^{*}_{\varPsi}(w+tw',w)w',\xi\rangle}\,
 e^{\langle J^{*}_{\varPhi}(z+tz',z)z'-w', \lambda ^{}_0\rangle}\bigr)\big{|}{}^{}_{w'=0}
\\
  & = e^{\langle J^{*}_{\varPsi}(w+J^{*}_{\varPhi}(z+tz',z)tz',w)J^{*}_{\varPhi}(z+tz',z)z',\xi\rangle}
=  e^{\langle J^{*}_{\varPhi \varPsi}(z+tz',z)z',\xi\rangle}.
\end{align*}
Therefore
\begin{align*}
\varPsi^* \varPhi^*P(t;v,\xi ,\eta) & = 
e^{\langle \partial_{\zeta'},\partial_{z'}\rangle}
P(t;z,\zeta^{}_0+\zeta',\eta)\,
e^{\langle J^{*}_{\varPhi \varPsi}(z+tz',z)z',\xi\rangle} e^{-\langle z',\zeta^{}_0\rangle}
\big{|}{}_{\atop{z'=0}{\zeta'=0}}
\\
 &= (\varPhi\varPsi)^*P(t;v,\xi ,\eta).
\qedhere
\end{align*}
\end{proof}
\begin{defn}
Under the notation above,  we define a coordinate transformation $\varPhi^*$ 
associated with $\varPhi$ by 
\[
 \varPhi^*(\wick{P})(t;w, \lambda,\eta):= \wick{\varPhi^*P(t;w, \lambda,\eta)}. 
\]
\end{defn}
\begin{lem}\label{S.lem3.14}
Let 
$\varOmega\subset T^*X$ be a conic  open subset, $d>0$, $\rho\in \big]0,1\big[$. 
Assume that $ P(z,\zeta,\eta) \in \varGamma(\varOmega^{}_ \rho [d^{}_ \rho]\times S;
\mathscr{O}^{}_{T^*X \times \mathbb{C}})$ and 
$\nu$,  $m$, $N \in \mathbb{N}^{}_0$ satisfy the following\textup{:}

for any $h>0$ and $Z\Subset  S$,  there exists a constant $C_{h,Z}>0$  such that 
for any  $\rho' \in \big]0, \rho\big[$,  on  $\varOmega^{}_{\rho'}[d^{}_{\rho'}] \times Z$
\begin{equation}\label{eq.a.3.8}
 |P(z,\zeta,\eta)| \leqslant \frac{C_{h,Z}e^{h\|\zeta\|}}{\,(\rho-\rho')^{N\nu}\|\eta\zeta\|^m \,}.
\end{equation}
Then  for any   $\alpha$, $\beta \in\mathbb{N}^{\,n}_0$ $(|\alpha|+|\beta|>0)$, 
$\rho' \in \bigl]0, \rho\bigr[$, and     $(z;\zeta,\eta)\in \varOmega^{}_{\rho'}[d^{}_ {\rho'}] \times Z $, 
the following hold\textup{:}

$(1)$ If $\nu\in \mathbb{N}$,  set  $C^{}_\nu:=\dfrac{\,\nu+1\,}{\nu} $. Then
\allowdisplaybreaks
\[
\begin{aligned}
 |\partial_z^{\,\alpha}\partial_\zeta^{\,\beta}P(z,\zeta,\eta)|
&\leqslant 
  \frac{\,C_{h,Z}C^{\,m}_\nu e^{N}(\nu+1)^{|\alpha+\beta|}\alpha!\,\beta!\,e^{2h\|\zeta\|}\,}
       {\, (\rho-\rho')^{|\alpha+\beta|+N\nu}\, \|\eta\zeta\|^{m+|\beta|}} &&(\beta\ne 0),
\\
 |\partial_z^{\,\alpha}P(z,\zeta,\eta)|
& \leqslant 
   \frac{C_{h,Z}\,e^{N}(\nu+1)^{|\alpha|}\alpha!\, e^{h\|\zeta\|}\,}
       {\, (\rho-\rho')^{|\alpha|+N\nu} \,\|\eta\zeta\|^{m}} &&(\beta= 0). 
\end{aligned}
\]
\par
$(2)$ If $\nu=0$, then 
\allowdisplaybreaks
\[
\begin{aligned}
 |\partial_z^{\,\alpha}\partial_\zeta^{\,\beta}P(z,\zeta,\eta)|
&\leqslant 
  \frac{ C_{h,Z}\,\alpha!\,\beta!\,e^{2h\|\zeta\|} \,}
       {\,(\rho -\rho')^{|\alpha+\beta|}(1-\rho +\rho')^m\|\eta\zeta\|^{m+|\beta|}} &&(\beta\ne 0),
\\
 |\partial_z^{\,\alpha}P(z,\zeta,\eta)|
& \leqslant 
   \frac{C_{h,Z}\,\alpha!\,e^{h\|\zeta\|}\,}
       {\, (\rho-\rho')^{|\alpha|} \|\eta\zeta\|^{m}} &&(\beta = 0).
\end{aligned}
\]
\end{lem}
\begin{proof} 
(1)  Set $\rho'':=\rho'+\dfrac{\rho-\rho'}{\nu+1}$. Note that 
for any  $(z;\zeta)\in \varOmega^{}_{\rho'}[d^{}_ {\rho'}]$ and $(z',\zeta')$ with 
 $\|z'\| \leqslant \dfrac{\,\rho-\rho'\,}{\nu+1}$ and 
 $\|\zeta'\| \leqslant \dfrac{\,(\rho-\rho')\|\eta\zeta\|\,}{\nu+1}<\|\zeta\|$, we have 
 $(z+z';\zeta+\zeta')\in \varOmega^{}_{\rho''}[d^{}_ {\rho''}]$. 
Indeed, 
by the definition there exists  $(z^{}_0;\zeta^{}_0)\in \varOmega$ such that $\|z-z^{}_0\|\leqslant \rho'$ and 
$\|\zeta-\zeta^{}_0\|\leqslant \rho'\|\zeta^{}_0\|$, 
hence we have 
$\|\zeta\|\leqslant (\rho'+1)\|\zeta^{}_0\|\leqslant 2\|\zeta^{}_0\|$. 
Recall that we assumed that $S\subset\{\eta\in \mathbb{C};\,|\eta|<\dfrac{1}{\,2\,}\}$. 
Therefore  
\[
\|\zeta+\zeta'-\zeta^{}_0\| \leqslant\rho'\|\zeta^{}_0\|+\dfrac{\rho-\rho'}{\nu+1}\|\eta\zeta\| 
\leqslant\rho'\|\zeta^{}_0\|+\dfrac{\rho-\rho'}{\nu+1}\|\zeta^{}_0\| =  \rho''\|\zeta^{}_0\|. 
\]
Further we  have
\begin{equation}\label{eq.6.9a}
2\|\zeta\| \geqslant
\|\zeta+\zeta'\| \geqslant \Bigl(1-\dfrac{|\eta|(\rho-\rho')}{\nu+1}\Bigr)\|\zeta\| 
\geqslant \Bigl(1-\dfrac{\rho-\rho'}{\nu+1}\Bigr)\|\zeta\| 
 \geqslant  d(1-\rho''), 
\end{equation}
hence  $(z+z';\zeta+\zeta')\in \varOmega^{}_{\rho''}[(j+1)d^{}_ {\rho''}] $. 
Thus
replacing  $\rho'$ with $\rho''$ in  \eqref{eq.a.3.8}, 
for any $h>0$ and  $(z;\zeta,\eta)\in \varOmega^{}_{\rho'}[d^{}_ {\rho'}] \times Z $
we have
\allowdisplaybreaks
\begin{align*}
 \smashoperator[r]{
\sup_{\atop{|z'_i|=(\rho-\rho')/(\nu+1)}{|\zeta'_i|=(\rho-\rho')\|\eta\zeta\|/(\nu+1)}} }&
|P(z+z',\zeta+\zeta',\eta)|
\leqslant
\frac{C_{h,Z}e^{2h\|\zeta\|}}{\,\bigl(\bigl(1- \dfrac{1}{\nu+1}\bigr)(\rho-\rho') \big)^{N\nu} 
((1 - \dfrac{\rho-\rho'}{\nu+1})\|\eta\zeta\|)^m\,}
\\
&
\leqslant
\frac{C_{h,Z}e^{2h\|\zeta\|}}{\,\big(\dfrac{\nu}{\nu+1}(\rho-\rho') \big)^{N\nu} 
\bigl(\dfrac{\nu}{\nu+1}\|\eta\zeta\|\bigr)^m\,}
\leqslant
  \frac{\, C_{h,Z} C^{\,m}_\nu e^{N}e^{2h\|\zeta\|}}{\,(\rho-\rho')^{N\nu}\,\|\eta\zeta\|^m \,}\,.
\end{align*}
Therefore if $\beta \ne 0$, we have
\allowdisplaybreaks
\begin{align*}
| \partial_z^{\alpha}\partial_\zeta^{\,\beta}P(z,\zeta,\eta)|&\leqslant
  \frac{\,(\nu+1)^{|\alpha+\beta|}\,\alpha!\,\beta!}{\,(\rho-\rho')^{|\alpha+\beta|}\,\|\eta\zeta\|^{|\beta|}\,}
  \smashoperator[r]{\sup_{\atop{|z'_i|=(\rho-\rho')/(\nu+1)}{|\zeta'_i|=(\rho-\rho')\|\eta\zeta\|/(\nu+1)}} }
|P(z+z',\zeta+\zeta',\eta)|
\\& 
\leqslant 
  \frac{\,C^{}_{h,Z}C^{\,m}_\nu e^{N}(\nu+1)^{|\alpha+\beta|} \alpha!\,\beta!\,e^{2h\|\zeta\|}\,}
       {(\rho-\rho')^{|\alpha+\beta|+N\nu}\,\|\eta\zeta\|^{m+|\beta|}}\,.
 \end{align*}
If $\beta=0$, we have
\allowdisplaybreaks
\begin{align*}
| \partial_z^{\,\alpha}&P(z,\zeta,\eta)|\leqslant
  \frac{\,(\nu+1)^{|\alpha|}\,\alpha!\,}{\,(\rho-\rho')^{|\alpha|}\,}
    \smashoperator[r]{\sup_{\scriptscriptstyle{|z'_i|=(\rho-\rho')/(\nu+1)}}}|P(z+z',\zeta,\eta)|
\\
& \leqslant
  \frac{\,C^{}_{h,Z}(\nu+1)^{|\alpha|}\alpha!\, e^{h\|\zeta\|}}{\,(\rho-\rho')^{|\alpha|}
\big((\rho-\rho')(1- \dfrac{1}{\nu+1}) \big)^{N\nu}\, \|\eta\zeta\|^m \,}
\leqslant 
\frac{C_{h,Z}\,e^{N}(\nu+1)^{|\alpha|}\alpha!\, e^{h\|\zeta\|}\,}
       {\, (\rho-\rho')^{|\alpha|+N\nu} \,\|\eta\zeta\|^{m}}\,.
 \end{align*}

(2) We may choose $|z'_i|= \rho-\rho'$ 
and $|\zeta'_i|= (\rho-\rho')\|\eta\zeta\|$. 
\end{proof}

\begin{thm}\label{S.thm3.15}
For any $P(t;z,\zeta)$, $Q(t;z,\zeta)\in \widehat{\mathfrak{S}}^{}_{\rm cl}(\varOmega; S)$, 
set 
\begin{align*}
Q\circ P(t;z,\zeta,\eta):= {}& e^{t\langle \partial_{\zeta'},\partial_{z'} \rangle}
 Q(t;z,\zeta', \eta)\,P(t; z',\zeta,\eta)\big{|}{}_{\atop{z'=z}{\zeta'=\zeta}}
\\
= {} & e^{t\langle \partial_{\zeta'},\partial_{z'} \rangle}
 Q(t;z,\zeta+\zeta', \eta)\,P(t;z+z', \zeta,\eta)\big{|}{}_{\atop{z'=0}{\zeta'=0}}\,.
\end{align*}

$(1)$
$Q\circ P(t;z,\zeta,\eta)\in  \widehat{\mathfrak{S}}^{}_{\rm cl}(\varOmega; S) $. 
Moreover  if  either $P(t; z,\zeta,\eta)$ or  $Q(t; z,\zeta,\eta)$ is an element of  
$\widehat{\mathfrak{N}}^{}_{\rm cl}(\varOmega; S) $, it follows that 
$Q\circ P(t;z,\zeta,\eta)\in  \widehat{\mathfrak{N}}^{}_{\rm cl}(\varOmega; S) $. 

$(2)$  $R\circ (Q \circ P)= (R\circ Q) \circ P$ holds.

$(3)$  Let $\varPhi(w)=z$ be a holomorphic coordinate transformation. Then 
\[
\varPhi^* Q\circ \varPhi^*P(t;w,\lambda,\eta)= \varPhi^* (Q\circ P)(t;w,\lambda,\eta).
\]
 \end{thm}
\begin{proof}
(1) We assume $P(t; z,\zeta,\eta)  = \Sum_{\nu= 0}^\infty t^ \nu  P^{}_ \nu(z,\zeta,\eta)$, $
Q(t; z,\zeta,\eta)   = \Sum_{\nu =0}^\infty t^\nu  Q^{}_ \nu(z,\zeta,\eta)  
\in \varGamma(\varOmega^{}_ \rho [d^{}_ \rho]\times S;
\mathscr{O}^{}_{T^*X \times \mathbb{C}}) \llbracket t\rrbracket$. 
If we set  $Q \circ P(t;z,\zeta,\eta) =\Sum_{\nu=0}^\infty t^\nu  R^{}_\nu(z,\zeta,\eta)$, we have
\[
R^{}_\nu(z,\zeta,\eta)
 = \Sum_{|\alpha|+k+l=\nu} 
  \frac{1}{\alpha!}\,\partial_\zeta^{\alpha} Q^{}_l(z,\zeta,\eta)\cdot \partial_z^{\alpha} P^{}_k(z,\zeta,\eta).
\]
Therefore 
$R(t;z,\zeta,\eta) \in \varGamma(\varOmega^{}_ \rho [d^{}_ \rho]\times S;
\mathscr{O}^{}_{T^*X \times \mathbb{C}})\llbracket t\rrbracket $. 
For any 
 $\rho' \in  \bigl]0, \rho\bigr[$, $k\in \mathbb{N}^{}_0$ and $Z\Subset   S$, 
on  $\varOmega^{}_ {\rho'}[d^{}_ {\rho'}]  \times 
Z\subset \varOmega^{}_ \rho[d^{}_ {\rho}]\times Z$  we have
$|P^{}_k(z,\zeta,\eta)|,$ $|Q^{}_k(z,\zeta,\eta)| \leqslant \dfrac{C_{h,Z}k! A^{k} e^{h\|\zeta\|}}{\,\|\eta\zeta\|^k \,}$\,. 
Hence by Lemma \ref{S.lem3.14}, we have 
\begin{align}
  |\partial_{\zeta}^{\,\alpha}Q^{}_l(z,\zeta,\eta)| &
\leqslant \frac{\,  C_{h,Z}\,\alpha!\, l! A^{l}e^{2h\|\zeta\|}\,}
       {(\rho-\rho')^{|\alpha|} (1-\rho+\rho')^{l}\|\eta\zeta\|^{l+|\alpha|}},
 \notag
\\ \label{S.eq3.7}
|\partial_{z}^{\alpha}P^{}_k(z,\zeta,\eta)|
&
\leqslant 
  \frac{ C_{h,Z}\,\alpha!\, k! A^{k}e^{h\|\zeta\|}}
       {\,(\rho-\rho')^{|\alpha|}\, \|\eta\zeta\|^{k}}.
\end{align}
We choose $\rho'\in\bigl]0, \rho\bigr[$ as 
  $C:=\dfrac{(\rho-\rho')^2A}{\,2(1-\rho+\rho')\,}<1$, and set $B:=\dfrac{1}{(\rho-\rho')^2\,}$. 
Then on   $\varOmega^{}_ {\rho'}[d^{}_{\rho'}] \times Z$ we have 
\begin{align*}
|R^{}_\nu(z,\zeta,\eta)|
  &\leqslant  \Sum_{k+l=0}^ \nu \frac{C_{h,Z}^{\;2}\,k!\,l! A^{k+l}e^{3h\|\zeta\|}}{(1-\rho+\rho')^l}
\smashoperator[l]{\Sum_{|\alpha|= \nu-k-l}}
 \frac{\,\alpha!\,}
       {\, (\rho-\rho')^{2|\alpha|}\,\|\eta\zeta\|^{k+l+|\alpha|}\,}
\\
 & \leqslant  
 \frac{\,2^{n-1} C_{h,Z}^{\;2}\,\nu! B^\nu e^{3\|\zeta\|}\,}       {\,\|\eta\zeta\|^{\nu}\,}
\smashoperator{\Sum_{k+l=0}^ \nu} C^{k+l}
\leqslant  
 \frac{\,2^{n-1} C_{h,Z}^{\;2}\,\nu!  B^\nu e^{3\|\zeta\|}\,}{(1-C)^2\,\|\eta\zeta\|^\nu}.
\end{align*}
Assume that 
$P(t;z,\zeta,\eta) \in \widehat{\mathfrak{N}}^{}_{\textrm{cl}}(\varOmega; S)$.  
Then by \eqref{S.eq3.7} for any   $m\in \mathbb{N}$ we have
\[
 \biggl|\Sum_{k=0}^{m-1}\partial_z^{\alpha}P^{}_k(z,\zeta,\eta)\biggr|
\leqslant 
  \frac{\, C_{h,Z}\,\alpha!\, m!A^{m}e^{2\|\zeta\|}\,}
       {(\rho-\rho')^{|\alpha|}\,\|\eta\zeta\|^{m}\,}
\]
Then on   $\varOmega^{}_ {\rho'}[d^{}_{\rho'}] \times Z$ we have 
\allowdisplaybreaks
\begin{align*}
\Bigl|&\Sum_{\nu=0}^{m-1} R^{}_ \nu(z,\zeta,\eta)\Bigr|  
=\Bigl|\Sum_{l+|\alpha|=0}^{m-1} \frac{1}{\,\alpha!\,}\,
         \partial_{\zeta}^{\,\alpha}Q^{}_l(z,\zeta,\eta) \smashoperator[r]{\Sum_{k=0}^{m-l-1-|\alpha|}}
 \partial_z^{\alpha}P^{}_k(z,\zeta,\eta)\Bigr| 
\\
& \leqslant  \smashoperator[r]{\Sum_{l+|\alpha|=0}^{m-1}} \,\,
  \frac{\,C_{h,Z}^{\;2}\,\alpha!\, l!\,(m-l-|\alpha|)!A^{m-|\alpha|} e^{3h\|\zeta\|}\,}
       {\|\eta\zeta\|^{m}(1-\rho+\rho')^{l}(\rho -\rho')^{2|\alpha|}}
\\
& \leqslant \frac{\,2^{n-1} C_{h,Z}^{\;2}\,m! B^{m}e^{3h\|\zeta\|}\,}
       {\|\eta\zeta\|^{m}\,}\smashoperator{\Sum_{l=0}^{m-1}}C^l
\smashoperator[r]{\Sum_{i=0}^{m-l-1}}\,((1-\rho+\rho')C)^{m-l-i} 
\leqslant  \frac{\,2^{n-1}C C_{h,Z}^{\;2}\,m! B^{m}e^{3h\|\zeta\|}\,}
       {\,(1-C)^2\|\eta\zeta\|^{m}\,}\,.
\end{align*}
Therefore 
 $Q\circ P(t;z,\zeta,\eta) \in \widehat{\mathfrak{N}}^{}_{\textrm{cl}}(\varOmega; S)$. 
For the same reasoning, we can show that if 
$Q(t;z,\zeta,\eta) \in \widehat{\mathfrak{N}}^{}_{\textrm{cl}}(\varOmega; S)$, 
we have
$Q\circ P(t;z,\zeta,\eta) \in \widehat{\mathfrak{N}}^{}_{\textrm{cl}}(\varOmega; S)$.
In particular, since 
\[
\partial_{\eta} (Q\circ P )(t;z,\zeta,\eta) = \partial_{\eta} Q\circ P (t;z,\zeta,\eta) 
+ Q\circ  \partial_{\eta} P (t;z,\zeta,\eta), 
\]
we see that if $P(t;z,\zeta,\eta)$, $Q(t;z,\zeta,\eta) \in 
\widehat{\mathfrak{S}}^{}_{\textrm{cl}}(\varOmega; S)$, 
we have 
$Q\circ P(t;z,\zeta,\eta) \in \widehat{\mathfrak{S}}^{}_{\textrm{cl}}(\varOmega; S)$.

The proof of (2)  is easy. 

(3) We may show the equality  around   each point $(\varPhi^{-1} (z^{}_0),
{}^{\mathsf{t}}\!\bigl[\dfrac{\partial z}{\partial w}(w^{}_0)\bigr]\zeta^{}_0)= 
(w^{}_0, \lambda^{}_0) $ as  a formal series. 
Set $z:= \varPhi(w)$. 
We remark that 
\[ 
\partial_{\lambda'}^{\,\alpha}
 e^{\langle J^{*}_{\varPhi}(z +z'',z)z'', \lambda +\lambda'\rangle}
=  e^{\langle J^{*}_{\varPhi}(z +z'',z)z'', \lambda +\lambda'\rangle}
 \,(\varPhi^{-1}(z +z'')- w)^{\alpha},
\]
hence as a formal power series with respect to $t$,  we have
\[ 
t^{|\alpha|}\partial_{\lambda'}^{\,\alpha}
 e^{\langle J^{*}_{\varPhi}(z+tz'',z)z'', \lambda +\lambda'\rangle}
=  e^{\langle J^{*}_{\varPhi}(z +tz'',z)z', \lambda +\lambda'\rangle}
 \,(\varPhi^{-1}(z +tz'')- w)^{\alpha}.
\] 
Further since
\[
J^{*}_{\varPhi}(z +z'',z)z''+
J^{*}_{\varPhi}(z +z''+z',z+z'')z'=\varPhi^{-1}(z+z''+z')- \varPhi^{-1}(z)
= J^{*}_{\varPhi}(z +z''+z',z)(z''+z'),
\]
as a formal power series with respect to $t$,  we have 
\[ 
J^{*}_{\varPhi}(z+tz'',z)z''+
J^{*}_{\varPhi}(z+tz''+tz',z+z'')z'
\\
= J^{*}_{\varPhi}(z +tz''+tz',z)(z''+z').
\] 
Therefore we have 
\allowdisplaybreaks
\begin{align*}
\varPhi^*Q\circ\varPhi ^*P&(t;w,\lambda,\eta)
\\
={}& e^{t\langle \partial_{\lambda'},\partial_{\widetilde{w}}\rangle}
 e^{\langle \partial_{\zeta''},\partial_{z''}\rangle}
 Q(t; z,\zeta^{}_0+\zeta'',\eta)\, e^{\langle J^{*}_{\varPhi}(z +tz'',z)z'', \lambda +\lambda'\rangle}
e^{-\langle z'',\zeta^{}_0\rangle}
\\*
& \cdot
e^{\langle \partial_{\zeta'},\partial_{z'}\rangle}
 P(t; \varPhi(w+ \widetilde{w}),\zeta^{}_0+\zeta',\eta)\, e^{\langle J^{*}_{\varPhi}(\varPhi(w+ \widetilde{w}) +tz',\varPhi(w+ \widetilde{w}))z', \lambda \rangle}
e^{-\langle z',\zeta^{}_0\rangle}
\big{|}{}_{\atop{z'=z''= \widetilde{w} =0}{\zeta'=\zeta''=\lambda'=0}}  
\\
={}&  e^{\langle \partial_{\zeta''},\partial_{z''}\rangle}\Sum_{\alpha}
Q(t; z,\zeta^{}_0+\zeta'',\eta)\, 
e^{\langle J^{*}_{\varPhi}(z +tz'',z)z'', \lambda \rangle}
e^{-\langle z'',\zeta^{}_0\rangle}
\,
\frac{(\varPhi^{\,-1}(z +tz'')-w)^{\alpha}}{\alpha!}
\\*
&\cdot
\partial_{w}^{\,\alpha}(
e^{\langle \partial_{\zeta'},\partial_{z'}\rangle}
 P(t; \varPhi(w),\zeta^{}_0+\zeta',\eta)\, e^{\langle J^{*}_{\varPhi}(\varPhi(w) +tz',\varPhi(w))z', \lambda \rangle}
e^{-\langle z',\zeta^{}_0\rangle}
\big{|}{}_{\atop{z'=z''=0}{\zeta'=\zeta''=0}} 
\\
={}& 
e^{\langle \partial_{\zeta''},\partial_{z''}\rangle}(
Q(t; z,\zeta^{}_0+\zeta'',\eta)\, 
e^{\langle J^{*}_{\varPhi}(z+tz'',z)z'', \lambda \rangle}
e^{-\langle z'',\zeta^{}_0\rangle}
\\*
& \cdot
e^{\langle \partial _{\zeta' },\partial_{z'}\rangle}
P(t; z +tz'',\zeta^{}_0+\zeta' ,\eta)\,
e^{\langle J^{*}_{\varPhi}(z+tz''+tz',z+tz'')z', \lambda \rangle}
e^{-\langle z',\zeta^{}_0\rangle}
\big{|}{}_{\atop{z'=z''=0}{\zeta'=\zeta''=0}} 
\\
={}& e^{\langle \partial_{\zeta' },\partial_{z'}\rangle}
e^{\langle \partial_{\zeta''},\partial_{z''}\rangle} 
Q(t; z ,\zeta^{}_0+\zeta'',\eta)\,P(t; z +tz'',\zeta^{}_0+\zeta' ,\eta)
\\* &\cdot
e^{\langle J^{*}_{\varPhi}(z +tz''+tz',z)(z''+z'), \lambda \rangle}
e^{-\langle z''+z',\zeta^{}_0\rangle}
\big{|}{}_{\atop{z'=z''=0}{\zeta'=\zeta''=0}}  
\\
={}& \Sum_{\alpha,\beta,\gamma} \frac{t^{|\alpha|}}{\alpha!\,\beta!\,\gamma!\,}\,
\partial_{\zeta'}^{\alpha+\beta} Q(t; z ,\zeta^{}_0+\zeta',\eta)\,\partial_{z}^{\,\alpha}
\partial_{\zeta'}^{\,\gamma} P(t; z,\zeta^{}_0+\zeta' ,\eta)
\\*& \cdot
\partial_{z'}^{\,\gamma}\partial_{z''}^{\,\beta}
e^{\langle J^{*}_{\varPhi}(z +tz''+tz',z)(z''+z'), \lambda \rangle}
e^{-\langle z''+z',\zeta^{}_0\rangle}
\big{|}{}_{\atop{z'=z''=0}{\zeta'=0}}
\\
={}& \Sum_{\beta,\gamma} \frac{1}{(\beta+\gamma)!\,}\,
\partial_{\zeta}^{\beta+\gamma}(Q \circ P)(t; z,\zeta^{}_0+\zeta' ,\eta)\,
\partial_{z'}^{\beta+\gamma}e^{\langle J^{*}_{\varPhi}(z+tz',z)z', \lambda \rangle}
e^{-\langle z',\zeta^{}_0\rangle}
\big{|}{}_{\atop{z'=0}{\zeta'=0}}
\\
={}& 
 e^{\langle \partial_{\zeta' },\partial_{z'}\rangle}(Q\circ P)(t;z,\zeta^{}_0+\zeta',\eta)\,
e^{\langle J^{*}_{\varPhi}(z+tz',z)z', \lambda \rangle}
e^{-\langle z',\zeta^{}_0\rangle}\big{|}{}_{\atop{z'=0}{\zeta'=0}}
\\
= {}&
  \varPhi^*(Q\circ P)(t;w,\lambda,\eta).
\qedhere
\end{align*}
\end{proof}
\begin{defn}
For any
 $\wick{P(t;z,\zeta,\eta)}$, $\wick{Q(t;z,\zeta,\eta)}\in 
\widehat{\mathfrak{S}}^{}_{\textrm{cl}}(\varOmega; S)\big/\widehat{\mathfrak{N}}^{}_{\textrm{cl}}(\varOmega; S)$, we define the product by:
\[
   \wick{Q(t;z,\zeta ,\eta)} \, \wick{P(t;z,\zeta ,\eta)}
:=\wick{Q\circ P(t;z,\zeta ,\eta)}.
\]
\end{defn}
\begin{thm}\label{S.thm3.17}
For any $P(t;z,\zeta,\eta)\in 
\widehat{\mathfrak{S}}^{}_{\rm cl}(\varOmega; S)$, set 
\[
P^*(t;z,\zeta,\eta):=e^{t\langle \partial_\zeta,\partial_z \rangle}
 P(t;z,-\zeta,\eta). 
\]

$(1)$
 $P^*(t;z,\zeta,\eta)\in  
\widehat{\mathfrak{S}}^{}_{\rm cl}(\varOmega^a; S)$,  where $\varOmega^a:=\{(z;\zeta);\,
(z;-\zeta)\in \varOmega\}$, and   $P^{**}=P$. 
Moreover if  $P(t;z,\zeta,\eta)\in \widehat{\mathfrak{N}}^{}_{\rm cl}(\varOmega; S) $, 
it follows that  $P^*(t;z,\zeta,\eta)\in \widehat{\mathfrak{N}}^{}_{\rm cl}(\varOmega^a; S) $.  

$(2)$  $
(Q \circ P)^*(t;z,\zeta,\eta)= P^* \circ Q^*(t;z,\zeta,\eta)$. 

$(3)$  Let $\varPhi(w)=z$ be a holomorphic coordinate transformation. Then  it follows that 
on $\widehat{\mathfrak{S}}^{}_{\rm cl}(\varOmega^a;S) \smashoperator{\tens_{\mathscr{O}^{}_{X}}} \varOmega^{}_{X}$,
\begin{equation}\label{eq6.11}
\varPhi^*(P^*(t;z,\zeta,\eta) \tens dz)= \varPhi^*(P(t;z,\zeta,\eta) \tens dz)^*.
\end{equation}
Here $\varPhi^*$ also stands for  the pull-back of differential forms, and   \eqref{eq6.11} 
 means that 
\[
\det \dfrac{\partial z}{\partial w}\,\varPhi^*(P^*) = (\det \dfrac{\partial z}{\partial w} \varPhi^* P)^* =
(\varPhi^* P)^* \circ\det \dfrac{\partial z}{\partial w}\,.
\]
\end{thm}
\begin{proof}
(1) For $P(t;z,\zeta,\eta)  = \Sum_{i=0}^\infty t^i  P^{}_i(z,\zeta,\eta)$, 
we set  $P^*(t;z,\zeta ,\eta) =\Sum\limits_{i=0}^\infty t^j  P^{*}_i(z,\zeta ,\eta)$. Then
\[
P^{*}_i(z,\zeta ,\eta)
 =\Sum_{|\alpha|+k=i} 
  \frac{1}{\alpha!}\,\partial_\zeta^{\,\alpha}\,
 \partial_z^{\,\alpha}P^{}_k(t;z,-\zeta ,\eta).
\]
As in the proof of Theorem  \ref{S.thm3.15},  we can show 
$P^*(t;z,\zeta,\eta)\in  
\widehat{\mathfrak{S}}^{}_{\rm cl}(\varOmega^a; S)$, 
and   that if  $P(t;z,\zeta)\in \widehat{\mathfrak{N}}^{}_{\rm cl}(\varOmega; S) $, 
we have   $P^*(t;z,\zeta)\in \widehat{\mathfrak{N}}^{}_{\rm cl}(\varOmega^a; S) $.  
Moreover
\begin{align*}
P^{**}(t;z,\zeta,\eta)
 &=e^{t\langle \partial_\zeta,\partial_z \rangle}P^*(t;z,-\zeta ,\eta)
=e^{t\langle \partial_\zeta,\partial_z \rangle}e^{-t\langle \partial_\zeta,\partial_z \rangle}P(t;z,\zeta ,\eta)
= P(t;z,\zeta,\eta).
\end{align*}

(2) We have
\begin{align*}
(Q &\circ P)^*(t;z,\zeta,\eta)  = e^{t\langle \partial_\zeta,\partial_z \rangle}(Q \circ P(t;z,-\zeta,\eta))
\\
&
=   e^{t\langle \partial_\zeta,\partial_z \rangle}(e^{-t\langle \partial_{\zeta''},\partial_{z'} \rangle}
 Q(t;z,-\zeta-\zeta'', \eta)\,P(t;z+z', -\zeta,\eta)\big{|}{}_{\atop{z'=0}{\zeta''=0}})
\\
&
=   e^{t\langle \partial_{\zeta'}+ \partial_{\zeta''},\partial_{z'}+ \partial_{z''} \rangle}e^{-t\langle \partial_{\zeta''},\partial_{z'} \rangle}
 Q(t;z+z'',-\zeta-\zeta'', \eta)\,P(t;z+z', -\zeta-\zeta',\eta)\big{|}{}_{\atop{z'=z''=0}{\zeta'=\zeta''=0}}
\\
&
=   e^{t\langle \partial_{\zeta'}, \partial_{z''} \rangle}e^{t\langle \partial_{\zeta'},\partial_{z'} \rangle}
P(t;z+z', -\zeta-\zeta',\eta)
e^{t\langle  \partial_{\zeta''}, \partial_{z''} \rangle}
 Q(t;z+z'',-\zeta-\zeta'', \eta)\,
\big{|}{}_{\atop{z'=z''=0}{\zeta'=\zeta''=0}}
\\
&
=   e^{t\langle \partial_{\zeta'}, \partial_{z''} \rangle}
P^*(t;z+z', \zeta+\zeta',\eta)\,
 Q^*(t;z+z'',\zeta+\zeta'', \eta)\,
\big{|}{}_{\atop{z'=z''=0}{\zeta'=\zeta''=0}}
\\
&
=   e^{t\langle \partial_{\zeta'}, \partial_{z''} \rangle}
P^*(t;z, \zeta+\zeta',\eta)\,
 Q^*(t;z+z'',\zeta, \eta)\,
\big{|}{}_{\atop{z''=0}{\zeta'=0}} =P^* \circ Q^*(t;z,\zeta,\eta).
\end{align*}

(3) (i) For any holomorphic function $a(z,\eta)$, we have 
\allowdisplaybreaks
\begin{align*}
\varPhi^*(a(z,\eta)^* \tens dz)
&=\varPhi^*(a(z,\eta) \tens dz) = a(\varPhi(z),\eta) 
 \tens \det \dfrac{\partial z}{\partial w}\,dw 
= a(\varPhi(z),\eta) \det \dfrac{\partial z}{\partial w}
 \tens dw, 
\\
\varPhi^*(a(z,\eta)\tens dz)^*&
= ( a(\varPhi(z),\eta) \det \dfrac{\partial z}{\partial w})^*\tens dw =
a(\varPhi(z),\eta) \det \dfrac{\partial z}{\partial w}\tens dw .
\end{align*}
Set $J:= \dfrac{\partial z}{\partial w}$ for short. 
Since  $\varPhi^*(\zeta^{}_i) = \Sum_{k=1}^n \dfrac{\partial w^{}_k}{\partial z^{}_i}\,\lambda^{}_k$, 
we have 
\allowdisplaybreaks
\begin{align*}
\varPhi^*((\zeta^{}_i)^* \tens dz)
&= -\varPhi^*(\zeta^{}_i \tens dz) = -
\Sum_{k=1}^n \dfrac{\partial w^{}_k}{\partial z^{}_i}\,\lambda^{}_{k} 
 \tens \det \dfrac{\partial z}{\partial w}\,dw 
=-\det J
\Sum_{k=1}^n \dfrac{\partial w^{}_k}{\partial z^{}_i}\,\lambda^{}_{k} \tens dw ,
\\
\varPhi^*(\zeta^{}_i\tens dz)^*&
= (\det J \Sum_{k=1}^n \dfrac{\partial w^{}_k}{\partial z^{}_i}\,\lambda^{}_{k} )^*\tens dw 
\\
& =
-\Sum_{k=1}^n \!\Bigl(\det J\,\dfrac{\partial w^{}_k}{\partial z^{}_i}\,\eta^{}_{k} 
+\frac{\partial  \det J}{\partial w^{}_k}\,
\dfrac{\partial w^{}_k}{\partial z^{}_i}
+\det J \dfrac{\partial }{\partial w^{}_k}\Bigl(
\frac{\partial w^{}_k}{\partial z^{}_i}\Bigr)\Bigr) \otimes dw.
\end{align*}
Here we remark 
\[
\frac{\partial \det J}{\partial w^{}_k}
=\det J\operatorname{Tr} \Bigl( J^{\,-1}
\frac{\partial J}{\partial w^{}_k}\Bigr),
\quad
\dfrac{\partial  J^{\,-1}}{\partial w^{}_k} =- J^{\,-1} \dfrac{\partial J}{\partial w^{}_k}
J^{\,-1}.
\] 
Hence we have
\begin{align*}
\Sum_{k=1}^n \!\Bigl(\frac{\partial  \det J}{\partial w^{}_k}\,
\dfrac{\partial w^{}_k}{\partial z^{}_i}
+ \det J& \dfrac{\partial}{\partial w^{}_k}\Bigl(
\frac{\partial w^{}_k}{\partial z^{}_i}\Bigr)\Bigr)
\\
&
=
\det J \Sum_{k=1}^n
\Bigl( \operatorname{Tr}\Bigl(J^{\,-1}
\frac{\partial  J}{\partial w^{}_k}\Bigr)\dfrac{\partial w^{}_k}{\partial z^{}_i} 
-\Sum_{\mu,\nu=1}^n\frac{\partial w^{}_{k}}{\partial z^{}_\mu}
\frac{\partial^2 z^{}_{\mu }}{\partial w^{}_k\partial w^{}_\nu}\frac{\partial w^{}_{\nu}}{\partial z^{}_i}\Bigr)
\\
&= 
\det J\Sum_{k,\mu,\nu=1}^n\!
 \Bigl(\frac{\partial w^{}_\nu}{\partial z^{}_\mu}
\dfrac{\partial^2 z^{}_{\mu}}{\partial w^{}_\nu \partial w^{}_k}\frac{\partial w^{}_k}{\partial z^{}_i}
-\frac{\partial w^{}_k}{\partial z^{}_\mu} \dfrac{\partial^2 z^{}_{\mu}}{\partial w^{}_k \partial w^{}_\nu} 
\frac{\partial w^{}_\nu}{\partial z^{}_i}\Bigr)=0.
\end{align*}
Therefore $\varPhi^*((\zeta^{}_i)^* \tens dz)=\varPhi^*(\zeta^{}_i\tens dz)^*$. 

(ii)  If $P$ and $Q$ satisfy \eqref{eq6.11}, then
\begin{align*}
\varPhi^*&((Q\circ P)^*\otimes dz)= \varPhi^*(P^* \circ  Q^*\otimes dz) =
\varPhi^*(P^*) \circ \varPhi^*(Q^*) \tens \det \dfrac{\partial z}{\partial w}\,dw
\\
& =\det \dfrac{\partial z}{\partial w}\,\varPhi^*(P^*)\circ\varPhi^*( Q^*) \tens dw
=  (\varPhi^* P)^*\circ \det \dfrac{\partial z}{\partial w}\,\varPhi^*(Q^*) \tens dw
\\
&= (\varPhi^* P)^* \circ \varPhi^*(Q^*) \circ  \det \dfrac{\partial z}{\partial w} \tens dw
= \varPhi^* (Q\circ P)^*  \circ  \det \dfrac{\partial z}{\partial w}   \tens dw
\\
&= (\det \dfrac{\partial z}{\partial w}\,\varPhi^*(Q \circ P))^*  \tens dw
= (\det \dfrac{\partial z}{\partial w}\,\varPhi^*(Q \circ  P)  \tens dw)^*
= \varPhi^*( Q \circ P\otimes dz)^*
\end{align*}

(iii) Take any point $(z^{}_0;\zeta^{}_0)\in \varOmega^{}_{\rho}$ and consider the Taylor expansion
\[
P(t;z,\zeta,\eta) =\Sum_{\alpha} p^{}_{\alpha}(t;z,\zeta^{}_0,\eta)\,(\zeta-\zeta^{}_0)^{\alpha}
=\Sum_{\alpha} p^{}_{\alpha}(t;z,\zeta^{}_0,\eta) \circ (\zeta-\zeta^{}_0)^{\alpha}.
\]
We may prove \eqref{eq6.11} in  a  formal  sense. 
Then by  induction and (i),  (ii), we obtain
\begin{align*}
\varPhi^* (P^*\tens dz) &=  \Sum_{\alpha}
\varPhi^* ((p^{}_{\alpha}(t;z,\zeta^{}_0,\eta) \circ (\zeta-\zeta^{}_0)^{\alpha })^* \tens dz)
\\
&=\Sum_{\alpha} \varPhi^* (p^{}_{\alpha}(t;z,\zeta^{}_0,\eta) \circ (\zeta-\zeta^{}_0)^{\alpha } \tens dz)^*
=
\varPhi^*(P\otimes dz)^*.
\qedhere
 \end{align*}
\end{proof}
 \begin{defn}
For any
 $\wick{P(t;z,\zeta,\eta)}\in 
\widehat{\mathfrak{S}}^{}_{\textrm{cl}}(\varOmega; S)
\big/\widehat{\mathfrak{N}}^{}_{\textrm{cl}}(\varOmega; S)$, we define the formal adjoint by 
\[
 (\wick{P(t;z,\zeta, \eta)})^*:=\wick{P^*(t;z,\zeta ,\eta)} \in 
\widehat{\mathfrak{S}}^{}_{\textrm{cl}}(\varOmega^a; S)
\big/\widehat{\mathfrak{N}}^{}_{\textrm{cl}}(\varOmega^a; S).
\]
\end{defn}

\begin{rem}
We identify $X$ with the diagonal set of $X \times X$. Then the 
sheaf  $\mathscr{D}^\infty_{X}$ of  holomorphic differential operators of infinite order on $X$ is 
defined by
\[
\mathscr{D}^\infty_{X}:= H_{X}^{n}(\varOmega^{\smash{(0,n)}}_{X\times X}).
\]  
Recall that for any open subset  $U\subset X$, a section 
\smash{$P(z,\partial^{}_z) =\Sum_{\alpha}a^{}_\alpha(z)\,\partial_z^{\alpha}\in  \varGamma(U;
\mathscr{D}^{\infty}_{X})$} is given by the following equivalent conditions:
\par
$(1)$ For any  $W \Subset U$ and $h>0$ there exists $C>0$  such that 
\[
\sup\limits_{z\in W} |a^{}_\alpha(z)| \leqslant \frac{C h^{|\alpha|}}{|\alpha|!}\,.
\]

$(2)$  Set $P(z,\zeta) :=\Sum_{\alpha}a^{}_\alpha(z)\,\zeta^{\,\alpha}$.
For any  $W \Subset U$ and $h>0$ there exists $C>0$  such that 
\[
\sup_{z\in W}|P(z,\zeta)|
  \leqslant 
  C e^{h\|\zeta\|}.
\]
$\mathscr{D}^{\infty}_X$ is a sheaf of rings on $X$, and:
\begin{enumerate}[(1)] 
\item the coordinate transform  is given by
\[
\varPhi^*P(w,\lambda) =  e^{\langle \partial_{\zeta'},\partial_{z'} \rangle}P
(\varPhi(w),\zeta'+{}^{\mathsf{t}}\!J^{*}_{\varPhi}(\varPhi(w) +z', \varPhi(w))\lambda)
\big{|}{}_{\atop{z'=0}{\zeta'=0}}\,,
\]
\item the product is given by 
\[
Q\circ P(z,\zeta)=  e^{\langle \partial_{\zeta'},\partial_{z'} \rangle}
 Q(z,\zeta+\zeta')\,P(z+z', \zeta)\big{|}{}_{\atop{z'=0}{\zeta'=0}}\,,
\]
\item the formal adjoint is given by
\[
P^*(z,\zeta)=e^{\langle \partial_\zeta,\partial_z \rangle}
 P(z,-\zeta). 
\]
\end{enumerate}
From our point of view, 
the sheaf $\mathscr{E}^{\mathbb{R}}_{X}$ on  $T^*X$  can be defined 
 $\mathscr{E}^{\mathbb{R}}_{X}|{}^{}_{X}:=\mathscr{D}^\infty_{X}$,  and on   $\dot{T}^*X$, 
associated conic sheaf with
\[
\mathrm{Op}(\dot{T}^*X) \owns V \mapsto 
\varinjlim_{ S}\mathfrak{S}(\mathbb{R}^{}_{>0}V; S)
\big/\mathfrak{N}(\mathbb{R}^{}_{>0}V; S) =
\mathscr{S}(\mathbb{R}^{}_{>0}V)\big/\mathscr{N}(\mathbb{R}^{}_{>0}V), 
\]
or equivalently 
\[
\mathrm{Op}(\dot{T}^*X) \owns V \mapsto 
\varinjlim_{ S}\widehat{\mathfrak{S}}_{\rm cl}(\mathbb{R}^{}_{>0}V; S)\big/
\widehat{\mathfrak{N}}_{\rm cl}(\mathbb{R}^{}_{>0}V; S) =
\widehat{\mathscr{S}}_{\rm cl}(\mathbb{R}^{}_{>0}V)\big/
\widehat{\mathscr{N}}_{\rm cl}(\mathbb{R}^{}_{>0}V) .
\] 
\end{rem}

\begin{thm} Let $[\psi(z,w,\eta)\,dw]$, $[\varphi(z,w,\eta)\,dw]\in \mathscr{E}^{\smash{\mathbb{R}}}_{X, z^*_0}$.  
Set for short
 \[
\sigma(\psi) \odot\sigma(\varphi)(z,\zeta,\eta):= 
\Sum_{\alpha}\dfrac{1}{\,\alpha!\,} 
\partial_ \zeta ^ \alpha\sigma(\psi) (z,\zeta,\eta) 
\, \partial^ \alpha_z\sigma(\varphi) (z,\zeta,\eta).\]
Then the following hold\textup{:}

$(1)$ 
 $\sigma(\psi)\odot \sigma(\varphi)(z,\zeta,\eta)\in \mathfrak{S}^{}_{z^*_0}$.

$(2)$ $\sigma(\psi)\odot \sigma(\varphi)(z,\zeta,\eta)
-\sigma(\psi)\circ \sigma(\varphi)(t;z,\zeta,\eta)
\in \widehat{\mathfrak{N}}^{}_{{\rm cl},z^*_0}$. 

$(3)$
$\sigma(\mu(\psi\tens \varphi))(z,\zeta,\eta) - \sigma(\psi)\odot \sigma(\varphi)(z,\zeta,\eta)\in 
\mathfrak{N}^{}_{z^*_0}$.

\noindent 
In other words,  the product in  $\mathscr{E}^{\smash{\mathbb{R}}}_{X,z^*_0}$ coincides with 
that  of  the classical symbols given  in Theorem \ref{S.thm3.15} through the symbol mapping $\sigma$. 
\end{thm}
\begin{proof}
 We assume that 
$\sigma(\psi)(z,\zeta,\eta)$, $\sigma(\varphi)(z,\zeta,\eta) \in  \varGamma(\varOmega^{}_ \rho[d^{}_ \rho]\times   S ;
\mathscr{O}^{}_{T^*X \times \mathbb{C}})$. 

(1) Take $\rho'\in \bigl]0,\rho\bigr [ $.
Fix any $Z\Subset  S$ and $h>0$. 
In $\gamma(0, \eta;\varrho,\theta)$, 
we can change $\gamma_i(0, \eta;\varrho)$ as $ \{w_i = |\eta|s' e^{2\pi\iim t};\, 0 \leqslant t \leqslant 1\}$ with 
$0<\varrho^{-1}<s'$ $(2\leqslant i \leqslant n)$, and $ \gamma_1(0, \eta;\varrho, \theta)\subset \{|w_1| \leqslant |\eta|s' \}$. 
Therefore, we have
\begin{align*}
|\partial_ \zeta ^ \alpha\sigma(\psi) (z,\zeta,\eta)| & =\Bigl| \partial_ \zeta ^ \alpha
\smashoperator{\int\limits_{\hspace{6ex}\gamma(0, \eta;\varrho,\theta) }} 
\psi(z,z+\widetilde{w},\eta) \,e^{\langle \widetilde{w}, \zeta\rangle} d\widetilde{w}\Bigr|
 =\Bigl| 
\smashoperator{\int\limits_{\hspace{6ex}\gamma(0, \eta;\varrho,\theta) }} \widetilde{w}^\alpha\psi(z,z+\widetilde{w},\eta) \,e^{\langle \widetilde{w}, \zeta\rangle} 
d\widetilde{w}\Bigr|
\\& 
\leqslant 
(|\eta|s')^{|\alpha|} C^{}_{h,Z}e^{h\|\zeta\|}.
\end{align*}
For the same reason, we have
\[
| \partial_ \eta\partial_ \zeta ^ \alpha\sigma(\psi) (z,\zeta,\eta)| =
| \partial_ \zeta ^ \alpha\partial_ \eta\sigma(\psi) (z,\zeta,\eta)| 
\leqslant (|\eta|s')^{|\alpha|} Ce^{-\delta\|\eta\zeta\|}.
\]
In the same way taking $\|z\|\leqslant  r'_0  <r'$, for some $R>0$ we have 
\begin{align*}
|\partial_ z ^ \alpha\sigma(\varphi) (z,\zeta,\eta)| &\leqslant \dfrac{ \,C_{h,Z}\,\alpha!\, e^{h\|\zeta\|}}{R^{|\alpha|}},
\\
| \partial_ \eta\partial_z^ \alpha\sigma(\varphi) (z,\zeta,\eta)| &\leqslant \dfrac{ \,C^{}_{Z}\, \alpha!\,e^{-\delta\|\eta\zeta\|}}{R^{|\alpha|}}.
\end{align*}
Hence taking $r$ small enough as $rs'<R$, we have 
\[
|\sigma(\psi)\odot \sigma(\varphi)(z,\zeta,\eta)|
\leqslant \dfrac{C_{h,Z}^{\;2}\, e^{2h\|\zeta\|}}{(1-|\eta| s'/R)^n}\,.
\]
For any $(z; \zeta,\eta )\in \varOmega ^{}_ {\rho'}[d ^{}_ {\rho'}]\times Z$, choosing $h= \delta m^{}_Z/2$, 
we have
\begin{align*}
|\partial_ \eta( & \sigma(\psi)\odot \sigma(\varphi))(z,\zeta,\eta)|
\\
&= \Sum_{\alpha}\dfrac{1}{\,\alpha!\,} 
( \partial_ \eta\partial_ \zeta ^ \alpha\sigma(\psi) (z,\zeta,\eta) 
\, \partial^ \alpha_z\sigma(\varphi) (z,\zeta,\eta)+
\partial_ \zeta ^ \alpha\sigma(\psi) (z,\zeta,\eta) 
\, \partial_ \eta \partial^ \alpha_z\sigma(\varphi) (z,\zeta,\eta))
\\
& \leqslant  
\dfrac{4C^{}_{h,Z}Ce^{-\delta\|\eta\zeta\|/2}}{(1-|\eta| s'/R)^n}\,.
\end{align*}

(2) Fix   $m\in \mathbb{N}$. Then by Lemma \ref{S.lem3.14}, 
 under the same notation of 
proof of (1), for any  $\beta,\gamma\in \mathbb{N}_0^{\,n}$ with $|\beta|=m$, on $\varOmega ^{}_ {\rho'}[d ^{}_ {\rho'}]\times Z$ 
we have
\begin{align*}
|\partial_ \zeta ^{\beta+\gamma}\sigma(\psi) (z,\zeta,\eta)| & =\Bigl| \partial_ \zeta ^{\beta+\gamma}
\smashoperator{\int\limits_{\hspace{6ex}\gamma(0, \eta;\varrho,\theta)}}
 \psi(z,z+\widetilde{w},\eta) \,e^{\langle \widetilde{w}, \zeta\rangle} d\widetilde{w}\Bigr|
 \\
&=\Bigl| \partial_ \zeta ^{\;\beta}
\smashoperator{\int\limits_{\hspace{6ex}\gamma(0, \eta;\varrho,\theta) }}
 \widetilde{w}^\gamma\psi(z,z+\widetilde{w},\eta) \,e^{\langle \widetilde{w}, \zeta\rangle} 
d\widetilde{w}\Bigr|
 \leqslant  \dfrac{
(|\eta|s')^{|\gamma|} C^{}_{h,Z}\,m!A^{m}e^{2h\|\zeta\|}}{(\rho-\rho')^{m}\|\eta\zeta\|^{m}}.
\end{align*}
Therefore, setting $B:= \dfrac{2A}{(\rho-\rho')R}$  we obtain
\begin{align*}
\Bigl| \sigma(\psi)&\odot \sigma(\varphi)(z,\zeta,\eta)
       - \Sum_{|\alpha|=0}^{m-1}\dfrac{1}{\,\alpha!\,} 
\partial_ \zeta ^ \alpha\sigma(\psi) (z,\zeta,\eta) 
\, \partial^ \alpha_z\sigma(\varphi) (z,\zeta,\eta)
\Bigr|
\\
       &
 = \Bigl|\Sum_{|\alpha| \geqslant m}\dfrac{1}{\,\alpha!\,} 
\partial_ \zeta ^ \alpha\sigma(\psi) (z,\zeta,\eta) 
\, \partial^ \alpha_z\sigma(\varphi) (z,\zeta,\eta)
           \Bigr|
\\
&
 = \Bigl| \Sum_{|\beta|= m} \Sum_{|\gamma|= 0}^\infty \dfrac{1}{\, (\beta+\gamma)!\,} 
\partial_ \zeta ^{\beta+\gamma}\sigma(\psi) (z,\zeta,\eta) 
\, \partial ^{\beta+\gamma}_z\sigma(\varphi) (z,\zeta,\eta)
           \Bigr|
\\
 &\leqslant 
\dfrac{C^{\,2}_{h,Z}\,m!A^{m}e^{3h\|\zeta\|}}{(\rho-\rho')^m\|\eta\zeta\|^{m}}\Sum_{|\beta|= m} \Sum_{|\gamma|= 0}^\infty 
\dfrac{ (|\eta|s')^{|\gamma|}}{R^{| \beta +\gamma|}}
 \leqslant 
\dfrac{2^{n-1}C^{\,2}_{h,Z}\,m! B^m e^{3h\|\zeta\|}}{(1-|\eta| s'/R)^n\|\eta\zeta\|^{m}}\, .
\end{align*}

(3) Take two paths $\gamma(0,\eta;\tilde{\varrho},\tilde{\theta})$ 
and $\gamma(0,\eta;\varrho',\theta')$. Here 
we take $\tilde{\varrho}$  is sufficiently smaller  
than $\varrho$,  
and next we take $\varrho'$  is sufficiently smaller   than $\tilde{\varrho}$. 
Hence we may assume
\begin{align*}
\mu(\psi\tens \varphi)(z,z+w,\eta)
&=\smashoperator{\int\limits_{\hspace{6ex}\gamma(0, \eta;\tilde{\varrho},\tilde{\theta})}} \psi(z,z+\widetilde{w},\eta) 
 \, \varphi(z+\widetilde{w},z+w,\eta) \,d\widetilde{w},
\\
\sigma(\mu(\psi\tens \varphi))(z,\zeta,\eta)&
=\smashoperator{\int\limits_{\hspace{6ex}\gamma(0, \eta;\varrho',\theta')}}
\mu(\psi\tens \varphi)(z,z+w,\eta)\,e^{\langle w, \zeta\rangle} dw
\\
& =\int\limits_{\gamma(0, \eta;\varrho',\theta') } \hspace{-2ex} dw
\smashoperator{\int\limits_{\hspace{6ex}\gamma(0, \eta;\tilde{\varrho},\tilde{\theta})}} \psi(z,z+\widetilde{w},\eta) 
 \, \varphi(z+\widetilde{w},z+w,\eta)\,e^{\langle w, \zeta\rangle} d\widetilde{w}
\\
&=
\int\limits_{\gamma(0, \eta;\tilde{\varrho},\tilde{\theta}) } \hspace{-2ex} d\widetilde{w}
\smashoperator{\int\limits_{\hspace{7ex}\gamma(0, \eta;\varrho',\theta')}}\psi(z,z+\widetilde{w},\eta) 
 \, \varphi(z+\widetilde{w},z+w,\eta) \,e^{\langle w, \zeta\rangle} dw.
\end{align*}
Then   we find
\allowdisplaybreaks
\begin{align*}
 \sigma(\psi)\odot \sigma(\varphi)(z,\zeta,\eta)
=
&{}
 \Sum_{\alpha}\dfrac{1}{\,\alpha!\,}\smashoperator{\int\limits_{\hspace{6ex}\gamma(0, \eta;\varrho,\theta) }} 
 \widetilde{w}^\alpha\psi(z,z+\widetilde{w},\eta) \,
e^{\langle \widetilde{w}, \zeta\rangle} d\widetilde{w}
\smashoperator{\int\limits_{\hspace{6ex}\gamma(0, \eta;\varrho,\theta)}} \partial_ z ^ \alpha\varphi(z,z+w,\eta) \,e^{\langle w, \zeta\rangle} dw
\\
=
&{}
 \Sum_{\alpha}\dfrac{1}{\,\alpha!\,}\smashoperator{\int\limits_{\hspace{6ex}\gamma(0, \eta;\tilde{\varrho},\tilde{\theta})}} 
\widetilde{w}^\alpha\psi(z,z+\widetilde{w},\eta) \,
e^{\langle \widetilde{w}, \zeta\rangle} d\widetilde{w}
\smashoperator{\int\limits_{\hspace{6ex}\gamma(0, \eta;\varrho,\theta) }} \partial_ z ^ \alpha\varphi(z,z+w,\eta) \,e^{\langle w, \zeta\rangle} dw
\\*
&
+
\Sum_{\alpha}\dfrac{1}{\,\alpha!\,}
\smashoperator{\int\limits_{\hspace{16ex}\gamma(0, \eta;\varrho,\theta) \vee (-\gamma(0, \eta;\tilde{\varrho},\tilde{\theta})) } }
 \widetilde{w}^\alpha\psi(z,z+\widetilde{w},\eta) \,
e^{\langle \widetilde{w}, \zeta\rangle} d\widetilde{w}
\smashoperator{\int\limits_{\hspace{6ex}\gamma(0, \eta;\varrho,\theta)}} 
 \partial_ z ^ \alpha\varphi(z,z+w,\eta) \,e^{\langle w, \zeta\rangle} dw
\end{align*}
By the Cauchy integration theorem, $\gamma^{}_1(0, \eta;\varrho,\theta) \vee (-\gamma^{}_1(0, \eta;\tilde{\varrho},\tilde{\theta})) $ 
can be changed to the  following two  segment paths: 
\[
\bigl[\dfrac{\tilde{\varrho}\eta}{2} \,e^{-\iim (\pi + \tilde{\theta})/2},\dfrac{\varrho\eta}{2} \,e^{-\iim (\pi + \theta)/2}\bigr],
\quad
\bigl[\dfrac{\varrho\eta}{2} \,e^{\iim (\pi + \theta)/2} ,\dfrac{\tilde{\varrho}\eta}{2} \,e^{\iim (\pi +\tilde{ \theta})/2}\bigr].
\]
Then we  can find 
 $\delta>0$ such that on the two paths above  $\Re\langle w^{}_1, \zeta^{}_1\rangle \leqslant -\delta|\eta\zeta^{}_1|$ holds. 
Thus  as in (1)  we 
can see 
\[
\Sum_{\alpha}\dfrac{1}{\,\alpha!\,}
\smashoperator{\int\limits_{\hspace{16ex}\gamma(0, \eta;\varrho,\theta)\vee( -\gamma(0, \eta;\tilde{\varrho},\tilde{\theta})) }} 
\widetilde{w}^\alpha\psi(z,z+\widetilde{w},\eta) \,
e^{\langle \widetilde{w}, \zeta\rangle} d\widetilde{w}
\smashoperator{\int\limits_{\hspace{6ex}\gamma(0, \eta;\varrho,\theta)}} \partial_ z ^ \alpha\varphi(z,z+w,\eta) \,e^{\langle w, \zeta\rangle} dw
\in \mathfrak{N}^{}_{z^*_0}\,.
\]
Next   we find
\begin{align*}
 \Sum_{\alpha}\dfrac{1}{\,\alpha!\,}&
\smashoperator{\int\limits_{\hspace{6ex}\gamma(0, \eta;\tilde{\varrho},\tilde{\theta})}}  \widetilde{w}^\alpha\psi(z,z+\widetilde{w},\eta) \,
e^{\langle \widetilde{w}, \zeta\rangle} d\widetilde{w}
\smashoperator{\int\limits_{\hspace{6ex}\gamma(0, \eta;\varrho,\theta)}}
 \partial_ z ^ \alpha\varphi(z,z+w,\eta) \,e^{\langle w, \zeta\rangle} dw
\\
={}
&{} \smashoperator{\int\limits_{\hspace{14ex}\gamma(0, \eta;\tilde{\varrho},\tilde{\theta}) \times\gamma(0, \eta;\varrho,\theta) }} 
\psi(z,z+\widetilde{w},\eta)
\Bigl( \Sum_{\alpha}\dfrac{\widetilde{w}^\alpha}{\,\alpha!\,}
\partial_ z ^ \alpha\varphi(z,z+w,\eta)\Bigr) e^{\langle \widetilde{w}+w, \zeta\rangle} d\widetilde{w}\,dw
\\
={}& \smashoperator{\int\limits_{\hspace{14ex}\gamma(0, \eta;\tilde{\varrho},\tilde{\theta}) \times\gamma(0, \eta;\varrho,\theta)} }
 \psi(z,z+\widetilde{w},\eta) 
 \, \varphi(z+\widetilde{w},z+\widetilde{w}+w,\eta) \,e^{\langle \widetilde{w}+w, \zeta\rangle} d\widetilde{w}\,dw
\\
={}&\int\limits_{\gamma(0, \eta;\tilde{\varrho},\tilde{\theta}) }\hspace{-2ex}d\widetilde{w}
\smashoperator{\int\limits_{\hspace{6ex}\gamma(\widetilde{w}, \eta;\varrho,\theta) }}  \psi(z,z+\widetilde{w},\eta) 
 \, \varphi(z+\widetilde{w},z+w,\eta) \,e^{\langle w, \zeta\rangle} dw
\\
={}&\int\limits_{\gamma(0, \eta;\tilde{\varrho},\tilde{\theta})}\hspace{-2ex} d\widetilde{w}
\smashoperator{\int\limits_{\hspace{6ex}\gamma(0, \eta;\varrho,\theta) } }\psi(z,z+\widetilde{w},\eta) 
 \, \varphi(z+\widetilde{w},z+w,\eta) \,e^{\langle w, \zeta\rangle} dw
\\*
&
+ \int\limits_{\gamma(0, \eta;\tilde{\varrho},\tilde{\theta})}\hspace{-2ex} d\widetilde{w}
\smashoperator{\int\limits_{\hspace{17ex}\gamma(\widetilde{w}, \eta;\varrho,\theta) \vee(-\gamma(0, \eta;\varrho,\theta))}}
 \psi(z,z+\widetilde{w},\eta) 
 \, \varphi(z+\widetilde{w},z+w,\eta) \,e^{\langle w, \zeta\rangle} dw.
\end{align*}
We consider
\allowdisplaybreaks
\begin{align*}
\sigma(\mu(\psi\tens \varphi)) (z,\zeta,\eta)& -\int\limits_{\gamma(0, \eta;\tilde{\varrho},\tilde{\theta}) } \hspace{-2ex} d\widetilde{w}
\smashoperator{\int\limits_{\hspace{6ex}\gamma(0, \eta;\varrho,\theta)} }\psi(z,z+\widetilde{w},\eta) 
 \, \varphi(z+\widetilde{w},z+w,\eta) \,e^{\langle w, \zeta\rangle} dw
\\
&=\int\limits_{\gamma(0, \eta;\tilde{\varrho},\tilde{\theta}) } \hspace{-2ex} d\widetilde{w}
 \smashoperator{\int\limits_{\hspace{17ex}\gamma(0, \eta;\varrho',\theta') \vee (-\gamma(0, \eta;\varrho,\theta))}}\psi(z,z+\widetilde{w},\eta) 
 \, \varphi(z+\widetilde{w},z+w,\eta) \,e^{\langle w, \zeta\rangle} dw.
\end{align*}
By the Cauchy integration theorem, $\gamma^{}_1(0, \eta;\varrho',\theta')\vee ( -\gamma^{}_1(0, \eta;\varrho,\theta))$ 
can be changed to the  following two  segment paths: 
\[
\bigl[\dfrac{\varrho\eta}{2} \,e^{-\iim (\pi + \theta)/2},\dfrac{\varrho'\eta}{2} \,e^{-\iim (\pi + \theta')/2}\bigr],
\quad
\bigl[\dfrac{\varrho'\eta}{2} \,e^{\iim (\pi + \theta')/2} ,\dfrac{\varrho\eta}{2} \,e^{\iim (\pi + \theta)/2}\bigr].
\]
Then we  can find 
 $\delta>0$ such that on the two paths above  $\Re\langle w^{}_1, \zeta^{}_1\rangle \leqslant -\delta|\eta\zeta^{}_1|$ holds. 
Thus    we  can see 
\[
\sigma(\mu(\psi\tens \varphi)) (z,\zeta,\eta) -\int\limits_{\gamma(0, \eta;\tilde{\varrho},\tilde{\theta})}\hspace{-2ex} d\widetilde{w}
 \smashoperator{\int\limits_{\hspace{6ex}\gamma(0, \eta;\varrho,\theta)}}\psi(z,z+\widetilde{w},\eta) 
 \, \varphi(z+\widetilde{w},z+w,\eta) \,e^{\langle w, \zeta\rangle} dw\in \mathfrak{N}^{}_{z^*_0}.
\]
Next we consider two  segment paths 
\[
\bigl[\dfrac{\varrho\eta}{2} \,e^{-\iim (\pi + \theta)/2},\dfrac{\varrho\eta}{2} \,e^{-\iim (\pi + \theta)/2}+\widetilde{w}_1\bigr],
\quad
\bigl[\dfrac{\varrho\eta}{2}\, e^{\iim (\pi + \theta)/2}+\widetilde{w}_1 ,\dfrac{\varrho\eta}{2} \,e^{\iim (\pi + \theta)/2}\bigr].
\]
Since $\tilde{\varrho}$ is sufficiently smaller than   $\varrho$ and  $|\widetilde{w}_1|\leqslant \dfrac{\tilde{\varrho}|\eta|}{2}$,  
we  can find  $\delta>0$ such that on the two paths above  
 $\Re\langle w^{}_1, \zeta^{}_1\rangle \leqslant -\delta|\eta\zeta^{}_1|$ holds.   Therefore   we 
can conclude that
\[
\int\limits_{\gamma(0, \eta;\tilde{\varrho},\tilde{\theta}) }\hspace{-2ex} d\widetilde{w}
\smashoperator{\int\limits_{\hspace{14ex}\gamma(\widetilde{w}, \eta;\varrho,\theta)-\gamma(0, \eta;\varrho,\theta) }} 
 \psi(z,z+\widetilde{w},\eta) 
 \, \varphi(z+\widetilde{w},z+w,\eta) \,e^{\langle w, \zeta\rangle} dw\in \mathfrak{N}^{}_{z^*_0}.
\]
The proof is complete. 
\end{proof}
\begin{rem}
 Let $[\psi(z,w,\eta)\,dw]\in \mathscr{E}^{\smash{\mathbb{R}}}_{X, z^*_0}$. Then 
we can also prove the following:
\begin{enumerate}[(1)]
\item We have  
\[\wick{P^*(t;z,\zeta,\eta)} = \wick{\smashoperator{\int\limits_{\hspace{6ex}\gamma(0, \eta;\varrho,\theta)}}
\psi(z-w,z,\eta)\,e^{-\langle w,\zeta \rangle}\,dw}.
\]
\item
Let $z=\varPhi(w)$ be a complex coordinate transformation. Then (see \eqref{B4})
\[
\wick{\varPhi^*P(t;w,\lambda,\eta)} = \wick{
\smashoperator{\int\limits_{\hspace{6ex}\gamma(z, \eta;\varrho,\theta)} } \psi(z,z',\eta)\,e^{\langle \varPhi^{-1}(z')-
\varPhi^{-1}(z),\lambda \rangle} dz'}.
\]
\end{enumerate}
\end{rem}

\section{Formal Symbols with an Apparent  Parameter}\label{sec:formal-symbol}

\begin{defn}[see \cite{a2}, \cite{aky}]  Let $t$ be an indeterminate. 

(1) 
$  
  P(t;z,\zeta) = \Sum_{\nu=0}^\infty t^\nu P^{}_\nu(z,\zeta)$
is an element of $\widehat{\mathscr{S}}(\varOmega) $  if  $ 
    P_\nu(z,\zeta)\in \varGamma(\varOmega^{}_ \rho [(\nu+1)d_ \rho ];
\mathscr{O}^{}_{T^*X})$ for some $d>0$ and $\rho \in \bigl]0,1\bigr[$,  
and there exists a   constant $A \in \bigl]0,1\bigr[$    satisfying  the following: 
for any  $h >0 $  there exists a  constant   $C_h >0$  such that   
\[
          |P^{}_\nu(z,\zeta)| \leqslant C_h A^{\nu} e^{h\|\zeta\|} \quad  
(\nu \in \mathbb{N}^{}_0,\,(z;\zeta) \in \varOmega^{}_ \rho [(\nu+1)d^{}_ \rho]).
\]

(2) Let 
$  
  P(t;z,\zeta) = \Sum_{\nu=0}^\infty t^\nu P^{}_\nu(z,\zeta)\in \widehat{\mathscr{S}} (\varOmega)$. 
Then $ P(t;z,\zeta)$ 
is an element of $\widehat{\mathscr{N}} (\varOmega)$  if 
there exists a   constant $A \in \bigl]0,1\bigr[$    satisfying  the following: 
for any  $h >0 $  there exists a  constant   $C_h>0$  such that 
\[
          \Bigl|\smashoperator[r]{\Sum_{\nu=0}^{m-1}} P^{}_\nu(z,\zeta)\Bigr|
 \leqslant C_h A^{m}e^{h\|\zeta\|} \quad (m \in\mathbb{N},\,(z;\zeta) \in \varOmega_ \rho [md_ \rho ]).
\]

(3) For  $z^*_0 \in \dot{T}^*X$, we set  
\[
\widehat{\mathscr{S}}^{}_{z^*_0 }  := \varinjlim_{ \varOmega}
\widehat{\mathscr{S}}(\varOmega) \supset
\widehat{\mathscr{N}}^{}_{z^*_0 }  := \varinjlim _{ \varOmega}
\widehat{\mathscr{N}}(\varOmega). 
\]
\end{defn}
We  call each element of $ \widehat{\mathscr{S}} (\varOmega) $ (resp.\ 
$\widehat{\mathscr{N}}(\varOmega)$) a \textit{formal  symbol}  (resp.\ \textit{formal null-symbol}) 
\textit{on $\varOmega$}.

For $U \subset  S$ and $m\in \mathbb{N}$, we set 
\[
(\varOmega^{}_\rho * U)[md^{}_\rho] := \{(z;\zeta,\eta) \in \varOmega_ \rho \times U;\,\|\eta\zeta\|\geqslant m d^{}_\rho\}
\subset \varOmega_ \rho[md^{}_\rho] \times U.
\]
\begin{defn}  Let $t$ be an indeterminate. 
We say that
$    P(t;z,\zeta,\eta) = \Sum_{\nu=0}^\infty t^\nu P^{}_\nu(z,\zeta,\eta)$
is an element of $\widehat{\mathfrak{N}}(\varOmega; S)$  if 
\begin{enumerate}[(i)]
\item $ 
    P_\nu(z,\zeta,\eta)\in
\varGamma((\varOmega^{}_ \rho *S )[(\nu+1)d^{}_ \rho];
\mathscr{O}^{}_{T^*X \times \mathbb{C}})$ for some $d>0$ and $\rho \in \bigl]0,1\bigr[$, 
\item
there exists a    constant $A \in \bigl]0,1\bigr[$,
and for any $Z \Subset   S $, $h>0$    there exists   $C_{h,Z}>0$ such that
\begin{equation}
\label{S.eq4.1a}
   \Bigl|\smashoperator[r]{\Sum_{\nu=0}^{m-1}} P^{}_\nu(z,\zeta,\eta)\Bigr|
 \leqslant C_{h,Z}A^{m}e^{h\|\zeta\|}
\quad (m \in\mathbb{N},\,(z;\zeta,\eta) \in(\varOmega^{}_\rho * Z)[md^{}_\rho]).
\end{equation}
\end{enumerate}
\end{defn}

\begin{defn} 
(1) We say that  $P(t;z,\zeta,\eta) = \Sum_{\nu=0}^\infty t^\nu P^{}_\nu(z,\zeta,\eta) $
is an element of $\widehat{\mathfrak{S}}(\varOmega; S)$  if 
\begin{enumerate}[(i)]
\item $ 
    P_\nu(z,\zeta,\eta)\in
\varGamma((\varOmega^{}_ \rho * S)[(\nu+1)d^{}_ \rho];
\mathscr{O}^{}_{T^*X \times \mathbb{C}})$ for some $d>0$ and $\rho\in \bigl]0,1\bigr[$, 
\item
there exists a    constant $A \in \bigl]0,1\bigr[$, 
and for any $Z \Subset   S $, $h>0$,     there exists $C_{h,Z}>0$ such that
\[
| P^{}_\nu(z,\zeta,\eta)|
 \leqslant C_{h,Z}A^{\nu}e^{h\|\zeta\|}
\quad   
(\nu \in\mathbb{N}^{}_0,\,(z;\zeta,\eta) \in (\varOmega^{}_\rho * Z)[(\nu+1)d^{}_\rho]).
\] 
\item $ 
   \partial_\eta  P(t;z,\zeta,\eta)
\in \widehat{\mathfrak{N}} (\varOmega; S)
$.
\end{enumerate}
\end{defn}
We  call each element of $ \widehat{\mathfrak{S}}(\varOmega; S) $ (resp.\ 
$\widehat{\mathfrak{N}} (\varOmega; S)$) a \textit{formal  symbol}  (resp.\ \textit{formal null-symbol}) 
\textit{on $\varOmega$ with an apparent parameter in $ S$}.
\begin{lem}\label{S.lem4.3}
$\widehat{\mathfrak{N}}(\varOmega; S) \subset \widehat{\mathfrak{S}}(\varOmega; S)$. 
\end{lem}
\begin{proof}
We assume \eqref{S.eq4.1a}.
 For any $\nu\in \mathbb{N}$ and $(z;\zeta,\eta) \in 
(\varOmega^{}_\rho * Z)[(\nu+1)d^{}_\rho]\subset (\varOmega^{}_\rho * Z)[\nu d^{}_\rho]$, we have
\begin{align*}
|P_\nu(z,\zeta,\eta)|& =  \Bigl|
\smashoperator[r]{\Sum_{i=0}^{\nu}}  
 P _i(z,\zeta,\eta)-\smashoperator{\Sum_{i=0}^{\nu-1}}  
  P _i(z,\zeta,\eta)\Bigr| 
\leqslant  
\Bigl|\smashoperator[r]{\Sum_{i=0}^{\nu}}
 P _i(z,\zeta,\eta)\Bigr| + \Bigl|\smashoperator[r]{\Sum_{i=0}^{\nu-1}} 
 P _i(z,\zeta,\eta)\Bigr| 
\\
&\leqslant C_{h,Z}A^{\nu+1} e^{h\|\zeta\|}+
C_{h,Z}A^{\nu} e^{h\|\zeta\|} \leqslant C_{h,Z}(A+1)A^{\nu} e^{h\|\zeta\|}.
\end{align*}
 Next, 
for any $Z \Subset  S$,  
take  $\delta'$ and $Z'$ as in \eqref{S.eq.1.3a}.
Then by the Cauchy inequality, 
for any $h>0$ there exist constants  $C_{h,Z'}$, $R >0$ such that 
for any $m\in \mathbb{N}$ and $(z;\zeta,\eta) \in 
(\varOmega^{}_\rho * Z)[m(2d)^{}_\rho]$, 
 the following holds:
\[
\Bigl|\smashoperator[r]{\Sum_{\nu=0}^{m-1}} \partial_\eta P^{}_\nu(z,\zeta,\eta)\Bigr| \leqslant 
\dfrac{1}{\delta'|\eta|}\sup_{|\eta-\eta'|= \delta'|\eta|}\Bigl|\smashoperator[r]{\Sum_{\nu=0}^{m-1}}  P^{}_\nu(z,\zeta,\eta')\Bigr|\leqslant\frac{C_{h,Z'}A^{m} e^{h\|\zeta\|}}{\delta'm^{}_Z}\,.
\qedhere
\]
\end{proof}
We set  
\[
\widehat{\mathfrak{S}}_{z^*_0}   := \smashoperator{\varinjlim_{ \varOmega,  S}}
\widehat{\mathfrak{S}}(\varOmega; S) \supset
\widehat{\mathfrak{N}} _{z^*_0}  := \smashoperator{\varinjlim _{ \varOmega,  S}}
\widehat{\mathfrak{N}}(\varOmega; S). 
\]

\begin{prop}\label{S.prop4.5}
Let   $P(t;z, \zeta, \eta)\in  \widehat{\mathfrak{S}}(\varOmega; S)$. 
Then for any  $\eta^{}_0 \in  S$, 
$P(t;z,\zeta,\eta^{}_0)\in \widehat{\mathscr{S}}(\varOmega) $ and 
$P(t;z,\zeta,\eta)- P(t;z,\zeta,\eta^{}_0) \in \widehat{\mathfrak{N}}(\varOmega; S) $. 
\end{prop}
\begin{proof}
Set $d':= d/|\eta^{}_0|>d$. Then for any $h>0$, there exists a constant   $C_{h,\eta^{}_0}>0$ such that 
\[
          |P^{}_\nu(z,\zeta, \eta^{}_0)|
\leqslant C_{h,\eta^{}_0} A^{\nu} e^{h\|\zeta\|} 
\quad ((z;\zeta) \in \varOmega_ \rho[(\nu+1)d'_{\rho}]).
\]
Therefore $P(t;z,\zeta,\eta^{}_0)\in \widehat{\mathscr{S}}(\varOmega) $. 
For any $Z \Subset  S$,  let  $ Z' \Subset  S$ be the  convex hull 
of  $Z \cup\{\eta^{}_0\}$. 
Since 
\[P_\nu(z,\zeta,\eta)=P _\nu(z, \zeta, \eta ^{}_0) + \displaystyle \int_{\eta ^{}_0}^{\eta}
\partial_\eta  P _\nu(z,\zeta,\tau)\,d\tau
\] 
and
 $\partial_\eta  P(t;z,\zeta,\eta) \in \widehat{\mathfrak{N}}  (\varOmega; S) $, 
there exists a    constant $A \in \bigl]0,1\bigr[$, and  for any  $h >0$
 we can find  a  constant  $C_{h,Z'}>0$ such that for any $m\in \mathbb{N}$ and $(z;\zeta,\eta) 
\in (\varOmega^{}_\rho * Z)[md'_\rho]
\subset (\varOmega^{}_\rho * Z')[m d'_\rho]$,  the following holds:
\begin{align*}
  \Bigl|\smashoperator[r]{\Sum_{\nu=0}^{m-1}} (P^{}_\nu(z,\zeta,\eta)-P_\nu(z,\zeta,\eta^{}_0))\Bigr|
& = \Bigl| \smashoperator[r]{\Sum_{\nu=0}^{m-1}} \int_{\eta^{}_0}^{\eta}
\partial_\eta  P _\nu(z,\zeta,\tau)\,d\tau\Bigr|
 = \Bigl|  \int_{\eta^{}_0}^{\eta}\smashoperator[r]{\Sum_{\nu=0}^{m-1}}
\partial_\eta  P _\nu(z,\zeta,\tau)\,d\tau\Bigr|  
\\
&
 \leqslant |\eta-\eta^{}_0|  C_{h,Z'} A^{m}e^{h\|\zeta\|}
\leqslant  r  C_{h,Z'} A^{m}e^{h\|\zeta\|}.
\end{align*}
Hence $P(t;z,\zeta,\eta)- P(t;z,\zeta,\eta^{}_0) \in \widehat{\mathfrak{N}}(\varOmega; S) $. 
\end{proof}

\begin{prop}
There exists the following isomorphism\textup{:}
\[
\widehat{\mathscr{S}} (\varOmega)/\widehat{\mathscr{N}}(\varOmega)\earrow
\widehat{\mathfrak{S}}(\varOmega;  S)  /\widehat{\mathfrak{N}}(\varOmega;  S).
\]
\end{prop}
\begin{proof}
We regard that
\begin{align*} 
\widehat{\mathscr{S}}(\varOmega) & =\{ P(t;z,\zeta,\eta)\in 
\widehat{\mathfrak{S}}(\varOmega;  S);\, \partial_\eta  P(t;z,\zeta,\eta)
=0\} \subset \widehat{\mathfrak{S}}(\varOmega;  S),
\\  
\widehat{\mathscr{N}}(\varOmega) & =
\widehat{\mathscr{S}}(\varOmega) \cap  \widehat{\mathfrak{N}}(\varOmega;  S)
\subset \widehat{\mathfrak{N}}(\varOmega;  S). 
\end{align*}
  Hence we have an  injective 
mapping $
\widehat{\mathscr{S}} (\varOmega)/\widehat{\mathscr{N}}(\varOmega)\hookrightarrow  
\widehat{\mathfrak{S}}(\varOmega;  S)  /\widehat{\mathfrak{N}}(\varOmega;  S) $.

Let $P(t;z,\zeta,\eta)\in \widehat{\mathfrak{S}}(\varOmega;  S)  $ and 
fix $\eta^{}_0 \in S$ arbitrary. 
Then by Proposition  \ref{S.prop4.5}, we have  
$P(t;z,\zeta,\eta^{}_0) \in \widehat{\mathscr{S}}(\varOmega)$ and
 $P(t;z,\zeta,\eta) -P(t;z,\zeta,\eta^{}_0) \in  \widehat{\mathfrak{N}}(\varOmega;  S)$.
\end{proof}
\begin{lem}\label{S.lem4.7}
$\widehat{\mathfrak{S}}_{\rm cl }(\varOmega;  S) \subset \widehat{\mathfrak{S}}(\varOmega;  S)$ and 
$\widehat{\mathfrak{N}}_{\rm cl }(\varOmega;  S) \subset \widehat{\mathfrak{N}}(\varOmega;  S)$.
\end{lem}
\begin{proof}
Let  $P(t;z,\zeta,\eta)\in \widehat{\mathfrak{S}}_{\rm cl }(\varOmega;  S)  $, and  assume \eqref{S.eq3.2a}. 
We replace $d$ as  $B:=\dfrac{\,A\,}{d^{}_\rho}\in \bigl]0,1\bigr[$  if necessary. 
Hence 
on $(\varOmega^{}_\rho * Z)[(\nu+1)d^{}_\rho]$, we have 
\[
   |P_\nu(z,\zeta,\eta)| \leqslant
\frac{C^{}_{h,Z}\,\nu!\,e^{h\|\zeta\|}}{(\nu+1)^\nu}\Bigl(\frac{A}{d^{}_\rho}\Bigr)^{\!\nu}
\leqslant C^{}_{h,Z}B^{\nu}e^{h\|\zeta\|}.
\]
The proof of $\widehat{\mathfrak{N}}_{\rm cl }(\varOmega;  S) \subset \widehat{\mathfrak{N}}(\varOmega;  S)$
is similar.
\end{proof}

\begin{prop}\label{S.prop4.8}  $ \widehat{\mathfrak{N}}(\varOmega; S)
 \cap\varGamma(\varOmega^{}_ \rho [d^{}_ \rho]\times S;
\mathscr{O}^{}_{T^*X \times \mathbb{C}}) 
= \mathfrak{N}(\varOmega; S)$.  
\end{prop}
\begin{proof} 
If  $ P(z,\zeta,\eta)\in \widehat{\mathfrak{N}}(\varOmega; S) \cap\varGamma(\varOmega^{}_ \rho [d^{}_ \rho]\times S;
\mathscr{O}^{}_{T^*X \times \mathbb{C}}) 
$,   
we set $\delta:= -\dfrac{\,2\log A\,}{d^{}_{\rho}}>0$, 
and for any $Z\Subset  S$,  take $h= \delta m^{}_Z$.
For each $(z;\zeta,\eta) \in \varOmega^{}_ \rho[d^{}_\rho] \times Z$, we take $m$ as the integral part of $ \dfrac{\|\eta\zeta\|}{d^{}_\rho}$, 
hence $(m+1)d^{}_{\rho} > \|\eta\zeta\| \geqslant m d^{}_\rho$.
Thus there exists   $C^{}_{Z}>0$ such that
\[
          |P(z,\zeta,\eta)|
\leqslant C_{Z}A^{m} e^{\delta m^{}_Z\|\zeta\|} \leqslant  C^{}_{Z}A^{ \|\eta\zeta\|/d^{}_\rho-1}e^{\delta m^{}_Z\|\zeta\|}
\leqslant\dfrac{C^{}_{Z}}{A}
 e^{\delta m^{}_Z\|\zeta\| -2\delta\|\eta\zeta\|}
\leqslant 
\dfrac{C^{}_{Z}}{A}
 e^{-\delta\|\eta\zeta\|}.
\]
Hence we have \eqref{S.eq1.2a}. 
Conversely, by Proposition \ref{S.prop3.7} and Lemma \ref{S.lem4.7} we have 
\[
\mathfrak{N}(\varOmega; S)= \widehat{\mathfrak{N}}^{}_{\rm cl}(\varOmega; S)\cap\varGamma(\varOmega^{}_ \rho [d^{}_ \rho]\times S;
\mathscr{O}^{}_{T^*X \times \mathbb{C}}) 
\subset  \widehat{\mathfrak{N}}(\varOmega; S)\cap\varGamma(\varOmega^{}_ \rho [d^{}_ \rho]\times S;
\mathscr{O}^{}_{T^*X \times \mathbb{C}}).
\qedhere\]
\end{proof}
\begin{thm}\label{S.thm4.9}
Let $z^*_0\in \dot{T}^*X$ and $P(t;z,\zeta,\eta) \in \widehat{\mathfrak{S}}^{}_{z^*_0} $. Then 
there exists    $\widetilde{P}(z,\zeta,\eta)  \in \mathfrak{S}^{}_{z^*_0}$ such that 
\[  
P(t;z,\zeta,\eta)  -\widetilde{P}(z,\zeta,\eta)\in \widehat{\mathfrak{N}} ^{}_{z^*_0}\,.
\]
\end{thm}
\begin{proof}
We may assume that
  $z_0^*=(0;1,0,\ldots,0)$. We fix $\eta^{}_0\in  S \cap \mathbb{R}$. 
Then by Proposition \ref{S.prop4.5} we have 
$P(t;z,\zeta,\eta^{}_0)\in \widehat{\mathscr{S}}(\varOmega) $ and 
$P(t;z,\zeta,\eta)- P(t;z,\zeta,\eta^{}_0) \in \widehat{\mathfrak{N}}(\varOmega; S) $.
We use the notation of the proof in Theorem \ref{S.thm2.1}; 
We develop  $P_\nu(z,\zeta,\eta^{}_0)$ into  the Taylor series with respect to
$(\zeta^{}_2/\zeta^{}_1,\dots, \zeta ^{}_n/\zeta ^{}_1)$:
\[
P_\nu(z,\zeta,\eta^{}_0)=\smashoperator[r]{\Sum_{\alpha\in{\mathbb{N}_0^{n-1}}}
}P^{}_{\nu,\alpha}(z,\zeta^{}_1,\eta^{}_0)
\biggl(\frac{\zeta'}{\zeta^{}_1}\biggr)^{\!\!\alpha},
\]
Then there exist sufficiently small $r^{}_0$, $\theta'>0$ and sufficiently large $d>0$ such that $P^{}_{\nu,\alpha}(z,\zeta^{}_1,\eta^{}_0)$ is
holomorphic on a common neighborhood of $D[(\nu+1)d]$ for each $\alpha \in \mathbb{N}_0^{n-1}$, where 
\[
D[(\nu+1)d]:=\{(z,\zeta^{}_1)\in  \mathbb{C}^{n+1};\,
\|z\|\leqslant r^{}_0,\:\lvert\arg\zeta^{}_1| \leqslant \theta',\, |\zeta^{}_1|\geqslant d (\nu+1)\}.
\]
It follows from the Cauchy inequality
that we can take  constants  $K>0$  and $A\in \big]0,1\big[$
so that for each  $h>0$  there exists $C^{}_{h}>0$ such that 
for every $\alpha \in \mathbb{N}_0^{n-1}$,
\[
|P^{}_{\nu,\alpha}(z,\zeta^{}_1 ,\eta^{}_0)|
\leqslant
C^{}_{h} A^\nu K^{|\alpha|}e^{h|\zeta^{}_1|} \quad ((z,\zeta^{}_1) \in D[(\nu+1)d]).
\]
We set 
$P^{\mathcal{B}}_{\nu,\alpha}(z,\zeta^{}_1 ,\eta)$ and $P^{\mathcal{B}}_\nu(z,\zeta,\eta)$ as in \eqref{S.eq2.8a}.
Then as in \eqref{S.eq2.11}, for  there exists  $\delta^{}_0>0$ and for any $Z \Subset S$, there exists $C'_Z$ such that 
for any  $(z,\zeta^{}_1,\eta) \in D[(\nu+1)d]\times Z$ and  $|\zeta^{}_i| \leqslant \varepsilon |\zeta^{}_1|$  we have
\[
|P^{}_\nu(z,\zeta,\eta^{}_0)- P^{\mathcal{B}}_\nu(z,\zeta,\eta)| \leqslant 2^{n-1}C'_{Z}A^\nu e^{-\delta^{}_0|\eta\zeta^{}_1|/2}.
\] 
Thus shrinking $\varOmega$ if necessary, setting  $A_1:= e^{-\delta^{}_0 d^{}_\rho/2}\in \big]0,1\bigr[$,
for  any 
$m \in\mathbb{N}$, we have on
$(\varOmega^{}_\rho * Z)[md^{}_\rho]$ 
\[
   \Bigl|\smashoperator[r]{\Sum_{\nu=0}^{m-1}} (P_\nu(z,\zeta,\eta^{}_0)- P^{\mathcal{B}}_\nu(z,\zeta,\eta))\Bigr|
 \leqslant \dfrac{\,C_{Z}A_1^{m}\,}{1-A}\,.
\]
Hence 
\begin{multline*}
P(t;z,\zeta,\eta)-  P^{\mathcal{B}}(t;z,\zeta,\eta) 
\\
=P(t;z,\zeta,\eta)-  P(t;z,\zeta,\eta^{}_0) +P(t;z,\zeta,\eta^{}_0)-  P^{\mathcal{B}}(t;z,\zeta,\eta) \in \widehat{\mathfrak{N}}(\varOmega; S).
\end{multline*} 
We set
\[
\varpi^{}_\alpha(P_{\nu})(x,w^{}_1,\eta)
:=
\int_{(\nu+1)d}^\infty \hspace{-1ex}P^{\mathcal{B}}_{\nu,\alpha}(z,\zeta^{}_1 ,\eta)
\,\dfrac{e^{- w^{}_1 \zeta^{}_1}}{\zeta_1^{|\alpha|}}\,d\zeta^{}_1.
\]
Recall $L$ of \eqref{S.eq2.16} and $L^{}_{k}$  of \eqref{S.eq2.18}.  
Hence as in \eqref{S.eq2.19}, 
$\varpi^{}_\alpha(P_{\nu})(x,w^{}_1,\eta)$  extends analytically to the domain $L \times  S$, and  for any $\eta \in   S$ we have 
\begin{equation}
\label{S.eq4.1}
\sup\{|\varpi^{}_\alpha(P_{\nu})(x,w^{}_1,\eta)|;\, (x,w^{}_1)\in L^{}_{k}\}
\leqslant
\dfrac{2kC^{}_{k,Z}A^\nu}{\, \delta^{}_1 |\alpha|!\,}\, (K|\eta|)^{|\alpha|}.
\end{equation}
Now we define
\[
\varpi(P_{\nu})(z,z+w,\eta)
:= \Sum_{\alpha\in \mathbb N_0^{n-1}}
\frac{\,\alpha!\, \varpi^{}_\alpha(P_{\nu})(z,w^{}_1,\eta)} 
{\,(2\pi\im\,)^{n}\,
(w')^{\alpha+ \mathbf{1}_{n-1}}\,}\,.
\]
The right-hand side converges locally uniformly $V^{}_{k}$ of \eqref{S.eq2.20}.  Hence $\varpi(P_{\nu})(z,z+w,\eta)$ is 
a holomorphic function
 defined on $V$ of \eqref{S.eq2.21}.
Hence we can define
\allowdisplaybreaks
\begin{align*}
\widetilde{P}^{}_\nu(z,\zeta,\eta):= {}&
\smashoperator{\int\limits_{\hspace{6ex}\gamma(0,\eta;\varrho,\theta)}}\varpi(P_{\nu})(z,z+w,\eta) 
e^{\langle w, \zeta\rangle} dw
\\
={}&  \smashoperator[r]{\Sum_{\alpha\in \mathbb{ N}_0^{n-1}}}\,  (\zeta')^{\alpha}
   \int\limits_{\gamma^{}_1(0, \eta;\varrho,\theta) } \hspace{-2ex}dw^{}_1\,\frac{\,e^{w^{}_1\zeta^{}_1}} {2\pi\im}
\int_{(\nu+1)d}^\infty \hspace{-1ex} P^{\mathcal{B}}_{\nu,\alpha}(z,\xi^{}_1 ,\eta)
\,\dfrac{e^{- w^{}_1 \xi^{}_1}}{\xi_1^{|\alpha|}}\,d\xi^{}_1\,.
\end{align*}
By virtue of \eqref{S.eq4.1}, 
there exist conic neighborhood $\varOmega$ of $z^*_0$, $\rho\in \bigl]0,1\bigr[$ and $d>0$ such 
that  $\widetilde{P}^{}_\nu(z,\zeta,\eta)\in\varGamma(\varOmega^{}_ \rho [d^{}_ \rho]\times S;
\mathscr{O}^{}_{T^*X \times \mathbb{C}})$ 
and for any  $h>0$ and  $Z\Subset  S$  there exist constants  $C_{h,Z}>0$  such that 
\begin{equation}\label{S.eq4.2}
|\widetilde{P}^{}_\nu(z , \zeta, \eta)| \leqslant     C_{h,Z}A^{\nu} e^{h \|\zeta\|} \quad ((z;\zeta,\eta) \in \varOmega^{}_\rho[d^{}_\rho] \times Z).
\end{equation}
We   set 
\[
\widetilde{V}_{\varepsilon}[(\nu+1)d]
:= \smashoperator[r]{\Bcap_{i=2}^n}\{(z,\zeta)\in \mathbb C^{2n};\, \|z\| < r^{}_0,\,
|\zeta^{}_1|\geqslant \dfrac{\,(\nu+1)d\,}{\,\varepsilon\,},\,
\lvert\arg\zeta^{}_1| \leqslant \varepsilon,\,|\zeta^{}_i|
\leqslant \varepsilon|\zeta^{}_1|\}. 
\]
 Recall  $\varSigma^{}_{\pm}$ (cf.\ Figure~\ref{S.pic3}), 
and we set $\varSigma^{\nu}_{\pm}:=\{(\nu+1)\xi_1\in \mathbb{C};\, \xi_1 \in \varSigma^{}_{\pm}\}$.  
Then we have $\widetilde{P}^{}_\nu(z,\zeta,\eta) = I^{}_\nu +I^{-}_\nu +I^{+}_\nu$, where
\allowdisplaybreaks
\begin{align*}
I^{}_\nu
 & :=
\smashoperator[r]{ \Sum_{\alpha\in\mathbb{N}_0^{n-1}}}\,
  (\zeta')^{\alpha}
    \int_{\varSigma^{\nu}_- - \varSigma^{\nu}_+}
    \frac{\,  P^{\mathcal{B}}_{\nu,\alpha}(z,\xi^{}_1 ,\eta) \,e^{a(\zeta^{}_1- \xi^{}_1)}\,}
         {\,2\pi\im \xi_1^{|\alpha|}(\xi^{}_1-\zeta^{}_1)\,}\,    d\xi^{}_1,
\\
I^{-}_\nu
& := -\smashoperator{\Sum_{\alpha\in\mathbb{ N}_0^{n-1}}}\,
  (\zeta')^{\alpha}
    \int_{\varSigma^{\nu}_-}\!\!
    \frac{\, P^{\mathcal{B}}_{\nu,\alpha}(z,\xi^{}_1 ,\eta) \,e^{\,\beta_1(\eta)(\zeta^{}_1-\xi^{}_1)}}
{\,2\pi\im \xi_1^{|\alpha|}(\xi^{}_1-\zeta^{}_1)\,}\,
    d\xi^{}_1,
\\
I^{+}_\nu
& := \smashoperator[r]{\Sum_{\alpha'\in\mathbb{ N}_0^{n-1}}}\,
(\zeta')^{\alpha}
    \int_{\varSigma^{\nu}_+}\!\!
    \frac{\, P^{}_{\nu,\alpha}(z,\xi^{}_1 ,\eta)\, e^{\,\beta^{}_0(\eta)(\zeta^{}_1-\xi^{}_1)}}
{\,2\pi\im \xi_1^{|\alpha|}(\xi^{}_1-\zeta^{}_1)\,}\,d\xi^{}_1.
\end{align*}
On $\widetilde{V}_{\varepsilon}[(\nu+1)d] \times Z$, as in \eqref{S.eq2.23} and \eqref{S.eq2.23a}  we have
\begin{equation}\label{S.eq4.5}
\begin{aligned}
  |I_\nu^{\,-}|
&\leqslant 
\frac{\,2^{n-2}C^{}_{h_Z}A^\nu e^{-c |\eta\zeta^{}_1|}\,}{\,c\,}\,
  \biggl( e^{(h^{}_Z+|\beta^{}_1|r) d}+
              \frac{e^{-h^{}_0dr}}
{\, 2\pi h^{}_0dm^{}_Z\,}
         \biggr), 
\\
 |I_\nu^{\,+}|
&\leqslant 
\frac{\,2^{n-2}C^{}_{h_Z}A^\nu e^{-c |\eta\zeta^{}_1|}\,}{\,c\,}\,
  \biggl( e^{(h^{}_Z+|\beta^{}_0|r) d}+
              \frac{e^{-h^{}_0dr}}
{\, 2\pi h^{}_0dm^{}_Z\,}
         \biggr).
\end{aligned}
\end{equation}
Further 
\[
I^{}_\nu =
\smashoperator[r]{\Sum_{\alpha\in\mathbb N_0^{n-1}}}
\,\,\,P^{\mathcal{B}}_{\nu,\alpha}(z,\zeta^{}_1,\eta)
\biggl(\frac{\zeta'}{\zeta^{}_1}\biggr)^{\!\!\!\alpha} = P^{\mathcal{B}}_{\nu}(z,\zeta,\eta)
\]
holds if $\zeta^{}_1$ is located in the domain surrounded by 
$\varSigma^{\nu}_{-}-\varSigma^{\nu}_+$. 
Therefore, shrinking $\varOmega$ and replacing $d$ with a larger one  if necessary,  by  \eqref{S.eq4.5},  
there exists $\delta>0$ and 
for any  $Z\Subset  S$ there exist $C^{}_{Z}>0$ such that 
on $\varOmega^{}_\rho[(\nu+1)d^{}_\rho] \times Z$, the following holds:
\[
 |\widetilde{P}^{}_\nu (z,\zeta,\eta)-P^{\mathcal{B}}_\nu(z,\zeta,\eta) |
\leqslant 
 C^{}_Z A^{\nu}e^{-\delta \|\eta\zeta\|}.
\]
We set $\widetilde{P}(t;z,\zeta,\eta):= \Sum_{\nu=0}^\infty t^\nu \widetilde{P}^{}_\nu (z,\zeta,\eta)$.
We set  $A_2:= e^{-\delta d^{}_\rho}\in \big]0,1\bigr[$. Then
for  any 
$m \in\mathbb{N}$, on $(\varOmega^{}_\rho * Z)[m d^{}_\rho]$   we have
\[
   \Bigl|\smashoperator[r]{\Sum_{\nu=0}^{m-1}} (\widetilde{P}^{}_\nu(z,\zeta,\eta)-P^{\mathcal{B}}_\nu (z,\zeta,\eta))\Bigr|
 \leqslant \dfrac{\,C_{Z}A_2^{m}\,}{1-A}\,,
\]
i.e.\  $ \widetilde{P}(t;z,\zeta,\eta)- P^{\mathcal{B}}(t;z,\zeta,\eta) \in  \widehat{\mathfrak{N}}^{}_{z^*_0}\,.$ 
In particular by Lemma \ref{S.lem4.3},
\[
\partial_{\eta}\widetilde{P}(t;z,\zeta,\eta) =\partial_{\eta}(\widetilde{P}(t;z,\zeta,\eta)
- P^{\mathcal{B}}(t;z,\zeta,\eta)) +\partial_{\eta} P^{\mathcal{B}}(t;z,\zeta,\eta)\in  \widehat{\mathfrak{N}}^{}_{z^*_0}\,.
\] 
Thus $\widetilde{P}(t;z,\zeta,\eta)\in \widehat{\mathfrak{S}}^{}_{z^*_0}$. By \eqref{S.eq4.2}, 
we can define 
\[
\widetilde{P}(z,\zeta,\eta) := \Sum_{\nu=0}^\infty \widetilde{P}^{}_\nu (z,\zeta,\eta)
\in \varGamma(\varOmega^{}_ \rho [d^{}_ \rho]\times S;
\mathscr{O}^{}_{T^*X \times \mathbb{C}}), 
\]
and  we have 
\[
|\widetilde{P} (z,\zeta,\eta)|
 \leqslant 
 \frac{\, C^{}_{h,Z}e^{h\|\zeta\|}\,}{1-A} \qquad  ((z;\zeta,\eta) \in \varOmega^{}_\rho[d^{}_\rho] \times Z),
\]
\[
  \Bigl|\widetilde{P} (z,\zeta,\eta)-\smashoperator{\Sum_{\nu=0}^{m-1}} \widetilde{P}^{}_\nu (z,\zeta,\eta)\Bigr|
= \Bigl|\smashoperator[r]{\Sum_{\nu =m}^\infty} \widetilde{P}^{}_\nu (z,\zeta,\eta)\Bigr|
\leqslant 
  \frac{\,C^{}_{h,Z}\,A^{m} e^{h\|\zeta\|}\,}{1-A}\quad ((z;\zeta,\eta) \in \varOmega^{}_\rho[md^{}_\rho] \times Z).
\]
i.e.\  $\widetilde{P}(z,\zeta,\eta)- \widetilde{P}(t;z,\zeta,\eta) \in  \widehat{\mathfrak{N}}^{}_{z^*_0}\,.$ 
Moreover by Lemma \ref{S.lem4.3} and   Proposition  \ref{S.prop4.8}, we have 
\begin{align*}
\partial_{\eta}\widetilde{P} (z,\zeta,\eta) & =
\partial_{\eta}(\widetilde{P} (z,\zeta,\eta)- \widetilde{P}(t;z,\zeta,\eta))
+ \partial_{\eta}\widetilde{P}(t;z,\zeta,\eta) 
\\
&\in 
\widehat{\mathfrak{N}} ^{}_{z^*_0}\cap\varGamma(\varOmega^{}_ \rho [d^{}_ \rho]\times S;
\mathscr{O}^{}_{T^*X \times \mathbb{C}})  
 = \mathfrak{N} ^{}_{z^*_0}\,.
\end{align*}
Therefore $\widetilde{P}(z,\zeta,\eta)  \in \mathfrak{S}^{}_{z^*_0}$ and 
\begin{align*}
P(t;z,\zeta,\eta)  -\widetilde{P}(z,\zeta,\eta) ={} &
P(t;z,\zeta,\eta)-P^{\mathcal{B}}(t;z,\zeta,\eta)+P^{\mathcal{B}}(t;z,\zeta,\eta)  -\widetilde{P}(t;z,\zeta,\eta) 
\\*
&+\widetilde{P}(t;z,\zeta,\eta) -\widetilde{P}(z,\zeta,\eta)
\in \widehat{\mathfrak{N}} ^{}_{z^*_0}\,.
\qedhere
\end{align*}
\end{proof}
\begin{thm}
For any $z^*_0\in \dot{T}^*X$, the inclusions $
\mathfrak{S}^{}_{z^*_0}\subset \widehat{\mathfrak{S}}^{}_{{\rm cl},z^*_0}
\subset  \widehat{\mathfrak{S}}^{}_{z^*_0}
$ and $
\mathfrak{N}^{}_{z^*_0}\subset \widehat{\mathfrak{N}}^{}_{{\rm cl},z^*_0}
\subset  \widehat{\mathfrak{N}}^{}_{z^*_0}
$ induce
\[
\mathfrak{S}^{}_{z^*_0}\big/\mathfrak{N}^{}_{z^*_0} \earrow \widehat{\mathfrak{S}}^{}_{{\rm cl},z^*_0}\big/
\widehat{\mathfrak{N}}^{}_{{\rm cl},z^*_0} \earrow\widehat{\mathfrak{S}}^{}_{z^*_0}\big/\widehat{\mathfrak{N}}^{}_{z^*_0}\,.
\]
\end{thm}
\begin{proof} 
By Proposition \ref{S.prop4.8} and Theorem  \ref{S.thm4.9}, we obtain an isomorphism 
$
\mathfrak{S}^{}_{z^*_0}\big/\mathfrak{N}^{}_{z^*_0} \earrow\widehat{\mathfrak{S}}^{}_{z^*_0}
\big/\widehat{\mathfrak{N}}^{}_{z^*_0}$,  and we shall show that this isomorphism  
 is compatible with $\mathfrak{S}^{}_{z^*_0}\big/\mathfrak{N}^{}_{z^*_0} \earrow 
\widehat{\mathfrak{S}}^{}_{{\rm cl},z^*_0}\big/\widehat{\mathfrak{N}}^{}_{{\rm cl},z^*_0}$ in Corollary \ref{S.cor3.9}. 
For any $
P(t;z,\zeta,\eta) \in \widehat{\mathfrak{S}}^{}_{\textrm{cl},z^*_0}\subset \widehat{\mathfrak{S}}^{}_{z^*_0}
$, by Theorems  \ref{S.thm3.8} and  \ref{S.thm4.9}, there exist  $P'(z,\zeta,\eta)$, $P''(z,\zeta,\eta) \in\mathfrak{S}^{}_{z^*_0}$ such that
\begin{equation*}
\left\{
\begin{aligned}
P(t;z,\zeta,\eta) - P'(z,\zeta,\eta) & \in \widehat{\mathfrak{N}}^{}_{\textrm{cl},z^*_0}\, ,
\\
P(t;z,\zeta,\eta) - P''(z,\zeta,\eta) & \in \widehat{\mathfrak{N}}^{}_{z^*_0}\, .
\end{aligned}
\right.
\end{equation*}
Then, by Propositions  \ref{S.prop3.7} and \ref{S.prop4.8} we have
\begin{align*}
 P'(z,\zeta,\eta) - P''(z,\zeta,\eta)  &\in \mathfrak{S}^{}_{z^*_0} \cap \widehat{\mathfrak{N}}^{}_{z^*_0} \cap\varGamma(\varOmega^{}_ \rho [d^{}_ \rho]\times S;
\mathscr{O}^{}_{T^*X \times \mathbb{C}}) 
\\
&
= \mathfrak{N}^{}_{z^*_0}
= \widehat{\mathfrak{N}}^{}_{\textrm{cl},z^*_0}\cap\varGamma(\varOmega^{}_ \rho [d^{}_ \rho]\times S;
\mathscr{O}^{}_{T^*X \times \mathbb{C}}).
\qedhere
\end{align*}
\end{proof}
\begin{rem}
Summing up, we have the following  commutative diagram:
\[
\xymatrix @C=1em @R=3ex{
\mathscr{E}^{\smash{\mathbb{R}}}_{X, z^*_0} \ar[r]^-{\dsim} \ar[d]^-{\displaystyle \!\wr}&
 \mathscr{S}^{}_{z^*_0}/\mathscr{N}^{}_{z^*_0} \ar[r]^-{\dsim}\ar[d]^-{\displaystyle \!\wr}
 & \widehat{\mathscr{S}} _{{\rm cl}, z^*_0}/\widehat{\mathscr{N}} _{{\rm cl}, z^*_0} 
\ar[d]^-{\displaystyle \!\wr}\ar[r]^-{\dsim}
&\widehat{\mathscr{S}} _{ z^*_0}/\widehat{\mathscr{N}} _{ z^*_0} \ar[d]^-{\displaystyle \!\wr}
\\
\varinjlim\limits_{\boldsymbol{\kappa}} E^{\mathbb{R}}_X(\boldsymbol{\kappa}) \ar[r]^-{\dsim} &\mathfrak{S}^{}_{z^*_0}/\mathfrak{N}^{}_{z^*_0} 
\ar[r]^-{\dsim}
& \widehat{\mathfrak{S}}_{{\rm cl}, z^*_0} /\widehat{\mathfrak{N}}_{{\rm cl}, z^*_0} \ar[r]^-{\dsim}
& \widehat{\mathfrak{S}}_{ z^*_0} /\widehat{\mathfrak{N}}_{ z^*_0} \,.
} 
\]
\end{rem}
\begin{defn} 
As in the case of   $ \mathfrak{S}(\varOmega; S)$ and $\widehat{\mathfrak{S}}_{\rm cl}(\varOmega; S)$,
 for any $P(t;z,\zeta,\eta)\in \widehat{\mathfrak{S}}(\varOmega; S)$ we set 
\[\wick{P(t;z,\zeta,\eta)}: = P(t;z,\zeta,\eta) \bmod 
 \widehat{\mathfrak{N}}(\varOmega; S) \in 
 \widehat{\mathfrak{S}}(\varOmega; S)
\big/ \widehat{\mathfrak{N}}(\varOmega; S)
\]
which is also called   the \textit{normal product} or the  \textit{Wick product} of $P(t;z,\zeta,\eta)$. 
\end{defn}

We use the notation of Theorem   \ref{S.thm3.11}. For any   $P(t;z,\zeta,\eta)\in
 \widehat{\mathfrak{S}}(\varOmega; S)$, we also set 
\[
\varPhi^*P(t;w,\lambda,\eta)
:= e^{t\langle \partial_{\zeta'},\partial^{}_{z'}\rangle} 
   P(t;\varPhi(w),\zeta'+ {}^{\mathsf{t}}\!J^{*}_{\varPhi}(\varPhi(w)+z', \varPhi(w) )\lambda, \eta)
\big{|}{}_{\atop{z'=0}{\zeta'=0}}\,.\]

\begin{thm}
$(1)$
$\varPhi^*P(t;w,\lambda, \eta) \in  \widehat{\mathfrak{S}}(\varOmega; S) $ with respect 
to coordinate system $(w;\lambda)$. 
Further if $P(t;z,\zeta, \eta) \in  \widehat{\mathfrak{N}}(\varOmega; S) $, it follows that
 $\varPhi^*P(t;w,\lambda, \eta) \in  \widehat{\mathfrak{N}}(\varOmega; S) $.
\par
$(2)$ $\mathds{1}^*$ is the identity, and for complex coordinate transformations  $z=\varPhi(w)$ and  $w=\varPsi(v)$, 
it follow that 
$\varPsi^*\varPhi^*P(t;v,\xi,\eta)-(\varPhi\varPsi)^*P(t;v,\xi,\eta)\in \widehat{\mathfrak{N}}^{}_{(v;\xi)}$. 
\end{thm}
\begin{proof}
(1) 
Suppose   that $P^{}_k(z,\zeta,\eta)\in 
\varGamma((\varOmega^{}_ \rho * S)[(k+1)d_{\rho}];
\mathscr{O}^{}_{T^*X \times \mathbb{C}})$. 
We also assume  \eqref{S.eq3.2}, 
\eqref{S.eq3.3} and \eqref{S.eq3.4}, hence for any $h>0$ there exists $C^{}_{h,Z}>0$ such that for  any 
 $(z;\zeta'+{}^{\mathsf{t}}\! J^{*}_{\varPhi}(z+z',z) \lambda)\in (\varOmega'_{\rho'}* Z)[(k+1)d^{}_{\rho'}]$ we have
\[
 | P^{}_k(z,\zeta'+{}^{\mathsf{t}}\!J^{*}_{\varPhi}(z+z',z)\lambda,\eta)| 
\leqslant 
 C^{}_{h,Z}A^{k} e^{h\|\zeta'+{}^{\mathsf{t}}\!J^{*}_{\varPhi}(z+z',z)\lambda\|}.
\]
Hence if $(z;\zeta,\eta) \in
(\varOmega^{}_{\rho'} * Z)[(k+1)d^{}_{\rho'}]$, 
instead of \eqref{S.eq3.5} we have
\allowdisplaybreaks
\begin{align*}
\frac{1}{\,\alpha!\,}\Bigl| &\partial_{\zeta'}^{\,\,\alpha}\partial_{z'}^{\,\,\alpha}
P^{}_k (z,\zeta'+{}^{\mathsf{t}}\!J^{*}_{\varPhi}(z+z',z)\lambda,\eta)\big{|}{}_{\atop{z'=0}{\zeta'=0}}\Bigr|
         \\
& \leqslant \frac{C^{}_{h,Z}\,\alpha!\,A^{k} }
       {(\varepsilon\delta c'\| \lambda\|)^{|\alpha|}}
\smashoperator[r]{\sup_{\atop{|z'_i|=\delta}{|\zeta'_i|=\varepsilon\|\zeta\|}}}
e^{h\|\zeta'+{}^{\mathsf{t}}\!J^{*}_{\varPhi}(z+z',z)\lambda)\|}
\leqslant
   \frac{C^{}_{h,Z}\alpha!\,A^{k} e^{2hc\|\lambda\|}\,}
       {(\varepsilon \delta c'\| \eta\lambda\|)^{|\alpha|}}\,.
\end{align*}
We  may assume that $\dfrac{1}{\,2\,}<A<1$. 
Replacing $d>0$ if necessary, we may assume  $C:= \dfrac{\varepsilon\delta d^{}_{\rho'}}{2}>4$, 
hence $C A>2$. Hence  
if $\|\eta\zeta\|\geqslant c'\|\eta\lambda\|\geqslant (\nu +1)d^{}_{\rho'}$, we have
\begin{align*}
\Big| (\varPhi^*&P)^{}_\nu(w,\lambda,\eta) \Big|
 \leqslant C^{}_{h,Z}\Sum_{k+|\alpha|= \nu}  \frac{\,\alpha! A^{k}e^{2hc\|\lambda\|}\,}
       {(\varepsilon\delta c'\| \eta\lambda\|)^{|\alpha|}}
\leqslant C^{}_{h,Z} e^{2hc\|\lambda\|}\Sum_{k=0}^ \nu  \frac{\,2^{n+\nu-k-1}A^{k}(\nu-k)!\,}
       {(\varepsilon\delta d^{}_{\rho'})^{\nu-k}\,(\nu +1)^{\nu-k}}
\\&
\leqslant 
\dfrac{2^{n-1}C^{}_{h,Z}e^{2hc\|\lambda\|}}{C^\nu}\Sum_{k=0}^ \nu   (CA)^k
=
\dfrac{2^{n-1}C^{}_{h,Z}e^{2hc\|\lambda\|}((CA)^{\nu+1}-1)}{C^\nu (CA-1)}
\leqslant \dfrac{2^{n}C^{}_{h,Z}\,A^{\nu}e^{2hc\|\lambda\|}}{CA-1}\,.
\end{align*}
If $P(t;z,\zeta, \eta) \in  \widehat{\mathfrak{N}}(\varOmega; S) $, 
for any $m \in \mathbb{N}$ and $(z;\zeta,\eta) \in 
(\varOmega^{}_{\rho'}* Z)[md^{}_{\rho'}]$ 
we have
\[
\frac{1}{\,\alpha!\,}\Bigl| \Sum_{k=0}^{m-1}\partial_{\zeta'}^{\,\,\alpha}\partial_{z'}^{\,\,\alpha}
P^{}_k(z,\zeta'+{}^{\mathsf{t}}\!J^{*}_{\varPhi}(z+z',z) \lambda,\eta)\big{|}{}_{\atop{z'=0}{\zeta'=0}}\Bigr|
\leqslant
   \frac{\,C^{}_{h,Z}\,\alpha!A^{m}e^{2hc\|\lambda\|}\,}
       {(\varepsilon \delta c'\| \lambda\|)^{|\alpha|}}\,.
\]
Hence if $\| \eta\lambda\|  \geqslant md^{}_{\rho'}/c'$, we have
\allowdisplaybreaks
\begin{align*}
\Big|&\Sum_{\nu=0}^{m-1} (\varPhi^*P)^{}_\nu(w,\lambda,\eta)\Big|
\leqslant 
\Sum_{|\alpha|=0}^{m-1}\frac{\,C^{}_{h,Z}\,\alpha! A^{m-|\alpha|}e^{2hc\|\lambda\|}\,}
       {(\varepsilon \delta c'\| \lambda\|)^{|\alpha|}}
\\
&
\leqslant 
2^{n-1}C^{}_{h,Z}A^{m}e^{2hc\|\lambda\|}
\Sum_{k=0}^{m-1}\Bigl(\frac{1}{CA}\Bigr)^{\!k} 
\Bigl(\frac{m-1}{m}\Bigr)^{\!k} 
\leqslant 
\frac{\, C^{}_{h,Z}CA^{m+1}e^{2hc\|\lambda\|}\,}{CA-1}\,.
\end{align*}
\par
(2) Set $v^*:=(v;\xi)$. By Theorem \ref{S.thm4.9}, we can find 
$P^{}_0(z,\zeta,\eta)\in \mathfrak{S}^{}_{z^*_0}\subset \widehat{\mathfrak{S}}^{}_{\textrm{cl},z^*_0}$ 
such that 
 \[
P(t;z,\zeta,\eta)-P^{}_0(z,\zeta,\eta)\in \widehat{\mathfrak{N}}^{}_{z^*_0}\,.
\] 
By Theorem  \ref{S.thm3.11}, we have
$ \varPsi^*\varPhi^* P^{}_0(v,\xi,\eta) = (\varPhi \varPsi)^*P^{}_0(v,\xi,\eta) $. 
Hence by (1)  we obtain
\begin{align*}
\varPsi^* &\varPhi^*P(t;v,\xi,\eta) - (\varPhi\varPsi)^*P(t;v,\xi,\eta) 
\\
=&{}( \varPsi^*\varPhi^* P(t;v,\xi,\eta) -\varPsi^*\varPhi^* P^{}_0(v,\xi,\eta)) 
-  ((\varPhi\varPsi)^*P(t;v,\xi,\eta) - (\varPhi \varPsi)^*P^{}_0(v,\xi,\eta)) \in \widehat{\mathfrak{N}}^{}_{v^*}\,.
\qedhere
\end{align*}
\end{proof}

\begin{thm}\label{S.thm7.13}
For any $P(t;z,\zeta,\eta)$, $Q(t;z,\zeta,\eta)\in \widehat{\mathfrak{S}}(\varOmega; S)$, 
set 
\begin{align*}
Q\circ P(t;z,\zeta,\eta):= {}& e^{t\langle \partial_{\zeta'},\partial^{}_{z'} \rangle}
 Q(t;z,\zeta', \eta)\,P(t; z',\zeta,\eta)\big{|}{}_{\atop{z'=z}{\zeta'=\zeta}}
\\
= {} & e^{t\langle \partial_{\zeta'},\partial^{}_{z'} \rangle}
 Q(t;z,\zeta+\zeta', \eta)\,P(t;z+z', \lambda,\eta)\big{|}{}_{\atop{z'=0}{\zeta'=0}}\,.
\end{align*}

$(1)$
$Q\circ P(t;z,\zeta,\eta)\in  \widehat{\mathfrak{S}}(\varOmega; S) $. 
Moreover  if  either $P(t; z,\zeta,\eta)$ or  $Q(t; z,\zeta,\eta)$ is an element of  
$\widehat{\mathfrak{N}}(\varOmega; S) $, it follows that
$Q\circ P(t;z,\zeta,\eta)\in  \widehat{\mathfrak{N}}(\varOmega; S) $. 

$(2)$  $R\circ (Q \circ P)= (R\circ Q) \circ P$ holds.

$(3)$  Let $\varPhi(w)=z$ be a holomorphic coordinate transformation. Then 
\[
\varPhi^* Q\circ \varPhi^*P(t;w,\lambda,\eta)- \varPhi^* (Q\circ P)(t;w,\lambda,\eta) \in 
\widehat{\mathfrak{N}}^{}_{(w;\lambda)}.
\]
\end{thm}
\begin{proof} 
(1) We assume  that 
 $P^{}_i(z,\zeta,\eta)$, $Q^{}_i(z,\zeta,\eta) \in
\varGamma((\varOmega^{}_ \rho * S)[(i+1)d_{\rho}];
\mathscr{O}^{}_{T^*X \times \mathbb{C}})$ 
for  $P(t;z,\zeta,\eta)  = \Sum_{i=0}^\infty t^i  P^{}_i(z,\zeta,\eta) $ 
and  $Q(t;z,\zeta,\eta)   = \Sum_{i=0}^\infty t^i  Q^{}_i(z,\zeta,\eta) $. Set $
Q \circ P(t;z,\zeta,\eta) =\Sum_{i=0}^\infty t^i  R^{}_i(z,\zeta,\eta)$. Then 
\[
R^{}_\nu(z,\zeta,\eta)
 = \Sum_{|\alpha|+k+l=\nu} 
  \frac{1}{\alpha!}\,\partial^{\,\alpha}_\zeta Q^{}_l(z,\zeta,\eta)\,
 \partial^{\alpha}_z P^{}_k(z,\zeta,\eta).
\]
Hence $R^{}_\nu(z,\zeta,\eta) \in
\varGamma((\varOmega^{}_ \rho * S)[(\nu+1)d_{\rho}];
\mathscr{O}^{}_{T^*X \times \mathbb{C}})$.
We shall prove   $R(t;z,\zeta,\eta)\in \widehat{\mathfrak{S}}(\varOmega; S)$. 
Let $Z\Subset  S$. 
Note that 
for any  
$(z,\zeta,\eta)\in 
(\varOmega^{}_\rho * Z)[(k+1)d^{}_{\rho'}]
$ and $(z',\zeta')$ with 
 $\|z'\| \leqslant \rho-\rho'$ and 
 $\|\zeta'\| \leqslant (\rho-\rho')\|\eta\zeta\|<\|\zeta\|$, we have 
 $(z+z',\zeta+\zeta')\in \varOmega^{}_{\rho}[(k+1)d^{}_ {\rho}]$. Moreover as in \eqref{eq.6.9a} we  have
\[
\|\eta(\zeta+\zeta')\| \geqslant  (1-\rho+\rho')\|\eta\zeta\| 
 \geqslant (k+1) d(1-\rho'').
\]
For any $\rho'\in\bigl]0,\rho\bigr[$ and  $h>0$,  on
$(\varOmega^{}_{\rho'} * Z)[(k+1)d^{}_{\rho'}]$  we have
\[ 
|P^{}_k(z,\zeta,\eta)|, \, |Q^{}_k(z,\zeta,\eta)|\leqslant C^{}_{h,Z}A^{k} e^{h\|\zeta\|}.\]
Hence in  the same way as in the proof of    Lemma \ref{S.lem3.14}, 
on $(\varOmega^{}_{\rho'} * Z)[(k+1)d^{}_{\rho'}]$ we have
\begin{align}
 |\partial_{\zeta}^{\,\alpha}Q^{}_k(z,\zeta,\eta)| &
\leqslant 
  \frac{\,C^{}_{h,Z}\,\alpha! A^{k} e^{2h\|\zeta\|}\,}
       {\, (\rho-\rho')^{|\alpha|}\|\eta\zeta\|^{|\alpha|}\,}\,, \notag
\\ \label{S.eq4.6}
|\partial_{z}^{\,\alpha}P^{}_k(z,\zeta,\eta)|
&
\leqslant 
  \frac{\,C^{}_{h,Z}\,\alpha! A^{k} e^{h\|\zeta\|}\,}{\,(\rho- \rho')^{|\alpha|}\,}\,.  
\end{align}
 We replace $d$, $\rho'>0$ as 
$C:= \dfrac{2}{Ad^{}_{\rho'}(\rho-\rho')^2} <1$, 
and   chose $C'>0$ and $B\in \bigl[A,1\bigr[$ as $(\nu+1) A^\nu \leqslant C'B^\nu$ for any $\nu \in \mathbb{N}^{}_0$. 
Since $\#\{(k,l)\in \mathbb{N}_0^{2};\,k+l=\nu-|\alpha|\} \leqslant  \nu-|\alpha|+1\leqslant  \nu+1$, 
for any $(z;\zeta,\eta) \in (\varOmega^{}_{\rho'}* Z)[(\nu+1)d^{}_{\rho'}]$,   we have 
\allowdisplaybreaks
\begin{align*}
|R^{}_\nu (z,\zeta,\eta)|
 & \leqslant  \Sum_{\nu=|\alpha|+k+l}
  \frac{\, C^{\;2}_{h,Z}\,\alpha!A^{k+l}  e^{3h\|\zeta\|}\,}
       {\,\|\eta\zeta\|^{|\alpha|}\, (\rho -\rho')^{2|\alpha|}\,}
 \leqslant\Sum_{i=0}^\nu 
  \frac{\,  2^{i+n-1} C^{\;2}_{h,Z}\,(\nu+1)\,i!\,A^{\nu-i}e^{3h\|\zeta\|}}
       {\,(d_{\rho'}(\rho-\rho')^2)^{i}\,(\nu+1)^{i}\,}
\\
&
  \leqslant  2^{n-1} C^{\;2}_{h,Z}\,(\nu+1)A^\nu e^{3h\|\zeta\|}\Sum_{i=0}^\nu C^i
\leqslant  \frac{\,  2^{n-1} C'C^{\;2}_{h,Z}\,B^\nu e^{3h\|\zeta\|}}
       {1-C}\,.
\end{align*}
Next, we assume 
$P(t;z,\zeta,\eta) \in \widehat{\mathfrak{N}}(\varOmega; S)$. Then, 
instead of \eqref{S.eq4.6}, we have that  for any $m\in \mathbb{N}$, on  
$(\varOmega^{}_{\rho'} * Z)[md^{}_{\rho'}]$ we have
\[
  \biggl|\Sum_{k=0}^{m-1}\partial_z^{\alpha}P^{}_k(z,\zeta,\eta)\biggr|
\leqslant 
  \frac{\,C^{}_{h,Z}\,\alpha! A^{m}e^{h\|\zeta\|}\,}
       {(\rho-\rho')^{|\alpha|}}\,,
\]
thus we have
\allowdisplaybreaks
\begin{align*}
\Bigl|\Sum_{\nu=0}^{m-1}R^{}_\nu(z,\zeta,\eta)\Bigr| &
=\Bigl|\Sum_{i+|\alpha|=0}^{m-1} \frac{1}{\,\alpha!\,}\,
         \partial_{\zeta}^{\,\alpha}Q^{}_i(z,\zeta,\eta)\smashoperator[r]{\Sum_{k=0}^{m-i-1-|\alpha|}}
 \partial_z^{\,\alpha}P^{}_k(z,\zeta,\eta)\Bigr| 
\\
&
 \leqslant  \Sum_{i+|\alpha|=0}^{m-1} \,
  \frac{\,C_{h,Z}^{\;2}\,\alpha! A^{m-|\alpha|} e^{3h\|\zeta\|}\,}
       {\|\zeta\|^{|\alpha|} \,(\rho-\rho')^{2|\alpha|}}
\leqslant  2^{n-1} C_{h,Z}^{\;2}\, A^{m} e^{3h\|\zeta\|}
\smashoperator[r]{\Sum_{i+\nu=0}^{m-1} }C^{\nu}
\\
&
 \leqslant \frac{\, 2^{n-1} C_{h,Z}^{\;2}\,m A^{m} e^{3h\|\zeta\|}\,}{1-C}
\leqslant \frac{\, 2^{n-1} C'C_{h,Z}^{\;2}B^{m} e^{3h\|\zeta\|}\,}{1-C}\,.
\end{align*}
The proof in the case of  $Q(t;z,\zeta,\eta) \in \widehat{\mathfrak{N}}(\varOmega; S)$ is  
similar. 
In particular, since 
\[
\partial_{\eta} (Q\circ P )(t;z,\zeta,\eta) = \partial_{\eta} Q\circ P (t;z,\zeta,\eta) 
+ Q\circ  \partial_{\eta} P (t;z,\zeta,\eta), 
\]
we see that if $P(t;z,\zeta,\eta)$, $Q(t;z,\zeta,\eta) \in 
\widehat{\mathfrak{S}}(\varOmega; S)$, 
we have 
$Q\circ P(t;z,\zeta,\eta) \in \widehat{\mathfrak{S}}(\varOmega; S)$.

(2) is easily obtained. 

(3) 
 Set $v^*:=(v;\xi)$. By Theorem \ref{S.thm4.9}, we can find 
$P^{}_0(z,\zeta,\eta)$, 
$Q^{}_0(z,\zeta,\eta)\in \widehat{\mathfrak{S}}^{}_{z^*_0}\subset \widehat{\mathfrak{S}}^{}_{\textrm{cl},z^*_0}$ 
such that 
 \[
P(t;z,\zeta,\eta)-P^{}_0(z,\zeta,\eta),\quad Q(t;z,\zeta,\eta)-Q^{}_0(z,\zeta,\eta)\in \widehat{\mathfrak{N}}^{}_{z^*_0}\,.
\] 
By Theorem  \ref{S.thm3.15}, we have
\[
\varPhi^* Q^{}_0\circ \varPhi^*P^{}_0(w,\lambda,\eta)= \varPhi^* (Q^{}_0\circ P^{}_0)(w,\lambda,\eta).
 \] 
Hence by (1)  we obtain
\begin{align*}
\varPhi^* Q\circ \varPhi^*P(t;w,\lambda,\eta)={}& 
\varPhi^* Q^{}_0\circ \varPhi^*P^{}_0(t;w,\lambda,\eta) +
\varPhi^* (Q-Q^{}_0)\circ \varPhi^*P(t;w,\lambda,\eta)
\\* & +
\varPhi^* Q^{}_0\circ \varPhi^*(P-P^{}_0)(t;w,\lambda,\eta)
\\
={}& 
\varPhi^* (Q^{}_0\circ P^{}_0)(t;w,\lambda,\eta) +
\varPhi^* (Q-Q^{}_0)\circ \varPhi^*P(t;w,\lambda,\eta)
\\* & +
\varPhi^* Q^{}_0\circ \varPhi^*(P-P^{}_0)(t;w,\lambda,\eta)
\\
\equiv {}& \varPhi^* (Q^{}_0\circ P^{}_0)(t;w,\lambda,\eta)
\\
={}& 
\varPhi^* (Q\circ P)(t;w,\lambda,\eta) +
\varPhi^*( (Q ^{}_0-Q)\circ P ^{}_0)(t;w,\lambda,\eta)
\\* & +
\varPhi^* (Q\circ (P ^{}_0-P))(t;w,\lambda,\eta)
\\
\equiv {}& \varPhi^* (Q\circ P)(t;w,\lambda,\eta).
 \qedhere\end{align*}
\end{proof}
 
We can also prove (cf.\  see the  proof of Theorems \ref{S.thm3.17} and \ref{S.thm7.13}):
\begin{thm}
For any 
 $P(t;z,\zeta,\eta)\in \widehat{\mathfrak{S}}(\varOmega; S)$ set 
\[
P^*(t;z,\zeta,\eta):=e^{t\langle \partial_\zeta,\partial_z \rangle}
 P(t;z,-\zeta,\eta).
\]

$(1)$  $P^*(t;z,\zeta,\eta)\in \widehat{\mathfrak{S}}(\varOmega^a ; S) $ and  $P^{**}=P$. 
Moreover if  $P(t;z,\zeta,\eta)\in \widehat{\mathfrak{N}}(\varOmega; S) $, 
it follows that   $P^*(t;z,\zeta,\eta)\in \widehat{\mathfrak{N}}(\varOmega^a; S) $.

$(2)$ 
$
(Q \circ P)^*(t;z,\zeta,\eta)= P^*\circ Q^*(t;z,\zeta,\eta)$.  

$(3)$  Let $\varPhi(w)=z$ be a holomorphic coordinate transformation. Then 
on $\widehat{\mathfrak{S}}(\varOmega^a;S) \smashoperator{\tens_{\mathscr{O}^{}_{X}}} \varOmega^{}_{X}$
\[
\varPhi^*(P^*(t;z,\zeta,\eta) \tens dz)= \varPhi^*(P(t;z,\zeta,\eta) \tens dz)^*.
\] 
\end{thm}

\appendix
\section{The Compatibility of Actions}\label{ap:actions}
The purpose of this appendix is to show Theorem \ref{th:action-commutative}.
We follow the same notation as those in  Section \ref{sec:actions_ER}.
Set  $\widehat{\pi}_1 \colon  \widehat{X}^2 \ni (z,w, \eta) \mapsto (z,\eta)\in \widehat{X}$  and 
 $\widehat{\pi}_2 \colon \widehat{X}^2\ni (z,w,\eta) \mapsto (w,\eta) \in \widehat{X}$.
We also define the canonical projections $\pi_1 \colon  X^2 \to X$
and $\pi_2 \colon  X^2 \to X$ in the same way.
Note that we consider the problem at $z^*_0=(0;1,0,\dots,0)$.
Then we have the following commutative diagram:
\begin{equation}
\label{eq.a.1}
\vcenter{
\xymatrix @R=2ex{
\mathscr{E}^{\mathbb{R}}_{X,z^*_0} \tens_{\mathbb{C}} 
\mathscr{C}^{\mathbb{R}}_{Y|X,z^*_0}
\ar[r]&
\mathscr{C}^{\mathbb{R}}_{Y|X,z^*_0} 
\\
H^{n}_{G_{\varDelta,\boldsymbol{\kappa}}\cap U_{\varDelta,\boldsymbol{\kappa}}}\!
(U_{\varDelta,\boldsymbol{\kappa}}; \Oo) \tens_{\mathbb{C}}
H^d_{G_{\boldsymbol{\kappa}}\cap U_{\boldsymbol{\kappa}}}\!(U_{\boldsymbol{\kappa}}; \mathscr{O}_{X}) 
\ar@{>->}[d] \ar[u]\ar[r]^-{\mu^c}
&
H^d_{G_{\widetilde{\boldsymbol{\kappa}}}\cap U_{\widetilde{\boldsymbol{\kappa}}}}\!
(U_{\widetilde{\boldsymbol{\kappa}}}; \mathscr{O}_{X}) \ar@{>->}[d] \ar[u]
\\
H^n_{\widehat{G}_{\varDelta,\boldsymbol{\kappa}}\cap \widehat{U}_{\varDelta,\boldsymbol{\kappa}}}\!
(\widehat{U}_{\varDelta,\boldsymbol{\kappa}};\OOO)
\tens_{\mathbb{C}}H^d_{\widehat{G}_{\boldsymbol{\kappa}}\cap \widehat{U}_{\boldsymbol{\kappa}}}\!
(\widehat{U}_{\boldsymbol{\kappa}}; \mathscr{O}_{\WX})
\ar[r]^-{\mu^c}
&
H^d_{\widehat{G}_{\widetilde{\boldsymbol{\kappa}}}\cap \widehat{U}_{\widetilde{\boldsymbol{\kappa}}}}\!
(\widehat{U}_{\widetilde{\boldsymbol{\kappa}}}; \mathscr{O}_{\WX}), 
}}
\end{equation}
where the down injective morphisms are described in \S \ref{sec:rep-micro}.
We will explain the other morphisms appearing in the  diagram above.
The top horizontal arrow in \eqref{eq.a.1} is associated with the cohomological action
of $\mathscr{E}^{\mathbb{R}}_X$ to $\mathscr{C}^{\mathbb{R}}_{Y|X}$.
The second horizontal arrow $\mu^c$  in \eqref{eq.a.1}  is given by the chain of morphisms
\begin{equation}
\label{eq.a.2}
\begin{aligned}
H^{n}_{G_{\varDelta,\boldsymbol{\kappa}}\cap U_{\varDelta,\boldsymbol{\kappa}}}\!
(U_{\varDelta,\boldsymbol{\kappa}}; \Oo) \tens_{\mathbb{C}}
H^d_{G_{\boldsymbol{\kappa}}\cap U_{\boldsymbol{\kappa}}}\!(U_{\boldsymbol{\kappa}}; \mathscr{O}_{X})&\to
H^{n+d}_{G\cap U}(U;\Oo) 
\\
&\to 
H^d_{G_{\widetilde{\boldsymbol{\kappa}}}\cap U_{\widetilde{\boldsymbol{\kappa}}}}\!
(U_{\widetilde{\boldsymbol{\kappa}}}; \mathscr{O}_{X}).
\end{aligned}\end{equation}
Here we set
\[
G:= G_{\varDelta,\boldsymbol{\kappa}} \cap \pi_2^{-1}(G_{\boldsymbol{\kappa}}), \quad
U := U_{\varDelta,\boldsymbol{\kappa}}\cap\pi_2^{-1}(U_{\boldsymbol{\kappa}}).
\]
The first morphism in \eqref{eq.a.2} is the usual cohomological cup product and the second morphism in \eqref{eq.a.2}
is the cohomological residue morphism. Note that since 
$
G \subset \pi^{-1}_1 (G_{\widetilde{\boldsymbol{\kappa}}})$ and $G \cap \pi^{-1}_1(K) \Subset U$
for any compact subset $K \Subset U_{\widetilde{\boldsymbol{\kappa}}}$, the second morphism in \eqref{eq.a.2} is well defined.
The third horizontal arrow $\mu^c$ in \eqref{eq.a.1} is defined by
the cohomological cup product and residue mapping in the same way as that
for the second horizontal arrow.
Therefore, to show the theorem, it suffices to prove that
the third horizontal arrow $\mu^c$ and $\mu$ defined in Theorem \ref{th:action-morphism}
coincide. Furthermore clearly the following diagram with respect to
the cup product
\[
\xymatrix @R=2ex{
H^n_{\widehat{G}_{\varDelta,\boldsymbol{\kappa}}\cap \widehat{U}_{\varDelta,\boldsymbol{\kappa}}}\!
(\widehat{U}_{\varDelta,\boldsymbol{\kappa}};\OOO)
\tens_{\mathbb{C}}H^d_{\widehat{G}_{\boldsymbol{\kappa}}\cap \widehat{U}_{\boldsymbol{\kappa}}}\!
(\widehat{U}_{\boldsymbol{\kappa}}; \mathscr{O}_{\WX})
 \ar[r] \ar@{=}[d]
&
H^{n+d}_{\widehat{G}\cap \widehat{U}}(\widehat{U};\OOO)\ar@{=}[d]
\\
\dfrac{\varGamma( \widehat{V}^{(*)}_{\varDelta,\boldsymbol{\kappa}}; \OOO)}
{\smashoperator[r]{\Sum_{\alpha \in \mathcal{P}^\vee_n}}
\varGamma( \widehat{V}_{\varDelta,\boldsymbol{\kappa}}^{(\alpha)}; \OOO)} 
\tens_{\mathbb{C}}
\dfrac{\varGamma(\widehat{V}^{(*)}_{\boldsymbol{\kappa}};\mathscr{O}_{\WX})}
{\smashoperator[r]{\Sum_{\beta \in \mathcal{P}^\vee_d}}
\varGamma(\widehat{V}^{(\beta)}_{\boldsymbol{\kappa}};\mathscr{O}_{\WX})}\ar[r]
&
\dfrac{\varGamma(\widehat{W}^{(*,*)}_{\boldsymbol{\kappa}};\OOO)}
{\smashoperator[r]{\Sum_{(\alpha,\,\beta) \in \varLambda}}
 \varGamma(\widehat{W}^{(\alpha,\,\beta)}_{\boldsymbol{\kappa}};\OOO)}
}
\]
commutes. Here we set
\[
\widehat{G}:= \widehat{G}_{\varDelta,\boldsymbol{\kappa}} \cap
\widehat{\pi}_2^{\,-1}(\widehat{G}_{\boldsymbol{\kappa}}),\quad
\widehat{U}:= \widehat{U}_{\varDelta,\boldsymbol{\kappa}}
\cap
\widehat{\pi}_2^{\,-1}(\widehat{U}_{\boldsymbol{\kappa}}).
\]
Hence the problem is reduced to the following proposition.
\begin{prop}
The diagram below commutes\textup{:}
\begin{equation}
\label{eq.a.3}
\vcenter{
\xymatrix @R=2ex{
\smash{H^{n+d}_{\widehat{G}\cap \widehat{U}}(\widehat{U};\OOO}) \ar[r]^-{\mu^c} \ar@{=}[d]
& 
 \smash{H^d_{\widehat{G}_{\widetilde{\boldsymbol{\kappa}}}\cap \widehat{U}_{\widetilde{\boldsymbol{\kappa}}}}\!
(\widehat{U}_{\widetilde{\boldsymbol{\kappa}}}; \mathscr{O}_{\WX}})\ar@{=}[d]
\\
\dfrac{\varGamma(\widehat{W}^{(*,*)}_{\boldsymbol{\kappa}};\OOO)}
{\smashoperator[r]{\Sum_{(\alpha,\,\beta) \in \varLambda}}
 \varGamma(\widehat{W}^{(\alpha,\,\beta)}_{\boldsymbol{\kappa}};\OOO)}
\ar[r]^-{\mu}
&
\dfrac{
\varGamma(\widehat{V}^{(*)}_{\widetilde{\boldsymbol{\kappa}}};\mathscr{O}^{}_{\WX})}
{\smashoperator[r]{\Sum_{\beta \in  \mathcal{P}^\vee_d}}
\varGamma(\widehat{V}^{(\beta)}_{\widetilde{\boldsymbol{\kappa}}};\mathscr{O}^{}_{\WX})}
}}
\end{equation}
Here $\mu^c$ is given by the cohomological residue morphism
and $\mu$ is given by
\[
u(z,w,\eta)\,dw \mapsto
\smashoperator{\int\limits_{\hspace{6ex}\gamma(z,\eta;\varrho,\theta)}}u(z,w,\eta)\,dw.
\]
\end{prop}
\begin{proof}
We first define the closed subsets in 
$T:=\{(w_1,\,\eta) \in \mathbb{C}^2;\,
\lvert \arg \eta| < \dfrac{\,\theta\,}{4}\}$ 
by 
\begin{align*}
L_{\varrho,\theta}
&:= \{(w_1,\eta) \in T;\,\lvert\arg w_1| \leqslant \frac{\,\pi\,}{2} - \frac{\,3\theta\,}{4} + \arg \eta\}, 
\\
L'_{\varrho, \theta}
&:= \{(w_1,\eta) \in L_{\varrho,\theta};\, |w_1| \leqslant \dfrac{\,\varrho\,}{4}|\eta|\}.
\end{align*}
Note that $T \smallsetminus L_{\varrho,\theta}$ and
$T \smallsetminus L'_{\varrho,\theta}$ are pseudoconvex open subsets.
Then the top horizontal morphism $\mu^c$ in \eqref{eq.a.3}
can be decomposed to the chain of morphisms:
\begin{align*}
H^{n+d}_{\widehat{G}\cap\widehat{U}}(\widehat{U};\OOO) 
&\xrightarrow{\psi^c_1}
H^{n+d}_{\widehat{G}_1\cap \widehat{U}'}(\widehat{U}';\OOO) 
\xleftarrow{\psi^c_2}
H^{n+d}_{\widehat{G}_2\cap \widehat{U}'}(\widehat{U}';\OOO) 
 \xleftarrow{\psi^c_3}
H^{n+d}_{\widehat{G}_3\cap \widehat{U}'}(\widehat{U}';\OOO) 
\\
&\xrightarrow{\psi^c_4}
H^{n+d}_{\widehat{G}_4\cap \widehat{U}'}(\widehat{U}';\OOO)  
\xrightarrow{\psi^c_5}
H^d_{\widehat{G}_{\widetilde{\boldsymbol{\kappa}}}\cap \widehat{U}_{\widetilde{\boldsymbol{\kappa}}}}\!(\widehat{U}_{\widetilde{\boldsymbol{\kappa}}}; \mathscr{O}^{}_{\WX}).
\end{align*}
Here we will explain all the subsets appearing in the  chain above.
Set
\allowdisplaybreaks
\begin{align*}
\widehat{U}' &:= \widehat{U} \cap \widehat{\pi}_{1}^{\,-1}(\widehat{U}_{\widetilde{\boldsymbol{\kappa}}})
= \widehat{U}_{\varDelta,\boldsymbol{\kappa}} \cap \widehat{\pi}_{1}^{\,-1}(\widehat{U}_{\widetilde{\boldsymbol{\kappa}}})
\cap \widehat{\pi}_{2}^{\,-1}(\widehat{U}_{\boldsymbol{\kappa}}),
\\
\widehat{K} &:= \smashoperator{\Bcap_{i=2}^n}\{(z,w,\eta);\,
(z_1 - w_1,\eta) \in L_{\varrho, \theta},\,\varrho|z_i - w_i| \leqslant |\eta|\}, \\
\widehat{K}' &:= \smashoperator{\Bcap_{i=2}^n}\{(z,w,\eta);\, 
(z_1 - w_1,\eta) \in L'_{\varrho, \theta},\,\varrho|z_i - w_i| \leqslant |\eta|\}.
\end{align*}
Note that $\widehat{G}_{\varDelta,\boldsymbol{\kappa}} \cap \widehat{U}' \subset \widehat{K} \cap \widehat{U}'$
holds.
Then $\widehat{G}_k$ $(1\leqslant k\leqslant 4)$ are defined by
\begin{alignat*}{3}
\widehat{G}_1 &:= \widehat{K} \cap \widehat{\pi}^{\,-1}_2(\widehat{G}_{\boldsymbol{\kappa}}),
&\widehat{G}_2 &:= \widehat{K}' \cap \widehat{\pi}^{\,-1}_2(\widehat{G}_{\boldsymbol{\kappa}}),\\
\widehat{G}_3 &:= \widehat{G}_2  \cap 
\widehat{\pi}^{\,-1}_1(\widehat{G}_{\widetilde{\boldsymbol{\kappa}}}) , 
\quad&\widehat{G}_4& := \widehat{K}' \cap \widehat{\pi}^{\,-1}_1(\widehat{G}_{\widetilde{\boldsymbol{\kappa}}}).
\end{alignat*}
The morphism $\psi^c_5$ is nothing but the residue morphism. The other morphisms
are canonical ones associated with the inclusion of sets.
If we take $\tilde{\varrho}$ of $\widetilde{\boldsymbol{\kappa}}$ sufficiently small, 
we have
$\widehat{U}' \cap \widehat{G}_1 = \widehat{U}' \cap \widehat{G}_2$. 
Therefore the canonical morphism $\psi^c_2$ becomes an isomorphism. Furthermore, as 
$\widehat{G}_2 \subset \widehat{\pi}_1^{\,-1}(\widehat{G}_{\widetilde{\boldsymbol{\kappa}}})$ holds,
we get
$
\widehat{U}' \cap \widehat{G}_2 = \widehat{U}' \cap \widehat{G}_3
$, thus the canonical morphism $\psi^c_3$ is also an isomorphism.
The corresponding morphisms of \v{C}ech cohomology groups are given by
the following chain:
\begin{multline}
\label{eq:chain}
\dfrac{\varGamma(\widehat{W}^{(*,*)}_{\boldsymbol{\kappa}};\OOO)}
{\smashoperator[r]{\Sum_{(\alpha,\,\beta) \in \varLambda}}
 \varGamma(\widehat{W}^{(\alpha,\,\beta)}_{\boldsymbol{\kappa}};\OOO)}
\xrightarrow[\psi_1]{}
\dfrac{\varGamma(\widehat{W}^{(*,*)}_1;\OOO)}
{\smashoperator[r]{\Sum_{(\alpha,\,\beta) \in \varLambda}}
 \varGamma(\widehat{W}^{(\alpha,\,\beta)}_1;\OOO)}
\xleftarrow[\psi_2]{\dsim}
\dfrac{\varGamma(\widehat{W}^{(*,*)}_2;\OOO)}
{\smashoperator[r]{\Sum_{(\alpha,\,\beta) \in \varLambda}}
 \varGamma(\widehat{W}^{(\alpha,\,\beta)}_2;\OOO)}
\\
\xleftarrow[\psi_3]{\dsim}
\dfrac{Z^{n+d}(\widehat{\mathfrak{W}}^{}_3;\OOO)}
{B^{n+d}(\widehat{\mathfrak{W}}^{}_3;\OOO)}
 \xrightarrow[\psi_4]{}
\dfrac{\varGamma(\widehat{W}^{(*,*)}_4;\OOO)}
{\smashoperator[r]{\Sum_{(\alpha,\,\beta) \in \varLambda}}
 \varGamma(\widehat{W}^{(\alpha,\,\beta)}_4;\OOO)}
\xrightarrow[\psi_5]{}
\dfrac{
\varGamma(\widehat{V}^{(*)}_{\widetilde{\boldsymbol{\kappa}}};\mathscr{O}^{}_{\WX})}
{\smashoperator[r]{\Sum_{\beta \in  \mathcal{P}^\vee_d}}
\varGamma(\widehat{V}^{(\beta)}_{\widetilde{\boldsymbol{\kappa}}};\mathscr{O}^{}_{\WX})}\,.
\end{multline}
We also explain all the subsets appearing in \eqref{eq:chain}.
Set 
\begin{align*}
\widehat{O}^{(1)}&:= 
\{(z,w,\eta) \in \widehat{U}';\,
(z_1 - w_1,\eta) \notin L_{\varrho,\theta}\}, \\
\widehat{O}^{(i)}&:= 
\{(z,w,\eta) \in \widehat{U}';\,
\varrho|z_i - w_i| > |\eta|\} \quad (i=2,\dots,n),
\\
\widehat{O}^{\prime (1)}&:= 
\{(z,w,\eta) \in \widehat{U}';\,
(z_1 - w_1, \eta) \notin L'_{\varrho, \theta}\}, \\
\widehat{O}^{\prime (i)}&:= \{(z,w,\eta) \in \widehat{U}';\,
\varrho|z_i - w_i| > |\eta| \} \quad (i=2,\dots,n).
\end{align*}
Note that these open subsets are pseudoconvex.  Then the coverings
$\{W_1^{(\alpha,\,\beta)}\}$ etc.\  appearing in \eqref{eq:chain} are given by
\begin{alignat*}{3}
\widehat{W}_1^{(\alpha, \,\beta)} &:=
\widehat{O}^{(\alpha)} \cap \widehat{\pi}^{\,-1}_2(\widehat{V}^{(\beta)}_{\boldsymbol{\kappa}}),\quad
&\widehat{W}_2^{(\alpha, \,\beta)} &:=
\widehat{O}'^{(\alpha)} \cap \widehat{\pi}^{\,-1}_2(\widehat{V}^{(\beta)}_{\boldsymbol{\kappa}}),\\
\widehat{W}_3^{(\alpha,\, \beta,\,\beta')} &:=
\widehat{O}'^{(\alpha)} \cap \widehat{\pi}^{\,-1}_2(\widehat{V}^{(\beta)}_{\boldsymbol{\kappa}})
\cap \widehat{\pi}^{\,-1}_1(\widehat{V}^{(\beta')}_{\widetilde{\boldsymbol{\kappa}}}),
\quad
&\widehat{W}_4^{(\alpha,\, \beta')} &:=
\widehat{O}'^{(\alpha)} \cap \widehat{\pi}^{\,-1}_1(\widehat{V}^{(\beta')}_{\widetilde{\boldsymbol{\kappa}}}),
\end{alignat*}
and 
$Z^{n+d}(\widehat{\mathfrak{W}}^{}_3;\OOO)$ (resp.\ $B^{n+d}(\widehat{\mathfrak{W}}^{}_3;\OOO)$) stands for  the $n+d$ cocycle group 
(resp.\ the $n+d$ coboundary group) of 
\v{C}ech complex $C^{\bullet}(\widehat{\mathfrak{W}}^{}_3;\OOO)$  with respect  to
the  covering $\widehat{\mathfrak{W}}^{}_3:= \{\widehat{W}_3^{(i,j,k)}\}_{\atop{1\leqslant i\leqslant n}{1\leqslant j,k\leqslant d}}$.
Let $P dw\in H^{n+d}_{\widehat{G}\cap \widehat{U}} (\widehat{U};\OOO)$, and $u\,dw=u(z,w,\eta)\,dw 
\in \varGamma(\widehat{W}^{(*,*)}_{\boldsymbol{\kappa}};\OOO)$  the corresponding representative 
of the \v{C}ech cohomology group. 
Let us trace the images of $P$ and $u$ by the chain of morphisms.  
\paragraph{\textbf{Step 1.}}
Set $P_1 dw:= \psi_1^c(Pdw)$ and $u_1\,dw := \psi_1(u\,dw)$.
Then clearly $u_1\,dw$ is a representative of $P_1\,dw$ and we have
\[
\mu(u\,dw) = \mu(u_1\,dw) \mod \smashoperator{\Sum_{\beta \in \mathcal{P}^\vee_d}}
\varGamma(\widehat{V}^{(\beta)}_{\widetilde{\boldsymbol{\kappa}}};\mathscr{O}_{\WX}),
\]
where $\mu$ was defined in the statement of the proposition.
\paragraph{\textbf{Step 2.}}
As $\psi_2^c$ is an  isomorphism, there exists $P_2dw$
with $P_1 dw= \psi_2^c(P_2dw)$. Then we can find a representative 
$u_2\,dw \in \varGamma(\widehat{W}^{(*,*)}_2;\OOO)$
of $P_2dw$ such that
\[
u_1\,dw - \psi_2(u_2\,dw) \in 
\smashoperator{\Sum_{(\alpha, \,\beta) \in \varLambda}}
\varGamma(\widehat{W}_1^{(\alpha,\,\beta)};\OOO).
\]
Since we have  similar claims as  in Lemma \ref{lem:path-fundamental},
we have $\mu(u_2\,dw) \in
\varGamma(\widehat{V}^{(*)}_{\widetilde{\boldsymbol{\kappa}}}; \mathscr{O}^{}_{\WX})$ and
\[
\mu(u_1\,dw) = \mu(u_2\,dw)
\mod \smashoperator{\Sum_{\beta \in \mathcal{P}^\vee_d}}
\varGamma(\widehat{V}^{(\beta)}_{\widetilde{\boldsymbol{\kappa}}};\mathscr{O}_{\WX}).
\]
Furthermore, we set
\begin{align*}
\widetilde{\gamma}(z,\eta;\varrho, \theta)
& :=
\gamma(z,\eta;\varrho, \theta) \vee
(-\overline{\gamma}(z,\eta;\varrho,\theta)),
\\
\mu'(u_2\,dw) & := \smashoperator{\int\limits_{\hspace{6ex}\widetilde{\gamma}(z,\eta;\varrho, \theta)}}
u_2(z,w,\eta) \,dw.
\end{align*}
Note that the real $n$-dimensional chain
${\widetilde{\gamma}(z,\eta;\varrho, \theta)}$ in $X$
becomes a product of closed paths where each path is homotopic 
to the circle in $\mathbb{C}_{w_i}$ ($i=1,\dots,n$),
in particular, we have $\partial \tilde{\gamma} = \emptyset$. By the same claim as  Lemma \ref{lem:path-fundamental} (3), we get
$
\mu'(u_2\,dw) \in\smashoperator{\Sum_{\beta \in \mathcal{P}^\vee_d}}
\varGamma(\widehat{V}^{(\beta)}_{\widetilde{\boldsymbol{\kappa}}};\mathscr{O}_{\WX})$. 
Hence we have obtained
\[
\mu(u\,dw) = \mu'(u_2\,dw) 
\mod 
\smashoperator{\Sum_{\beta \in \mathcal{P}^\vee_d}}
\varGamma(\widehat{V}^{(\beta)}_{\widetilde{\boldsymbol{\kappa}}};\mathscr{O}_{\WX}).
\]
\paragraph{\textbf{Step 3.}} As $\psi_3^c$ is an isomorphism, there exists $P_3dw$
with $P_2\,dw = \psi^c_3(P_3\,dw)$. Then we can take a representative
\[
u_3 \,dw= \{u_3^{(\alpha,\,\beta,\,\beta')}dw\}_{(\alpha,\,\beta,\,\beta')\in \varLambda^3_{n+d}}
\in Z^{n+d}(\widehat{\mathfrak{W}}^{}_3;\OOO) \subset 
\smashoperator[r]{\Sum_{(\alpha,\, \beta, \,\beta') \in \varLambda^3_{n+d}}}
\varGamma(\widehat{W}_3^{(\alpha,\,\beta,\,\beta')};\OOO)
\] 
of $P_3dw$, where we set
\[
\varLambda^3_{n+d} := \{(\alpha,\beta,\beta')
\in \mathcal{P}_n \times \mathcal{P}_d \times \mathcal{P}_d;\,
\#\alpha + \#\beta + \#\beta' = n+d\}.
\]
Since the covering $\{\widehat{W}_2^{(\alpha,\,\beta)}\}$ is finer than
$\{\widehat{W}_3^{(\alpha,\,\beta,\,\beta')}\}$, 
we get
\[
u_2\,dw - u_3^{(*,*,\emptyset)}dw \in \smashoperator[r]{\Sum _{(\alpha, \beta) \in \varLambda}}
\varGamma(\widehat{W}_2^{(\alpha,\,\beta)};\OOO),
\]
and thus, we obtain
\[
\mu'(u_2\,dw) = \mu'(u_3^{(*,*,\emptyset)}dw)
\mod 
\smashoperator{\Sum_{\beta \in \mathcal{P}^\vee_d}}
\varGamma(\widehat{V}^{(\beta)}_{\widetilde{\boldsymbol{\kappa}}};\mathscr{O}_{\WX})
\]
for which we have:
\begin{lem}\label{lem:compat_act_cocycle}
The following holds\textup{:}
\[
\mu'(u_3^{(*,\emptyset,*)}dw) = \mu'(u_3^{(*,*,\emptyset)}dw)
\mod 
\smashoperator{\Sum_{\beta \in \mathcal{P}^\vee_d}}
\varGamma(\widehat{V}^{(\beta)}_{\widetilde{\boldsymbol{\kappa}}};\mathscr{O}_{\WX}).
\]
\end{lem}
\begin{proof}
We set $*^{{\vee}k}:= *\smallsetminus\{k\}$.
By the cocycle condition for $u_3\,dw$, we have
\[
(-1)^{n+d}(u_3^{(*,*,\emptyset)} dw- 
      u_3^{(*,*^{\vee d},\{d\})}dw) + \smashoperator{\Sum_{i=1}^n}\, (-1)^i u_3^{(*^{\vee i},*,\{d\})}dw 
+ \smashoperator{\Sum_{i=1}^n}\, (-1)^{n+i} u_3^{(*,*^{\vee i},\{d\})}dw = 0.
\]
Hence, by the same claim as  Lemma \ref{lem:path-fundamental} (2), we obtain
\[
\mu'(u_3^{(*,*,\emptyset)}dw)
= \mu'(u_3^{(*,*^{\vee d},\{d\})}dw)
\mod 
\smashoperator{\Sum_{\beta \in \mathcal{P}^\vee_d}}
\varGamma(\widehat{V}^{(\beta)}_{\widetilde{\boldsymbol{\kappa}}};\mathscr{O}_{\WX}).
\]
By repeating the same argument, we obtain the result.
\end{proof}
Summing up, we have
\[
\mu(u\,dw)
= \mu'(u_3^{(*,\emptyset,*)}dw)
\mod 
\smashoperator{\Sum_{\beta \in \mathcal{P}^\vee_d}}
\varGamma(\widehat{V}^{(\beta)}_{\widetilde{\boldsymbol{\kappa}}};\mathscr{O}_{\WX}).
\]
\paragraph{\textbf{Step 4.}} Set $P_4 \,dw:= \psi^c_4(P_3\,dw)$. Clearly 
$u_4\,dw := \psi_4(u_3\,dw)$ is given by $u_3^{(*,\emptyset,*)}dw$
which is a representative of $P_4$. By the previous step,
we have
\[
\mu(u\,dw)
= \mu'(u_4\,dw)
\mod 
\smashoperator{\Sum_{\beta \in \mathcal{P}^\vee_d}}
\varGamma(\widehat{V}^{(\beta)}_{\widetilde{\boldsymbol{\kappa}}};\mathscr{O}_{\WX}).
\]

\paragraph{\textbf{The Final Step.}}
$\psi^c_5$ is given by the residue morphism.
Then the subsets $\widehat{U}'$, $\widehat{G}_4$ and 
the chain $\tilde{\gamma}(z,\eta; \varrho, \theta)$
satisfy the geometrical situation under which the following Lemma 
\ref{lem:residue-simple} holds.
Hence it follows from the lemma that
the representative of $\psi^c_5(v_4\,dw)$ is given by
$\mu'(u_4\,dw)$. 
Therefore we have the conclusion that
a representative of $\mu^c(P\,dw)$ is given by $\mu(u\,dw)$.
This completes the proof for the proposition.
\end{proof}
We first clarify a geometrical situation.
Let 
$X := \mathbb{C}_z^{\ell}$ and $Y :=  \mathbb{C}_w^n$.
Let $Z$ (resp.\ $U$) be a closed (resp.\ a Stein open) subset in $X$, and
let $K_i$ (resp.\ $W_i$) be a closed (resp.\ a Stein open) subset 
in $X\times \mathbb{C}_{w_i}$ $(i=1,\dots,n)$. 
The mappings $\pi \colon X \times Y \to X$,
$\pi_i \colon  X \times Y \to X \times \mathbb{C}_{w_i}$ and
$\tau_i\colon  X \times \mathbb{C}_{w_i} \to X$ denote
the canonical projections respectively. 
In this situation, the following conditions are also assumed.
\begin{enumerate}[(i)]
\item The subset $U \smallsetminus Z \subset X$ has a covering 
$\{U^{(j)}\}^m_{j=1}$ of Stein open subsets for an $m \leqslant \ell$.
\item  The subset $W_i \smallsetminus K_i$ is Stein in $X \times \mathbb{C}_{w_i}$
for $1 \leqslant  i \leqslant n$.
\item The mapping $\tau_i \colon K_i \cap W_i\to X$  is proper for $1 \leqslant  i \leqslant n$.
\item $U \subset \pi(\smashoperator[r]{\Bcap_{i=1}^n}\pi_1^{-1}(W_i))$.
\end{enumerate}
Set $V^{(i)} := \pi^{-1}_i(W_i \smallsetminus K_i)$ and 
\begin{align*}
K &:= \pi^{-1}(Z) \cap \smashoperator[r]{\Bcap_{i=1}^n} 
	\pi^{-1}_i(K_i),\quad
W := \pi^{-1}(U) \cap \smashoperator[r]{\Bcap_{i=1}^n} \pi^{-1}_i(W_i),\\
W^{(\alpha,\, \beta)} &:= \pi^{-1}(U^{(\alpha)}) \cap V^{(\beta)}\quad
(\alpha \in \mathcal{P}_m,\, \beta \in \mathcal{P}_n).
\end{align*}
We also denote by $\gamma_i(z) \subset \mathbb{C}_{w_i}$ a closed path in 
$\tau_i^{-1}(z) \cap W_i$ (regarded as a subset in $\mathbb{C}_{w_i}$)
turning around each component of $\tau_i^{-1}(z) \cap K_i$ once 
with anti-clockwise direction.
\begin{lem}{\label{lem:residue-simple}}
Under the situation described above, there exists  the following commutative
diagram\textup{:}
\[
\xymatrix @R=1ex{
H^{m+n}_{K \cap W} (W; \mathscr{O}_{X \times Y}^{\smash{(0,n)}})
\ar[r]^-{\mu^c}\ar@{=}[d]
& 
H^m_{Z \cap U}(U; \mathscr{O}_{X}) \ar@{=}[d] 
\\
\dfrac{\varGamma(W^{(*,*)}; \mathscr{O}_{X \times Y}^{\smash{(0,n)}})}
{\smashoperator[r]{\Sum_{(\alpha,\, \beta) \in (\mathcal{P}_m \times \mathcal{P}_n)^\vee }}
\varGamma(W^{(\alpha,\,\beta)};\mathscr{O}_{X \times Y}^{\smash{(0,n)}})}
 \ar[r]^-{\mu}
&
\dfrac{\varGamma(U^{(*)}; \mathscr{O}_{X}^{})}
{\smashoperator[r]{\Sum_{\alpha \in\mathcal{P}^\vee_m}}
\varGamma(U^{(\alpha)};\mathscr{O}_{X}^{})}.
}
\]
Here $\mu$ is defined by
\[
u(z,w)\,dw \mapsto
\smashoperator{\int\limits_{\hspace{10ex}\gamma_1(z) \times \dots \times \gamma_n(z)}}u(z,w)\, dw,
\]
and
$(\mathcal{P}_m \times \mathcal{P}_n)^\vee $ denotes $ \{(\alpha,\beta)
\in \mathcal{P}_m \times \mathcal{P}_n;\,
\#\alpha + \#\beta =m+n-1\}$. 
\end{lem}
\begin{proof}
The lemma immediately follows form \cite[Corollary 3.1.4]{K-K1}.
However, for the reader's convenience, we will give
a proof in what follows.
Clearly we can apply the induction with respect to $n$, and thus,
we may assume from the beginning that $n=1$; i.e.\ $Y=\mathbb{C}$.
Then we show the claim by the induction for $m \geqslant 0$. 

First we prove the lemma for $m=0$.
Let $\mathcal{S}^\bullet_{X \times \mathbb{C}}$ (resp.\ $\mathcal{S}^\bullet_{X}$)
be the $\bar{\partial}$ complex of $\mathscr{O}_{X \times \mathbb{C}}^{\smash{(0,1)}}$ (resp. $\mathscr{O}_{X}$) 
with coefficients in
the sheaf of distributions on $X \times \mathbb{C} = \mathbb{R}^{2\ell} \times \mathbb{R}^2$ 
(resp.\ $X = \mathbb{R}^{2\ell}$).
Then we have the following diagram:
\begin{equation}
\label{eq.a.4}
\vcenter{
\xymatrix @R=2.9ex{
H^{1}_{K \cap W} (W;\mathscr{O}_{X \times \mathbb{C}}^{\smash{(0,1)}})
\ar[d]_-{\iota}
& \varGamma(W \smallsetminus K;\mathscr{O}_{X \times \mathbb{C}}^{\smash{(0,1)}})\big/
\varGamma (W;\mathscr{O}_{X \times \mathbb{C}}^{\smash{(0,1)}})\ar[d]\ar[l]_-{\dsim}^-{\delta'}
\\
\smash{\smashoperator{\varinjlim_{\widetilde{K}}}
H^{1}\varGamma_{\widetilde{K} \cap W}(W;\mathcal{S}^\bullet_{X \times \mathbb{C}}})\ar[d]_-{\int_c}
&
\smash{\smashoperator{\varinjlim_{\widetilde{K}}}\varGamma(W \smallsetminus \widetilde{K};
\mathscr{O}_{X \times \mathbb{C}}^{\smash{(0,1)}})}\big/
\varGamma (W;\mathscr{O}_{X \times \mathbb{C}}^{\smash{(0,1)}})\ar[l]_-{\delta}\ar[d]_-{\mu}
\\
H^0\varGamma(U; \mathcal{S}^\bullet_{X}) \ar@{=}[r]
&
\varGamma(U; \mathscr{O}_{X}).
}}
\end{equation}
Here $\widetilde{K}$ ranges through closed subsets in $W$ such that
$K \subset \operatorname{Int}\widetilde{K}$ and $\pi|_W \colon \widetilde{K} \to X$ is proper. 
Here $\operatorname{Int}\widetilde{K}$ denotes the interior of $\widetilde{K}$. 
The morphism $\delta$ is given by $
u \mapsto \bar{\partial} \tilde{u}$, 
where $\tilde{u}$ is a distribution extension of $u$ to $W$
with $u = \tilde{u}$ on $W \smallsetminus \widetilde{K}$. The morphism $\int_c$
is nothing but the integration along the fiber of 
$\pi \colon  X \times \mathbb{C} \to X$ for distributions. 
Note that the element in
$H^{1}
\varGamma_{\widetilde{K}\cap W}(W;\mathcal{S}^\bullet_{X \times \mathbb{C}})$ is a real differential $2$-form
as $\mathcal{S}^\bullet_{X \times \mathbb{C}}$ is the $\bar{\partial}$ complex of $\mathscr{O}_{X \times \mathbb{C}}^{\smash{(0,1)}}$.
Then the commutativity of the lower square 
in \eqref{eq.a.4} comes from the Stokes formula.
Let $\mathcal{B}^\bullet_{X \times \mathbb{C}}$ be a $\bar{\partial}$ complex of $\mathscr{O}_{X \times \mathbb{C}}^{\smash{(0,1)}}$
with coefficients in the sheaf of hyperfunctions
on $X \times \mathbb{C} = \mathbb{R}^{2\ell} \times \mathbb{R}^2$. Then we have
$
H^{1}_{K \cap W} (W;\mathscr{O}_{X \times \mathbb{C}}^{\smash{(0,1)}})
=
H^{1}\varGamma_{K\cap W}(W; \mathcal{B}^\bullet_{X \times \mathbb{C}}).
$
The morphism $\delta'$ is given by $u \mapsto \bar{\partial} \tilde{u}$, 
where $\tilde{u}$ is an extension of $u$ to $W$ as an element of the flabby sheaf.
The morphism $\iota$ is the composition of morphisms
\[
H^{1}\varGamma_{K \cap W}(W; \mathcal{B}^\bullet_{X \times \mathbb{C}})
\to 
\varinjlim_{\widetilde{K}}
H^{1}\varGamma_{\widetilde{K} \cap W}(W;\mathcal{B}^\bullet_{X \times \mathbb{C}})
\eleftarrow
\varinjlim_{\widetilde{K}}
H^{1}\varGamma_{\widetilde{K} \cap W}(W; \mathcal{S}^\bullet_{X \times \mathbb{C}}).
\]
Let $\tilde{u}'$ (resp.\ $\tilde{u}$) be a distribution (resp.\ hyperfunction) extension of 
$u \in \Gamma(W \smallsetminus K; \mathscr{O}_{X \times \mathbb{C}}^{\smash{(0,1)}})$ 
to $W$ with $\tilde{u}' = u$ on $W \smallsetminus \widetilde{K}$ 
(resp.\ $\tilde{u} = u$ on $W \smallsetminus K$).
Since $\operatorname{supp}(\tilde{u} - \tilde{u}') \subset \widetilde{K}$,
we have
\[
\partial{\tilde{u}} - \partial{\tilde{u}'} = 0 \in 
\varinjlim\limits_{\widetilde{K}}H^{1}\varGamma_{\widetilde{K} \cap W}(W; \mathcal{B}^\bullet_{X \times \mathbb{C}}),
\]
which implies the commutativity of the upper square in \eqref{eq.a.4}.
Hence, as the residue morphism $\mu^c$ is the composition 
$\int_c \circ \iota$ by definition, 
we have obtained the claim of the lemma for the case $m = 0$.
Now suppose that the claim of the lemma is true for $0,\dots,m-1$.
We will show the lemma for $m$. Let us consider the commutative diagram
between exact sequences:
\begin{equation}
\label{eq.a.6}
\vcenter{
\xymatrix @C=1em @R=2ex{ 
H_{K' \cap W}^{m}(W;\mathscr{O}_{X \times \mathbb{C}}^{\smash{(0,1)}}) 
\ar[r] \ar[d]_-{\mu^c}&
H_{K' \cap W^{(m)}}^{m}(W^{(m)};\mathscr{O}_{X \times \mathbb{C}}^{\smash{(0,1)}})
\ar[r] \ar[d]_-{\mu^c}
&
H_{K \cap W}^{m+1}(W;\mathscr{O}_{X \times \mathbb{C}}^{\smash{(0,1)}}) \ar[r] \ar[d]_-{\mu^c}&  0 
\\
H_{Z' \cap U}^{m-1}(U;\mathscr{O}^{}_X) 
\ar[r]&
H_{Z' \cap U^{(m)}}^{m-1}(U^{(m)}; \mathscr{O}^{}_X)
\ar[r] &
H_{Z \cap U}^{m}(U; \mathscr{O}^{}_X) \ar[r]& 0,
}}
\end{equation}
where $W^{(m)}:= W \cap \pi^{-1}(U^{(m)})$, $Z' := U \smallsetminus \smashoperator{\Bcup_{i=1}^{m-1}}U^{(i)}$
and $K' := \pi^{-1}(Z') \cap \pi_1^{-1}(K_1)$.
We also have the commutative diagram between exact sequences:
\begin{equation}
\label{eq.a.7}
\vcenter{
\xymatrix @C=1em @R=1.9ex{
\dfrac{\varGamma(W'^{(*,1)};\mathscr{O}_{X \times \mathbb{C}}^{\smash{(0,1)}})}
{\smash{\smashoperator{\Sum_{\alpha\in\mathcal{P}^\vee_{m-1}}}
\varGamma(}W'^{(\alpha,1)};\mathscr{O}_{X \times \mathbb{C}}^{\smash{(0,1)}})} \ar[r] \ar[d]_-{\mu}
&
\dfrac{\varGamma(W''^{(*,1)};\mathscr{O}_{X \times \mathbb{C}}^{\smash{(0,1)}})}
{\smash{\smashoperator{\Sum_{\alpha\in\mathcal{P}^\vee_{m-1}}}
\varGamma(W''^{(\alpha,1)}};\mathscr{O}_{X \times \mathbb{C}}^{\smash{(0,1)}})}\ar[r] \ar[d]_-{\mu}
&
\dfrac{\varGamma(W^{(*,1)};\mathscr{O}_{X \times \mathbb{C}}^{\smash{(0,1)}})}
{\smash{\smashoperator{\Sum_{\alpha\in\mathcal{P}^\vee_{m}}}
\varGamma(W^{(\alpha,1)}};\mathscr{O}_{X \times \mathbb{C}}^{\smash{(0,1)}})}\ar[r] \ar[d]_-{\mu}
&0 
\\
\dfrac{\varGamma(U'^{(*)};\mathscr{O}_{X }^{})}
{\smashoperator{\Sum_{\alpha\in\mathcal{P}^\vee_{m-1}}}
\varGamma(U'^{(\alpha)};\mathscr{O}_{X }^{})}\ar[r]
&
\dfrac{\varGamma(U''^{(*)};\mathscr{O}_{X }^{})}
{\smashoperator{\Sum_{\alpha\in\mathcal{P}^\vee_{m-1}}}
\varGamma(U''^{(\alpha)};\mathscr{O}_{X }^{})} \ar[r] &
\dfrac{\varGamma(U^{(*)};\mathscr{O}_{X }^{})}
{\smashoperator{\Sum_{\alpha\in\mathcal{P}^\vee_{m}}}
\varGamma(U^{(\alpha)};\mathscr{O}_{X }^{})} \ar[r]&
0. }}
\end{equation}
Here $\{W'^{(\alpha,1)}\}$, $\{W''^{(\alpha,1)}\}$, $\{U'^{(\alpha)}\}$, $\{U''^{(\alpha)}\}$  are 
the corresponding the coverings of $W \smallsetminus K'$, $W^{(m)} \smallsetminus K'$, $U \smallsetminus Z'$, 
$U^{(m)} \smallsetminus Z'$    respectively.
By the induction hypothesis, the first and the second $\mu^c$ and $\mu$ 
in \eqref{eq.a.6} and \eqref{eq.a.7} coincide. Hence the third ones in the both diagrams
also coincide.
The proof is  complete.
\end{proof}
\section{General Construction of $C^\mathbb{R}_{Y|X,z^*_0}$}\label{ap:general_C}
In this appendix, we will extend theories developed in Sections
\ref{sec:rep-micro} and \ref{sec:actions_ER} to a general family of \v{C}ech coverings, that
enables us to define the symbol mapping $\sigma$ in a general complex manifold.  
We continue to use the same notation as those in  Section \ref{sec:rep-micro} unless
we specify them.
Let $X$ be an $n$-dimensional complex manifold 
with a system of local coordinates $z=(z_1,\dots,z_n)$, and
$Y$  a closed complex submanifold of $X$ which is defined locally by $\{z'=0\}$ 
where $z = (z',z'')$ with $z':=(z_1,\dots,z_d)$ for 
some $1 \leqslant d \leqslant n$. 
Set $\widehat{ X}:= X\times \mathbb{C}$, and
let $\pi_\eta \colon \widehat{ X} \ni (z, \eta) \mapsto z\in  X$ be  
the canonical projection.
Let $z_0 = (0, z''_0) \in Y$ and $z^*_0=(z''_0;\zeta'_0) \in T^*_YX$ with $\zeta'_0 \ne 0$. 

Let $\chi = \{f_1(z), \dots,f_{d}(z)\}$ be a sequence of holomorphic functions in an open neighborhood
of $z_0$ satisfying the following conditions:
\begin{enumerate}[(1)]
 \item $df_1(z_0) \wedge df_2(z_0) \wedge \dots \wedge df_{d}(z_0) \ne 0$.
 \item $f_1, \dots, f_{d}$ belong to the defining ideal $\mathscr{I}_Y$ of $Y$.
 \item We have
\[
{}^{\mathsf{t}}\!\bigl[\dfrac{\partial f}{\partial z}(z_0)\bigr]\boldsymbol{e} = (\zeta'_0, 0) \in (T^*X)_{z_0}
\]
where $f(z) := (f_1(z),\dots, f_{d}(z))$ and $\boldsymbol{e} := (1,0,\dots,0) \in \mathbb{C}^d$.
\end{enumerate}
We denote by $\varXi(z^*_0)$ the set of sequences satisfying the conditions above.
Set 
\begin{align*}
f(z) & := (f_1(z),f_2(z),\dots,f_{d}(z)) = (f_1(z), f'(z)),
\\
G^\chi_{\varrho,L} & :=\{z \in X;\,\varrho^2|f'(z)| \leqslant |f_1(z)|,\,
f_1(z) \in L\},
\end{align*}
where $\varrho > 0$ and
$L \subset \mathbb{C}$ 
is a closed convex cone 
with $L \subset \{\tau \in \mathbb{C};\, \Re \tau > 0\} \cup \{0\}$. 
We also set, for an open neighborhood $U$ of $z_0$ in $X$,
\begin{align*}
\widehat{G}^\chi_{\varrho,L} &:= \{(z,\eta) \in \widehat{X};\,
\varrho|f'(z)| \leqslant |\eta|,\, f_1(z) \in L\},
\\
\widehat{U}^\chi_{\varrho,r,\theta} &:= 
\{(z,\eta) \in  U \times  S_{r,\theta};\,|f_1(z)| < \varrho |\eta|\}.
\end{align*}
Now we define
\begin{align*}
\widehat{C}^{\,\mathbb{R}, \chi}_{\smash{Y|X,z^*_0}} & :=
\smashoperator{\varinjlim_{\varrho,r,\theta,L,U}}
H^d_{\widehat{G}^{\chi}_{\varrho,L} \cap\widehat{U}^{\chi}_{\varrho,r,\theta}}
\!( \widehat{U}^{\chi}_{\varrho,r,\theta}; \mathscr{O}_{\WX}),
\\
C^{\mathbb{R}, \chi}_{\smash{Y|X,z^*_0}}&  := \Ker(\partial^{}_\eta\colon 
\widehat{C}^{\,\mathbb{R},\chi}_{\smash{Y|X,z^*_0}} \to \widehat{C}^{\,\mathbb{R},\chi}_{\smash{Y|X,z^*_0}}).
\end{align*}
Then, by the same reasoning as that in  Section \ref{sec:rep-micro}, we have the isomorphisms
\[
\mathscr{C}^{\mathbb{R}}_{Y|X,z^*_0}
\eleftarrow
\smashoperator[r]{\varinjlim_{\varrho,L,U}}H^d_{G^\chi_{\varrho,L}\cap U}(U;\mathscr{O}_{X})
\earrow
C^{\mathbb{R}, \chi}_{\smash{Y|X,z^*_0}},
\]
where these isomorphisms are associated with the natural inclusions of sets
and the canonical morphism $\pi_\eta^{-1}
\mathscr{O}_{X} \to \mathscr{O}_{\WX}$ as we have seen in 
 Section \ref{sec:rep-micro}.
Hence, for any $\chi_1$ and $\chi_2$ in $\varXi(z^*_0)$, two modules
$C^{\mathbb{R}, \chi_1}_{\smash{Y|X,z^*_0}}$ and
$C^{\mathbb{R}, \chi_2}_{\smash{Y|X,z^*_0}}$ are isomorphic through
$\mathscr{C}^{\mathbb{R}}_{Y|X,z^*_0}$.
Using this fact, we replace the definition of $C^{\mathbb{R}}_{Y|X,z^*_0}$ introduced
in  Section \ref{sec:rep-micro} with a slightly generalized one.
From now on, we write 
$M^{\mathbb{R},\chi}_{\smash{Y|X,z^*_0}} 
:= \smashoperator[r]{\varinjlim_{\varrho,L,U}}H^d_{G^\chi_{\varrho,L}\cap U}(U;\mathscr{O}_{X})$
for short.
\begin{defn}
 We denote by $C^{\mathbb{R}}_{Y|X,z^*_0}$ the isomorphism class 
$\{C^{\mathbb{R}, \chi}_{\smash{Y|X,z^*_0}}\}_{\chi \in \varXi(z^*_0)}$
consisting of $C^{\mathbb{R}, \chi}_{\smash{Y|X,z^*_0}}$ indexed by $\chi \in \varXi(z_0^*)$.
In the same way, the isomorphism class $M^{\mathbb{R}}_{Y|X,z^*_0}$ is defined by
$\{M^{\mathbb{R}, \chi}_{\smash{Y|X,z^*_0}}\}_{\chi \in \varXi(z^*_0)}$.
\end{defn}
By a direct consequence of the construction above, we have the 
morphism of $C^{\mathbb{R}}_{Y|X, z^*_0}$ associated with a coordinates transformation.
Let $w =(w',w'')$ be a system of local coordinates of a copy of $X$ where $Y$ is 
locally defined by $w' = 0$, and 
$z=\varPhi(w)$ a local coordinates transformation in an open neighborhood of $w_0 \in Y$
satisfying $\varPhi(Y) \subset Y$ and $z_0 = \varPhi(w_0)$. Set $\widehat{\varPhi} := \varPhi \times \mathds{1}_\eta$.
Then it induces the sheaf morphism
\[
\widehat{\varPhi}^{\,-1} \mathscr{O}_{\WX} \ni \varphi \mapsto \varphi \circ \widehat{\varPhi}\in \mathscr{O}_{\WX}\,.
\]
Let $w^*_0 := (w_0; {}^{\mathsf{t}}\!\bigl[\dfrac{\partial \varPhi}{\partial w}(w_0)\bigr](\zeta'_0, 0)) \in T^*_YX$.
It is easy to see
\[
\chi \circ \varPhi := \{f_1 \circ \varPhi, \dots, f_{d} \circ \varPhi\} \in \varXi(w^*_0)
\] 
for any $\chi = \{f_1, \dots, f_{d}\} \in \varXi(z^*_0)$.
Hence we have the morphism
$
C^{\mathbb{R}, \chi}_{\smash{Y|X,z^*_0}} \to
C^{\mathbb{R}, \chi\circ \varPhi}_{\smash{Y|X,w^*_0}}
$
defined by $[u(z,\eta)] \to [u(\varPhi(w),\eta)]$,
which gives 
$\widehat{\varPhi}^* \colon C^{\mathbb{R}}_{\smash{Y|X,z^*_0}} \to C^{\mathbb{R}}_{\smash{Y|X,w^*_0}}
$\,.
This morphism is compatible with the morphism 
$\varPhi^* \colon \mathscr{C}^{\mathbb{R}}_{Y|X,z^*_0} \to
\mathscr{C}^{\mathbb{R}}_{Y|X,w^*_0}$ 
associated with the coordinates transformation $\varPhi$
because the both morphisms are induced from
the same coordinates transformation 
of holomorphic functions.

Next we consider a \v{C}ech representation of $C^{\mathbb{R}, \chi}_{\smash{Y|X,{z}^*_0}}$
for $\chi = \{f_1, \dots,f_{d}\} \in \varXi(z^*_0)$.
Set
\begin{align*}
U^{\chi}_{\boldsymbol{\kappa}}& :=\smashoperator{\Bcap_{i=2}^{d}}
\{z = (z',z'') \in X;\, |f_1(z)| < \varrho r,\, |f_i(z)| < r', \,
\lVert z'' - z''_0\rVert < r'\}, \\
\widehat{U}^{\chi}_{\boldsymbol{\kappa}}& := 
\smashoperator{\Bcap_{i=2}^{d}}\{(z,\eta) = (z',z'',\eta) 
\in X \times S_{\boldsymbol{\kappa}};\,
 |f_1(z)| < \varrho|\eta|,\, |f_i(z)| < r', \,\rVert z'' - z''_0\lVert < r'\},
\end{align*}
where $\lVert z''\rVert$ denotes $\max\{|z_{d+1}|,\dots, |z_n|\}$.
We also define
\allowdisplaybreaks
\begin{align*}
V^{\chi,(1)}_{\boldsymbol{\kappa}}&:= 
\{z \in U^{\chi}_{\boldsymbol{\kappa}};\,
\dfrac{\,\pi\,}{2} - \theta <  \arg f_1(z) < \dfrac{\,3\pi\,}{2} + \theta\}, \\
V^{\chi,(i)}_{\boldsymbol{\kappa}} &:= 
\{z \in U^{\chi}_{\boldsymbol{\kappa}};\,\varrho^{2}|f_i(z)| > |f_1(z)| \} \qquad (2 \leqslant i \leqslant d),
\\
\widehat{V}^{\chi,(1)}_{\boldsymbol{\kappa}} &:= 
\{(z,\eta) \in \widehat{U}^{\chi}_{\boldsymbol{\kappa}};\,
\dfrac{\,\pi\,}{2} - \theta <  \arg f_1(z) < \dfrac{\,3\pi\,}{2} + \theta\}, \\
\widehat{V}^{\chi,(i)}_{\boldsymbol{\kappa}}&:= 
\{(z,\eta) \in \widehat{U}^{\chi}_{\boldsymbol{\kappa}};\,
\varrho|f_i(z)| > |\eta|\} \qquad (2 \leqslant i \leqslant d).
\end{align*}
Then it follows from the same arguments in  Section \ref{sec:rep-micro} that we have
\begin{align*}
\widehat{C}^{\,\mathbb{R},\chi}_{\smash{Y|X,z^*_0}} & =
\varinjlim\limits_{\boldsymbol{\kappa}} 
\varGamma(\widehat{V}^{\chi,(*)}_{\boldsymbol{\kappa}};\mathscr{O}^{}_{\WX})\big/
\smashoperator{\Sum_{\alpha \in\mathcal{P}^\vee_d}}
\varGamma(\widehat{V}^{\chi,(\alpha)}_{\boldsymbol{\kappa}};\mathscr{O}^{}_{\WX}), 
\\
{C}^{\mathbb{R},\chi}_{\smash{Y|X,z^*_0}} & =
\varinjlim\limits_{\boldsymbol{\kappa}} 
\{u \in 
\varGamma(\widehat{V}^{\chi,(*)}_{\boldsymbol{\kappa}};\mathscr{O}^{}_{\WX})\big/
\smashoperator{\Sum_{\alpha \in\mathcal{P}^\vee_d}}
\varGamma(\widehat{V}^{\chi,(\alpha)}_{\boldsymbol{\kappa}};\mathscr{O}^{}_{\WX});\, 
\partial_\eta u  = 0 \},\\
M^{\mathbb{R},\chi}_{\smash{Y|X,z^*_0}} &=
\varinjlim\limits_{\boldsymbol{\kappa}} 
\varGamma(V^{\chi,(*)}_{\boldsymbol{\kappa}};\mathscr{O}_{X})\big/
\smashoperator{\Sum_{\alpha \in \mathcal{P}^\vee_d}}
\varGamma(V^{\chi,(\alpha)}_{\boldsymbol{\kappa}};\mathscr{O}_{X}).
\end{align*}

Let us recall the definitions of the paths $\gamma_1(z,\eta;\varrho,\theta)$ and
$\gamma_i(z,\eta;\varrho)$ in $\mathbb{C}$ which were given in  Section \ref{sec:rep-micro}.
In this appendix, we take slightly modified paths.
Set 
\[
\gamma_1(\eta;\varrho,\theta) := - \gamma_1(0,\eta; \varrho, \theta),\quad
\gamma_i(\eta;\varrho) := \gamma_i(0,\eta;\varrho)\,\, (i > 1).
\]
We define the real $d$-dimensional chain in $\mathbb{C}^d$
\[
\gamma(\eta;\varrho,\theta) := \gamma_1(\eta;\varrho,\eta) \times
\gamma_2(\eta;\varrho) \times \dots \times \gamma_{d}(\eta;\varrho).
\]
Then, for any $(0,z'') \in Y$ near $z_0$, 
we also define the real $d$-dimensional chain in $\mathbb{C}^d_{z'}$
by
\[
\gamma^\chi(z'',\eta;\varrho,\theta)
:= \{z' \in \mathbb{C}^d;\, f(z',z'') \in \gamma(\eta;\varrho,\theta)\},
\]
where $f(z) := (f_1(z), \dots,f_d(z))\colon \mathbb{C}^n \to \mathbb{C}^d$
for $\chi = \{f_1,\dots,f_d\} \in \varXi(z^*_0)$ and the orientation
of $\gamma^\chi$ is determined by that of $\gamma$ through $f$.

Let us introduce the symbol spaces 
\begin{align*}
\mathfrak{S}^{}_{Y|X,z^*_0}
:= \varinjlim_ {\varOmega,   S}
\mathfrak{S}^{}_{Y|X} (\varOmega; S)&  \supset
\mathfrak{N}^{}_{Y|X,z^*_0}  := \varinjlim_ {\varOmega,   S}
\mathfrak{N}^{}_{Y|X} (\varOmega; S), 
\\
\mathscr{S}^{}_{Y|X,z^*_0}  :=     \smashoperator{\varinjlim_{\varOmega\owns z^*_0}}
\mathscr{S}^{}_{Y|X}(\varOmega) &
\supset  
\mathscr{N}^{}_{Y|X,z^*_0}  :=  \smashoperator{\varinjlim_{\varOmega\owns z^*_0}}
\mathscr{N}^{}_{Y|X}(\varOmega).
\end{align*}
Here  $\varOmega \cset T_Y^*X$ ranges through open conic neighborhoods of $z^*_0$, 
and the inductive limits with respect to $S$ are taken by $r^{}_0$, $\theta\to 0$.
The sets
$\mathfrak{S}^{}_{Y|X} (\varOmega; S)$,
$\mathfrak{N}^{}_{Y|X} (\varOmega; S)$,
$\mathscr{S}^{}_{Y|X}(\varOmega)$ and
$\mathscr{N}^{}_{Y|X}(\varOmega)$
are defined in the same way as  in  Section \ref{sec:symbol-with-ap}. 
Then we can define the mapping  
$\hat{\sigma}^\chi \colon {C}^{\mathbb{R},\chi}_{\smash{Y|X,z^*_0}} \to
\mathfrak{S}^{}_{Y|X,z^*_0}/\mathfrak{N}^{}_{Y|X,z^*_0}$ 
by
\[
 \hat{\sigma}^\chi([u])(z'', \zeta', \eta):= 
\smashoperator{\int\limits_{\hspace{8ex}\gamma^\chi(z'',\eta;\varrho,\theta)}}u(z',z'',\eta) 
\,e^{-\langle z', \zeta'\rangle} dz'
\]
for $u(z',z'',\eta) \in \varGamma(\widehat{V}^{\chi,(*)}_{\boldsymbol{\kappa}};\mathscr{O}^{}_{\WX})$
with a suitable ${\boldsymbol{\kappa}}$.
Similarly we get the mapping 
$
M^{\mathbb{R},\chi}_{\smash{Y|X,z^*_0}}
\to \mathscr{S}^{}_{Y|X,z^*_0}/\mathscr{N}^{}_{Y|X,z^*_0}$ by
\[
 \sigma^\chi([v])(z'', \zeta'):= 
\smashoperator{\int\limits_{\hspace{8ex}\gamma^\chi(z'',\eta_0;\varrho,\theta)}}v(z',z'') 
\,e^{-\langle z', \zeta'\rangle} dz'
\]
for $v(z',z'') \in \varGamma(V^{\chi,(*)}_{\boldsymbol{\kappa}};\mathscr{O}_{X})$
with a suitable ${\boldsymbol{\kappa}}$ and a sufficiently small fixed $\eta_0 > 0$.

Now we have the following theorem.
\begin{thm}{\label{theorem:invariant-symbol-definition}}
 The morphisms $\hat{\sigma}^\chi$ and $\sigma^\chi$ induce the well-defined mappings
$\hat{\sigma}\colon C^{\,\mathbb{R}}_{Y|X,z^*_0}
\to 
\mathfrak{S}^{}_{Y|X,z^*_0}/\mathfrak{N}^{}_{Y|X,z^*_0}$ 
and $\sigma \colon M^{\mathbb{R}}_{Y|X,z^*_0}
\to 
\mathscr{S}^{}_{z^*_0}/\mathscr{N}^{}_{z^*_0}$ respectively.
To be more precise, if $\chi_1$,  $\chi_2 \in \varXi(z^*_0)$
and  $[u_1] \in C^{\,\mathbb{R},\chi_1}_{\smash{Y|X,z^*_0}}$
and $[u_2] \in C^{\,\mathbb{R},\chi_2}_{\smash{Y|X,z^*_0}}$ determining
the same element in $\mathscr{C}^{\mathbb{R}}_{Y|X,z^*_0}$, it follows that 
$\hat{\sigma}^{\chi_1}([u_1]) = \hat{\sigma}^{\chi_2}([u_2])
\in \mathfrak{S}^{}_{Y|X,z^*_0}/\mathfrak{N}^{}_{Y|X,z^*_0}$\,. 
Similarly, for $[v_1] \in 
M^{\mathbb{R},\chi_1}_{\smash{Y|X,z^*_0}}$
and $[v_2] \in M^{\mathbb{R},\chi_2}_{\smash{Y|X,z^*_0}}$
giving the same element in $\mathscr{C}^{\mathbb{R}}_{Y|X,z^*_0}$, it follows that
$
\sigma^{\chi_1}([v_1]) = \sigma^{\chi_2}([v_2]) 
\in \mathscr{S}^{}_{Y|X,z^*_0}/\mathscr{N}^{}_{Y|X,z^*_0}$\,. 
\end{thm}
\begin{proof}
 By a linear coordinate transformation, we may assume $z^*_0 = (0; 1,0,\dots,0) \in T^*_YX$ and
$z_0 = 0 \in X$.
We need the following easy lemma.
\begin{lem}{\label{lem:first_entry_preserve_esitimate}}
Let $g(z)$ be a holomorphic function in an open neighborhood of $z_0$. Assume that
$g(z) \in \mathscr{I}_Y$ and ${}^{\mathsf{t}}\!\bigl[\dfrac{\partial g}{\partial z}(z_0)\bigr] = (1,0,\dots,0)$. 
Then,
for $\varrho$ and $\theta$, there exists a sufficiently
small $\varepsilon > 0$ 
such that
\[
\Re g(z) \geqslant \varepsilon |\eta| \qquad
(z' \in \partial \gamma(\eta; \varrho, \theta),\,\, 
\eta \in S_{\boldsymbol{\kappa}},
\,
|z''| \leqslant \varepsilon, \,|\eta| < \varepsilon),
\]
where $\partial \gamma$ denotes the boundary of $\gamma(\eta;\varrho,\theta)$.
\end{lem}
\begin{proof}
The Taylor expansion of $g(z)$ along $Y$ is given by
\[
g(z) = \psi_1(z'')z_1 + \psi_2(z'') z_2 + \dots + \psi_d(z'') z_d + O(|z'|^2)
\]
with $\psi_1(0) = 1$ and $\psi_k(0) = 0$ ($k \geqslant 2$).  
The claim immediately follows from this.
\end{proof}
Let $\chi = \{f_1,\dots,f_d\} \in \varXi(z^*_0)$.
Set $f(z) := (f_1(z), \dots, f_d(z))$ and, for $z''$ near $0$, 
we write by $f_{z''}(z')$ the
mapping $f(z',z'')$ regarded as a mapping of the variable $z'$ with a fixed $z''$.
Then, by the coordinates transformation, we have
\begin{align*}
\hat{\sigma}^{\chi}([u]) (z'',\zeta', \eta)  &=
\smashoperator{\int\limits_{\hspace{4ex}\gamma(\eta;\varrho,\theta)}} u (f_{z''}^{\,-1}(w'), z'', \eta) 
\,e^{-\langle f_{z''}^{-1}(w'), \zeta'\rangle} \det[\partial_{w'}f^{\,-1}_{z''}]\,dw',
\\
\sigma^{\chi}([v]) (z'',\zeta') & =
\smashoperator{\int\limits_{\hspace{5ex}\gamma(\eta_0;\varrho,\theta)}} v (f_{z''}^{\,-1}(w'), z'') 
\,e^{-\langle f_{z''}^{-1}(w'), \zeta'\rangle} \det[\partial_{w'}f^{\,-1}_{z''}]\,dw'.
\end{align*}
Therefore, by applying Lemma \ref{lem:first_entry_preserve_esitimate} to
the first coordinate function of $f_{z''}^{\,-1}(w')$,  we have the commutative 
diagram below:
\[
\xymatrix @R=3ex{
M^{\mathbb{R},\chi}_{\smash{Y|X,z^*_0}}\ar[r]^-{\sigma^{\chi}}\ar[d]^-{\displaystyle \!\wr}
&
\mathscr{S}^{}_{Y|X,z^*_0}/\mathscr{N}^{}_{Y|X,z^*_0} \ar[d]
\\
C^{\,\mathbb{R},\chi}_{\smash{Y|X,z^*_0} }
\ar[r]^-{\hat{\sigma}^{\chi}}
&\mathfrak{S}^{}_{Y|X,z^*_0}/\mathfrak{N}^{}_{Y|X,z^*_0}\,. 
}
\]

Since the first down-arrow 
$M^{\mathbb{R},\chi}_{\smash{Y|X,z^*_0}} \earrow
C^{\,\mathbb{R},\chi}_{\smash{Y|X,z^*_0}}$
is isomorphic,
to show the theorem, it suffices to prove the last claim in the theorem.
We first consider a special case.
\begin{lem}{\label{lemma:replace-last}}
 Let $\chi_1 = \{f_1, \dots, f_d\}$,  
$\chi_2 = \{f_1, \dots, f_{d-1},  g\} \in \varXi(z^*_0)$.  Then the last claim in Theorem \ref{theorem:invariant-symbol-definition}
holds for these $\chi_1$ and $\chi_2$.
\end{lem}
\begin{proof}
Let $v_1 \in \varGamma(V^{\chi_1,(*)}_{\boldsymbol{\kappa}};\mathscr{O}_{X})$
and $v_2 \in \varGamma(V^{\chi_2,(*)}_{\boldsymbol{\kappa}};\mathscr{O}_{X})$ 
some ${\boldsymbol{\kappa}}$ 
which give 
the same element in $\mathscr{C}^{\mathbb{R}}_{Y|X,z^*_0}$.
Let us consider the coordinates transformation 
\[
w = (w',w'') = f(z) = (f_1(z),\dots,f_d(z), z_{d+1},\dots,z_n),
\]
and let $w_0 =  f(z_0) = 0$, $w^*_0 = (0;1,0,\dots,0)$.
Clearly the coordinates transformation changes
$\chi_1$ and $\chi_2$ to
$
\tilde{\chi}_1 = (w_1, \dots, w_d)$ and $\tilde{\chi}_2 = (w_1, \dots, w_{d-1}, h)$
with $h(w) = g \circ f^{\,-1}$ respectively.  Further, we have
\begin{equation}{\label{eq:coordinate_transformed_integral_1}}
\begin{aligned}
\sigma^{\chi_1}([v_1]) (w'',\zeta') & =
\smashoperator{\int\limits_{\hspace{5ex}\gamma(\eta_0;\varrho,\theta)}} v_1 (f^{\,-1}(w)) 
\,e^{-\langle f_{w''}^{-1}(w'), \zeta'\rangle} \det[\partial_{w'}f^{\,-1}_{w''}]\,dw', \\
\sigma^{\chi_2}([v_2]) (w'',\zeta')& =
\smashoperator{\int\limits_{\hspace{13ex}f_{w''}(\gamma^{\chi_2}(w'',\eta_0;\varrho,\theta))}} v_2 (f^{\,-1}(w)) 
\,e^{-\langle f_{w''}^{-1}(w'), \zeta'\rangle} \det[\partial_{w'}f^{\,-1}_{w''}]\,dw'.
\end{aligned}
\end{equation}
 It follows from definitions that $v_1(f^{\,-1}(w))$ 
and $v_2(f^{\,-1}(w))$
are holomorphic in $V^{(*)}_{\boldsymbol{\kappa}} = 
V^{\tilde{\chi}_1,(*)}_{\boldsymbol{\kappa}}$ and
$V^{\tilde{\chi}_2, (*)}_{\boldsymbol{\kappa}}$ respectively.  We can also see
\[
f_{w''}(\gamma^{\chi_2}(w'',\eta_0;\varrho,\theta)) = 
\gamma^{\tilde{\chi}_2}(w'',\eta_0;\varrho,\theta).
\]
Since $\tilde{\chi}_2$ belongs to $\varXi(w^*_0)$,  we have $\partial_{w_d} h(w_0) \ne 0$.
This implies that, 
by keeping $\eta_0$ of $\gamma_d(\eta_0; \rho)$ unchanged and by taking $\eta_0$
of other $\gamma_k$ ($1 \leqslant k \leqslant d-1$) so small if needed, the chain
$\gamma(\eta_0; \rho, \theta) \times \{w''\}$ belongs the common domain of
$V^{\tilde{\chi}_1,(*)}_{\boldsymbol{\kappa}}$ and
$V^{\tilde{\chi}_2, (*)}_{\boldsymbol{\kappa}}$ for a sufficiently small $w''$.
 In particular, we can replace 
$f_{w''}(\gamma^{\chi_2}(w'',\eta_0;\varrho,\theta))$ 
in \eqref{eq:coordinate_transformed_integral_1}
with
$\gamma(\eta_0;\varrho, \theta)$.
Therefore the problem can be reduced to the case where
$\chi_1 = \{z_1, \dots, z_d\}$, $\chi_2 = \{z_1, \dots, z_{d-1}, g\}$ and 
the morphisms $\sigma^{\chi^{}_1}$ and $\sigma^{\chi^{}_2}$ are replaced with the same morphism
\begin{equation}{\label{eq:reduced_integral}}
\sigma(v) (z'', \zeta'):= \smashoperator{\int\limits_{\hspace{5ex}\gamma(\eta_0; \varrho,\theta)}} v(z',z'')
\,e^{-\langle \varphi(z), \zeta'\rangle} dz'
\end{equation}
for some $\varphi(z)=(\varphi^{}_1(z),\dots,\varphi^{}_d(z))$ with $\{\varphi^{}_1(z),\dots,\varphi^{}_d(z)\}\in \varXi(z^*_0)$ 
$(z^*_0= (0;1,0,\dots,0))$.
Let us show the lemma for this case. We have the diagram
\begin{equation}{\label{eq:reduced_diagram}}
M^{\mathbb{R},\chi_1}_{\smash{Y|X,z^*_0}} 
\eleftarrow
\smashoperator[r]{\varinjlim_{\varrho,L,U}}H^d_{G^{\chi_1}_{\varrho,L}\cap
G^{\chi_2}_{\varrho,L}\cap U}(U;\mathscr{O}_{X})
\earrow
M^{\mathbb{R},\chi_2}_{\smash{Y|X,z^*_0}}. 
\end{equation}
The corresponding diagram of \v{C}ech representations is given by
\begin{align*}
\varinjlim\limits_{\boldsymbol{\kappa}} 
\varGamma(V^{\chi_1,(*)}_{\boldsymbol{\kappa}};\mathscr{O}_{X})\big/
\smashoperator{\Sum_{\alpha \in \mathcal{P}^\vee_d}}
\varGamma(V^{\chi_1,(\alpha)}_{\boldsymbol{\kappa}};\mathscr{O}_{X})
&\underset{\iota_1}{\eleftarrow}
\varinjlim_{\boldsymbol{\kappa}} 
Z^d(\mathfrak{T}^{}_{\boldsymbol{\kappa}};\mathscr{O}_{X})\big/
B^d(\mathfrak{T}^{}_{\boldsymbol{\kappa}};\mathscr{O}_{X}) \\
&\underset{\iota_2}{\earrow}
\varinjlim\limits_{\boldsymbol{\kappa}} 
\varGamma(V^{\chi_2,(*)}_{\boldsymbol{\kappa}};\mathscr{O}_{X})\big/
\smashoperator{\Sum_{\alpha \in \mathcal{P}^\vee_d}}
\varGamma(V^{\chi_2,(\alpha)}_{\boldsymbol{\kappa}};\mathscr{O}_{X}),
\end{align*}
where the covering $\mathfrak{T}^{}_{\boldsymbol{\kappa}}= \{T_{\boldsymbol{\kappa}}^{(i)}\}_{i=1}^{d+1}$ is given by
\[
T^{(i)}_{\boldsymbol{\kappa}} = V^{\chi_1,(i)}_{\boldsymbol{\kappa}}
\cap V^{\chi_2,(i)}_{\boldsymbol{\kappa}} \,\, (1 \leqslant i \leqslant d-1), \,\,\,
T^{(d)}_{\boldsymbol{\kappa}} = V^{\chi_1,(d)}_{\boldsymbol{\kappa}}
\cap U^{\chi_2}_{\boldsymbol{\kappa}}, \,\,\,
T^{(d+1)}_{\boldsymbol{\kappa}} = 
V^{\chi_2,(d)}_{\boldsymbol{\kappa}}
\cap U^{\chi_1}_{\boldsymbol{\kappa}}.
\]
We also note that,  for $v = \{v^{(\beta)}\} \in Z^d(\mathfrak{T}^{}_{\boldsymbol{\kappa}};\mathscr{O}_{X}) \subset
\smashoperator{\Sum_{\beta \in \varLambda_d}}
\varGamma(T^{(\beta)}_{\boldsymbol{\kappa}};\mathscr{O}_{X})$,
we have 
\[
v_1 := \iota_1(v) = v^{(\{1,\dots,d\})},\,
v_2 := \iota_2(v) = v^{(\{1,\dots,d-1,d+1\})}.
\]
Hence, to complete the proof, it suffices to show $\sigma([v_1]) = \sigma([v_2])$.
Since $v$ satisfies a cocycle condition, we have
\[
v_2 - v_1 = (-1)^d \smashoperator{\Sum_{1 \leqslant k < d}}\,(-1)^k v^{(*^{\vee k})}.
\]
By modifying the path of the integration, we obtain
$\sigma(v^{(\{2,\dots,d+1\})}) \in \mathscr{N}_{Y|X,z^*_0}$.
Furthermore we get $\sigma(v^{(*^{\vee k})}) = 0$ if
$2 \leqslant k < d$.  Hence
we have obtain that $\sigma(v_2) = \sigma(v_1) \in\mathscr{S}_{Y|X,z^*_0}\big/\mathscr{N}_{Y|X,z^*_0}$. 
\end{proof}
By repeated application of Lemma 
\ref{lemma:replace-last},
we can show the last claim of the theorem
for the case $\chi_1 = \{f_1, f_2, \dots, f_d\}$, $\chi_2 = \{f_1, g_2, \dots, g_d\} \in \varXi(z^*_0)$.
Hence the theorem immediately follows from the lemma below. This completes the proof.
\end{proof}
\begin{lem}
 Let $\chi_1 = \{f_1, f_2,\dots, f_d\}$, 
$\chi_2 = \{g, f_2, \dots, f_{d}\} \in \varXi(z^*_0)$.  Then the last claim in Theorem \ref{theorem:invariant-symbol-definition}
holds for these $\chi_1$ and $\chi_2$.
\end{lem}
\begin{proof}
By the same argument in that of the proof of Lemma \ref{lemma:replace-last} 
and by noticing Lemma \ref{lem:first_entry_preserve_esitimate},
the problem can be reduced to the case $\chi_1=(z_1, z_2,\dots, z_d)$
and $\chi_2 = (g, z_2, \dots, z_n)$
with the morphism $\sigma$ defined by \eqref{eq:reduced_integral}.
Then we have the diagram \eqref{eq:reduced_diagram} and the corresponding
one by \v{C}ech representations is given by
\begin{align*}
\varinjlim_{\boldsymbol{\kappa}} 
\varGamma(V^{\chi_1,(*)}_{\boldsymbol{\kappa}};\mathscr{O}_{X})\big/
\smashoperator{\Sum_{\alpha \in \mathcal{P}^\vee_d}}
\varGamma(V^{\chi_1,(\alpha)}_{\boldsymbol{\kappa}};\mathscr{O}_{X})
&\underset{\iota_1}{\eleftarrow}
\varinjlim_{\boldsymbol{\kappa}}
Z^{d}(\mathfrak{T}^2_{\boldsymbol{\kappa}};\mathscr{O}_{X})\big/
B^{d}(\mathfrak{T}^2_{\boldsymbol{\kappa}};\mathscr{O}_{X}) 
\\
&\underset{\iota_2}{\earrow}
\varinjlim\limits_{\boldsymbol{\kappa}} 
\varGamma(V^{\chi_2,(*)}_{\boldsymbol{\kappa}};\mathscr{O}_{X})\big/
\smashoperator{\Sum_{\alpha \in \mathcal{P}^\vee_d}}
\varGamma(V^{\chi_2,(\alpha)}_{\boldsymbol{\kappa}};\mathscr{O}_{X}),
\end{align*}
where the covering
$\mathfrak{T}^2_{\boldsymbol{\kappa}}=\{T^{(i,j)}_{\boldsymbol{\kappa}}\}_{i,j=1}^d$ is defined by
\[
T^{(i,j)}_{\boldsymbol{\kappa}} := 
V^{\chi_1,(i)}_{\boldsymbol{\kappa}} \cap
V^{\chi_2,(j)}_{\boldsymbol{\kappa}} \cap
U^{\chi_1}_{\boldsymbol{\kappa}} \cap
U^{\chi_2}_{\boldsymbol{\kappa}}.
\]
Furthermore the morphisms $\iota_1$ and $\iota_2$ are given by
\[
v_1 := \iota_1(v) = v^{(*,\emptyset)},\,
v_2 := \iota_2(v) = v^{(\emptyset,*)}
\]
respectively for $v = \{v^{(\alpha,\beta)}\} \in Z^{d}(\mathfrak{T}^2_{\boldsymbol{\kappa}};\mathscr{O}_{X})\subset
\smashoperator{\Sum_{(\alpha, \beta) \in \varLambda_{d}}}
\varGamma(T^{(\alpha,\beta)}_{\boldsymbol{\kappa}};\mathscr{O}_{X})$.
Then, by employing the same argument in the proof of Lemma \ref{lem:compat_act_cocycle}
in Appendix
\ref{ap:actions}, we have $\sigma(v_1) = \sigma(v_2) \in
\mathscr{S}_{Y|X,z^*_0}\big/\mathscr{N}_{Y|X,z^*_0}$.  The proof is complete.
\end{proof}

Now we compute behavior of a symbol by a coordinates transformation.
Let $z = (z',z'') = (\varPhi'(w),\varPhi''(w)) = \varPhi(w)$ be a local coordinates transformation near $w_0 \in Y$ with
$\varPhi(Y) \subset Y$ and $z_0 = \varPhi(w_0)$ where $Y$ is also defined by $w'= 0$ under the system of 
local coordinates $w = (w',w'')$.
We denote by $\varPhi'_{w''}(w')$ the
mapping $z' = \varPhi'(w',w'')$ regarded as a mapping of the variable 
$w'$ with a fixed $w''$.
Set $w^*_0 := (w_0;d\varPhi(w_0) (\zeta'_0, 0)) \in T^*_YX$.
Let $\chi = \{z_1, \dots, z_d\}$ and $[u] \in {C}^{\mathbb{R},\chi}_{\smash{Y|X,z^*_0}}$ with
$u(z,\eta) \in \varGamma(\widehat{V}^{\chi,(*)}_{\boldsymbol{\kappa}};\mathscr{O}^{}_{\WX})$
for some $\boldsymbol{\kappa}$.
Then we get $\widehat{\varPhi}^*([u]) = [u(\varPhi(w), \eta)] \in {C}^{\smash{\mathbb{R},\chi \circ \varPhi}}_{Y|X,w^*_0}$. Hence
we have obtained
\begin{align*}
\hat{\sigma}(\widehat{\varPhi}^*([u])) (w'',\lambda', \eta) &=
\smashoperator{\int\limits_{\hspace{10ex}\gamma^{\chi \circ \varPhi}(w'',\eta;\varrho,\theta)}}
u(\varPhi(w), \eta) 
\,e^{-\langle w', \lambda'\rangle} dw' \\
&=
\smashoperator{\int\limits_{\hspace{4ex}\gamma(\eta;\varrho,\theta)}}
u(z', \varPhi''(\varPhi_{w''}^{\prime -1}(z'), w''), \eta) 
\,e^{-\langle \varPhi_{w''}^{\prime -1}(z'), \lambda'\rangle} 
\det[\partial_{z'}\varPhi^{\prime -1}_{w''}]\,dz'.
\end{align*}

Let us consider the corresponding generalization of $E^{\mathbb{R}}_{X,z^*_0}$.
Hereafter we follow the same notations as those in  Section \ref{sec:actions_ER}.
Let $X$ be an $n$-dimensional complex manifold.
Set $X^2:= X \times X$ with a system of local coordinates $(z,w)$ and
$\widehat{ X}^2:= X^2 \times \mathbb{C}$ with local coordinates $(z,w,\eta)$.
Let $\varDelta\subset X^2$  be the diagonal set identified with $X$
and $z^*_0 = (z^{}_0;\zeta^{}_0) \in T^*X = T^*_\varDelta X^2$ with $\zeta^{}_0 \ne 0$. 
Let $\{f_1(z), \dots, f_n(z)\}$ be a sequence of holomorphic functions in an open neighborhood
of $z_0$ of $X$ satisfying the conditions:
\begin{enumerate}[(1)]
 \item $df_1(z_0) \wedge \dots \wedge df_n(z_0) \ne 0$.
 \item we have $f(z_0) = 0$ and 
${}^{\mathsf{t}}\!\bigl[\dfrac{\partial f}{\partial z}(z_0)\bigr]\boldsymbol{e} = 
\zeta_0 \in (T^*X)_{z_0}$ where
$f(z): = (f_1(z), \dots,  f_n(z))$ and $\boldsymbol{e} := (1,0,\dots,0) \in \mathbb{C}^n$.
\end{enumerate}
We denote by $\varXi_{\varDelta}(z^*_0)$ the set of such a sequence.
Let $\chi = \{f_1, \dots, f_n\} \in \varXi_{\varDelta}(z^*_0)$ and set $f_{\varDelta,i}(z,w) := f_i(z) - f_i(w)$,
\[
f_\varDelta(z,w) = (f_{\varDelta,1}(z, w), \dots,  f_{\varDelta,n}(z, w)) := (f_{\varDelta,1}(z,w) , f'_{\varDelta}(z,w)).
\]
Define, for an open neighborhood $U \subset X^2$ of $(z_0,z_0)$ and 
a closed convex cone $L \subset \mathbb{C}$ with $L \subset \{\tau \in \mathbb{C};\, \Re \tau > 0\} \cup \{0\}$, 
\[
G^{\chi}_{\varDelta,\varrho,L} := \{(z,w) \in X^2;\,
\varrho^2|f'_\varDelta(z,w)| \leqslant |f_{\varDelta,1}(z, w)|,\,
f_{\varDelta,1}(z, w) \in L\},
\]
and
\begin{align*}
\widehat{U}^{\chi}_{\varDelta,\varrho,r,\theta} &:= \{(z,w,\eta) \in U \times S_{r,\theta};\,
|f_{\varDelta,1}(z,w)| < \varrho |\eta|\}, \\
\widehat{G}^{\chi}_{\varDelta,\varrho,L} &:= \{(z,w,\eta) \in\widehat{ X}^2;\,
\varrho|f'_{\varDelta}(z,w)| \leqslant |\eta|,\,
f_{\varDelta,1}(z,w) \in L
\}.
\end{align*}
We also define
\begin{align*}
\widehat{E}^{\smash{\mathbb{R},\chi}}_{X,z^*_0} &:=
\smashoperator{\varinjlim_{\varrho,r,\theta,L,U}}
H^n_{\widehat{G}^{\chi}_{\varDelta,\varrho,L}\cap 
\widehat{U}^{\chi}_{\varDelta,\varrho,r,\theta}}\!
(\widehat{U}^{\chi}_{\varDelta,\varrho,r,\theta}; \OOO),\\
E^{\smash{\mathbb{R},\chi}}_{X,z^*_0} &:=
\Ker(\partial^{}_\eta\colon \widehat{E}^{\mathbb{R},\chi}_{X,z^*_0} 
\to \widehat{E}^{\mathbb{R},\chi}_{X,z^*_0}),
\\
M^{\smash{\mathbb{R},\chi}}_{X,z^*_0}
& := \smashoperator{\varinjlim_{\varrho,L,U}}
H^n_{G_{\varDelta,\varrho,L} \cap U}(U;\Oo).
\end{align*}
Then we obtain isomorphisms
\[
\mathscr{E}^{\mathbb{R}}_{X,z^*_0}
\eleftarrow
M^{\smash{\mathbb{R},\chi}}_{X,z^*_0}
\earrow
E^{\smash{\mathbb{R}, \chi}}_{X,z^*_0}\,.
\]
Hence, by the same reasoning as that for $\mathscr{C}^{\mathbb{R}}_{Y|X,z^*_0}$,
we can introduce the following definition.
\begin{defn}
 We denote by $E^{\mathbb{R}}_{X,z^*_0}$ 
(resp.\ $M^{\mathbb{R}}_{X,z^*_0}$) 
the isomorphism class 
$\{E^{\smash{\mathbb{R}, \chi}}_{X,z^*_0}\}_{\chi \in \Xi_{\varDelta}(z^*_0)}$
(resp.\ $\{M^{\smash{\mathbb{R}, \chi}}_{X,z^*_0}\}_{\chi \in \Xi_{\varDelta}(z^*_0)}$).
\end{defn}
We also give \v{C}ech representations of these cohomology groups.
Set
\begin{align*}
\widehat{U}^{\chi}_{\varDelta,\boldsymbol{\kappa}} & := \smashoperator{\Bcap_{i=2}^n}
\{(z,w,\eta) \in \widehat{X}^2;\, \rVert f(z) \lVert < r',\,\eta \in S_{r,\theta},\,|f_{\varDelta,1}(z,w)| < 
\varrho |\eta|,\,|f_{\varDelta,i}(z,w)| < r'\},
\\ 
U^{\chi}_{\varDelta,\boldsymbol{\kappa}} &:= \smashoperator{\Bcap_{i=2}^n}
\{(z,w) \in X^2;\, \rVert f(z) \lVert < r',\, |f_{\varDelta,1}(z,w)| < \varrho r,\,
 |f_{\varDelta,i}(z,w)| <r'\}.
\end{align*}
We also set
\begin{align*}
\widehat{V}_{\varDelta,\boldsymbol{\kappa}}^{\chi,(1)} &:= 
\{(z,w,\eta) \in\widehat{U}^{\chi}_{\varDelta,\boldsymbol{\kappa}};\,
\dfrac{\,\pi\,}{2} - \theta <  \arg f_{\varDelta,1}(z,w) < \dfrac{\,3\pi\,}{2} + \theta\}, \\
\widehat{V}^{\chi,(i)}_{\varDelta,\boldsymbol{\kappa}} &:= 
\{(z,w,\eta) \in \widehat{U}^{\chi}_{\varDelta,\boldsymbol{\kappa}};\,
\varrho|f_{\varDelta,i}(z,w)| > |\eta|\}\quad (2 \leqslant i \leqslant n),
\\
V_{\varDelta,\boldsymbol{\kappa}}^{\chi,(1)} &:= 
\{(z,w) \in U^{\chi}_{\varDelta,\boldsymbol{\kappa}};\,
\dfrac{\,\pi\,}{2} - \theta <  \arg f_{\varDelta,1}(z,w) < \dfrac{\,3\pi\,}{2} + \theta\},\\
V_{\varDelta,\boldsymbol{\kappa}}^{\chi,(i)} &:= 
\{(z,\,w) \in U^{\chi}_{\varDelta,\boldsymbol{\kappa}};\,
\varrho^{2}|f_{\varDelta,i}(z,w)| > |f_{\varDelta,1}(z,w)|\} \quad  (2 \leqslant i\leqslant n).
\end{align*}
Then we get
\begin{align*}
\widehat{E}^{\smash{\mathbb{R},\chi}}_{X} & =
\varinjlim_{\boldsymbol{\kappa}} 
\varGamma( \widehat{V}_{\varDelta,\boldsymbol{\kappa}}^{\chi,(*)}; 
\OOO)
\big/\smashoperator{\Sum_{\alpha \in \mathcal{P}^\vee_n}}
\varGamma( \widehat{V}^{\chi,(\alpha)}_{\varDelta,\boldsymbol{\kappa}}; \OOO), 
\\
{E}^{\smash{\mathbb{R},\chi}}_{X} &=
\varinjlim\limits_{\boldsymbol{\kappa}} 
\{K \in 
\varGamma( \widehat{V}_{\varDelta,\boldsymbol{\kappa}}^{\chi,(*)}; 
\OOO)
\big/\smashoperator{\Sum_{\alpha \in \mathcal{P}^\vee_n}}
\varGamma( \widehat{V}^{\chi,(\alpha)}_{\varDelta,\boldsymbol{\kappa}}; \OOO);\,
\partial_\eta K  = 0 \},
\\
{M}^{\smash{\mathbb{R},\chi}}_{X} &=
\varinjlim\limits_{\boldsymbol{\kappa}} 
\varGamma(V^{\chi,(*)}_{\varDelta,\boldsymbol{\kappa}}; \Oo)
\big/\smashoperator{\Sum_{\alpha \in \mathcal{P}^\vee_n}}
\varGamma( V_{\varDelta,\boldsymbol{\kappa}}^{\chi,(\alpha)}; \Oo).
\end{align*}
Let $\gamma(z,\eta;\varrho,\theta)$ be an $n$-dimension real chain in $\mathbb{C}^n$ defined in
 Section \ref{sec:actions_ER}. Then we define the $n$-dimensional real chain in $\mathbb{C}^n$ by
\[
 \gamma^{\chi}(z,\eta;\varrho,\theta) = f_{\varDelta,z}^{\,-1}(-\gamma(0,\eta;\varrho,\theta))
= f^{\,-1}(\gamma(f(z),\eta;\varrho,\theta)),
\]
where $f_{\varDelta,z}(w)$ is the morphism $f_\varDelta(z,w) = f(z) - f(w)$ regarded as a function of $w$ for a fixed $z$ 
and the orientation of $\gamma^\chi$ is induced from that of $\gamma$ by $f^{-1}$.
Then we can define the mapping  
$\hat{\sigma}^\chi \colon  {E}^{\smash{\mathbb{R},\chi}}_{X,z^*_0} \to
\mathfrak{S}^{}_{z^*_0}/\mathfrak{N}^{}_{z^*_0}$ 
by
\[
 \hat{\sigma}^\chi([Kdw])(z, \zeta, \eta):= 
\smashoperator{\int\limits_{\hspace{6ex}\gamma^\chi(z,\eta;\varrho,\theta)}} K(z,w,\eta) 
\,e^{\langle w - z , \zeta\rangle} dw
\]
for $K(z,w,\eta)\,dw \in \varGamma( \widehat{V}_{\varDelta,\boldsymbol{\kappa}}^{\chi,(*)};\, \OOO)$
with a suitable ${\boldsymbol{\kappa}}$.
Similarly we have the mapping 
$
M^{\mathbb{R},\chi}_{X,z^*_0} 
\to \mathscr{S}^{}_{z^*_0}/\mathscr{N}^{}_{z^*_0}$ by
\[
 \sigma^\chi([Kdw])(z, \zeta):= 
\smashoperator{\int\limits_{\hspace{7ex}\gamma^\chi(z,\eta_0;\varrho,\theta)}} K(z,w) 
\,e^{\langle w-z, \zeta\rangle} dw
\]
for $K(z,w)\,dw \in \varGamma(V^{\chi,(*)}_{\varDelta,\boldsymbol{\kappa}}; \Oo)$
with a suitable ${\boldsymbol{\kappa}}$ and a sufficiently small fixed $\eta_0 > 0$.
As an immediate consequence of Theorem {\ref{theorem:invariant-symbol-definition}}, we have obtained
the following corollary.
\begin{cor}
There exist  the well-defined symbol morphisms
\[\hat{\sigma} \colon {E}^{\mathbb{R}}_{X,z^*_0} \to
\mathfrak{S}^{}_{z^*_0}/\mathfrak{N}^{}_{z^*_0},
\qquad 
\sigma\colon M^{\smash{\mathbb{R}}}_{X,z^*_0} \to \mathscr{S}^{}_{z^*_0}/\mathscr{N}^{}_{z^*_0}\,,
\] 
induced by  $\hat{\sigma}^\chi$ and $\sigma^\chi$  
respectively.
\end{cor}
Let us consider a coordinates transformation.
Let $z = \varPhi(w)$ be a local coordinates transformation of $X$ near $w_0 \in X$ with
$z_0 = \varPhi(w_0)$. We take 
$(z,z',\eta)$ and $(w,w', \eta)$ 
as the corresponding systems of local coordinates of $\widehat{X}^2$ respectively and 
the associated local coordinates transformation $\widehat{\varPhi}$ of 
$\widehat{X}^2$ is defined by $(z,z',\eta) = (\varPhi(w), \varPhi(w'), \eta)$.
Set $w^*_0 := (w_0; d\varPhi(w_0) (\zeta_0)) \in T^*X$.
Let $\chi = \{z_1, \dots, z_n\}$ and $[Kdz'] \in {E}^{\smash{\mathbb{R},\chi}}_{X,z^*_0}$ with
$K(z,z',\eta)\,dz' \in \varGamma(\widehat{V}^{\chi,(*)}_{\varDelta,\boldsymbol{\kappa}}; \OOO)$
for some $\boldsymbol{\kappa}$.
Then, by the same argument as in $C^{\smash{\mathbb{R}, \chi}}_{Y|X,z^*_0}$, 
we get 
\[
\widehat{\varPhi}^*([Kdz']) = 
[\widehat{\varPhi}^*(Kdz')] = 
[K(\varPhi(w), \varPhi(w'), \eta)\,\det[\partial_{w'}\varPhi(w')]\, dw'] \in {E}^{\smash{\mathbb{R},\chi \circ \varPhi}}_{X,w^*_0}.
\]
Hence we have obtained
\begin{equation}\label{B4}
\begin{aligned}
\hat{\sigma}(\widehat{\varPhi}^*([Kdz'])) (w,\lambda, \eta) & =
\smashoperator{\int\limits_{\hspace{9ex}\gamma^{\chi \circ \varPhi}(w,\eta;\varrho,\theta)}}
\,e^{\langle w' - w, \lambda\rangle}
\widehat{\varPhi}^*(Kdz')
\\
&=
\smashoperator{\int\limits_{\hspace{6ex}\gamma(z,\eta;\varrho,\theta)}}
K(z, z', \eta) 
\,e^{\langle \varPhi^{-1}(z') - \varPhi^{-1}(z), \lambda \rangle} dz'.
\end{aligned}
\end{equation}

Finally we shall consider the action on $C^{\mathbb{R}}_{Y|X,z^*_0}$ associated with $E^{\mathbb{R}}_{X,z^*_0}$.
Let $z^*_0 = (z_0;\zeta_0) = (0, z''_0; \zeta'_0,0) \in \dot{T}^*_YX \subset \dot{T}^*X$,
$\chi_C \in \varXi(z^*_0)$ and $\chi_E \in \varXi_{\varDelta}(z^*_0)$.
Assume $\chi_C \subset \chi_E$; that is,  $\chi_C$ and $\chi_E$ are given by
$\{f_1, \dots, f_d\}$ and $\{f_1, \dots, f_d, \dots, f_n\}$ respectively.
Note that, for any $\chi_C \in \varXi(z^*_0)$, we can always find a $\chi_E \in \varXi_{\varDelta}(z^*_0)$ 
with $\chi_C \subset \chi_E$.
Let $[u] \in C^{\smash{\mathbb{R}, \chi_C}}_{Y|X,z^*_0}$ with $u(w,\eta) \in 
\varGamma(V^{\chi_C,(*)}_{\boldsymbol{\kappa}};\mathscr{O}_{X})$ and
$[Kdw] \in E^{\smash{\mathbb{R}, \chi_E}}_{X,z^*_0}$ with $K(z,w,\eta)\,dw \in 
\varGamma(V^{\chi_E,(*)}_{\varDelta,\boldsymbol{\kappa}};\OOO)$ for some $\boldsymbol{\kappa}$.
Then we have the morphism 
\[
\mu^{\chi_E}\colon
E^{\smash{\mathbb{R}, \chi_E}}_{X,z^*_0} \tens_{\mathbb{C}}
C^{\smash{\mathbb{R}, \chi_C}}_{Y|X,z^*_0} \ni [Kdw] \otimes [u] \to
[\smashoperator{\int\limits_{\hspace{8ex}\gamma^{\chi_E}(z,\eta;\varrho,\theta)}}K(z,w,\eta)\,u(w,\eta) \,dw] \in
C^{\smash{\mathbb{R}, \chi_C}}_{Y|X,z^*_0} , 
\]
which is well defined. Indeed, by the coordinates transformation
$(\tilde{z}, \tilde{w}) = (f(z), f(w))$, the situation can be reduced to one studied in  Section \ref{sec:actions_ER}.
\begin{thm}{\label{thm:generic-point-action}}
 The family $\{\mu^{\chi_E}\}_{\chi_E \in \varXi_{\varDelta}(z^*_0)}$ of morphisms constructed above 
induces the well-defined morphism $\mu \colon
E^{\mathbb{R}}_{X,z^*_0} \tens_{\mathbb{C}}C^{\mathbb{R}}_{Y|X,z^*_0} 
\to
C^{\mathbb{R}}_{Y|X,z^*_0}$.
Furthermore $\mu$ coincides with the action of
$\mathscr{E}^{\mathbb{R}}_{X,z^*_0}$ on 
$\mathscr{C}^{\mathbb{R}}_{Y|X,z^*_0}$.
\end{thm}
\begin{proof}
 It suffices to show that, for $\chi_C \subset \chi_E$, the following diagram commutes:
\[
\xymatrix @R=3ex{
 \mathscr{E}^{\mathbb{R}}_{X,z^*_0}
\tens_{\mathbb{C}}
\mathscr{C}^{\mathbb{R}}_{Y|X,z^*_0} \ar[r]^-{\mu^c } \ar[d]
&\mathscr{C}^{\mathbb{R}}_{Y|X,z^*_0}  \ar[d]
\\
E^{\smash{\mathbb{R}, \chi_E}}_{X,z^*_0}
\tens_{\mathbb{C}}
C^{\smash{\mathbb{R}, \chi_C}}_{Y|X,z^*_0} \ar[r]^-{\mu^{\chi^{}_E}} &
C^{\smash{\mathbb{R}, \chi_C}}_{Y|X,z^*_0}.
}
\]
We denote by $\iota$ the isomorphism
$\mathscr{E}^{\mathbb{R}}_{X,z^*_0} 
\earrow
E^{\smash{\mathbb{R}, \chi_E}}_{X,z^*_0}$ and by the same symbol the one
$\mathscr{C}^{\mathbb{R}}_{Y|X,z^*_0} 
\earrow
C^{\smash{\mathbb{R}, \chi_C}}_{Y|X,z^*_0}$.
Let $u(w,\eta) \in 
\varGamma(V^{\chi_C,(*)}_{\boldsymbol{\kappa}};\mathscr{O}_{X})$ and
$K(z,w,\eta)dw \in 
\varGamma(V^{\chi_E,(*)}_{\varDelta,\boldsymbol{\kappa}};\OOO)$ for some $\boldsymbol{\kappa}$.
We define the coordinates transformations
\begin{alignat*}{4}
\varPhi(z) &= \varPhi_1(z) = f^{\,-1}(\,\widetilde{z}\,) ,\quad &
\varPhi_2(z,w) &= (f^{\,-1}(\,\widetilde{z}\,), f^{\,-1}(\widetilde{w})),
\\
\widehat{\varPhi}_1(z,\eta)& = (f^{\,-1}(\,\widetilde{z}\,), \eta) ,\quad&
\widehat{\varPhi}_2(z,w,\eta)& = (f^{\,-1}(\,\widetilde{z}\,), f^{\,-1}(\widetilde{w}), \eta).
\end{alignat*}
It follows from 
the fact $\chi_E \circ \varPhi = \{\,\widetilde{z}_1, \dots, \widetilde{z}_n\}$ and
Theorem \ref{th:action-commutative} in  Section \ref{sec:actions_ER} that we have
\[
\mu^c(\iota^{-1}(\widehat{\varPhi}_2^*[Kdw]) \otimes
\iota^{-1}(\widehat{\varPhi}_1^*[u]))
=
\iota^{-1} \circ \mu^{\chi_E \circ \varPhi}(\widehat{\varPhi}_2^{*}([Kdw]) \otimes \widehat{\varPhi}_1^{*}([u])).
\]
By the coordinates transformation law of the integration, we get
\[
\mu^{\chi_E \circ \varPhi}(\widehat{\varPhi}^{*}_2([Kdw]) \otimes \widehat{\varPhi}_1^{*}([u]))
=
\widehat{\varPhi}_1^* \circ \mu^{\chi_E}([Kdw] \otimes [u]).
\]
Furthermore, it follows from functorial properties that
$\iota^{-1} \circ \widehat{\varPhi}_k^{*} = \varPhi_k^* \circ \iota^{-1}$ ($k=1,2$) and $\mu^c$ and $\Psi^*$ commute.
Hence we have obtained
\[
\varPhi_1^*\circ\mu^c(\iota^{-1}([Kdw]) \otimes
\iota^{-1}([u]))
=
\varPhi_1^* \circ \iota^{-1} \circ \mu^{\chi_E}([Kdw] \otimes [u]),
\]
which implies
$
\mu^c(\iota^{-1}([Kdw]) \otimes
\iota^{-1}([u]))
=
\iota^{-1} \circ \mu^{\chi_E}([Kdw] \otimes [u])$.
This completes the proof.
\end{proof}

\end{document}